\documentclass[12pt,a4paper,french,english,PSfrag,BCOR8mm, most, usenames, dvipsnames]{article}
\usepackage[T1]{fontenc}
\usepackage[latin9]{inputenc}
\usepackage{color}
\usepackage{babel}
\usepackage{float}
\usepackage{mathrsfs}
\usepackage{tcolorbox}
\usepackage{amsmath}
\usepackage{amsthm}
\usepackage{amssymb}
\usepackage{graphicx}
\usepackage[all]{xy}
\usepackage[pdfusetitle,
 bookmarks=true,bookmarksnumbered=false,bookmarksopen=false,
 breaklinks=false,pdfborder={0 0 1},backref=false,colorlinks=false]
 {hyperref}

\makeatletter

\pdfpageheight\paperheight
\pdfpagewidth\paperwidth

\newcommand{\noun}[1]{\textsc{#1}}

\numberwithin{equation}{section}
\numberwithin{figure}{section}
\numberwithin{table}{section}
\theoremstyle{plain}
\newtheorem{thm}{\protect\theoremname}[section]
\theoremstyle{remark}
\newtheorem{rem}[thm]{\protect\remarkname}
\theoremstyle{definition}
\newtheorem{defn}[thm]{\protect\definitionname}
\theoremstyle{plain}
\newtheorem{lem}[thm]{\protect\lemmaname}
\theoremstyle{plain}
\newtheorem{prop}[thm]{\protect\propositionname}
\theoremstyle{plain}
\newtheorem{cor}[thm]{\protect\corollaryname}

\usepackage[a4paper, total={17cm, 27cm}]{geometry}
\usepackage{color}
\usepackage{graphics}

\usepackage{stmaryrd}

\makeatother

\addto\captionsenglish{\renewcommand{\corollaryname}{Corollary}}
\addto\captionsenglish{\renewcommand{\definitionname}{Definition}}
\addto\captionsenglish{\renewcommand{\lemmaname}{Lemma}}
\addto\captionsenglish{\renewcommand{\propositionname}{Proposition}}
\addto\captionsenglish{\renewcommand{\remarkname}{Remark}}
\addto\captionsenglish{\renewcommand{\theoremname}{Theorem}}
\addto\captionsfrench{\renewcommand{\corollaryname}{Corollaire}}
\addto\captionsfrench{\renewcommand{\definitionname}{Définition}}
\addto\captionsfrench{\renewcommand{\lemmaname}{Lemme}}
\addto\captionsfrench{\renewcommand{\propositionname}{Proposition}}
\addto\captionsfrench{\renewcommand{\remarkname}{Remarque}}
\addto\captionsfrench{\renewcommand{\theoremname}{Théorème}}
\providecommand{\corollaryname}{Corollary}
\providecommand{\definitionname}{Definition}
\providecommand{\lemmaname}{Lemma}
\providecommand{\propositionname}{Proposition}
\providecommand{\remarkname}{Remark}
\providecommand{\theoremname}{Theorem}

\begin{document}
\selectlanguage{english}
\title{From geodesic flow to wave dynamics on hyperbolic surfaces}
\author{\href{https://www-fourier.univ-grenoble-alpes.fr/~faure/}{Frédéric Faure}{\small}\\
{\small Institut Fourier, UMR 5582, Laboratoire de Mathématiques}\\
{\small Université Grenoble Alpes, CS 40700, 38058 Grenoble cedex 9,
France}\\
{\small\href{mailto:frederic.faure@univ-grenoble-alpes.fr}{frederic.faure@univ-grenoble-alpes.fr}}}
\maketitle
\begin{abstract}
We study the geodesic flow on the unit cotangent bundle $M=S^{*}\mathcal{N}$
of a closed hyperbolic surface $\mathcal{N}$, using the representation
theory of $SL_{2}(\mathbb{R})$. We construct explicit $X$-adapted
Hilbert spaces, obtained by completing propagated dense domains of
$L^{2}(M)$, which are tailored to the spectral analysis of the geodesic
generator $X$. In these spaces, $X$ becomes a normal operator with
discrete spectrum, except at the threshold $\mu=1/4$, where Jordan
blocks of size two may occur.

In this Hilbert model, the propagator $e^{tX}$ factorizes into a
damped harmonic oscillator with eigenvalues $e^{-t(n+1/2)}$, $n\in\mathbb{N}$,
and a transverse part involving the shifted wave group $e^{\pm it\sqrt{\Delta-1/4}}$
on $\mathcal{N}$, together with the holomorphic and anti-holomorphic
discrete series.

The model clarifies two classical links between geodesic dynamics
and the Laplace spectrum. Comparing the spectral trace of the propagator
in the $X$-adapted Hilbert model with the Atiyah--Bott--Guillemin
flat trace gives a dynamical form of the Selberg trace formula: closed
geodesics arise from the flat trace, while the spectral side comes
from the explicit $SL_{2}(\mathbb{R})$-decomposition. The same factorization
also explains the large-time structure of spherical mean operators
on $\mathcal{N}$: after the natural $e^{t/2}$-renormalization and
the removal of a finite-rank low-energy part, the shifted wave equation
on $\mathcal{N}$ emerges as the leading effective dynamics. Thus
the construction provides an explicit Hilbert-space structure relating
classical geodesic dynamics, Ruelle resonances, and the spectral theory
of the surface. 
\end{abstract}
\newpage

\selectlanguage{french}%

\newtcolorbox{cBoxA}[1][]{ enhanced, breakable, frame style=purple!80, interior style=red!0, #1 }

\newtcolorbox{cBoxB}[2][]{ enhanced, breakable, frame style=teal!80, interior style=cyan!0, #1, #2 }

\selectlanguage{english}
\global\long\def\eq#1{\underset{(#1)}{=}}%
\global\long\def\ineq#1{\underset{(#1)}{\leq}}%
\global\long\def\eqq#1{\underset{(#1)}{\approx}}%

\tableofcontents{}


\medskip{}


\newpage{}

\section{Introduction}

\subsection{Main result}

Let $\mathcal{N}$ be a closed hyperbolic surface, and let $M=S^{*}\mathcal{N}$
be its unit cotangent bundle. We denote by $X$ the geodesic vector
field on $M$, viewed as a first-order differential operator on $C^{\infty}(M)$,
and by $\Delta$ the positive Laplace operator on $\mathcal{N}$.
The main result of this paper, Theorem~\ref{thm:main_result}, gives
an explicit Hilbert-space model for the pull-back operator $e^{tX}$.
In this model, $e^{tX}$ factorizes as the tensor product of a universal
damped harmonic oscillator $e^{-t(A+\frac{1}{2})}$, where $Ae_{n}=ne_{n}$
on $\ell^{2}(\mathbb{N})$, and a transverse part involving the shifted
wave propagators $e^{\pm it\sqrt{\Delta-\frac{1}{4}}}$ on $\mathcal{N}$,
together with the holomorphic and anti-holomorphic discrete series.

We introduce the notation used in the main theorem. We denote $K:=\mathrm{SO}_{2}\left(\mathbb{R}\right)$
be the maximal compact subgroup of $SL_{2}(\mathbb{R})$. Let $\mathcal{P}(M)$,
defined precisely in \eqref{eq:def_calP_M}, be the algebraic core
of $L^{2}(M)$ consisting of finite sums of $K$-finite vectors belonging
to finitely many irreducible components of the $SL_{2}(\mathbb{R})$-decomposition.
$\mathcal{P}(M)$ is dense in $L^{2}(M)$.

Let $\mathcal{H}_{\mathrm{hol}}$ be the multiplicity space of the
discrete series: 
\[
\mathcal{H}_{\mathrm{hol}}=\bigoplus_{q\ge1}\left(\mathcal{H}_{\mathrm{hol},+}(2q)\oplus\mathcal{H}_{\mathrm{hol},-}(-2q)\right).
\]
Here 
\[
\mathcal{H}_{\mathrm{hol},+}(2q)\simeq H^{0}(\mathcal{N},K_{\mathcal{N}}^{\otimes q})
\]
is the finite dimensional space of holomorphic $q$-differentials
on $\mathcal{N}$, $K_{\mathcal{N}}$ denotes the canonical line bundle,
and 
\[
\mathcal{H}_{\mathrm{hol},-}(-2q)\simeq\overline{\mathcal{H}_{\mathrm{hol},+}(2q)}
\]
is the corresponding anti-holomorphic space. We define the operator
$\Lambda_{\mathrm{d.s.}}$ on $\mathcal{H}_{\mathrm{hol}}$ by 
\[
\Lambda_{\mathrm{d.s.}}v=\left(q-\frac{1}{2}\right)v,\qquad v\in\mathcal{H}_{\mathrm{hol},+}(2q)\oplus\mathcal{H}_{\mathrm{hol},-}(-2q).
\]

We decompose 
\[
L^{2}(\mathcal{N})=L_{0}^{2}(\mathcal{N})\overset{\perp}{\oplus}\mathbb{C}\boldsymbol{1},
\]
where $L_{0}^{2}(\mathcal{N})$ is the space of zero-average functions.
We further split 
\begin{equation}
L_{0}^{2}(\mathcal{N})=E_{\neq1/4}\overset{\perp}{\oplus}E_{1/4},\label{eq:decomp_L0N}
\end{equation}
where 
\[
E_{1/4}:=\ker\left(\Delta-\frac{1}{4}\right)
\]
is the possible threshold eigenspace of the Laplacian $\Delta$ on
$\mathcal{N}$ that is finite-dimensional and may be trivial, and
\[
E_{\neq1/4}:=\mathbf{1}_{(0,\infty)\setminus\{1/4\}}(\Delta)L_{0}^{2}(\mathcal{N}).
\]

On $E_{\neq1/4}$, the operator $\sqrt{\Delta-\frac{1}{4}}$ is defined
by spectral calculus with the convention 
\[
\sqrt{\mu-\frac{1}{4}}=i\sqrt{\frac{1}{4}-\mu}\qquad\text{for }0<\mu<\frac{1}{4}.
\]

\begin{cBoxB}{}

\begin{thm}[Main result]
\label{thm:main_result} Define the Hilbert space 
\[
\mathcal{K}:=\left(\ell^{2}(\mathbb{N})\otimes\mathcal{B}\right)\overset{\perp}{\oplus}\mathbb{C}\boldsymbol{1},
\]
with 
\begin{equation}
\mathcal{B}:=\left(L_{0}^{2}(\mathcal{N})\otimes\mathbb{C}^{2}\right)\overset{\perp}{\oplus}\mathcal{H}_{\mathrm{hol}}.\label{eq:def_cal_B}
\end{equation}
For every $\tau>0$, the propagated domain 
\[
\mathcal{D}_{\tau}:=e^{\tau X}\mathcal{P}(M)
\]
is dense in $L^{2}\left(M\right)$ and can be completed for a Hilbert
norm into a Hilbert space $\mathcal{H}_{\tau}$ such that there exists
a unitary isomorphism 
\begin{equation}
\mathbb{T}:\mathcal{H}_{\tau}\xrightarrow{\sim}\mathcal{K},\label{eq:TT}
\end{equation}
and for every $t\ge0$, 
\begin{equation}
\mathbb{T}e^{tX}\mathbb{T}^{-1}=\left(e^{-t(A+\frac{1}{2})}\otimes e^{tX_{\mathcal{B}}}\right)\oplus\mathrm{Id}_{\mathbb{C}\boldsymbol{1}},\label{eq:result_X}
\end{equation}
where, with respect to the decompositions (\ref{eq:def_cal_B}) and
\eqref{eq:decomp_L0N}, 
\begin{equation}
e^{tX_{\mathcal{B}}}=\left(\begin{array}{cc}
e^{it\sqrt{\Delta-\frac{1}{4}}} & 0\\
0 & e^{-it\sqrt{\Delta-\frac{1}{4}}}
\end{array}\right)_{E_{\neq1/4}}\oplus\left(\begin{array}{cc}
1 & t\\
0 & 1
\end{array}\right)_{E_{1/4}}\oplus e^{-t\Lambda_{\mathrm{d.s.}}}.\label{eq:exp_t_XB}
\end{equation}

Consequently, on the non-threshold part $E_{\neq1/4}$ and on the
discrete series, the operator $X$ is normal in $\mathcal{H}_{\tau}$,
with discrete spectrum
\begin{align}
\mathrm{Spec}_{\mathcal{H}_{\tau}}(X) & =\left\{ -n-\frac{1}{2}\pm i\sqrt{\mu-\frac{1}{4}}\ ;\ n\in\mathbb{N},\ \mu\in\mathrm{Spec}(\Delta_{/L_{0}^{2}(\mathcal{N})}),\ \mu\neq\frac{1}{4}\right\} \label{eq:Spec_X}\\
 & \quad\cup\left\{ -n-q\ ;\ n\in\mathbb{N},\ q\in\mathbb{N}^{*}\right\} \cup\{0\}.\nonumber 
\end{align}
At the threshold $\mu=\frac{1}{4}$, the two spherical branches coalesce
into Jordan blocks of size $2$. Multiplicities of $\left(-n-q\right)$
are described in Theorem~\ref{thm:global_Hilbert_model_XUS}.
\end{thm}

\end{cBoxB}

Theorem~\ref{thm:main_result} is a corollary of the more detailed
global statement, Theorem~\ref{thm:global_Hilbert_model_XUS}. The
latter also gives the action of the other generators $U$ and $S$
of $\mathfrak{sl}_{2}(\mathbb{R})$. In the Hilbert model, these operators
become weighted creation and annihilation operators, coupled to the
transverse spectral parameters. Thus the construction does not merely
diagonalize $X$, but realizes the whole infinitesimal $\mathfrak{sl}_{2}$-action
in an explicit oscillator model.

Figure~\ref{fig:Pollicott=002013Ruelle-spectrum-} illustrates Theorem~\ref{thm:main_result}.

\begin{figure}[h]

\begin{centering}
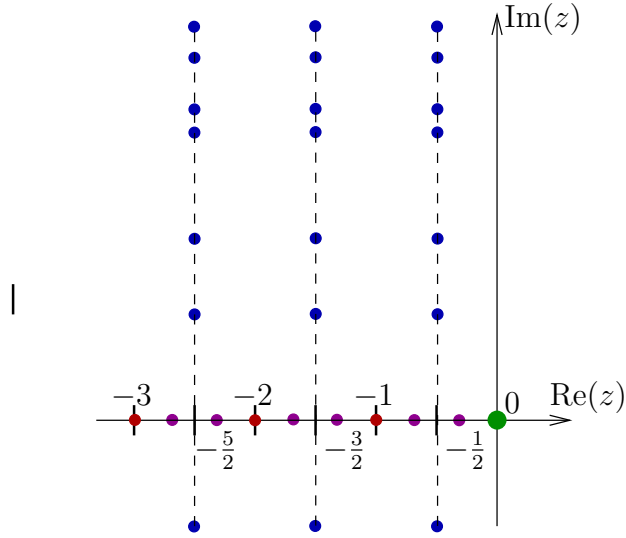
\par\end{centering}
\caption{\protect\label{fig:Pollicott=002013Ruelle-spectrum-}Pollicott--Ruelle
spectrum $\mathrm{Spec}_{\mathcal{H}_{\tau}}(X)$ of the geodesic
vector field $X$, as described in \eqref{eq:Spec_X}. The trivial
resonance $0$ is shown in green. The red points come from the discrete
series. The blue points correspond to the spherical principal series,
for $\mu\protect\geq1/4$, whereas the magenta points correspond to
the spherical complementary series, for $0<\mu<1/4$.}

\end{figure}

\begin{rem}
The three terms in \eqref{eq:exp_t_XB} correspond respectively to
the non-threshold spherical part, the threshold spherical part, and
the holomorphic/anti-holomorphic discrete series. The Jordan block
is present only when $1/4$ is an eigenvalue of $\Delta$. 

On the principal spectrum $\mu>1/4$, the first term in (\ref{eq:exp_t_XB})
is the shifted wave group. On the complementary spectrum $0<\mu<1/4$,
the same spectral-calculus formula gives the corresponding real exponential
branches, according to the convention for $\sqrt{\Delta-\frac{1}{4}}$
stated above.
\end{rem}

$\quad$
\begin{rem}
The only non-normal part of the model occurs at the threshold $\mu=\frac{1}{4}$.
The possible threshold Jordan blocks are consistent with the description
of Ruelle resonant states in Guillarmou--Hilgert--Weich~\cite[Theorem~3.3 and Proposition~1.3]{guillarmou_weich_resonances_16}.
We define 
\[
\operatorname{Spec}_{J}(X):=\begin{cases}
\left\{ -n-\frac{1}{2}\ ;\ n\in\mathbb{N}\right\} , & \text{if }E_{1/4}\neq\{0\},\\
\varnothing, & \text{if }E_{1/4}=\{0\}.
\end{cases}
\]
On the normal part of the model, the resolvent satisfies the usual
normal estimate. On each threshold Jordan block, one has 
\[
\left(z-\begin{pmatrix}z_{n} & 1\\
0 & z_{n}
\end{pmatrix}\right)^{-1}=(z-z_{n})^{-1}I+(z-z_{n})^{-2}\begin{pmatrix}0 & 1\\
0 & 0
\end{pmatrix},\qquad z_{n}=-n-\frac{1}{2}.
\]
Consequently, for $z\in\mathbb{C}\setminus\operatorname{Spec}_{\mathcal{H}_{\tau}}(X)$,
\[
\|(z-X)^{-1}\|_{\mathcal{H}_{\tau}}\le\frac{1}{\operatorname{dist}\!\left(z,\operatorname{Spec}_{\mathcal{H}_{\tau}}(X)\right)}+\frac{1}{\operatorname{dist}\!\left(z,\operatorname{Spec}_{J}(X)\right)^{2}}.
\]
If $E_{1/4}=\{0\}$, the second term is absent. Thus the only pseudospectral
effect is the elementary one produced by the explicit size-two Jordan
blocks at the threshold. 
\end{rem}

$\quad$
\begin{rem}
Theorem~\ref{thm:main_result} gives a Hilbert-space realization
of the geodesic generator $X$ adapted to dynamical questions. In
the usual $L^{2}(M)$ space, $X$ is skew-adjoint with essential spectrum
on $i\mathbb{R}$ and the Pollicott--Ruelle spectrum is not visible
as an ordinary Hilbert-space spectrum. In the spaces $\mathcal{H}_{\tau}$
constructed here, the same generator becomes a normal operator with
discrete spectrum (away from the threshold $\mu=1/4$), while the
threshold contribution is described by explicit Jordan blocks of size
two. Thus the pseudospectral behaviour is completely controlled by
the explicit size-two Jordan blocks at the threshold.

The discrete spectrum of $X$ given in (\ref{eq:Spec_X}) is the Pollicott--Ruelle
spectrum and has already been described in \cite{dyatlov_faure_guillarmou_2014,kuster2017quantum,guillarmou_weich_resonances_16,hilgert2021higher,bonthonneau2026spectrumanosovrepresentations}.
The possible Jordan blocks at the threshold are also consistent with
the description of invariant distributions by Flaminio--Forni~\cite{flaminio_forni_2003}
and with the analysis of Ruelle resonant states in Guillarmou--Hilgert--Weich~\cite[Theorem~3.3 and Proposition~1.3]{guillarmou_weich_resonances_16}.

This result also suggests a broader program: to construct analogous
Hilbert realizations for compact quotients of more general semisimple
groups, adapted to non-elliptic elements of the Lie algebra. Such
models could be useful in particular for questions related to quantum
chaos, where one seeks precise links between classical hyperbolic
dynamics, representation theory, and spectral theory.

Some manifestations of this structure are already present in classical
formulas, notably the Selberg trace formula and the Harish--Chandra
formula for spherical functions for which we give some results in
section \ref{subsec:Spherical-mean-operators}. In the last part of
the paper we revisit these formulas from the point of view of Theorem~\ref{thm:main_result}.
In particular, comparing the spectral trace in the Hilbert model with
the Atiyah--Bott--Guillemin flat trace of the geodesic flow recovers
the Selberg trace formula. 
\end{rem}

\subsection{Context and related works}

The initial motivation for this paper is the study of Pollicott--Ruelle
resonances, originating in the works of Ruelle~\cite{Ruelle_76}
and Pollicott~\cite{pollicott1986meromorphic}. A powerful functional-analytic
approach consists in realizing the generator on anisotropic spaces,
developed in various forms in \cite{liverani_02,baladi_05,fred-roy-sjostrand-07,fred_flow_09,dyatlov_Ruelle_resonances_2012,dyatlov2019mathematical,dang2020spectralI,dyatlov_guillarmou_2014}.

In the case of hyperbolic surfaces, and more generally homogeneous
manifolds, the discrete Pollicott--Ruelle spectrum has been studied
using algebraic Lie group techniques in \cite{dyatlov_faure_guillarmou_2014,kuster2017quantum,guillarmou_weich_resonances_16,hilgert2021higher,bonthonneau2026spectrumanosovrepresentations}.
The purpose of the present paper is different: we construct explicit
$X$-adapted Hilbert models in which this spectrum appears as an ordinary
Hilbert-space spectrum, with $X$ normal except for the explicit threshold
Jordan blocks. Thus we do not aim at rediscovering the Pollicott--Ruelle
spectrum itself, but rather at realizing it in explicit Hilbert spaces
where the generator $X$ and the full infinitesimal $\mathfrak{sl}_{2}(\mathbb{R})$-action
take an oscillator form.

Another motivation was to reinterpret, from the point of view of spectral
theory and microlocal analysis, the asymptotic results of Ratner~\cite{ratner_87}
and the invariant-distribution structure described by Flaminio--Forni~\cite{flaminio_forni_2003}.

A related Hilbert-space construction was carried out for a similar
model, namely the prequantum cat map~\cite{fred-PreQ-06}, also called
a contact extension of a hyperbolic automorphism of the torus. In
that model, the transfer operator has discrete spectrum and becomes
normal in the constructed Hilbert space. The key mechanism for the
construction is \cite[Lemma~14, p.~279]{fred-PreQ-06}, and it reappears
in the present paper in Theorem~\ref{thm:For-,-the}.

A closely related point of view was developed by Anantharaman--Zelditch~\cite{anantharaman2012intertwining},
who constructed an explicit intertwining operator between the geodesic
flow and the Schrödinger group on compact hyperbolic surfaces, acting
on suitable Hilbert spaces of symbols. Their construction is different
from the $X$-adapted Hilbert models constructed here, but it is very
close in spirit to the relation between geodesic dynamics and quantum
or wave propagation.

In the second part of the paper, in Section~\ref{sec:Atiyah=002013Bott-and-Selberg},
the derivation of the Selberg trace formula is close in spirit to
Guillemin's viewpoint in~\cite{guillemin1977lectures}: the formula
is obtained by comparing a dynamical Atiyah--Bott--Guillemin flat
trace formula with a spectral decomposition of the regular representation.
The difference is that we do not use the polarized complex appearing
in~\cite{guillemin1977lectures}; instead, we use the constructed
Hilbert spaces $\mathcal{H}_{\tau}$ adapted to the generator $X$.
On each irreducible component, the operator $e^{tX}$ is then trace
class for $t>0$ in the corresponding $X$-adapted Hilbert model,
whereas the geometric side remains the flat trace of the geodesic
flow. A related viewpoint is also discussed in Anantharaman's lectures
at the Collège de France \cite[around 38 min]{Anantharaman2023SpectresGraphesSurfaces4}.

\subsection{\protect\protect\label{subsec:Spherical-mean-operators} Spherical
mean operators and emergence of the wave equation}

In Section~\ref{sec:Spherical-averages}, we study a family of operators
$\left(\mathcal{L}_{t}\right)_{t\in\mathbb{R}}$ on $L^{2}(\mathcal{N})$
called spherical mean operators. Spherical means are classical objects
in the analysis of wave propagation and Huygens' principle, going
back to the works of Huygens~\cite{Huygens1690}, Poisson~\cite{Poisson1816},
and Hadamard~\cite{Hadamard1923}. In Euclidean space, their role
in the analysis of the wave equation is classical; see for instance
John~\cite{John1955}. On symmetric and hyperbolic spaces, the relation
between spherical analysis, shifted wave equations and Huygens' principle
has been studied by Helgason~\cite{Helgason1992}, Lax--Phillips~\cite{LaxPhillips1979},
{Ó}lafsson--Schlichtkrull~\cite{OlafssonSchlichtkrull1992}, and
Anker--Pierfelice--Vallarino~\cite{AnkerPierfeliceVallarino2012}.
For asymptotic expansions of large hyperbolic circles on compact hyperbolic
surfaces, see Corso--Ravotti~\cite{CorsoRavotti2025}.

The point emphasized here is that, in constant negative curvature,
the leading oscillatory part of the large-radius spherical mean operator
is governed, after the universal factor $e^{-t/2}$, by the shifted
wave group $e^{\pm it\sqrt{\Delta-1/4}}.$

Let 
\[
\pi:S^{*}\mathcal{N}\longrightarrow\mathcal{N}
\]
be the canonical projection. We denote by 
\[
\pi^{\circ}:L^{2}(\mathcal{N})\longrightarrow L^{2}(S^{*}\mathcal{N}),\qquad\pi^{\circ}u:=u\circ\pi,
\]
the pullback map. Its $L^{2}$-adjoint is the fiberwise average 
\[
((\pi^{\circ})^{\dagger}v)(q)=\int_{S_{q}^{*}\mathcal{N}}v(q,p)\,d\omega_{q}(p),
\]
where $d\omega_{q}$ is the normalized measure on the fiber $S_{q}^{*}\mathcal{N}$.

\begin{cBoxA}{}

\begin{defn}[Spherical mean operator]
For $t\in\mathbb{R}$, we define 
\begin{equation}
\mathcal{L}_{t}:=(\pi^{\circ})^{\dagger}e^{tX}\pi^{\circ}\qquad:L^{2}(\mathcal{N})\longrightarrow L^{2}(\mathcal{N}).\label{eq:def_L_t}
\end{equation}
\end{defn}

\end{cBoxA}

\begin{rem}
Equivalently, 
\[
(\mathcal{L}_{t}u)(q)=\int_{S_{q}^{*}\mathcal{N}}u\bigl(\pi(\varphi^{t}(q,p))\bigr)\,d\omega_{q}(p).
\]
Thus $\left(\mathcal{L}_{t}u\right)(q)$ is the angular average of
$u$ over the geodesic circle of center $q$ and radius $|t|$, counted
with the measure induced from the sphere of directions. See Figure
\ref{fig:spherical_mean}. In particular, 
\[
\mathcal{L}_{t}^{\dagger}=\mathcal{L}_{-t}=\mathcal{L}_{t},
\]
but the family $(\mathcal{L}_{t})_{t\in\mathbb{R}}$ is not a group.

\begin{figure}[h]
\begin{centering}
{\small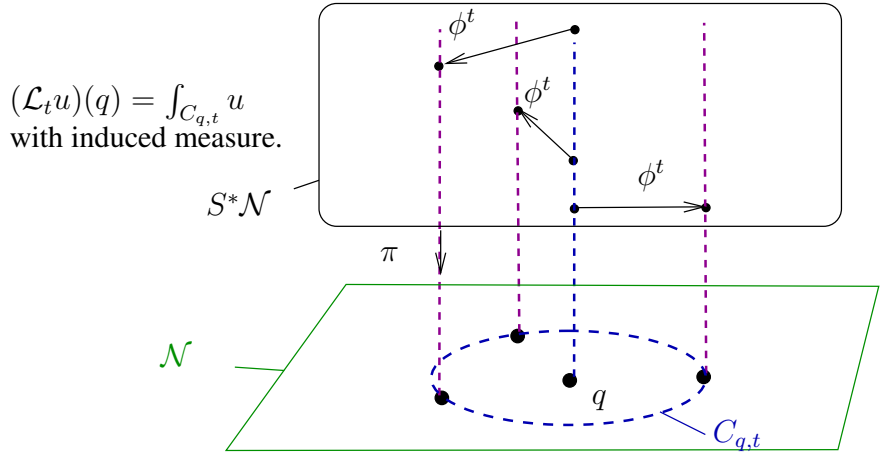}{\small\par}
\par\end{centering}
\caption{\protect\label{fig:spherical_mean} Spherical mean defined in (\ref{eq:def_L_t}).}
\end{figure}
\end{rem}

A theorem of Ratner \cite{ratner_87} implies that, as $t\to+\infty$,
\begin{equation}
\mathcal{L}_{t}=F_{t}^{\mathrm{Rat}}+O_{L^{2}\to L^{2}}(e^{-t/2}),\label{eq:decay_ratner}
\end{equation}
where $F_{t}^{\mathrm{Rat}}$ is a finite-rank operator with uniformly
bounded rank. For instance, if $\Delta$ has no eigenvalues in $(0,\frac{1}{4}]$,
then $F_{t}^{\mathrm{Rat}}$ reduces to the orthogonal projection
onto constant functions.

The following result refines the exponentially decaying remainder
by showing that, after the natural $e^{t/2}$-renormalization, its
leading part is governed by the shifted wave equation.

Fix $\eta>0$, and let $\Pi_{\mathrm{exc}}$ be the finite rank spectral
projector of $\Delta$ onto the finite-dimensional subspace generated
by the constants, the complementary spectrum $0<\mu<\frac{1}{4}$,
the possible threshold eigenspace $E_{1/4}$, and the principal eigenvalues
in the window $\frac{1}{4}<\mu<\frac{1}{4}+\eta^{2}.$ Define the
regular spectral subspace by 
\[
L_{\mathrm{reg}}^{2}(\mathcal{N}):=(1-\Pi_{\mathrm{exc}})L^{2}(\mathcal{N}).
\]
On this subspace, $\Delta-\frac{1}{4}\ge\eta^{2}$, so the operators
\[
\left(\Delta-\frac{1}{4}\right)^{1/2}\quad\text{and}\quad\left(\Delta-\frac{1}{4}\right)^{-1/4}
\]
are defined by spectral calculus.

\begin{cBoxB}{}

\begin{thm}[Emergence of the wave equation from spherical means]
For $t>0$, define 
\begin{equation}
E_{t}:=e^{t/2}\mathcal{L}_{t}(1-\Pi_{\mathrm{exc}}).\label{eq:def_E_t_renormalized_spherical_mean-1}
\end{equation}
Then, on the regular spectral subspace $L_{\mathrm{reg}}^{2}(\mathcal{N})$,
one has 
\begin{equation}
E_{t}=W_{+}e^{it\left(\Delta-\frac{1}{4}\right)^{1/2}}+W_{+}^{\dagger}e^{-it\left(\Delta-\frac{1}{4}\right)^{1/2}}+e^{-2t}Q_{t},\label{eq:E_t_wave_asymptotic-1}
\end{equation}
where, for every $s\in\mathbb{R}$, 
\[
\sup_{t\ge t_{0}}\|Q_{t}\|_{H^{s}(\mathcal{N})\to H^{s+1/2}(\mathcal{N})}<\infty.
\]
Moreover, 
\begin{equation}
W_{+}=\frac{1}{\sqrt{\pi}}e^{-i\pi/4}\left(\Delta-\frac{1}{4}\right)^{-1/4}+\Psi^{-3/2}.\label{eq:W_plus_principal_symbol-1}
\end{equation}
Consequently, $E_{t}$ satisfies the shifted wave equation modulo
an exponentially small remainder: 
\begin{equation}
\left(\frac{\partial^{2}}{\partial t^{2}}+\Delta-\frac{1}{4}\right)E_{t}=e^{-2t}P_{t},\label{eq:E_t_wave_equation_remainder-1}
\end{equation}
where 
\[
\sup_{t\ge t_{0}}\|P_{t}\|_{H^{s}(\mathcal{N})\to H^{s-1/2}(\mathcal{N})}<\infty.
\]
\end{thm}

\end{cBoxB}

This theorem is proved in Proposition~\ref{prop:wave_asymptotics_spherical_means}.

\subsection{Outline of the paper}

In Section~\ref{sec:Standard-properties-of}, we recall standard
facts and fix the notation used throughout the paper. We describe
hyperbolic surfaces in the group-theoretic setting of $\mathrm{SL}_{2}(\mathbb{R})$,
with $M=\Gamma\backslash\mathrm{SL}_{2}(\mathbb{R})$ identified with
the unit tangent bundle of the hyperbolic surface $\mathcal{N}=\bar{\Gamma}\backslash\mathbb{H}$.
In Subsection~\ref{subsec:Decomposition-of-}, we recall the decomposition
of $L^{2}(M)$ into unitary irreducible representations, including
the spherical principal and complementary series, as well as the discrete
series. This is summarized in Theorem~\ref{thm:We-have-a}.

In Section~\ref{sec:Computation-with-the}, we explain how the hyperbolic
vector field $x\frac{\partial}{\partial x}$ on $\mathbb{R}$ is conjugated
to the quantum harmonic oscillator. The corresponding statement, given
in Corollary~\ref{cor:With-setting-}, is one of the main tools used
in the proof of the main theorem.

In Section~\ref{sec:Principal-series}, we study the operators $X,U,S$
restricted to the spherical representations. Starting from the standard
unitary representation on $L^{2}(S^{1})$, we use two stereographic
charts from $S^{1}$ to $\mathbb{R}$ in Subsection~\ref{subsec:Stereographic-projection-and}.
In these charts, the operator $X$ becomes the linear vector field
$x\frac{\partial}{\partial x}$ plus some constant. Using the results
of the previous section, we show in Lemma~\ref{lem:Tprime_pm} that
$X$ is conjugated to a harmonic oscillator in each chart, and we
derive the corresponding estimates in the following subsections. In
Subsection~\ref{subsec:Merging-the-two}, we merge the two projective
realizations. A key point is that the two realizations are defined
on disjoint dense domains; this leads, in Subsection~\ref{subsec:Hilbert-space-realization-from},
to a natural Hilbert completion of their direct sum. The main result
of this section is Theorem~\ref{thm:summary_spectral}. The complementary
series are treated in the same way in Subsection~\ref{subsec:Complementary-series}.
At the threshold $\mu=1/4$, discussed in Subsection~\ref{subsec:Threshold-case-mu=00003D1/4},
the two projective realizations coalesce and form a Jordan block of
size two.

In preparation for the trace formula, we prove in Subsection~\ref{subsec:Trace-of--1}
that the flat trace and the spectral trace coincide in the present
setting.

In Section~\ref{sec:Discrete-series}, we study the operators $X,U,S$
restricted to the discrete series. The analysis is parallel to the
spherical case, but simpler: no merging of two projective realizations
is needed, and there is no threshold phenomenon.

In Section~\ref{sec:Global-Hilbert-model}, we assemble the previous
results and state the global Hilbert model in Theorem~\ref{thm:global_Hilbert_model_XUS}.

We conclude the paper with two applications of the main theorem. First,
in Section~\ref{sec:Atiyah=002013Bott-and-Selberg}, taking the trace
of the operator identity gives the Selberg trace formula. Second,
in Section~\ref{sec:Spherical-averages}, using the spectral decomposition
of the propagator $e^{tX}$ provided by the main theorem, we show
that the large-time behavior of the spherical mean operator $\mathcal{L}_{t}$
is governed by the shifted wave group $e^{\pm it\sqrt{\Delta-\frac{1}{4}}}.$

\section*{Acknowledgements}

This work was partially supported by the ANR project ADYCT grant ANR-20-CE40-0017. 

\section{\protect\label{sec:Standard-properties-of}Standard properties of
$\mathrm{SL}_{2}(\mathbb{R})$ and $\Gamma\backslash\mathrm{SL}_{2}(\mathbb{R})$}

This section fixes the conventions and normalizations used throughout
the paper. The reader familiar with the representation theory of $SL_{2}(\mathbb{R})$
may skip directly to Theorem~\ref{thm:We-have-a}; however, the explicit
formulas for $X,U,S,\Omega,\Theta$ and $N_{\pm}$ will be used in
the construction of the Hilbert models below.

We let $\Gamma\subset\mathrm{SL}_{2}(\mathbb{R})$ be a discrete cocompact
subgroup. Its projection to $\mathrm{PSL}_{2}(\mathbb{R})$ will define
the closed hyperbolic surface $\mathcal{N}$ used below. Since the
left action of $\Gamma$ on $\mathrm{SL}_{2}(\mathbb{R})$ is free
and properly discontinuous, the quotient $\Gamma\backslash\mathrm{SL}_{2}(\mathbb{R})$
is a compact smooth $3$-manifold. Throughout the paper we assume
that $-\mathrm{Id}\in\Gamma$, so that only even $K$-types occur.

For general references on $\mathrm{SL}_{2}(\mathbb{R})$ and its representations,
see for instance \cite[p.~56, p.~124]{carter}, \cite[Ch.~2]{bump1998automorphic},
and also \cite{Knapp_representation_86,helgason}.

\subsection{Geometry of $\mathrm{SL}_{2}(\mathbb{R})$}

Recall that 
\[
\mathrm{SL}_{2}(\mathbb{R})=\{g\in M_{2}(\mathbb{R})\ ;\ \det g=1\}.
\]
By the Iwasawa decomposition, every $g\in\mathrm{SL}_{2}(\mathbb{R})$
can be written uniquely in the form 
\begin{equation}
g=\begin{pmatrix}1 & x\\
0 & 1
\end{pmatrix}\begin{pmatrix}y^{1/2} & 0\\
0 & y^{-1/2}
\end{pmatrix}\begin{pmatrix}\cos\theta & \sin\theta\\
-\sin\theta & \cos\theta
\end{pmatrix},\qquad x\in\mathbb{R},\quad y>0,\quad\theta\in\mathbb{R}/2\pi\mathbb{Z}.\label{eq:iwasawa}
\end{equation}
Thus 
\[
(x,y,\theta)\in\mathbb{R}\times(0,\infty)\times S^{1}
\]
give global product coordinates on $\mathrm{SL}_{2}(\mathbb{R})$.

Setting $z=x+iy\in\mathbb{H}$, and writing $k(\theta)\in SO(2)$
for the last factor in \eqref{eq:iwasawa}, the map 
\begin{equation}
\mathrm{SL}_{2}(\mathbb{R})\longrightarrow\mathbb{H}\times SO(2),\qquad g\longmapsto(z,k(\theta)),\label{eq:product_SL2R}
\end{equation}
is a smooth diffeomorphism. Equivalently, using the Cayley transform
\[
w=\frac{z-i}{z+i},
\]
one may identify the upper half-plane $\mathbb{H}$ with the unit
disk $\mathbb{D}$, and hence identify 
\[
\mathrm{SL}_{2}(\mathbb{R})\simeq\mathbb{D}\times SO(2).
\]

\subsection{\protect\label{subsec:Lie-algebra}Lie algebra $\mathfrak{sl}_{2}(\mathbb{R})$}

We briefly recall the conventions used for the Lie algebra of $\mathrm{SL}_{2}(\mathbb{R})$.

\begin{figure}[h]
\begin{centering}
{\small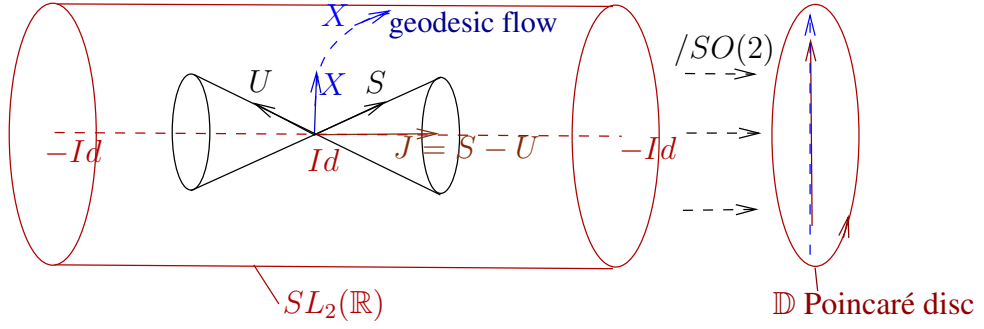}{\small\par}
\par\end{centering}
\caption{The geodesic flow on the Poincaré disk is generated by $X\in\mathfrak{sl}_{2}(\mathbb{R})$.}
\end{figure}

The Lie algebra of $\mathrm{SL}_{2}(\mathbb{R})$ is 
\[
\mathfrak{sl}_{2}(\mathbb{R}):=T_{\mathrm{Id}}\mathrm{SL}_{2}(\mathbb{R})=\left\{ V\in\mathrm{Mat}_{2\times2}(\mathbb{R})\ ;\ \mathrm{Tr}(V)=0\right\} .
\]
Indeed, for any $V\in\mathrm{Mat}_{2\times2}(\mathbb{R})$, one has
\[
\det(e^{V})=e^{\mathrm{Tr}(V)}.
\]
Thus $e^{V}\in\mathrm{SL}_{2}(\mathbb{R})$ if and only if $\mathrm{Tr}(V)=0$.

We use the following basis of $\mathfrak{sl}_{2}(\mathbb{R})$: 
\begin{equation}
X=\frac{1}{2}\begin{pmatrix}1 & 0\\
0 & -1
\end{pmatrix},\qquad U=\begin{pmatrix}0 & 0\\
1 & 0
\end{pmatrix},\qquad S=\begin{pmatrix}0 & 1\\
0 & 0
\end{pmatrix}.\label{eq:X,U,S}
\end{equation}
The corresponding Lie brackets are 
\begin{equation}
[X,U]\eq{\ref{eq:X,U,S}}-U,\qquad[X,S]\eq{\ref{eq:X,U,S}}S,\qquad[S,U]\eq{\ref{eq:X,U,S}}2X.\label{eq:commutt}
\end{equation}

Every $V\in\mathfrak{sl}_{2}(\mathbb{R})$ defines a left-invariant
vector field on $\mathrm{SL}_{2}(\mathbb{R})$, still denoted by $V$,
by 
\begin{equation}
V(g)=gV,\qquad g\in\mathrm{SL}_{2}(\mathbb{R}).\label{eq:V_g}
\end{equation}
Its flow is given by right multiplication: 
\begin{equation}
\phi_{V}^{t}(g):=g\,e^{tV}.\label{eq:def_Phi_t_V}
\end{equation}
Indeed, 
\[
\left.\frac{d}{dt}\right|_{t=0}\phi_{V}^{t}(g)\eq{\ref{eq:def_Phi_t_V}}gV\eq{\ref{eq:V_g}}V(g).
\]
Since these vector fields are left-invariant, they descend to the
left quotient $\Gamma\backslash\mathrm{SL}_{2}(\mathbb{R})$.

In the Iwasawa coordinates \eqref{eq:iwasawa}, one may write 
\[
g\eq{\ref{eq:iwasawa}}e^{xS}e^{(\ln y)X}e^{\theta J},
\]
where 
\begin{equation}
J:=\begin{pmatrix}0 & 1\\
-1 & 0
\end{pmatrix}\eq{\ref{eq:X,U,S}}S-U.\label{eq:def_J}
\end{equation}
The element $J$ generates the compact subgroup $\mathrm{SO}(2)\subset\mathrm{SL}_{2}(\mathbb{R})$.

\subsection{Casimir operator $\Omega$}

The Casimir operator is the second-order differential operator on
$\mathrm{SL}_{2}(\mathbb{R})$ defined by 
\begin{equation}
\Omega:=-X^{2}-\frac{1}{2}SU-\frac{1}{2}US.\label{eq:def_casimir}
\end{equation}
It belongs to the center of the universal enveloping algebra. Equivalently,
with our convention for left-invariant vector fields, it satisfies
\[
[\Omega,V]=0,\qquad\forall V\in\mathfrak{sl}_{2}(\mathbb{R}).
\]

Indeed, using \eqref{eq:commutt}, one may rewrite 
\[
\Omega\eq{\ref{eq:def_casimir},\ref{eq:commutt}}-X^{2}-X-US.
\]
Then 
\[
[\Omega,X]=-[US,X]=-\bigl(U[S,X]+[U,X]S\bigr)=-\bigl(-US+US\bigr)=0.
\]
Similarly, one checks that 
\[
[\Omega,U]=[\Omega,S]=0.
\]

\begin{cBoxB}{}
\begin{lem}
In the Iwasawa coordinates $(x,y,\theta)$ defined in \eqref{eq:iwasawa},
one has, see for instance \cite[Ch.~2]{bump1998automorphic}, 
\begin{equation}
\Omega=\Delta_{\mathbb{H}}+y\frac{\partial}{\partial\theta}\frac{\partial}{\partial x},\label{eq:C_coordinates}
\end{equation}
where 
\begin{equation}
\Delta_{\mathbb{H}}:=-y^{2}\left(\frac{\partial^{2}}{\partial x^{2}}+\frac{\partial^{2}}{\partial y^{2}}\right)\label{eq:def_Laplacian_H}
\end{equation}
is the positive hyperbolic Laplacian on $\mathbb{H}$. In particular,
on $K$-invariant functions, i.e. on functions independent of $\theta$,
the Casimir operator reduces to the positive Laplacian: 
\[
\Omega_{\mid K\text{-inv.}}=\Delta_{\mathbb{H}}.
\]
\end{lem}

\end{cBoxB}

\begin{proof}
We use the Iwasawa coordinates 
\[
g=e^{xS}e^{(\ln y)X}e^{\theta J}.
\]
With our convention $V(g)=gV$, the vector field generated by $V\in\mathfrak{sl}_{2}(\mathbb{R})$
is obtained by differentiating $ge^{tV}$ at $t=0$. Comparing with
the derivatives with respect to $x,y,\theta$, one obtains 
\begin{align*}
X & =-y\sin(2\theta)\partial_{x}+y\cos(2\theta)\partial_{y}+\frac{1}{2}\sin(2\theta)\partial_{\theta},\\
S & =y\cos(2\theta)\partial_{x}+y\sin(2\theta)\partial_{y}+\sin^{2}\theta\,\partial_{\theta},\\
U & =y\cos(2\theta)\partial_{x}+y\sin(2\theta)\partial_{y}-\cos^{2}\theta\,\partial_{\theta}.
\end{align*}
Substituting these expressions into (\ref{eq:def_casimir}), we get
(\ref{eq:C_coordinates}).
\end{proof}

\subsection{\protect\label{subsec:Hyperbolic-dynamics-on}Hyperbolic dynamics
on $\mathrm{SL}_{2}(\mathbb{R})$}

Let $\mathbf{g}$ be any left-invariant Riemannian metric on $\mathrm{SL}_{2}(\mathbb{R})$.
We consider the flow generated by $X$, namely 
\[
\phi_{X}^{t}(g)=ge^{tX}.
\]
For $g\in\mathrm{SL}_{2}(\mathbb{R})$, the vector $U_{g}$ is $gU$.
Hence 
\[
d\phi_{X}^{t}(U_{g})=gUe^{tX}
\]
as a tangent vector at $ge^{tX}$. By left-invariance of the metric,
its norm is computed after left-translation back to the identity:
\begin{align*}
\|d\phi_{X}^{t}(U_{g})\|_{\mathbf{g}} & =\left\Vert (ge^{tX})^{-1}gUe^{tX}\right\Vert _{\mathbf{g}}\\
 & =\left\Vert e^{-tX}Ue^{tX}\right\Vert _{\mathbf{g}}=\left\Vert e^{-t\operatorname{ad}_{X}}U\right\Vert _{\mathbf{g}}\\
 & \eq{\ref{eq:commutt}}e^{t}\|U\|_{\mathbf{g}}.
\end{align*}
Thus $U$ spans the unstable direction 
\[
E_{u}(g):=\mathbb{R}U_{g}
\]
with Lyapunov exponent $+1$.

Similarly, using $[X,S]=S$, one obtains 
\[
\|d\phi_{X}^{t}(S_{g})\|_{\mathbf{g}}=e^{-t}\|S\|_{\mathbf{g}}.
\]
Thus $S$ spans the stable direction 
\[
E_{s}(g):=\mathbb{R}S_{g}.
\]
Altogether, 
\[
T\mathrm{SL}_{2}(\mathbb{R})=\mathbb{R}X\oplus E_{s}\oplus E_{u}.
\]
Since the vector fields are left-invariant, this splitting descends
to the compact quotient $\Gamma\backslash\mathrm{SL}_{2}(\mathbb{R})$;
there it is the Anosov splitting of the geodesic flow.
\begin{rem}
Equivalently, the hyperbolic behavior can be read directly at the
group level. For all $u,s,t\in\mathbb{R}$, 
\[
e^{uU}e^{tX}=e^{tX}e^{ue^{t}U},\qquad e^{sS}e^{tX}=e^{tX}e^{se^{-t}S}.
\]
\end{rem}

The Anosov one-form $\mathcal{A}$ is the left-invariant one-form
defined by 
\[
\mathcal{A}(X)=1,\qquad\mathcal{A}(E_{u}\oplus E_{s})=0.
\]
It descends to $\Gamma\backslash\mathrm{SL}_{2}(\mathbb{R})$. From
$[S,U]=2X$, we get 
\[
(d\mathcal{A})(U,S)=U(\mathcal{A}(S))-S(\mathcal{A}(U))-\mathcal{A}([U,S])\eq{\ref{eq:commutt}}2.
\]
Therefore $d\mathcal{A}$ is non-degenerate on $E_{u}\oplus E_{s}=\mathrm{Span}(U,S)$.
Hence $\mathcal{A}$ is a contact one-form.
\begin{rem}
In summary, the commutation relations 
\[
[U,X]=U,\qquad[S,X]=-S
\]
express the hyperbolicity of the flow, while 
\[
[S,U]=2X
\]
expresses the preserved contact structure. 
\end{rem}

\subsection{Hyperbolic dynamics on $M=\Gamma\backslash\mathrm{SL}_{2}(\mathbb{R})$}

Let $\Gamma<\mathrm{SL}_{2}(\mathbb{R})$ be a discrete cocompact
subgroup, and set 
\begin{equation}
M:=\Gamma\backslash\mathrm{SL}_{2}(\mathbb{R}).\label{eq:def_M}
\end{equation}
Then $M$ is a compact smooth $3$-manifold.

For simplicity, we assume from now on that 
\begin{equation}
-\mathrm{Id}\in\Gamma,\label{eq:minusId_in_Gamma}
\end{equation}
and that the image $\bar{\Gamma}\subset\mathrm{PSL}_{2}(\mathbb{R})$
of $\Gamma$ is torsion-free. We define the base surface by 
\begin{equation}
\mathcal{N}:=M/\mathrm{SO}(2)\eq{\ref{eq:def_M}}\Gamma\backslash\mathrm{SL}_{2}(\mathbb{R})/\mathrm{SO}(2)\eq{\ref{eq:product_SL2R}}\bar{\Gamma}\backslash\mathbb{H}^{2}.\label{eq:def_N}
\end{equation}
Thus $\mathcal{N}$ is a smooth compact hyperbolic surface, and $M$
identifies with its unit cotangent bundle $S^{*}\mathcal{N}$.
\begin{rem}
The assumptions above are made only to keep the geometric presentation
simple. The condition $-\mathrm{Id}\in\Gamma$ implies that only even
$K$-types occur on $M$. The torsion-free assumption on $\bar{\Gamma}$
ensures that $\mathcal{N}=\bar{\Gamma}\backslash\mathbb{H}^{2}$ is
a smooth surface rather than a hyperbolic orbifold with conical points. 
\end{rem}

For every $V\in\mathfrak{sl}_{2}(\mathbb{R})$, the flow on $\mathrm{SL}_{2}(\mathbb{R})$
defined by 
\[
\phi_{V}^{t}(g)\eq{\ref{eq:def_Phi_t_V}}ge^{tV}
\]
commutes with the left action of $\Gamma$. Hence it descends to a
well-defined flow on 
\[
M=\Gamma\backslash\mathrm{SL}_{2}(\mathbb{R}),
\]
still denoted by $\phi_{V}^{t}$, and generated by the vector field
$V$.

In particular, the flow $(\phi_{X}^{t})_{t\in\mathbb{R}}$ generated
by $X$ is a contact Anosov flow on $M$. Under the identification
$M\simeq S^{*}\mathcal{N}$, it is the geodesic flow of the hyperbolic
surface $\mathcal{N}$; see for instance \cite{katok_hasselblatt}.

\begin{figure}[H]
\begin{centering}
\scalebox{0.9}[0.9]{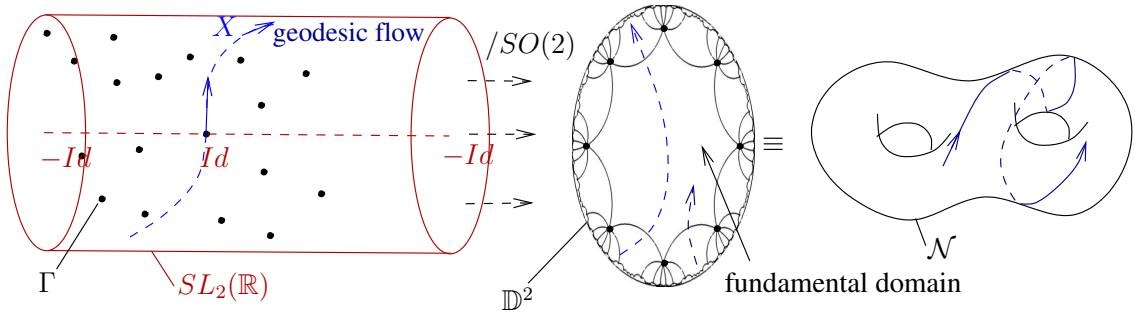}
\par\end{centering}
\caption{Geodesic flow on a surface $\mathcal{N}=\Gamma\backslash\mathrm{SL}_{2}(\mathbb{R})/\mathrm{SO}(2)$
with constant negative curvature.}
\end{figure}

\subsection{\protect\label{subsec:Decomposition-of-}Decomposition of $L^{2}(M)$
into unitary irreducible representations}

We have set 
\[
M\eq{\ref{eq:def_M}}\Gamma\backslash\mathrm{SL}_{2}(\mathbb{R}).
\]
We now recall the standard decomposition of $L^{2}(M)$ into unitary
irreducible representations of $\mathrm{SL}_{2}(\mathbb{R})$ for
the right regular action. The final result is stated in Theorem~\ref{thm:We-have-a};
the reader familiar with this decomposition may go directly to that
theorem. A classical reference is \cite{gelfand2014integral}.

\subsubsection{Definition of some operators}

For $V\in\mathfrak{sl}_{2}(\mathbb{R})$, the left-invariant vector
field $V$ on $M$ is viewed as a first-order differential operator
on $C^{\infty}(M)$. Since the corresponding flow is given by right
multiplication and preserves the invariant measure on $M$, these
operators are skew-adjoint on $L^{2}(M)$ up to the usual factor $i$.

We define 
\begin{equation}
\Theta:=-iJ\eq{\ref{eq:def_J}}i(U-S).\label{eq:def_Theta}
\end{equation}
In Iwasawa coordinates, $J=\partial_{\theta}$, hence 
\[
\Theta=-i\partial_{\theta}.
\]
Therefore $\Theta$ is essentially self-adjoint on $L^{2}(M)$. Since
$-\mathrm{Id}\in\Gamma$, functions on $M$ are invariant under $\theta\mapsto\theta+\pi$.
Consequently only even Fourier modes occur, and 
\[
\mathrm{Spec}(\Theta)=2\mathbb{Z}.
\]
Thus we obtain the orthogonal decomposition 
\begin{equation}
L^{2}(M)=\widehat{\bigoplus}_{l\in2\mathbb{Z}}L_{l}^{2}(M),\label{eq:decomp_L2M}
\end{equation}
where 
\begin{equation}
L_{l}^{2}(M):=\ker(\Theta-l\mathrm{Id})=\left\{ u\in L^{2}(M)\ ;\ \Theta u=lu\right\} .\label{eq:def_Ln}
\end{equation}

The Casimir operator $\Omega$ defined in \eqref{eq:def_casimir}
is essentially self-adjoint on $L^{2}(M)$. Since $\Omega$ commutes
with $\Theta$, each $K$-type 
\[
L_{l}^{2}(M),\qquad l\in2\mathbb{Z},
\]
is invariant under $\Omega$. We denote by $\Omega_{l}$ the restriction
of $\Omega$ to $L_{l}^{2}(M)$.

Using \eqref{eq:C_coordinates}, and the relation 
\[
\Theta=-i\partial_{\theta},
\]
we have, on $L_{l}^{2}(M)$, 
\[
\partial_{\theta}=il.
\]
Therefore 
\begin{equation}
\Omega_{l}\eq{\ref{eq:C_coordinates}}\Delta_{\mathbb{H}}+l\,y\left(i\frac{\partial}{\partial x}\right).\label{eq:C}
\end{equation}
Here $\Delta_{\mathbb{H}}$ is the positive hyperbolic Laplacian defined
in \eqref{eq:def_Laplacian_H}.

For fixed $l$, the space $L_{l}^{2}(M)$ may be viewed as the space
of $L^{2}$-sections of a complex line bundle over the compact surface
\[
\mathcal{N}=\bar{\Gamma}\backslash\mathbb{H}^{2}.
\]
In this interpretation, $\Omega_{l}$ is an elliptic self-adjoint
operator. Indeed, its principal symbol is 
\[
\mathrm{Symb}(\Omega_{l})=\mathrm{Symb}(\Delta_{\mathbb{H}})=\|\xi\|_{g_{\mathbb{H}}}^{2}=y^{2}(\xi_{x}^{2}+\xi_{y}^{2}).
\]
Hence, by elliptic spectral theory on compact manifolds, the spectrum
of $\Omega_{l}$ is real, discrete, and each eigenspace is finite
dimensional.

Thus we have the Hilbert orthogonal decomposition 
\begin{equation}
L_{l}^{2}(M)=\widehat{\bigoplus}_{\mu\in\mathrm{Spec}(\Omega_{l})}\mathcal{H}_{\mu,l},\label{eq:def_L2_n}
\end{equation}
where 
\begin{equation}
\mathcal{H}_{\mu,l}:=\ker(\Omega-\mu\mathrm{Id})\cap\ker(\Theta-l\mathrm{Id})=\left\{ u\in L^{2}(M)\ ;\ \Omega u=\mu u,\ \Theta u=lu\right\} .\label{eq:def_H_mu_n}
\end{equation}
The spaces $\mathcal{H}_{\mu,l}$ are finite-dimensional. In Figure~\ref{fig:irreps},
each such space is represented by a point of a lattice indexed by
the two quantum numbers $(\mu,l)$.

We now introduce the raising and lowering operators. Set 
\begin{equation}
Y:=\frac{1}{2}(U+S),\label{eq:def_Y}
\end{equation}
and 
\begin{equation}
N_{\pm}:=X\pm iY.\label{eq:def_N+-}
\end{equation}
Using \eqref{eq:def_Theta} and \eqref{eq:commutt}, one checks that
\[
[\Theta,N_{\pm}]=\pm2N_{\pm}.
\]
Thus, if $u\in L_{l}^{2}(M)$, namely $\Theta u=lu$, then 
\[
\Theta(N_{+}u)=([\Theta,N_{+}]+N_{+}\Theta)u=(l+2)N_{+}u,
\]
and similarly $\Theta(N_{-}u)=(l-2)N_{-}u$. Hence 
\begin{equation}
N_{\pm}:L_{l}^{2}(M)\longrightarrow L_{l\pm2}^{2}(M).\label{eq:N_pm_maps}
\end{equation}
These operators are the usual $K$-type raising and lowering operators.

Conversely, from \eqref{eq:def_Theta} and \eqref{eq:def_N+-}, we
recover 
\begin{equation}
X=\frac{1}{2}(N_{+}+N_{-}),\qquad U=\frac{i}{2}(N_{-}-N_{+}-\Theta),\qquad S=\frac{i}{2}(N_{-}-N_{+}+\Theta).\label{eq:X_from_R}
\end{equation}
In terms of $N_{+}$, $N_{-}$ and $\Theta$, the Casimir operator
becomes 
\begin{equation}
\Omega\eq{\ref{eq:def_casimir}}-\left(\frac{\Theta}{2}\right)^{2}-\frac{1}{2}N_{+}N_{-}-\frac{1}{2}N_{-}N_{+}=-\left(\frac{\Theta}{2}\right)^{2}-X^{2}-Y^{2}.\label{eq:Casimir_RL}
\end{equation}
Moreover, 
\begin{equation}
[N_{+},N_{-}]=\Theta.\label{eq:Com_RL}
\end{equation}
Combining \eqref{eq:Casimir_RL} and \eqref{eq:Com_RL}, we obtain
\begin{equation}
N_{+}N_{-}=-\Omega-\frac{1}{4}\Theta(\Theta-2),\label{eq:NpNm}
\end{equation}
and 
\begin{equation}
N_{-}N_{+}=-\Omega-\frac{1}{4}\Theta(\Theta+2).\label{eq:NmNp}
\end{equation}
It will be convenient to use the complex conjugation operator 
\begin{equation}
C:\begin{cases}
L^{2}(M) & \longrightarrow L^{2}(M),\\
v & \longmapsto\overline{v}.
\end{cases}\label{eq:def_Complex_conjugate_C}
\end{equation}
It is an anti-unitary involution. Since $C$ changes the sign of the
fiber Fourier mode and commutes with the Casimir operator, it induces
an anti-unitary isomorphism 
\[
C:\mathcal{H}_{\mu,l}\longrightarrow\mathcal{H}_{\mu,-l}.
\]
Moreover, 
\begin{equation}
C\Theta=-\Theta C,\qquad CN_{\pm}=N_{\mp}C,\qquad C\Omega=\Omega C.\label{eq:conjugation}
\end{equation}

\begin{rem}
For fixed $l\in2\mathbb{Z}$, the operator $\Omega_{l}$ on $L_{l}^{2}(M)$
may be interpreted as a magnetic Laplacian on the base surface $\mathcal{N}$.
In this interpretation, the operators $N_{\pm}$ in \eqref{eq:def_N+-}
play the role of covariant Cauchy--Riemann operators: they shift
the $K$-type by $\pm2$ and correspond to the holomorphic and anti-holomorphic
directions in $\mathrm{Span}(X,Y)^{\mathbb{C}}$. For a fixed $l$,
the spectrum of $\Omega_{l}$ is represented on a vertical line in
Figure~\ref{fig:irreps}. 
\end{rem}

\subsubsection{Construction of irreducible representations}

For $l\in2\mathbb{Z}$, the operators 
\[
N_{\pm}:L_{l}^{2}(M)\longrightarrow L_{l\pm2}^{2}(M)
\]
shift the $K$-type by $\pm2$. Viewed on the base surface $\mathcal{N}$,
these are first-order elliptic operators of Cauchy--Riemann type
between the corresponding line bundles. Hence their kernels are finite-dimensional.

For $l\in2\mathbb{N}^{*}$, we define the holomorphic lowest-weight
spaces by 
\begin{equation}
\mathcal{H}_{\mathrm{hol},+}(l):=\ker\left(N_{-}:L_{l}^{2}(M)\longrightarrow L_{l-2}^{2}(M)\right),\label{eq:H_hol_plus}
\end{equation}
and the anti-holomorphic highest-weight spaces by 
\begin{equation}
\mathcal{H}_{\mathrm{hol},-}(-l):=\ker\left(N_{+}:L_{-l}^{2}(M)\longrightarrow L_{-l+2}^{2}(M)\right).\label{eq:H_hol_minus}
\end{equation}

\begin{cBoxB}{}

\begin{prop}[Riemann--Roch multiplicities]
\label{prop:RR_discrete_multiplicities} Assume that $-\mathrm{Id}\in\Gamma$,
so that only even $K$-types occur. For $l\in2\mathbb{N}^{*}$, set
\begin{equation}
m_{l}:=\dim\mathcal{H}_{\mathrm{hol}^{+}}(l).\label{eq:ml}
\end{equation}
Then 
\begin{equation}
m_{l}=\dim\mathcal{H}_{\mathrm{hol}^{-}}(-l)=\begin{cases}
g, & l=2,\\[0.2em]
(l-1)(g-1), & l\ge4,
\end{cases}\label{eq:m_l}
\end{equation}
where $g$ is the genus of $\mathcal{N}$. Moreover, by Gauss--Bonnet,
\[
g-1=\frac{\operatorname{Area}(\mathcal{N})}{4\pi}.
\]
\end{prop}

\end{cBoxB}

\begin{proof}
The identification of the discrete-series multiplicities with spaces
of holomorphic powers of the canonical line bundle is standard; see
for example \cite[Theorem~1.1]{cosentino2005holder}, following Flaminio--Forni
\cite{flaminio_forni_2003}. More precisely, if $K_{\mathcal{N}}$
denotes the canonical line bundle of the compact Riemann surface $\mathcal{N}$,
then for $l=2q$, $q\ge1$, one has 
\[
\mathcal{H}_{\mathrm{hol},+}(2q)\simeq H^{0}(\mathcal{N},K_{\mathcal{N}}^{\otimes q}).
\]
Thus 
\[
m_{2q}=\dim H^{0}(\mathcal{N},K_{\mathcal{N}}^{\otimes q}).
\]

For $q=1$, this is the space of holomorphic one-forms, hence 
\[
m_{2}=\dim H^{0}(\mathcal{N},K_{\mathcal{N}})=g.
\]

For $q\ge2$, Riemann--Roch applied to $K_{\mathcal{N}}^{\otimes q}$
gives, see for instance \cite[Chapter~16]{forster1981riemann}, 
\[
h^{0}(K_{\mathcal{N}}^{\otimes q})-h^{0}(K_{\mathcal{N}}^{\otimes(1-q)})=\deg(K_{\mathcal{N}}^{\otimes q})+1-g.
\]
Since $q\ge2$, the line bundle $K_{\mathcal{N}}^{\otimes(1-q)}$
has negative degree, and therefore 
\[
h^{0}(K_{\mathcal{N}}^{\otimes(1-q)})=0.
\]
Moreover 
\[
\deg(K_{\mathcal{N}}^{\otimes q})=q\,\deg(K_{\mathcal{N}})=q(2g-2).
\]
Hence 
\[
h^{0}(K_{\mathcal{N}}^{\otimes q})=q(2g-2)+1-g=(2q-1)(g-1).
\]
Since $l=2q$, this becomes 
\[
m_{l}=(l-1)(g-1),\qquad l\ge4.
\]

The equality 
\[
\dim\mathcal{H}_{\mathrm{hol},+}(l)=\dim\mathcal{H}_{\mathrm{hol},-}(-l)
\]
follows from the anti-unitary complex conjugation $C$, which exchanges
$\mathcal{H}_{\mathrm{hol},+}(l)$ and $\mathcal{H}_{\mathrm{hol},-}(-l)$.

Finally, since $\mathcal{N}$ has constant curvature $-1$, Gauss--Bonnet
gives 
\[
-\mathrm{Area}(\mathcal{N})=2\pi\chi(\mathcal{N})=2\pi(2-2g).
\]
Thus 
\[
\mathrm{Area}(\mathcal{N})=4\pi(g-1),
\]
which proves the last identity. 
\end{proof}
\begin{cBoxB}{}

\begin{lem}
\label{lem:norm_Npm} Let $k\in2\mathbb{Z}$, let $\mu\in\mathrm{Spec}(\Omega_{k})$,
and let $0\neq v_{k}\in\mathcal{H}_{\mu,k}$. Then 
\begin{equation}
\|N_{+}v_{k}\|_{L^{2}}^{2}=\|v_{k}\|_{L^{2}}^{2}\left(\mu+\frac{1}{4}k(k+2)\right),\label{eq:Rvn}
\end{equation}
and 
\begin{equation}
\|N_{-}v_{k}\|_{L^{2}}^{2}=\|v_{k}\|_{L^{2}}^{2}\left(\mu+\frac{1}{4}k(k-2)\right).\label{eq:Lvn}
\end{equation}
\end{lem}

\end{cBoxB}

\begin{proof}
By elliptic regularity, $v_{k}$ is smooth, so all computations below
are legitimate. We use the unitarity of the right regular representation.
Since the real vector fields $X,U,S$ preserve the invariant measure
on $M$, they are skew-adjoint on $L^{2}(M)$, on their natural domains.
Hence 
\[
Y\eq{\ref{eq:def_Y}}\frac{1}{2}(U+S)
\]
is also skew-adjoint. Therefore, on the algebraic core, 
\begin{equation}
N_{+}^{\dagger}\eq{\ref{eq:def_N+-}}-N_{-},\qquad N_{-}^{\dagger}=-N_{+}.\label{eq:adjoint_Npm}
\end{equation}
Using \eqref{eq:adjoint_Npm}, we get 
\begin{align*}
\|N_{+}v_{k}\|_{L^{2}}^{2} & =\langle N_{+}v_{k},N_{+}v_{k}\rangle_{L^{2}}=\langle v_{k},N_{+}^{\dagger}N_{+}v_{k}\rangle_{L^{2}}\\
 & \eq{\ref{eq:adjoint_Npm}}-\langle v_{k},N_{-}N_{+}v_{k}\rangle_{L^{2}}\\
 & \eq{\ref{eq:NmNp}}\left\langle v_{k},\left(\Omega+\frac{1}{4}\Theta(\Theta+2)\right)v_{k}\right\rangle _{L^{2}}.
\end{align*}
Since $v_{k}\in\mathcal{H}_{\mu,k}$, one has 
\[
\Omega v_{k}=\mu v_{k},\qquad\Theta v_{k}=kv_{k}.
\]
Therefore 
\[
\|N_{+}v_{k}\|_{L^{2}}^{2}=\|v_{k}\|_{L^{2}}^{2}\left(\mu+\frac{1}{4}k(k+2)\right),
\]
which proves \eqref{eq:Rvn}.

Similarly, 
\begin{align*}
\|N_{-}v_{k}\|_{L^{2}}^{2} & =\langle v_{k},N_{-}^{\dagger}N_{-}v_{k}\rangle_{L^{2}}\eq{\ref{eq:adjoint_Npm}}-\langle v_{k},N_{+}N_{-}v_{k}\rangle_{L^{2}}\\
 & \eq{\ref{eq:NpNm}}\left\langle v_{k},\left(\Omega+\frac{1}{4}\Theta(\Theta-2)\right)v_{k}\right\rangle _{L^{2}}\\
 & =\|v_{k}\|_{L^{2}}^{2}\left(\mu+\frac{1}{4}k(k-2)\right).
\end{align*}
This proves \eqref{eq:Lvn}. 
\end{proof}
The norm identities \eqref{eq:Rvn} and \eqref{eq:Lvn} describe the
possible ladders of $K$-types. Let $l\in2\mathbb{Z}$, $\mu\in\mathrm{Spec}(\Omega_{l})$,
and let $0\neq v_{l}\in\mathcal{H}_{\mu,l}$. By applying successively
$N_{+}$ and $N_{-}$, one obtains one of the following types of irreducible
ladders.
\begin{enumerate}
\item \textbf{Trivial representation.} If 
\[
N_{-}v_{l}=0\qquad\text{and}\qquad N_{+}v_{l}=0,
\]
then \eqref{eq:Rvn} and \eqref{eq:Lvn} give 
\[
\mu+\frac{1}{4}l(l+2)=0,\qquad\mu+\frac{1}{4}l(l-2)=0.
\]
Subtracting the two identities gives $l=0$, hence $\mu=0$. From
\eqref{eq:C}, this implies that $v_{l}$ is constant on $M$. Thus
\[
\mathcal{H}^{\mathrm{triv}}:=\mathrm{Span}\{\mathbf{1}\}
\]
is the one-dimensional trivial representation.
\item \textbf{Holomorphic discrete series.} Suppose that the ladder has
a lowest $K$-type. Thus, after replacing $v_{l}$ by a nonzero lowest-weight
vector in the same ladder, we may assume 
\[
N_{-}v_{l}=0,\qquad N_{+}v_{l}\neq0.
\]
Then 
\begin{equation}
\mu\eq{\ref{eq:Lvn}}-\frac{1}{4}l(l-2).\label{eq:N-0}
\end{equation}
Using \eqref{eq:Rvn}, we get 
\[
\|N_{+}v_{l}\|_{L^{2}}^{2}\eq{\ref{eq:N-0},\ref{eq:Rvn}}l\,\|v_{l}\|_{L^{2}}^{2}.
\]
Hence $l>0$, and since $l\in2\mathbb{Z}$, one has $l\ge2$. Moreover,
for every $k\in\mathbb{N}$, $N_{+}^{k}v_{l}\neq0$. We denote by
\begin{equation}
\mathcal{H}_{l}^{\mathrm{d.s.},+}(v_{l}):=\widehat{\mathrm{Span}}\{N_{+}^{k}v_{l}\ ;\ k\in\mathbb{N}\}\label{eq:def_ds_plus}
\end{equation}
the corresponding irreducible unitary representation. It is the holomorphic
discrete series.
\item \textbf{Anti-holomorphic discrete series.} Suppose that the ladder
has a highest $K$-type. Thus, after replacing $v_{l}$ by a nonzero
highest-weight vector in the same ladder, we may assume 
\[
N_{+}v_{l}=0,\qquad N_{-}v_{l}\neq0.
\]
Then 
\begin{equation}
\mu\eq{\ref{eq:Rvn}}-\frac{1}{4}l(l+2).\label{eq:N+0}
\end{equation}
Using \eqref{eq:Lvn}, we get 
\[
\|N_{-}v_{l}\|_{L^{2}}^{2}\eq{\ref{eq:N+0},\ref{eq:Lvn}}(-l)\,\|v_{l}\|_{L^{2}}^{2}.
\]
Hence $l<0$, and since $l\in2\mathbb{Z}$, one has $l\le-2$. Moreover,
for every $k\in\mathbb{N}$, $N_{-}^{k}v_{l}\neq0$. We denote by
\begin{equation}
\mathcal{H}_{l}^{\mathrm{d.s.},-}(v_{l}):=\widehat{\mathrm{Span}}\{N_{-}^{k}v_{l}\ ;\ k\in\mathbb{N}\}\label{eq:def_ds_minus}
\end{equation}
the corresponding irreducible unitary representation. It is the anti-holomorphic
discrete series.
\item \textbf{Spherical principal or complementary series.} Suppose that
the ladder is infinite in both directions, namely no nonzero vector
in the ladder is killed by $N_{+}$ or $N_{-}$. Then the weights
in the ladder cover all even $K$-types. Applying \eqref{eq:Rvn}
and \eqref{eq:Lvn} to all of them gives 
\[
\mu+\frac{1}{4}k(k+2)>0,\qquad\mu+\frac{1}{4}k(k-2)>0,\qquad\forall k\in2\mathbb{Z}.
\]
Taking $k=0$, for instance, gives $\mu>0$.

In this case the ladder contains a $K$-invariant vector. Thus, for
\[
\mu\in\mathrm{Spec}\left(\Delta_{/L^{2}(\mathcal{N})}\right)\cap(0,\infty)
\]
and for any nonzero $v\in\mathcal{H}_{\mu,0}$, we define 
\begin{equation}
\mathcal{H}_{\mu}^{\mathrm{sph}}(v):=\widehat{\mathrm{Span}}\left\{ N_{-}^{k}v,\ N_{+}^{k}v\ ;\ k\in\mathbb{N}\right\} .\label{eq:def_H_sph}
\end{equation}
This is an irreducible unitary representation. Since it contains a
$K$-type of weight $0$, it is called spherical. It belongs to the
complementary series if 
\[
0<\mu<\frac{1}{4},
\]
and to the principal series if 
\[
\mu\ge\frac{1}{4}.
\]

\end{enumerate}
These representations are the natural candidates for the irreducible
summands in the decomposition of $L^{2}(M)$; the precise statement
is given in Theorem~\ref{thm:We-have-a}. See Figure~\ref{fig:irreps}.

\begin{figure}[h]
\begin{centering}
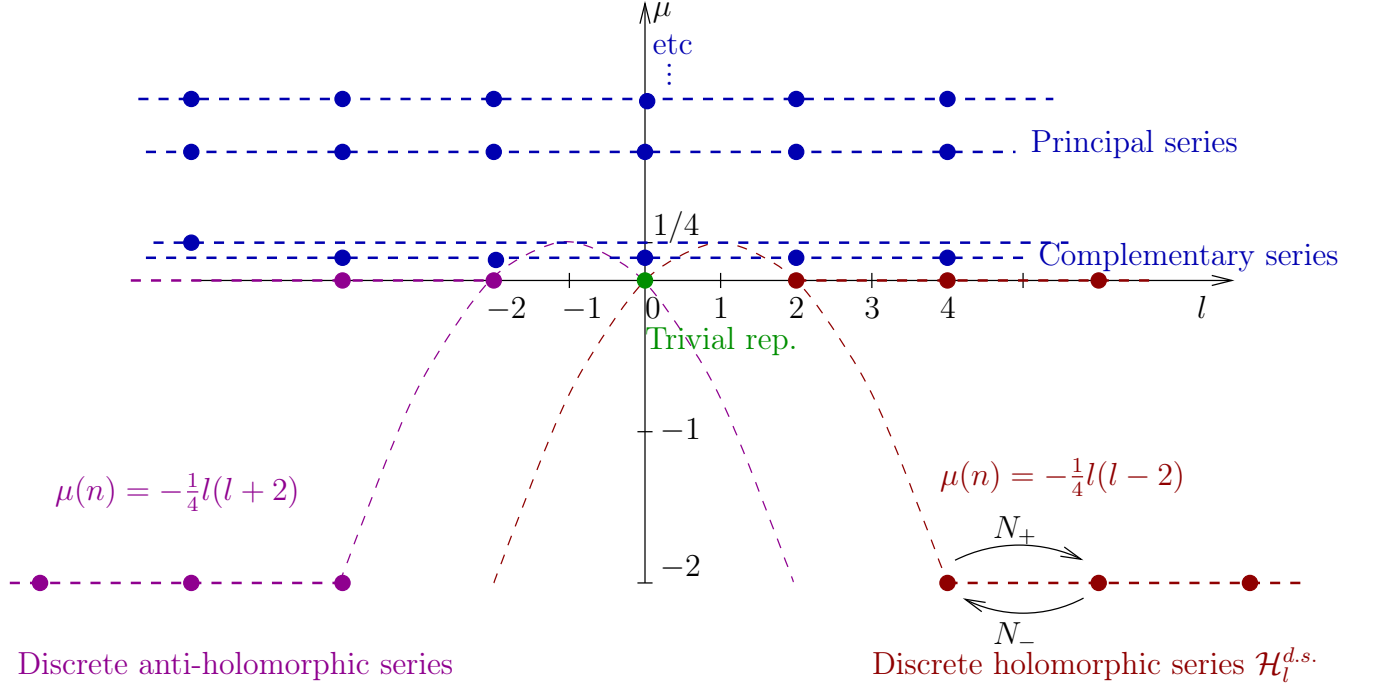
\par\end{centering}
\caption{\protect\label{fig:irreps} Each point at position $(l,\mu)$ represents
the finite-dimensional joint eigenspace $\mathcal{H}_{\mu,l}$ defined
in \eqref{eq:def_H_mu_n}. Since we assume $-\mathrm{Id}\in\Gamma$,
only even $K$-types $l\in2\mathbb{Z}$ occur. For fixed $l$, the
points on the vertical line of abscissa $l$ represent the spectral
decomposition of $L_{l}^{2}(M)$ in \eqref{eq:def_L2_n}. The points
connected by a full line or a semi-line generate an irreducible $\mathrm{SL}_{2}(\mathbb{R})$-representation
in $L^{2}(M)$: the trivial representation, a holomorphic or anti-holomorphic
discrete series, or a spherical principal/complementary series.}
\end{figure}

\subsubsection{Orthonormal basis for spherical principal series}

Let $\mu\geq\frac{1}{4}$, and define $\lambda\geq0$ by 
\begin{equation}
\mu=\lambda^{2}+\frac{1}{4}.\label{eq:def_lambda}
\end{equation}
For $l\in2\mathbb{Z}$, set 
\begin{equation}
c_{\pm,l}:=\frac{1}{2}(l\pm1)-i\lambda.\label{eq:def_cn_dn}
\end{equation}
Then 
\[
|c_{\pm,l}|^{2}=\lambda^{2}+\frac{1}{4}(l\pm1)^{2}=\mu+\frac{1}{4}l(l\pm2).
\]
Therefore, for every $v_{l}\in\mathcal{H}_{\mu,l}$, 
\begin{equation}
\|N_{\pm}v_{l}\|_{L^{2}}^{2}\eq{\ref{eq:Rvn},\ref{eq:Lvn}}|c_{\pm,l}|^{2}\|v_{l}\|_{L^{2}}^{2}.\label{eq:N+-vn}
\end{equation}

Fix one irreducible spherical representation 
\[
\mathcal{H}_{\mu}^{\mathrm{sph}}\subset L^{2}(M)
\]
as in \eqref{eq:def_H_sph}, and assume here that $\mu\ge\frac{1}{4}$.
Choose a nonzero $K$-invariant vector 
\[
v\in\mathcal{H}_{\mu,0}\cap\mathcal{H}_{\mu}^{\mathrm{sph}},
\]
and set 
\[
\varphi_{0}:=\frac{v}{\|v\|}.
\]

We define recursively vectors 
\[
\varphi_{k}\in\mathcal{H}_{\mu,2k},\qquad k\in\mathbb{Z},
\]
by 
\begin{equation}
\varphi_{k+1}:=\frac{1}{ic_{+,2k}}N_{+}\varphi_{k},\qquad k\ge0,\label{eq:def_phi_k_plus}
\end{equation}
and 
\begin{equation}
\varphi_{k-1}:=\frac{1}{i\overline{c_{-,2k}}}N_{-}\varphi_{k},\qquad k\le0.\label{eq:def_phi_k_minus}
\end{equation}
The denominators do not vanish. By \eqref{eq:N+-vn}, the recursion
preserves the norm, hence 
\[
\|\varphi_{k}\|_{L^{2}}=1,\qquad k\in\mathbb{Z}.
\]
Since $\varphi_{k}\in\mathcal{H}_{\mu,2k}$, and the spaces $\mathcal{H}_{\mu,2k}$
are mutually orthogonal for distinct $k$, we obtain an orthonormal
family. Since the representation is spherical and irreducible, this
family spans the whole representation: 
\begin{equation}
(\varphi_{k})_{k\in\mathbb{Z}}\text{ is an orthonormal basis of }\mathcal{H}_{\mu}^{\mathrm{sph}}.\label{eq:def_phi_k}
\end{equation}

By construction, and using $N_{+}^{\dagger}=-N_{-}$, the action of
$N_{\pm}$ on this basis is 
\begin{equation}
N_{+}\varphi_{k}=ic_{+,2k}\varphi_{k+1},\qquad N_{-}\varphi_{k}=i\overline{c_{-,2k}}\varphi_{k-1},\qquad k\in\mathbb{Z}.\label{eq:action_N_pm_phi_k}
\end{equation}

\subsubsection{\protect\label{subsec:Orthonormal-basis-for}Orthonormal basis for
discrete series}

For the discrete series, it is enough to describe the holomorphic
case. Indeed, for $l\in2\mathbb{N}^{*}$ and 
\[
0\neq v\in\mathcal{H}_{\mathrm{hol},+}(l)\qquad\text{that is,}\qquad N_{-}v=0,
\]
complex conjugation gives an anti-unitary isomorphism 
\[
C:\mathcal{H}_{l}^{\mathrm{d.s.},+}(v)\longrightarrow\mathcal{H}_{-l}^{\mathrm{d.s.},-}(Cv),
\]
which maps the holomorphic discrete series onto the anti-holomorphic
one.

We shall use the notation 
\[
\mathcal{H}_{l}^{\mathrm{d.s.},+}:=\widehat{\mathrm{Span}}\left\{ N_{+}^{k}v\ ;\ v\in\mathcal{H}_{\mathrm{hol},+}(l),\ k\in\mathbb{N}\right\} ,
\]
and 
\[
\mathcal{H}_{-l}^{\mathrm{d.s.},-}:=\widehat{\mathrm{Span}}\left\{ N_{-}^{k}v\ ;\ v\in\mathcal{H}_{\mathrm{hol},-}(-l),\ k\in\mathbb{N}\right\} .
\]
Then 
\[
\mathcal{H}_{-l}^{\mathrm{d.s.},-}=C\,\mathcal{H}_{l}^{\mathrm{d.s.},+}.
\]

Let us now fix $0\neq v\in\mathcal{H}_{\mathrm{hol},+}(l)$. Recall
that, for the holomorphic discrete series of lowest $K$-type $l$,
one has 
\[
\mu\eq{\ref{eq:N-0}}-\frac{1}{4}l(l-2)=\frac{l}{2}\left(1-\frac{l}{2}\right).
\]
For $k\in\mathbb{N}$, set 
\[
c_{l+2k}:=\sqrt{(l+k)(k+1)}.
\]
Then 
\[
c_{l+2k}^{2}=(l+k)(k+1)=\mu+\frac{1}{4}(l+2k)(l+2k+2).
\]

We normalize 
\[
\varphi_{0}:=\frac{v}{\|v\|},
\]
and define recursively 
\[
\varphi_{k+1}:=\frac{1}{c_{l+2k}}N_{+}\varphi_{k},\qquad k\in\mathbb{N}.
\]
Thus 
\[
\varphi_{k}\in\mathcal{H}_{\mu,l+2k}.
\]
Moreover, using \eqref{eq:Rvn}, 
\[
\|\varphi_{k+1}\|^{2}=\frac{1}{c_{l+2k}^{2}}\|N_{+}\varphi_{k}\|^{2}\eq{\ref{eq:Rvn}}\|\varphi_{k}\|^{2}.
\]
Hence, by induction, 
\[
\|\varphi_{k}\|=1,\qquad k\in\mathbb{N}.
\]
Since the $K$-types $\mathcal{H}_{\mu,l+2k}$ are pairwise orthogonal,
the family $(\varphi_{k})_{k\in\mathbb{N}}$ is an orthonormal basis
of $\mathcal{H}_{l}^{\mathrm{d.s.},+}(v)$.

By construction, 
\begin{equation}
N_{+}\varphi_{k}=c_{l+2k}\varphi_{k+1},\qquad k\in\mathbb{N}.\label{eq:action_Nplus_discrete_basis}
\end{equation}
Using $N_{+}^{\dagger}=-N_{-}$, we also get 
\begin{equation}
N_{-}\varphi_{0}=0,\qquad N_{-}\varphi_{k+1}=-c_{l+2k}\varphi_{k},\qquad k\in\mathbb{N}.\label{eq:action_Nminus_discrete_basis}
\end{equation}

\subsubsection{Final result}

We have obtained the following orthogonal decomposition of $L^{2}(M)$.

\begin{cBoxB}{}

\begin{thm}[Decomposition of $L^{2}(M)$]
\label{thm:We-have-a} \cite[Chap.~1, Sect.~2]{gelfand1969representation},
\cite[Thm.~2.7.1, p.~241]{bump1998automorphic}. The right regular
representation of $\mathrm{SL}_{2}(\mathbb{R})$ on 
\[
L^{2}(M),\qquad M=\Gamma\backslash\mathrm{SL}_{2}(\mathbb{R}),
\]
decomposes as the Hilbert orthogonal direct sum 
\begin{equation}
L^{2}(M)=\mathcal{H}^{\mathrm{triv}}\widehat{\oplus}\mathcal{H}^{\mathrm{sph}}\widehat{\oplus}\mathcal{H}^{\mathrm{d.s.}}.\label{eq:decomp_L2M-1}
\end{equation}
Here 
\[
\mathcal{H}^{\mathrm{triv}}=\mathbb{C}\boldsymbol{1}
\]
is the trivial representation.

The spherical part is 
\begin{equation}
\mathcal{H}^{\mathrm{sph}}:=\widehat{\bigoplus}_{\mu\in\mathrm{Spec}(\Delta_{/L_{0}^{2}(\mathcal{N})})}\mathcal{H}_{\mu}^{\mathrm{sph}},\label{eq:H_sph}
\end{equation}
where 
\begin{equation}
\mathcal{H}_{\mu}^{\mathrm{sph}}:=\widehat{\mathrm{Span}}\left\{ N_{-}^{k}v,\ N_{+}^{k}v\ ;\ v\in\mathcal{H}_{\mu,0},\ k\in\mathbb{N}\right\} .\label{eq:def_H_mu_sph_full}
\end{equation}
Thus $\mathcal{H}_{\mu}^{\mathrm{sph}}$ is the full spherical isotypic
component with multiplicity 
\[
\dim\mathcal{H}_{\mu,0}=\dim\ker(\Delta-\mu).
\]
Each nonzero $v\in\mathcal{H}_{\mu,0}$ generates one irreducible
spherical representation $\mathcal{H}_{\mu}^{\mathrm{sph}}(v)$. It
belongs to the complementary series if 
\[
0<\mu<\frac{1}{4},
\]
and to the principal series if 
\[
\mu\ge\frac{1}{4}.
\]

The discrete-series part is 
\begin{equation}
\mathcal{H}^{\mathrm{d.s.}}:=\widehat{\bigoplus}_{l\in2\mathbb{N}^{*}}\left(\mathcal{H}_{l}^{\mathrm{d.s.},+}\oplus\mathcal{H}_{-l}^{\mathrm{d.s.},-}\right),\label{eq:def_H_ds}
\end{equation}
where 
\begin{equation}
\mathcal{H}_{l}^{\mathrm{d.s.},+}:=\widehat{\mathrm{Span}}\left\{ N_{+}^{k}v\ ;\ v\in\mathcal{H}_{\mathrm{hol}^{+}}(l),\ k\in\mathbb{N}\right\} ,\label{eq:def_H_ds_plus_full}
\end{equation}
and 
\begin{equation}
\mathcal{H}_{-l}^{\mathrm{d.s.},-}:=\widehat{\mathrm{Span}}\left\{ N_{-}^{k}v\ ;\ v\in\mathcal{H}_{\mathrm{hol}^{-}}(-l),\ k\in\mathbb{N}\right\} .\label{eq:def_H_ds_minus_full}
\end{equation}
Each nonzero vector in $\mathcal{H}_{\mathrm{hol}^{+}}(l)$, respectively
in $\mathcal{H}_{\mathrm{hol}^{-}}(-l)$, generates one irreducible
holomorphic, respectively anti-holomorphic, discrete-series representation.
The multiplicity is 
\[
m_{l}=\dim\mathcal{H}_{\mathrm{hol}^{+}}(l)=\dim\mathcal{H}_{\mathrm{hol}^{-}}(-l).
\]
\end{thm}

\end{cBoxB}

\section{\protect\label{sec:Computation-with-the}Intertwining the hyperbolic
and elliptic generators}

\subsection{\protect\label{subsec:Spectrum-of-the}The quantum harmonic oscillator}

In Section~\ref{sec:Principal-series}, we will use standard facts
about the harmonic oscillator and the metaplectic representation.
We recall here the notation and normalizations.

On $L^{2}(\mathbb{R}_{x})$, we first define the operators on $\mathcal{S}(\mathbb{R})$
\begin{equation}
\hat{x}:=\mathcal{M}_{x},\qquad\hat{\xi}:=-i\frac{\partial}{\partial x}.\label{eq:def_p_hat}
\end{equation}
Thus 
\begin{equation}
[\hat{x},\hat{\xi}]=i\mathrm{Id}.\label{eq:commut_x_xi}
\end{equation}
We define the annihilation and creation operators by 
\begin{equation}
\tilde{a}^{-}:=\frac{1}{\sqrt{2}}\left(\hat{x}+i\hat{\xi}\right),\qquad\tilde{a}^{+}:=\frac{1}{\sqrt{2}}\left(\hat{x}-i\hat{\xi}\right),\label{eq:def_a_pm_tilde}
\end{equation}
and set 
\begin{equation}
\tilde{A}:=\tilde{a}^{+}\tilde{a}^{-},\qquad\hat{H}:=\frac{1}{2}\left(\hat{x}^{2}+\hat{\xi}^{2}\right)=\tilde{A}+\frac{1}{2}.\label{eq:def_A_H}
\end{equation}
The operator $\hat{H}$ is the quantum harmonic oscillator. The operators
$\tilde{a}^{-}$ and $\tilde{a}^{+}$ satisfy 
\begin{equation}
[\tilde{a}^{-},\tilde{a}^{+}]=\mathrm{Id},\label{eq:commut_a_minus_a_plus}
\end{equation}
and $\tilde{a}^{+}$ is the adjoint of $\tilde{a}^{-}$.

The operator $\tilde{A}$ is self-adjoint, positive, and has simple
discrete spectrum 
\[
\mathrm{Spec}(\tilde{A})=\mathbb{N}.
\]
We denote by $(\vartheta_{n})_{n\in\mathbb{N}}$ the normalized Hermite
basis of eigenvectors: 
\begin{equation}
\tilde{A}\vartheta_{n}=n\vartheta_{n},\qquad n\in\mathbb{N}.\label{eq:def_A}
\end{equation}
The normalization is chosen so that 
\begin{equation}
\tilde{a}^{+}\vartheta_{n}=\sqrt{n+1}\,\vartheta_{n+1},\qquad\tilde{a}^{-}\vartheta_{n}=\sqrt{n}\,\vartheta_{n-1},\label{eq:def_a+_a-}
\end{equation}
with the convention $\tilde{a}^{-}\vartheta_{0}=0$.

Finally, we shall use repeatedly the commutation relations 
\begin{equation}
[\tilde{A},\tilde{a}^{+}]\eq{\ref{eq:commut_a_minus_a_plus}}\tilde{a}^{+},\qquad[\tilde{A},\tilde{a}^{-}]=-\tilde{a}^{-}.\label{eq:rel}
\end{equation}

\subsection{Intertwining}

On $L^{2}(\mathbb{R})$, we introduce the quadratic operators 
\begin{equation}
\hat{P}:=\frac{i}{2}\left((\tilde{a}^{+})^{2}-(\tilde{a}^{-})^{2}\right)\eq{\ref{eq:def_a_pm_tilde}}\frac{1}{2}(\hat{x}\hat{\xi}+\hat{\xi}\hat{x})\eq{\ref{eq:def_p_hat}}-i\left(x\frac{\partial}{\partial x}+\frac{1}{2}\right),\label{eq:def_P}
\end{equation}
and 
\begin{equation}
\hat{Q}:=\frac{1}{2}\left((\tilde{a}^{-})^{2}+(\tilde{a}^{+})^{2}\right)=\frac{1}{2}(\hat{x}^{2}-\hat{\xi}^{2}).\label{eq:def_Q}
\end{equation}
These are the quadratic generators which will intertwine the elliptic
harmonic oscillator with the hyperbolic generator.

The operator $\hat{Q}$ is essentially self-adjoint on $\mathcal{S}(\mathbb{R})$.
Hence, by spectral calculus, $e^{\pm\alpha\hat{Q}}$ is a closed densely
defined operator for every $\alpha\in\mathbb{R}$. Since $\hat{Q}$
is unbounded above and below, these exponentials are not bounded in
general; the following statement specifies an explicit dense domain
on which the conjugation formulas are meaningful.

The abstract exponentiation of quadratic Hamiltonians is standard
in the metaplectic representation; see for instance Howe~\cite{Howe1988},
Folland~\cite[Chap.~5]{folland-88}, and Nelson~\cite{nelson1959}.
The explicit one-dimensional formulation below, using Gaussian test
functions with complex width, is taken from \cite[Lemma~14, p.~279]{fred-PreQ-06}.

\begin{cBoxB}{}

\begin{thm}[Complex oscillator intertwining]
\label{thm:For-,-the} \cite[Lemma~14, p.~279]{fred-PreQ-06}. Let
$\beta\in\mathbb{C}$ satisfy 
\[
0<\Re\beta<\frac{\pi}{2},
\]
and define 
\[
D_{\beta}:=\operatorname{span}\left\{ P(x)\exp\left(-\frac{1}{\tan\beta}\frac{x^{2}}{2}\right)\ ;\ P\text{ polynomial}\right\} \subset L^{2}(\mathbb{R}).
\]
Then $D_{\beta}$ is dense in $L^{2}(\mathbb{R})$. Moreover, if $\alpha\in(0,\frac{\pi}{2})$
and 
\[
0<\Re\beta<\frac{\pi}{2}-\alpha,
\]
then 
\[
e^{\alpha\hat{Q}}:D_{\beta}\longrightarrow D_{\beta+\alpha}
\]
is a bijection, with inverse $e^{-\alpha\hat{Q}}$. On $D_{\beta}$,
one has 
\begin{equation}
e^{-\alpha\hat{Q}}\hat{x}e^{\alpha\hat{Q}}=(\cos\alpha)\hat{x}-i(\sin\alpha)\hat{\xi},\label{eq:x_x}
\end{equation}
\begin{equation}
e^{-\alpha\hat{Q}}\hat{\xi}e^{\alpha\hat{Q}}=(\cos\alpha)\hat{\xi}-i(\sin\alpha)\hat{x},\label{eq:xi_xi}
\end{equation}
\begin{equation}
e^{-\alpha\hat{Q}}\hat{H}e^{\alpha\hat{Q}}=\cos(2\alpha)\hat{H}-i\sin(2\alpha)\hat{P},\label{eq:H_P}
\end{equation}
and 
\begin{equation}
e^{-\alpha\hat{Q}}\hat{P}e^{\alpha\hat{Q}}=\cos(2\alpha)\hat{P}-i\sin(2\alpha)\hat{H}.\label{eq:P_H}
\end{equation}
\end{thm}

\end{cBoxB}

\begin{proof}
Set 
\[
\mathscr{S}:=\left\{ \beta\in\mathbb{C}\ ;\ 0<\Re\beta<\frac{\pi}{2}\right\} .
\]
Since 
\[
\hat{Q}\eq{\ref{eq:def_Q}}\frac{1}{2}\left(x^{2}+\frac{\partial^{2}}{\partial x^{2}}\right),
\]
we can compute explicitly its action on Gaussian functions. Choose
a holomorphic branch of $(\sin\beta)^{-1/2}$ on $\mathscr{S}$, and
define 
\[
\phi_{\beta}(x):=(\sin\beta)^{-1/2}\exp\left(-\frac{\cot\beta}{2}x^{2}\right).
\]
For $\beta\in\mathscr{S}$, one has $\Re(\cot\beta)>0$, hence $\phi_{\beta}\in L^{2}(\mathbb{R})$.

A direct computation gives 
\[
\hat{Q}\phi_{\beta}=\frac{1}{2}\left(-\cot\beta+\frac{x^{2}}{\sin^{2}\beta}\right)\phi_{\beta}.
\]
On the other hand, 
\[
\partial_{\beta}\phi_{\beta}=\frac{1}{2}\left(-\cot\beta+\frac{x^{2}}{\sin^{2}\beta}\right)\phi_{\beta}.
\]
Thus 
\[
\partial_{\beta}\phi_{\beta}=\hat{Q}\phi_{\beta}.
\]
Consequently, if $\beta,\beta+\alpha\in\mathscr{S}$, then 
\[
e^{\alpha\hat{Q}}\phi_{\beta}=\phi_{\beta+\alpha}.
\]

We now show that the polynomial factor is preserved. Let 
\[
u(t,x)=P_{t}(x)\phi_{\beta+t}(x),
\]
where $P_{t}$ is a polynomial to be determined. Write 
\[
q_{t}:=\cot(\beta+t).
\]
Using 
\[
\partial_{x}(P\phi_{\beta+t})=(P'-q_{t}xP)\phi_{\beta+t},
\]
one obtains 
\[
\hat{Q}(P\phi_{\beta+t})=P\,\hat{Q}\phi_{\beta+t}+\frac{1}{2}\left(P''-2q_{t}xP'\right)\phi_{\beta+t}.
\]
Since $\partial_{t}\phi_{\beta+t}=\hat{Q}\phi_{\beta+t}$, the equation
\[
\partial_{t}u=\hat{Q}u
\]
is equivalent to 
\[
\partial_{t}P_{t}=\frac{1}{2}\left(P_{t}''-2\cot(\beta+t)\,xP_{t}'\right).
\]
This is a linear differential equation on the finite-dimensional space
of polynomials of degree at most $\deg P_{0}$. Hence its solution
remains a polynomial for all $t$ such that $\beta+t\in\mathscr{S}$.
Therefore 
\[
e^{\alpha\hat{Q}}D_{\beta}\subset D_{\beta+\alpha}.
\]
Applying the same argument to $-\alpha$ gives the reverse inclusion,
and hence 
\[
e^{\alpha\hat{Q}}:D_{\beta}\longrightarrow D_{\beta+\alpha}
\]
is a bijection, with inverse $e^{-\alpha\hat{Q}}$.

It remains to prove that $D_{\beta}$ is dense. Write 
\[
\cot\beta=a+ib,\qquad a>0.
\]
Then 
\[
D_{\beta}=M_{b}S_{a}(D_{\pi/4}),
\]
where 
\[
(M_{b}u)(x)=e^{-i\frac{b}{2}x^{2}}u(x),\qquad(S_{a}u)(x)=a^{1/4}u(\sqrt{a}\,x).
\]
The operators $M_{b}$ and $S_{a}$ are unitary on $L^{2}(\mathbb{R})$.
Thus it is enough to prove that $D_{\pi/4}$ is dense. But 
\[
D_{\pi/4}=\operatorname{span}\{P(x)e^{-x^{2}/2}\ ;\ P\text{ polynomial}\}=\operatorname{span}\{\vartheta_{n}\ ;\ n\in\mathbb{N}\},
\]
because the Hermite polynomials form a basis of the space of polynomials.
Since $(\vartheta_{n})_{n\in\mathbb{N}}$ is an orthonormal basis
of $L^{2}(\mathbb{R})$, $D_{\pi/4}$ is dense, and hence so is $D_{\beta}$.

We now prove the conjugation identities. Fix $\alpha\in(0,\pi/2)$,
and consider the common invariant core 
\[
D_{*}:=\operatorname{span}\left\{ D_{\beta}\ ;\ 0<\Re\beta<\frac{\pi}{2}-\alpha\right\} .
\]
For any operator $B$ preserving $D_{*}$, one has on $D_{*}$ 
\[
\frac{d}{ds}\left(e^{-s\hat{Q}}Be^{s\hat{Q}}\right)=-e^{-s\hat{Q}}[\hat{Q},B]e^{s\hat{Q}}.
\]
Using $[\hat{x},\hat{\xi}]=i\,\mathrm{Id}$, we get 
\[
[\hat{Q},\hat{x}]=i\hat{\xi},\qquad[\hat{Q},\hat{\xi}]=i\hat{x}.
\]
Set 
\[
\hat{x}(s):=e^{-s\hat{Q}}\hat{x}e^{s\hat{Q}},\qquad\hat{\xi}(s):=e^{-s\hat{Q}}\hat{\xi}e^{s\hat{Q}}.
\]
Then 
\[
\hat{x}'(s)=-i\hat{\xi}(s),\qquad\hat{\xi}'(s)=-i\hat{x}(s).
\]
Thus 
\[
\hat{x}''(s)=-\hat{x}(s),\qquad\hat{x}(0)=\hat{x},\qquad\hat{x}'(0)=-i\hat{\xi},
\]
and therefore 
\[
\hat{x}(s)=\cos s\,\hat{x}-i\sin s\,\hat{\xi}.
\]
Similarly, 
\[
\hat{\xi}(s)=\cos s\,\hat{\xi}-i\sin s\,\hat{x}.
\]
Taking $s=\alpha$ proves \eqref{eq:x_x} and \eqref{eq:xi_xi}.

Finally, since 
\[
\hat{H}=\frac{1}{2}(\hat{x}^{2}+\hat{\xi}^{2}),\qquad\hat{P}=\frac{1}{2}(\hat{x}\hat{\xi}+\hat{\xi}\hat{x}),
\]
we obtain 
\[
e^{-\alpha\hat{Q}}\hat{H}e^{\alpha\hat{Q}}=\frac{1}{2}\left(\hat{x}(\alpha)^{2}+\hat{\xi}(\alpha)^{2}\right)=\cos(2\alpha)\hat{H}-i\sin(2\alpha)\hat{P},
\]
and similarly 
\[
e^{-\alpha\hat{Q}}\hat{P}e^{\alpha\hat{Q}}=\cos(2\alpha)\hat{P}-i\sin(2\alpha)\hat{H}.
\]
This concludes the proof. 
\end{proof}

\subsection{\protect\label{sec:Useful-corollary}Useful corollary}

For later use, we express the operators $\tilde{a}^{+},\tilde{a}^{-}$
and $\tilde{A}=\tilde{a}^{+}\tilde{a}^{-}$ in the Hermite basis $(\vartheta_{n})_{n\in\mathbb{N}}$,
see \eqref{eq:def_A} and \eqref{eq:def_a+_a-}. Let $(e_{n})_{n\in\mathbb{N}}$
be the canonical basis of $\ell^{2}(\mathbb{N})$, 
\begin{equation}
e_{n}:=(0,\ldots,\underbrace{1}_{n},0,\ldots),\label{eq:def_e_k}
\end{equation}
and define the unitary isomorphism 
\begin{equation}
F:L^{2}(\mathbb{R}_{x})\longrightarrow\ell^{2}(\mathbb{N}),\qquad Fu:=\sum_{n\in\mathbb{N}}e_{n}\,\langle\vartheta_{n}\mid u\rangle_{L^{2}(\mathbb{R})}.\label{eq:def_F}
\end{equation}
We set 
\begin{equation}
A:=F\tilde{A}F^{-1},\qquad a^{-}:=F\tilde{a}^{-}F^{-1},\qquad a^{+}:=F\tilde{a}^{+}F^{-1}.\label{eq:def_A_a+-}
\end{equation}
Thus the operators $A,a^{-},a^{+}$ on $\ell^{2}(\mathbb{N})$ are
characterized by 
\begin{equation}
Ae_{n}\eq{\ref{eq:def_A}}ne_{n},\qquad a^{-}e_{n}\eq{\ref{eq:def_a+_a-}}\sqrt{n}\,e_{n-1},\qquad a^{+}e_{n}\eq{\ref{eq:def_a+_a-}}\sqrt{n+1}\,e_{n+1},\label{eq:A_in_l2}
\end{equation}
with the convention $a^{-}e_{0}=0$.

In the next corollary of Theorem~\ref{thm:For-,-the}, $x^{n}$ denotes
the monomial function on $\mathbb{R}$, and $\delta^{(n)}$ denotes
the $n$-th derivative of the Dirac measure. We use the notation $\langle\cdot\mid\cdot\rangle_{L^{2}(\mathbb{R})}$
also for the natural extension of the $L^{2}(dx)$-pairing between
distributions and test functions, whenever this pairing is well-defined.
Thus, for every $\varphi\in\mathcal{S}(\mathbb{R})$, 
\[
\langle x^{n}\mid\varphi\rangle_{L^{2}(\mathbb{R})}=\int_{\mathbb{R}}x^{n}\varphi(x)\,dx,\qquad\langle\delta^{(n)}\mid\varphi\rangle_{L^{2}(\mathbb{R})}=(-1)^{n}\varphi^{(n)}(0).
\]

\begin{cBoxB}{}

\begin{cor}
\label{cor:With-setting-} Applying Theorem~\ref{thm:For-,-the}
with $\alpha=\frac{\pi}{4}$, and using the inverse map for the opposite
direction, define 
\[
D_{+}:=\operatorname{span}\{D_{\beta}\ ;\ 0<\Re\beta<\tfrac{\pi}{4}\},\qquad D_{-}:=\operatorname{span}\{D_{\beta}\ ;\ \tfrac{\pi}{4}<\Re\beta<\tfrac{\pi}{2}\}.
\]
Set 
\begin{align}
T_{+} & :=(2\pi)^{-1/4}F\circ e^{-\frac{\pi}{4}\hat{Q}}:D_{-}\longrightarrow F(D_{+})\subset\ell^{2}(\mathbb{N}),\label{eq:def_Tplus_norm}\\
T_{-} & :=(2\pi)^{1/4}F\circ e^{\frac{\pi}{4}\hat{Q}}:D_{+}\longrightarrow F(D_{-})\subset\ell^{2}(\mathbb{N}).\label{eq:def_Tminus_norm}
\end{align}
Then $T_{\pm}$ are bijections onto their images. On their natural
domains, they satisfy 
\begin{equation}
T_{\pm}\left(\mp\left(x\frac{d}{dx}+\frac{1}{2}\right)\right)T_{\pm}^{-1}=-\left(A+\frac{1}{2}\right),\label{eq:H_conj}
\end{equation}
and 
\begin{equation}
T_{\pm}\hat{x}T_{\pm}^{-1}=a^{\pm},\qquad T_{\pm}\frac{d}{dx}T_{\pm}^{-1}=\pm a^{\mp}.\label{eq:x_conj}
\end{equation}
Moreover, the following weak expansions hold: 
\begin{align}
T_{+} & =\left(T_{-}^{-1}\right)^{\dagger}=\sum_{n\in\mathbb{N}}e_{n}\left\langle \frac{(-1)^{n}}{\sqrt{n!}}\delta^{(n)}\middle|\,\cdot\,\right\rangle _{L^{2}(\mathbb{R})},\label{eq:T_+_T_-}\\
T_{-} & =\left(T_{+}^{-1}\right)^{\dagger}=\sum_{n\in\mathbb{N}}e_{n}\left\langle \frac{1}{\sqrt{n!}}x^{n}\middle|\,\cdot\,\right\rangle _{L^{2}(\mathbb{R})}.\label{eq:T_-_T_+}
\end{align}
Here $\delta^{(n)}$ and $x^{n}$ are understood through the extended
$L^{2}(dx)$-pairing with test functions. More precisely, \eqref{eq:T_+_T_-}
is used on $D_{-}$, and \eqref{eq:T_-_T_+} is used on $D_{+}$. 
\end{cor}

\end{cBoxB}

\begin{proof}
The definition of $T_{\pm}$, and their bijectivity onto their images,
follow from Theorem~\ref{thm:For-,-the}, applied with $\alpha=\pi/4$,
together with the inverse map for the opposite direction.

Let 
\[
B:=x\frac{d}{dx}+\frac{1}{2}.
\]
Since $\hat{P}=-iB$, the identities \eqref{eq:P_H} at $\alpha=\pi/4$,
transported by $F$, give 
\[
T_{+}\hat{P}T_{+}^{-1}=-i\left(A+\frac{1}{2}\right),\qquad T_{-}\hat{P}T_{-}^{-1}=i\left(A+\frac{1}{2}\right).
\]
Equivalently, 
\[
T_{\pm}(\mp B)T_{\pm}^{-1}=-\left(A+\frac{1}{2}\right),
\]
which proves \eqref{eq:H_conj}.

Similarly, using \eqref{eq:x_x} and \eqref{eq:xi_xi} at $\alpha=\pi/4$,
we get 
\[
e^{-\frac{\pi}{4}\hat{Q}}\hat{x}e^{\frac{\pi}{4}\hat{Q}}=\frac{1}{\sqrt{2}}(\hat{x}-i\hat{\xi})=\tilde{a}^{+},
\]
and 
\[
e^{-\frac{\pi}{4}\hat{Q}}\hat{\xi}e^{\frac{\pi}{4}\hat{Q}}=\frac{1}{\sqrt{2}}(\hat{\xi}-i\hat{x})=-i\tilde{a}^{-}.
\]
After applying $F$, this gives 
\[
T_{+}\hat{x}T_{+}^{-1}=a^{+},\qquad T_{+}\frac{d}{dx}T_{+}^{-1}=a^{-},
\]
because $d/dx=i\hat{\xi}$. The inverse direction gives 
\[
T_{-}\hat{x}T_{-}^{-1}=a^{-},\qquad T_{-}\frac{d}{dx}T_{-}^{-1}=-a^{+}.
\]
This proves \eqref{eq:x_conj}.

It remains to identify the weak expansions. Define, on $D_{-}$, 
\[
(S_{+}u)_{n}:=\left\langle \frac{(-1)^{n}}{\sqrt{n!}}\delta^{(n)}\middle|u\right\rangle _{L^{2}(\mathbb{R})}=\frac{u^{(n)}(0)}{\sqrt{n!}},
\]
and, on $D_{+}$, 
\[
(S_{-}u)_{n}:=\left\langle \frac{x^{n}}{\sqrt{n!}}\middle|u\right\rangle _{L^{2}(\mathbb{R})}=\frac{1}{\sqrt{n!}}\int_{\mathbb{R}}x^{n}u(x)\,dx.
\]
Using $x\delta^{(n)}=-n\delta^{(n-1)}$ and integration by parts,
one checks that 
\[
S_{+}x=a^{+}S_{+},\qquad S_{+}\frac{d}{dx}=a^{-}S_{+},
\]
and 
\[
S_{-}x=a^{-}S_{-},\qquad S_{-}\frac{d}{dx}=-a^{+}S_{-}.
\]
Thus $S_{\pm}$ satisfy the same intertwining relations as $T_{\pm}$.

We compare the two maps on one Gaussian. Set 
\[
g_{\beta}(x):=\exp\left(-\frac{\cot\beta}{2}x^{2}\right).
\]
From the proof of Theorem~\ref{thm:For-,-the}, 
\[
e^{-\frac{\pi}{4}\hat{Q}}g_{\beta}=\left(\frac{\sin\beta}{\sin(\beta-\frac{\pi}{4})}\right)^{1/2}g_{\beta-\frac{\pi}{4}}.
\]
Taking the scalar product with 
\[
\vartheta_{0}(x)=\pi^{-1/4}e^{-x^{2}/2}
\]
gives 
\[
\left\langle \vartheta_{0}\middle|e^{-\frac{\pi}{4}\hat{Q}}g_{\beta}\right\rangle _{L^{2}(\mathbb{R})}=(2\pi)^{1/4}.
\]
Since $T_{+}=(2\pi)^{-1/4}Fe^{-\frac{\pi}{4}\hat{Q}}$, the zeroth
coefficient of $T_{+}g_{\beta}$ is $1$. This is also the zeroth
coefficient of $S_{+}g_{\beta}$, because $g_{\beta}(0)=1$. Since
both $T_{+}$ and $S_{+}$ intertwine multiplication by $x$ with
$a^{+}$, the equality extends from $g_{\beta}$ to $P(x)g_{\beta}$.
Hence 
\[
T_{+}=S_{+}\qquad\text{on }D_{-}.
\]

The same argument, using the normalization $T_{-}=(2\pi)^{1/4}Fe^{\frac{\pi}{4}\hat{Q}}$,
gives 
\[
T_{-}=S_{-}\qquad\text{on }D_{+}.
\]
Therefore 
\[
T_{+}=\sum_{n\in\mathbb{N}}e_{n}\left\langle \frac{(-1)^{n}}{\sqrt{n!}}\delta^{(n)}\middle|\,\cdot\,\right\rangle _{L^{2}(\mathbb{R})},
\]
and 
\[
T_{-}=\sum_{n\in\mathbb{N}}e_{n}\left\langle \frac{x^{n}}{\sqrt{n!}}\middle|\,\cdot\,\right\rangle _{L^{2}(\mathbb{R})}.
\]

Finally, since $e^{\pm\frac{\pi}{4}\hat{Q}}$ is self-adjoint and
$F$ is unitary, the chosen normalization gives 
\[
T_{+}=\left(T_{-}^{-1}\right)^{\dagger},\qquad T_{-}=\left(T_{+}^{-1}\right)^{\dagger}.
\]
This proves \eqref{eq:T_+_T_-} and \eqref{eq:T_-_T_+}. 
\end{proof}

\section{\protect\label{sec:Principal-series}Spherical principal series}

In this section, we fix one irreducible spherical representation 
\[
\mathcal{H}_{\mu}^{\mathrm{sph}}\subset L^{2}(M),\qquad\mu\ge\frac{1}{4},
\]
generated by a nonzero $K$-invariant vector. We use the orthonormal
basis $(\varphi_{k})_{k\in\mathbb{Z}}$ defined in \eqref{eq:def_phi_k},
with 
\[
\varphi_{k}\in\mathcal{H}_{\mu,2k},\qquad\varphi_{0}\in\mathcal{H}_{\mu,0}.
\]
Thus $\mathcal{H}_{\mu}^{\mathrm{sph}}$ is identified with $\ell^{2}(\mathbb{Z})$.

The aim of this section is to conjugate the hyperbolic generator $X$
on this representation to a harmonic-oscillator model. We proceed
in several steps. First, we use a unitary map 
\[
\mathcal{U}_{1}:\mathcal{H}_{\mu}^{\mathrm{sph}}\longrightarrow L^{2}(S_{\theta}^{1})
\]
to realize the representation in the compact picture. Then we pass
to two real-line models by unitary maps $\mathcal{U}_{\pm}$, and
finally apply the non-unitary intertwiners $T_{\pm}$ and diagonal
rescalings $G_{\pm}$. The resulting diagram is 
\begin{equation}
\xymatrix{\mathcal{H}_{\mu}^{\mathrm{sph}}\ar[r]_{\eqref{eq:def_U}}^{\mathcal{U}_{1}} & L^{2}(S_{\theta}^{1})\ar[r]_{\eqref{eq:def_Upm}}^{\mathcal{U}_{+}}\ar[rd]_{\eqref{eq:def_Upm}}^{\mathcal{U}_{-}} & L_{+}^{2}(\mathbb{R})\ar[r]_{\eqref{eq:def_Tplus_norm}}^{T_{+}} & \ell_{+}^{2}(\mathbb{N})\ar[r]_{\eqref{eq:def_G+-}}^{G_{+}} & \ell_{+}^{2}(\mathbb{N})\\
 &  & L_{-}^{2}(\mathbb{R})\ar[r]_{\eqref{eq:def_Tminus_norm}}^{T_{-}} & \ell_{-}^{2}(\mathbb{N})\ar[r]_{\eqref{eq:def_G+-}}^{G_{-}} & \ell_{-}^{2}(\mathbb{N}).
}
\label{eq:def_P_V_k}
\end{equation}
We will write 
\[
T'_{\pm}\eq{\eqref{eq:def_T3}}T_{\pm}\mathcal{U}_{\pm},\qquad\mathcal{T}_{\pm}\eq{\eqref{eq:def_cal_T_+-}}G_{\pm}T_{\pm}\mathcal{U}_{\pm},\quad\mathcal{T}\eq{\eqref{eq:def_cal_T}}\begin{pmatrix}\mathcal{T}_{+}\\
\mathcal{T}_{-}
\end{pmatrix},
\]
and finally 
\[
\mathbb{T}_{\mathrm{sph}}\eq{\eqref{eq:def_TT}}\mathcal{T}\circ\mathcal{U}_{1},
\]
that is the spherical part of \noun{$\mathbb{T}$} introduced in (\ref{eq:TT})
of the main theorem.

We use the parametrization 
\[
\mu=\lambda^{2}+\frac{1}{4},\qquad\lambda\ge0,
\]
and set 
\[
b_{\pm}:=-\frac{1}{2}\pm i\lambda.
\]
Thus 
\[
b_{+}b_{-}=\mu.
\]
For short, we write 
\begin{equation}
b:=b_{+}=-\frac{1}{2}+i\lambda.\label{eq:def_b}
\end{equation}

\subsection{\protect\label{subsec:Differential-operators-in}Differential operators
in $L^{2}(S^{1})$}

The next lemma realizes the spherical principal series in the compact
picture. Using the inverse Fourier transform $\mathcal{U}_{1}$, the
operators $X,U,S$ become first-order differential operators on $L^{2}(S_{\theta}^{1})$,
where $S_{\theta}^{1}=\mathbb{R}_{\theta}/2\pi\mathbb{Z}$.
\begin{rem}
The variable $\theta$ in this compact model is normalized so that
the geometric $K$-type $2k$ is represented by $e^{ik\theta}$. Thus
$\widetilde{\Theta}=-2i\partial_{\theta}$ in this model.
\end{rem}

\begin{cBoxB}{}

\begin{lem}
\label{lem:operators_compact_picture} Let 
\begin{equation}
\mathcal{U}_{1}:\begin{cases}
\mathcal{H}_{\mu}^{\mathrm{sph}} & \longrightarrow L^{2}(S_{\theta}^{1};d\theta),\\
\varphi_{k} & \longmapsto\psi_{k}(\theta):=\dfrac{1}{\sqrt{2\pi}}e^{ik\theta}.
\end{cases}\label{eq:def_U}
\end{equation}
Then 
\begin{equation}
\tilde{X}:=\mathcal{U}_{1}X\mathcal{U}_{1}^{-1}=\cos\theta\,\frac{d}{d\theta}+b\sin\theta,\label{eq:expression_X_sl2R}
\end{equation}
\begin{equation}
\tilde{U}:=\mathcal{U}_{1}U\mathcal{U}_{1}^{-1}=(\sin\theta-1)\frac{d}{d\theta}-b\cos\theta,\label{eq:expression_U_sl2R}
\end{equation}
and 
\begin{equation}
\tilde{S}:=\mathcal{U}_{1}S\mathcal{U}_{1}^{-1}=(\sin\theta+1)\frac{d}{d\theta}-b\cos\theta.\label{eq:expression_S_sl2R}
\end{equation}
\end{lem}

\end{cBoxB}

\begin{proof}
We first compute the matrix elements of $X$ in the basis $(\varphi_{k})_{k\in\mathbb{Z}}$.
Since 
\[
X\eq{\ref{eq:X_from_R}}\frac{1}{2}(N_{+}+N_{-}),
\]
and using \eqref{eq:action_N_pm_phi_k}, the only nonzero matrix elements
are 
\begin{align*}
\langle\varphi_{k+1}\mid X\varphi_{k}\rangle & =\frac{1}{2}\langle\varphi_{k+1}\mid N_{+}\varphi_{k}\rangle=\frac{i}{2}c_{+,2k}\eq{\ref{eq:def_cn_dn}}\frac{i}{2}\left(k+\frac{1}{2}-i\lambda\right),\\
\langle\varphi_{k-1}\mid X\varphi_{k}\rangle & =\frac{1}{2}\langle\varphi_{k-1}\mid N_{-}\varphi_{k}\rangle=\frac{i}{2}\overline{c_{-,2k}}\eq{\ref{eq:def_cn_dn}}\frac{i}{2}\left(k-\frac{1}{2}+i\lambda\right).
\end{align*}

On the other hand, the operator 
\[
\cos\theta\,\frac{d}{d\theta}+b\sin\theta=\frac{1}{2}(e^{i\theta}+e^{-i\theta})\frac{d}{d\theta}-\frac{ib}{2}(e^{i\theta}-e^{-i\theta})
\]
has, in the Fourier basis $(\psi_{k})_{k\in\mathbb{Z}}$, the nonzero
matrix elements 
\[
\langle\psi_{k+1}\mid\tilde{X}\psi_{k}\rangle=\frac{i}{2}(k-b)\eq{\ref{eq:def_b}}\frac{i}{2}\left(k+\frac{1}{2}-i\lambda\right),
\]
and 
\[
\langle\psi_{k-1}\mid\tilde{X}\psi_{k}\rangle=\frac{i}{2}(k+b)\eq{\ref{eq:def_b}}\frac{i}{2}\left(k-\frac{1}{2}+i\lambda\right).
\]
Thus $\tilde{X}$ has the same matrix elements as $X$, which proves
\eqref{eq:expression_X_sl2R}.

From \eqref{eq:expression_X_sl2R}, one obtains 
\[
\tilde{N}_{+}=e^{i\theta}\frac{d}{d\theta}-ib\,e^{i\theta},\qquad\tilde{N}_{-}=e^{-i\theta}\frac{d}{d\theta}+ib\,e^{-i\theta}.
\]
Moreover, 
\[
\tilde{\Theta}:=\mathcal{U}_{1}\Theta\mathcal{U}_{1}^{-1}=-2i\frac{d}{d\theta},
\]
because $\Theta\varphi_{k}=2k\varphi_{k}$. Therefore, using \eqref{eq:X_from_R},
\begin{align*}
\tilde{U} & =\frac{i}{2}\left(\tilde{N}_{-}-\tilde{N}_{+}-\tilde{\Theta}\right)\\
 & =\frac{i}{2}\left(e^{-i\theta}\frac{d}{d\theta}+ibe^{-i\theta}-e^{i\theta}\frac{d}{d\theta}+ibe^{i\theta}+2i\frac{d}{d\theta}\right)\\
 & =(\sin\theta-1)\frac{d}{d\theta}-b\cos\theta.
\end{align*}
This proves \eqref{eq:expression_U_sl2R}. Similarly, 
\begin{align*}
\tilde{S} & =\frac{i}{2}\left(\tilde{N}_{-}-\tilde{N}_{+}+\tilde{\Theta}\right)\\
 & =(\sin\theta+1)\frac{d}{d\theta}-b\cos\theta,
\end{align*}
which proves \eqref{eq:expression_S_sl2R}. 
\end{proof}

\subsubsection{\protect\label{subsec:Stereographic-projection-and}Projective charts
and normal form for the spherical series}

The vector field on $S^{1}$
\[
\cos\theta\,\frac{d}{d\theta}
\]
appearing in \eqref{eq:expression_X_sl2R} has two fixed points, 
\[
\theta=\frac{\pi}{2},\qquad\theta=-\frac{\pi}{2}.
\]
We introduce two affine charts on $S^{1}$, obtained by removing one
of these two fixed points: 
\begin{equation}
I_{+}:=\left(-\frac{\pi}{2},\frac{3\pi}{2}\right),\qquad I_{-}:=\left(-\frac{3\pi}{2},\frac{\pi}{2}\right).\label{eq:def_I_pm}
\end{equation}
Thus $I_{+}$ contains the attracting fixed point $\theta=\pi/2$
and excludes the repelling fixed point $\theta=-\pi/2$, whereas $I_{-}$
contains the repelling fixed point and excludes the attracting one.

The next lemma shows that, after a projective change of variables
on each chart $I_{\pm}$, the operator $\tilde{X}$ in \eqref{eq:expression_X_sl2R}
is conjugated to a simple normal form. See Figure~\ref{fig:projective}.

\begin{figure}[h]
\begin{centering}
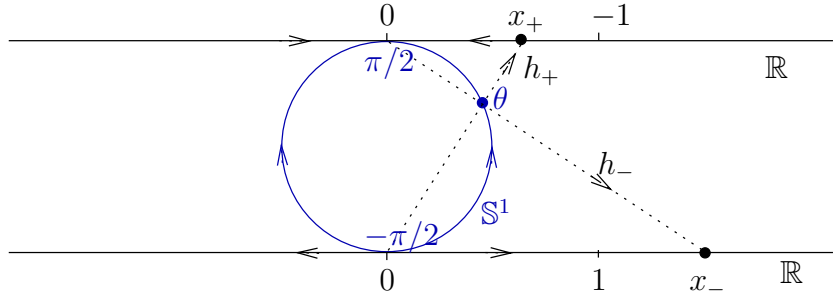
\par\end{centering}
\caption{\protect\label{fig:projective} The projective charts $h_{\pm}:I_{\pm}\to\mathbb{R}$
defined in \eqref{eq:def_h+-}. The vector field $\cos\theta\,\frac{d}{d\theta}$
on $S^{1}$ is transformed into a linear contracting vector field
in the $h_{+}$-chart, and into a linear expanding vector field in
the $h_{-}$-chart.}
\end{figure}

\begin{cBoxB}{}

\begin{lem}
\label{lem:For-each-sign} For each sign $\pm$, define the diffeomorphism
\begin{equation}
h_{\pm}:\begin{cases}
I_{\pm} & \longrightarrow\mathbb{R},\\
\theta & \longmapsto x=\tan\left(\frac{1}{2}\left(\theta\mp\frac{\pi}{2}\right)\right),
\end{cases}\label{eq:def_h+-}
\end{equation}
and the operator 
\begin{equation}
\mathcal{U}_{\pm}:=\sqrt{2}\,(1+x^{2})^{b}h_{\pm}^{-\circ}:L^{2}(S_{\theta}^{1};d\theta)\longrightarrow L^{2}(\mathbb{R};dx),\label{eq:def_Upm}
\end{equation}
where $h_{\pm}^{-\circ}u:=u\circ h_{\pm}^{-1}$. Then $\mathcal{U}_{\pm}$
is unitary and 
\begin{equation}
\mathcal{U}_{\pm}\tilde{X}\mathcal{U}_{\pm}^{-1}=\mp\left(x\frac{\partial}{\partial x}-b\right).\label{eq:def_Xtilde}
\end{equation}
Moreover, 
\begin{equation}
\mathcal{U}_{\pm}\tilde{U}\mathcal{U}_{\pm}^{-1}=\begin{cases}
-x^{2}\dfrac{\partial}{\partial x}+2bx, & \text{for }+,\\[0.5em]
-\dfrac{\partial}{\partial x}, & \text{for }-,
\end{cases}\label{eq:U_after_conj}
\end{equation}
and 
\begin{equation}
\mathcal{U}_{\pm}\tilde{S}\mathcal{U}_{\pm}^{-1}=\begin{cases}
\dfrac{\partial}{\partial x}, & \text{for }+,\\[0.5em]
x^{2}\dfrac{\partial}{\partial x}-2bx, & \text{for }-.
\end{cases}\label{eq:S_after_conj}
\end{equation}
\end{lem}

\end{cBoxB}

\begin{proof}
Set 
\[
\Upsilon(x):=1+x^{2}.
\]
Let 
\[
\theta'=\theta\mp\frac{\pi}{2},\qquad x=\tan\frac{\theta'}{2}.
\]
By the tangent half-angle formulas, 
\[
\sin\theta'=\frac{2x}{1+x^{2}}=2x\Upsilon^{-1},\qquad\cos\theta'=\frac{1-x^{2}}{1+x^{2}}=1-2x^{2}\Upsilon^{-1},
\]
and 
\[
\frac{d\theta'}{dx}=2\Upsilon^{-1}.
\]
Equivalently, 
\begin{equation}
\mp\cos\theta=2x\Upsilon^{-1},\qquad\pm\sin\theta=1-2x^{2}\Upsilon^{-1},\qquad\frac{dx}{d\theta}=\frac{\Upsilon}{2}.\label{eq:half_angle_pm}
\end{equation}

We first check unitarity. Since $b=-\frac{1}{2}+i\lambda$, one has
\[
\left|(1+x^{2})^{b}\right|^{2}=(1+x^{2})^{-1}=\Upsilon^{-1}.
\]
Also 
\[
d\theta=2\Upsilon^{-1}\,dx.
\]
Therefore, for $u\in L^{2}(S^{1};d\theta)$, 
\[
\|\mathcal{U}_{\pm}u\|_{L^{2}(\mathbb{R};dx)}^{2}=\int_{\mathbb{R}}2\Upsilon^{-1}|u(h_{\pm}^{-1}(x))|^{2}\,dx=\int_{S^{1}}|u(\theta)|^{2}\,d\theta.
\]
Thus $\mathcal{U}_{\pm}$ is unitary.

We now compute the conjugated vector fields. Before multiplication
by $\sqrt{2}\,\Upsilon^{b}$, the constant factor $\sqrt{2}$ being
irrelevant for conjugation, \eqref{eq:expression_X_sl2R} and \eqref{eq:half_angle_pm}
give 
\begin{align}
h_{\pm}^{-\circ}\tilde{X}h_{\pm}^{\circ} & =\cos\theta\,\frac{dx}{d\theta}\frac{\partial}{\partial x}+b\sin\theta\nonumber \\
 & =\mp x\frac{\partial}{\partial x}\pm b\left(1-2x^{2}\Upsilon^{-1}\right).\label{eq:X_chart_before_weight}
\end{align}
For any first-order operator $a(x)\partial_{x}+c(x)$, conjugation
by multiplication with $\Upsilon^{b}$ gives 
\[
\Upsilon^{b}(a\partial_{x}+c)\Upsilon^{-b}=a\partial_{x}+c-b\,a\,\frac{\Upsilon'}{\Upsilon}.
\]
Applying this to \eqref{eq:X_chart_before_weight}, and using $\Upsilon'=2x$,
we obtain 
\[
\mathcal{U}_{\pm}\tilde{X}\mathcal{U}_{\pm}^{-1}=\mp x\frac{\partial}{\partial x}\pm b=\mp\left(x\frac{\partial}{\partial x}-b\right),
\]
which proves \eqref{eq:def_Xtilde}.

The same computation gives, before multiplication by $\Upsilon^{b}$,
\[
h_{+}^{-\circ}\tilde{U}h_{+}^{\circ}=-x^{2}\frac{\partial}{\partial x}+2bx\Upsilon^{-1},\qquad h_{-}^{-\circ}\tilde{U}h_{-}^{\circ}=-\frac{\partial}{\partial x}-2bx\Upsilon^{-1},
\]
and 
\[
h_{+}^{-\circ}\tilde{S}h_{+}^{\circ}=\frac{\partial}{\partial x}+2bx\Upsilon^{-1},\qquad h_{-}^{-\circ}\tilde{S}h_{-}^{\circ}=x^{2}\frac{\partial}{\partial x}-2bx\Upsilon^{-1}.
\]
Conjugating again by $\Upsilon^{b}$ yields 
\[
\mathcal{U}_{+}\tilde{U}\mathcal{U}_{+}^{-1}=-x^{2}\frac{\partial}{\partial x}+2bx,\qquad\mathcal{U}_{-}\tilde{U}\mathcal{U}_{-}^{-1}=-\frac{\partial}{\partial x},
\]
and 
\[
\mathcal{U}_{+}\tilde{S}\mathcal{U}_{+}^{-1}=\frac{\partial}{\partial x},\qquad\mathcal{U}_{-}\tilde{S}\mathcal{U}_{-}^{-1}=x^{2}\frac{\partial}{\partial x}-2bx.
\]
This proves \eqref{eq:U_after_conj} and \eqref{eq:S_after_conj}. 
\end{proof}

\subsection{\protect\label{subsec:Conjugation-to-the}Conjugation to the quantum
harmonic oscillator}

We distinguish two copies of $\ell^{2}(\mathbb{N})$, denoted by $\ell_{+}^{2}(\mathbb{N})$
and $\ell_{-}^{2}(\mathbb{N})$. As in \eqref{eq:def_e_k}, we write
$e_{n}^{\pm}\in\ell_{\pm}^{2}(\mathbb{N})$ for the canonical basis
vectors, so that 
\begin{equation}
\mathrm{Id}_{\ell_{\pm}^{2}(\mathbb{N})}=\sum_{n\in\mathbb{N}}e_{n}^{\pm}\langle e_{n}^{\pm}\mid\,\cdot\,\rangle_{\ell_{\pm}^{2}(\mathbb{N})}.\label{eq:Id_l2}
\end{equation}

Recall that $T_{+}$ and $T_{-}$ were defined in \eqref{eq:def_Tplus_norm}
and \eqref{eq:def_Tminus_norm}. Thus 
\[
T_{\pm}:D_{\mp}\subset L^{2}(\mathbb{R})\longrightarrow F(D_{\pm})\subset\ell_{\pm}^{2}(\mathbb{N})
\]
are bijections onto their images. We also use the unitary maps $\mathcal{U}_{\pm}$
defined in \eqref{eq:def_Upm}.

\begin{cBoxB}{}

\begin{lem}
\label{lem:Tprime_pm} The operators 
\begin{equation}
T'_{\pm}:=T_{\pm}\mathcal{U}_{\pm}\qquad:\mathcal{U}_{\pm}^{-1}D_{\mp}\subset L^{2}(S_{\theta}^{1};d\theta)\longrightarrow F(D_{\pm})\subset\ell_{\pm}^{2}(\mathbb{N})\label{eq:def_T3}
\end{equation}
are injective, densely defined operators with dense range. They satisfy
\begin{equation}
T'_{\pm}\tilde{X}(T'_{\pm})^{-1}=-A+b_{\pm},\label{eq:X_conj}
\end{equation}
\begin{equation}
T'_{\pm}\tilde{U}(T'_{\pm})^{-1}=\begin{cases}
-a^{+}(A-2b_{+}), & \text{for }+,\\[0.4em]
a^{+}, & \text{for }-,
\end{cases}\label{eq:T3_U}
\end{equation}
and 
\begin{equation}
T'_{\pm}\tilde{S}(T'_{\pm})^{-1}=\begin{cases}
a^{-}, & \text{for }+,\\[0.4em]
-(A-2b_{-})a^{-}, & \text{for }-.
\end{cases}\label{eq:T3_S}
\end{equation}
\end{lem}

\end{cBoxB}

\begin{proof}
Since $\mathcal{U}_{\pm}$ is unitary and $T_{\pm}$ is injective
with dense domain and dense range by Corollary~\ref{cor:With-setting-},
the same holds for $T'_{\pm}=T_{\pm}\mathcal{U}_{\pm}$ with domain
$\mathcal{U}_{\pm}^{-1}D_{\mp}$.

For $X$, using \eqref{eq:def_Xtilde} and \eqref{eq:H_conj}, we
get 
\[
T'_{\pm}\tilde{X}(T'_{\pm})^{-1}=T_{\pm}\left(\mp\left(x\frac{\partial}{\partial x}-b\right)\right)T_{\pm}^{-1}=-A+b_{\pm}.
\]

In the case $(+)$, using \eqref{eq:U_after_conj}, \eqref{eq:x_conj}
and \eqref{eq:H_conj}, 
\[
T'_{+}\tilde{U}(T'_{+})^{-1}=T_{+}x\left(-x\frac{\partial}{\partial x}+2b\right)T_{+}^{-1}=a^{+}(-A+2b_{+})=-a^{+}(A-2b_{+}).
\]
Also, 
\[
T'_{+}\tilde{S}(T'_{+})^{-1}=T_{+}\frac{\partial}{\partial x}T_{+}^{-1}=a^{-}.
\]

In the case $(-)$, we have 
\[
T'_{-}\tilde{U}(T'_{-})^{-1}=-T_{-}\frac{\partial}{\partial x}T_{-}^{-1}=a^{+}.
\]
Finally, 
\[
T'_{-}\tilde{S}(T'_{-})^{-1}=T_{-}x\left(x\frac{\partial}{\partial x}-2b\right)T_{-}^{-1}=a^{-}(-A-1-2b).
\]
Since $b=b_{+}$, one has 
\[
-1-2b_{+}=1+2b_{-},
\]
hence 
\[
a^{-}(-A-1-2b)=a^{-}(-A+2b_{-}+1).
\]
Using $[A,a^{-}]=-a^{-}$, equivalently $a^{-}f(A)=f(A+1)a^{-}$,
we obtain 
\[
a^{-}(-A+2b_{-}+1)=-(A-2b_{-})a^{-}.
\]
This proves \eqref{eq:T3_S}. 
\end{proof}
We shall need two elementary estimates on the weak coefficients of
$T'_{\pm}\psi_{k}$, where $\psi_{k}$ was defined in \eqref{eq:def_U}
and $T'_{\pm}$ in \eqref{eq:def_T3}. The first estimate concerns
the limit $n\to+\infty$ with $k$ fixed; the second concerns the
limit $|k|\to+\infty$ with $n$ fixed. The coefficients are understood
through the weak expansions \eqref{eq:T_+_T_-} and \eqref{eq:T_-_T_+}.

\begin{cBoxB}{}

\begin{lem}
\label{lem:estimate_Tprime_pm} For fixed $k\in\mathbb{Z}$, as $n\to+\infty$,
\begin{equation}
\left|\langle e_{n}^{\pm}\mid T'_{\pm}\psi_{k}\rangle\right|=O\left(n^{|k|-\frac{1}{2}}(n!)^{\pm\frac{1}{2}}\right).\label{eq:estimate_decay}
\end{equation}
Moreover, 
\begin{equation}
\left|\langle e_{n}^{\pm}\mid(T_{\pm}'{}^{-1})^{\dagger}\psi_{k}\rangle\right|=O\left(n^{|k|-\frac{1}{2}}(n!)^{\mp\frac{1}{2}}\right).\label{eq:estimate_Tinvdagger_decay}
\end{equation}
\end{lem}

\end{cBoxB}

\begin{proof}
We first compute $\mathcal{U}_{\pm}\psi_{k}$. From \eqref{eq:def_h+-},
\[
h_{\pm}^{-1}(x)=2\arctan x\pm\frac{\pi}{2}.
\]
Since 
\[
\frac{1+ix}{1-ix}=e^{2i\arctan x},
\]
and since $\mathcal{U}_{\pm}$ contains the factor $\sqrt{2}$, we
get 
\begin{equation}
(\mathcal{U}_{\pm}\psi_{k})(x)=\frac{e^{\pm ik\pi/2}}{\sqrt{\pi}}(1+x^{2})^{b}\left(\frac{1+ix}{1-ix}\right)^{k}.\label{eq:Upsi_pm_common}
\end{equation}

For the $+$ branch, using the weak expansion \eqref{eq:T_+_T_-},
\[
\langle e_{n}^{+}\mid T'_{+}\psi_{k}\rangle=\frac{1}{\sqrt{n!}}\left(\frac{d^{n}}{dx^{n}}(\mathcal{U}_{+}\psi_{k})(x)\right)_{x=0}.
\]
By \eqref{eq:Upsi_pm_common}, 
\begin{equation}
(\mathcal{U}_{+}\psi_{k})(x)=\frac{e^{ik\pi/2}}{\sqrt{\pi}}(1+ix)^{b+k}(1-ix)^{b-k}.\label{eq:U+_Psi_k}
\end{equation}
Set 
\[
f_{k}(x):=(1+ix)^{b+k}(1-ix)^{b-k}.
\]
The function $f_{k}$ is holomorphic for $|x|<1$, and its only singularities
on $|x|=1$ are at $x=\pm i$. Near these points, 
\[
f_{k}(x)=C_{k,+}(x-i)^{b+k}(1+O(x-i)),
\]
and 
\[
f_{k}(x)=C_{k,-}(x+i)^{b-k}(1+O(x+i)).
\]
Since 
\[
\Re(b+k)=k-\frac{1}{2},\qquad\Re(b-k)=-k-\frac{1}{2},
\]
the Darboux estimate for Taylor coefficients of functions with algebraic
singularities gives, for fixed $k$, 
\[
[x^{n}]f_{k}(x)=O\left(n^{|k|-\frac{1}{2}}\right).
\]
Therefore 
\[
\left|\left(\frac{d^{n}}{dx^{n}}(\mathcal{U}_{+}\psi_{k})\right)_{x=0}\right|=O\left(n!\,n^{|k|-\frac{1}{2}}\right),
\]
and hence 
\[
\left|\langle e_{n}^{+}\mid T'_{+}\psi_{k}\rangle\right|=O\left(n^{|k|-\frac{1}{2}}(n!)^{1/2}\right).
\]

For the $-$ branch, using \eqref{eq:T_-_T_+}, 
\[
\langle e_{n}^{-}\mid T'_{-}\psi_{k}\rangle=\frac{1}{\sqrt{n!}}\int_{\mathbb{R}}x^{n}(\mathcal{U}_{-}\psi_{k})(x)\,dx,
\]
where the integral is understood by analytic continuation of the corresponding
beta integrals. From \eqref{eq:Upsi_pm_common}, if $k\ge0$, 
\[
(\mathcal{U}_{-}\psi_{k})(x)=C_{k}(1+x^{2})^{b-k}(1+ix)^{2k},
\]
whereas if $k\le0$, 
\[
(\mathcal{U}_{-}\psi_{k})(x)=C_{k}(1+x^{2})^{b-|k|}(1-ix)^{2|k|}.
\]
Thus in both cases, 
\[
(\mathcal{U}_{-}\psi_{k})(x)=\sum_{j=0}^{2|k|}c_{j,k}x^{j}(1+x^{2})^{b-|k|}.
\]
Hence 
\[
\langle e_{n}^{-}\mid T'_{-}\psi_{k}\rangle=\frac{1}{\sqrt{n!}}\sum_{j=0}^{2|k|}c_{j,k}I_{n+j},
\]
with 
\[
I_{m}:=\int_{\mathbb{R}}x^{m}(1+x^{2})^{b-|k|}\,dx,
\]
again defined by analytic continuation.

For $m=2r$, the beta integral gives 
\[
I_{2r}=B\left(r+\frac{1}{2},-b+|k|-r-\frac{1}{2}\right).
\]
Since $b=-\frac{1}{2}+i\lambda$, 
\[
I_{2r}=\frac{\Gamma\left(r+\frac{1}{2}\right)\Gamma\left(|k|-i\lambda-r\right)}{\Gamma\left(\frac{1}{2}-i\lambda+|k|\right)}.
\]
Using the reflection formula on the second Gamma factor and Stirling's
formula, we obtain 
\[
I_{2r}=O\left(r^{|k|-\frac{1}{2}}\right).
\]
Odd indices vanish by parity in the present expression for $I_{m}$.
Thus the same bound holds trivially for them. Therefore, uniformly
for $0\le j\le2|k|$, 
\[
I_{n+j}=O\left(n^{|k|-\frac{1}{2}}\right),
\]
and consequently 
\[
\left|\langle e_{n}^{-}\mid T'_{-}\psi_{k}\rangle\right|=O\left(n^{|k|-\frac{1}{2}}(n!)^{-1/2}\right).
\]
This proves \eqref{eq:estimate_decay}.

Finally, 
\[
(T'_{+}{}^{-1})^{\dagger}=(T_{+}^{-1})^{\dagger}\mathcal{U}_{+}=T_{-}\mathcal{U}_{+},
\]
and 
\[
(T'_{-}{}^{-1})^{\dagger}=(T_{-}^{-1})^{\dagger}\mathcal{U}_{-}=T_{+}\mathcal{U}_{-}.
\]
Thus the same estimates as above apply with $T_{+}$ and $T_{-}$
exchanged. This reverses the factorial exponent and gives \eqref{eq:estimate_Tinvdagger_decay}. 
\end{proof}
\begin{cBoxB}{}

\begin{lem}
\label{lem:estimate_Tprime_pm_absk} For fixed $n\in\mathbb{N}$,
as $|k|\to+\infty$, 
\begin{equation}
\left|\langle e_{n}^{\pm}\mid T'_{\pm}\psi_{k}\rangle\right|=O\left(|k|^{n}\right).\label{eq:estimate_absk_infty}
\end{equation}
Moreover, 
\begin{equation}
\left|\left\langle e_{n}^{\pm}\middle|(T'_{\pm}{}^{-1})^{\dagger}\psi_{k}\right\rangle \right|=O\left(|k|^{n}\right).\label{eq:estimate_Tinvdagger_absk}
\end{equation}
\end{lem}

\end{cBoxB}

\begin{proof}
We use the expression obtained in \eqref{eq:Upsi_pm_common}: 
\[
(\mathcal{U}_{\pm}\psi_{k})(x)=\frac{e^{\pm ik\pi/2}}{\sqrt{\pi}}(1+x^{2})^{b}\left(\frac{1+ix}{1-ix}\right)^{k}.
\]

We first consider the $+$ branch. By the weak expansion \eqref{eq:T_+_T_-},
\[
\langle e_{n}^{+}\mid T'_{+}\psi_{k}\rangle=\frac{1}{\sqrt{n!}}\left(\frac{d^{n}}{dx^{n}}(\mathcal{U}_{+}\psi_{k})(x)\right)_{x=0}.
\]
Since 
\[
(\mathcal{U}_{+}\psi_{k})(x)=\frac{e^{ik\pi/2}}{\sqrt{\pi}}(1+x^{2})^{b}e^{2ik\arctan x},
\]
and $n$ is fixed, repeated differentiation gives 
\[
\frac{d^{n}}{dx^{n}}e^{2ik\arctan x}=P_{n}(x,k)e^{2ik\arctan x},
\]
where $P_{n}(x,k)$ is polynomial in $k$ of degree at most $n$,
with coefficients smooth in $x$. Therefore 
\[
\left(\frac{d^{n}}{dx^{n}}(\mathcal{U}_{+}\psi_{k})(x)\right)_{x=0}=O(|k|^{n}),
\]
and hence 
\[
\left|\langle e_{n}^{+}\mid T'_{+}\psi_{k}\rangle\right|=O(|k|^{n}).
\]

We now consider the $-$ branch. By \eqref{eq:T_-_T_+}, 
\[
\langle e_{n}^{-}\mid T'_{-}\psi_{k}\rangle=\frac{1}{\sqrt{n!}}\int_{\mathbb{R}}x^{n}(\mathcal{U}_{-}\psi_{k})(x)\,dx,
\]
where the integral is understood by analytic continuation of the corresponding
beta integral. Using \eqref{eq:Upsi_pm_common}, 
\[
(\mathcal{U}_{-}\psi_{k})(x)=\frac{e^{-ik\pi/2}}{\sqrt{\pi}}(1+x^{2})^{b}\left(\frac{1+ix}{1-ix}\right)^{k}.
\]

Assume first that $k\ge0$. Expanding 
\[
x^{n}=(2i)^{-n}\bigl((1+ix)-(1-ix)\bigr)^{n},
\]
we obtain 
\[
\langle e_{n}^{-}\mid T'_{-}\psi_{k}\rangle=C_{n}e^{-ik\pi/2}\sum_{j=0}^{n}(-1)^{n-j}\binom{n}{j}I_{j,k},
\]
where $C_{n}=(\sqrt{\pi n!})^{-1}(2i)^{-n}$ and 
\[
I_{j,k}:=\int_{\mathbb{R}}(1+ix)^{b+k+j}(1-ix)^{b-k+n-j}\,dx.
\]
By analytic continuation of the beta integral, 
\[
I_{j,k}=C_{n,\lambda,j}\,\frac{\Gamma\left(k+j+\frac{1}{2}+i\lambda\right)}{\Gamma\left(k-n+j+\frac{1}{2}-i\lambda\right)}.
\]
Since $n$ and $j$ are fixed, Stirling's formula gives 
\[
I_{j,k}=O(k^{n}),\qquad k\to+\infty.
\]
Thus 
\[
\left|\langle e_{n}^{-}\mid T'_{-}\psi_{k}\rangle\right|=O(k^{n}),\qquad k\to+\infty.
\]

If $k\le0$, set $\kappa:=|k|=-k$. Then 
\[
(\mathcal{U}_{-}\psi_{k})(x)=\frac{e^{i\kappa\pi/2}}{\sqrt{\pi}}(1+x^{2})^{b}\left(\frac{1-ix}{1+ix}\right)^{\kappa}.
\]
The same expansion of $x^{n}$ gives 
\[
\langle e_{n}^{-}\mid T'_{-}\psi_{k}\rangle=C_{n}e^{i\kappa\pi/2}\sum_{j=0}^{n}(-1)^{n-j}\binom{n}{j}J_{j,\kappa},
\]
where 
\[
J_{j,\kappa}:=\int_{\mathbb{R}}(1+ix)^{b-\kappa+j}(1-ix)^{b+\kappa+n-j}\,dx.
\]
Again, by analytic continuation, 
\[
J_{j,\kappa}=C'_{n,\lambda,j}\,\frac{\Gamma\left(\kappa+n-j+\frac{1}{2}+i\lambda\right)}{\Gamma\left(\kappa-j+\frac{1}{2}-i\lambda\right)}.
\]
Stirling's formula gives 
\[
J_{j,\kappa}=O(\kappa^{n}),\qquad\kappa\to+\infty.
\]
Hence 
\[
\left|\langle e_{n}^{-}\mid T'_{-}\psi_{k}\rangle\right|=O(|k|^{n}),\qquad k\to-\infty.
\]
Combining the two cases proves \eqref{eq:estimate_absk_infty}.

Finally, \eqref{eq:estimate_Tinvdagger_absk} follows from the weak
relations 
\[
T_{+}=(T_{-}^{-1})^{\dagger},\qquad T_{-}=(T_{+}^{-1})^{\dagger},
\]
together with the same argument applied with $T_{+}$ and $T_{-}$
exchanged. 
\end{proof}
\begin{cBoxB}{}

\begin{cor}
\label{cor:weak_extensions_sobolev} For each $n\in\mathbb{N}$ and
every $\epsilon>0$, the weak extensions 
\[
(T'_{\pm})^{\dagger}e_{n}^{\pm},\qquad(T'_{\pm})^{-1}e_{n}^{\pm}
\]
belong to 
\[
H^{-n-\frac{1}{2}-\epsilon}(S^{1}).
\]
More precisely, they are first defined on trigonometric polynomials
by 
\[
\left\langle (T'_{\pm})^{\dagger}e_{n}^{\pm}\mid\psi_{k}\right\rangle :=\left\langle e_{n}^{\pm}\mid T'_{\pm}\psi_{k}\right\rangle ,
\]
and 
\[
\left\langle (T'_{\pm})^{-1}e_{n}^{\pm}\mid\psi_{k}\right\rangle :=\left\langle e_{n}^{\pm}\mid\left((T'_{\pm})^{-1}\right)^{\dagger}\psi_{k}\right\rangle .
\]
\end{cor}

\end{cBoxB}

\begin{proof}
Set, in the weak sense above, 
\[
u_{n,\pm}:=(T'_{\pm})^{\dagger}e_{n}^{\pm},\qquad v_{n,\pm}:=(T'_{\pm})^{-1}e_{n}^{\pm}.
\]
By Lemma~\ref{lem:estimate_Tprime_pm_absk}, for fixed $n$, 
\[
\left|\left\langle u_{n,\pm}\mid\psi_{k}\right\rangle \right|=\left|\left\langle e_{n}^{\pm}\mid T'_{\pm}\psi_{k}\right\rangle \right|\eq{\ref{eq:estimate_absk_infty}}O(|k|^{n}),\qquad|k|\to+\infty,
\]
and similarly 
\[
\left|\left\langle v_{n,\pm}\mid\psi_{k}\right\rangle \right|=\left|\left\langle e_{n}^{\pm}\mid\left((T'_{\pm})^{-1}\right)^{\dagger}\psi_{k}\right\rangle \right|\eq{\ref{eq:estimate_Tinvdagger_absk}}O(|k|^{n}).
\]

Using the Fourier characterization of Sobolev spaces on $S^{1}$,
a distribution $w$ belongs to $H^{s}(S^{1})$ if and only if 
\[
\sum_{k\in\mathbb{Z}}\langle k\rangle^{2s}\left|\left\langle w\mid\psi_{k}\right\rangle \right|^{2}<\infty,\qquad\langle k\rangle:=(1+k^{2})^{1/2}.
\]
For $w=u_{n,\pm}$ or $w=v_{n,\pm}$, the summand is 
\[
O\left(\langle k\rangle^{2(s+n)}\right).
\]
The series 
\[
\sum_{k\in\mathbb{Z}}\langle k\rangle^{2(s+n)}
\]
converges precisely when $2(s+n)<-1$, that is, 
\[
s<-n-\frac{1}{2}.
\]
Therefore 
\[
u_{n,\pm},\,v_{n,\pm}\in H^{s}(S^{1})\qquad\text{for every }s<-n-\frac{1}{2}.
\]
Taking $s=-n-\frac{1}{2}-\epsilon$ proves the result. 
\end{proof}

\subsection{Additional conjugation: algebraic gauge}

By \eqref{eq:estimate_decay}, for fixed $k$, 
\[
\left|\langle e_{n}^{-}\mid T'_{-}\psi_{k}\rangle\right|
\]
decays factorially as $n\to+\infty$, and therefore $T'_{-}\psi_{k}\in\ell^{2}(\mathbb{N})$.
This is not the case for the $+$-branch, where 
\[
\left|\langle e_{n}^{+}\mid T'_{+}\psi_{k}\rangle\right|
\]
has factorial growth.

We now introduce a diagonal algebraic gauge. It will be used to conjugate
the expressions in Lemma~\ref{lem:Tprime_pm}, in order to obtain
good decay for both signs and a more symmetric form of the operators.

\begin{cBoxB}{}

\begin{lem}[Diagonal algebraic gauge]
\label{lem:diagonal_algebraic_gauge} Let $\alpha\in\mathbb{R}$
and $c\in\mathbb{C}\setminus(-\mathbb{N})$. Choose once and for all
a determination of $(c+n)^{\alpha}$ for $n\in\mathbb{N}$, obtained
from a single branch of $z^{\alpha}$ on a domain avoiding $0$ and
containing all points $c+n$. For $n\in\mathbb{N}$, set 
\[
t_{n}:=\prod_{j=0}^{n-1}(c+j)^{\alpha},
\]
with the convention $t_{0}=1$. Define the diagonal operator $G_{c,\alpha}$
on $c_{00}(\mathbb{N})$ by 
\begin{equation}
G_{c,\alpha}e_{n}:=t_{n}e_{n}.\label{eq:def_G_c_alpha}
\end{equation}
Then, on $c_{00}(\mathbb{N})$, 
\[
G_{c,\alpha}A(G_{c,\alpha})^{-1}=A,
\]
\[
G_{c,\alpha}a^{+}(G_{c,\alpha})^{-1}=a^{+}(A+c)^{\alpha},\qquad G_{c,\alpha}a^{-}(G_{c,\alpha})^{-1}=(A+c)^{-\alpha}a^{-}.
\]
Consequently, for every $\beta\in\mathbb{R}$, 
\[
G_{c,\alpha}\,a^{+}(A+c)^{\beta}\,(G_{c,\alpha})^{-1}=a^{+}(A+c)^{\beta+\alpha},
\]
and 
\[
G_{c,\alpha}\,(A+c)^{\beta}a^{-}\,(G_{c,\alpha})^{-1}=(A+c)^{\beta-\alpha}a^{-}.
\]
\end{lem}

\end{cBoxB}

\begin{proof}
All identities are checked on the canonical basis $(e_{n})_{n\in\mathbb{N}}$
of $c_{00}(\mathbb{N})$.

Since $G_{c,\alpha}$ is diagonal, it commutes with $A$, and hence
\[
G_{c,\alpha}A(G_{c,\alpha})^{-1}=A.
\]

Next, 
\[
G_{c,\alpha}a^{+}(G_{c,\alpha})^{-1}e_{n}=t_{n}^{-1}\sqrt{n+1}\,t_{n+1}e_{n+1}.
\]
Since 
\[
\frac{t_{n+1}}{t_{n}}=(c+n)^{\alpha},
\]
we obtain 
\[
G_{c,\alpha}a^{+}(G_{c,\alpha})^{-1}e_{n}=\sqrt{n+1}(c+n)^{\alpha}e_{n+1}=a^{+}(A+c)^{\alpha}e_{n}.
\]
Therefore 
\[
G_{c,\alpha}a^{+}(G_{c,\alpha})^{-1}=a^{+}(A+c)^{\alpha}.
\]

Similarly, 
\[
G_{c,\alpha}a^{-}(G_{c,\alpha})^{-1}e_{n}=t_{n}^{-1}\sqrt{n}\,t_{n-1}e_{n-1}.
\]
Since 
\[
\frac{t_{n-1}}{t_{n}}=(c+n-1)^{-\alpha},
\]
we get 
\[
G_{c,\alpha}a^{-}(G_{c,\alpha})^{-1}e_{n}=\sqrt{n}(c+n-1)^{-\alpha}e_{n-1}=(A+c)^{-\alpha}a^{-}e_{n}.
\]
Hence 
\[
G_{c,\alpha}a^{-}(G_{c,\alpha})^{-1}=(A+c)^{-\alpha}a^{-}.
\]

Finally, $G_{c,\alpha}$ commutes with every function of $A$. Thus
\[
G_{c,\alpha}(A+c)^{\beta}(G_{c,\alpha})^{-1}=(A+c)^{\beta}.
\]
Combining this with the two previous identities gives 
\[
G_{c,\alpha}\,a^{+}(A+c)^{\beta}\,(G_{c,\alpha})^{-1}=a^{+}(A+c)^{\alpha}(A+c)^{\beta}=a^{+}(A+c)^{\beta+\alpha},
\]
and 
\[
G_{c,\alpha}\,(A+c)^{\beta}a^{-}\,(G_{c,\alpha})^{-1}=(A+c)^{\beta}(A+c)^{-\alpha}a^{-}=(A+c)^{\beta-\alpha}a^{-}.
\]
This proves the lemma. 
\end{proof}
\begin{cBoxB}{}

\begin{lem}
\label{lem:gauge_estimate} For each sign $\pm$, set 
\[
c_{\pm}:=-2b_{\pm}.
\]
Then $\Re c_{\pm}=1$. For fixed $k$, as $n\to+\infty$, one has
\[
\left|\left\langle e_{n}^{\pm}\mid G_{c_{\pm},\alpha}T'_{\pm}\psi_{k}\right\rangle \right|=O\left(n^{|k|-\frac{1}{2}}(n!)^{\pm\frac{1}{2}+\alpha}\right).
\]
\end{lem}

\end{cBoxB}

\begin{proof}
By Stirling's formula, for fixed $c\notin-\mathbb{N}$, 
\[
|\Gamma(c+n)|=O\left(\Gamma(n+1)n^{\Re c-1}\right),\qquad n\to+\infty.
\]
Therefore, with 
\[
t_{n}=\left(\frac{\Gamma(c+n)}{\Gamma(c)}\right)^{\alpha},
\]
we have 
\begin{equation}
|t_{n}|=O\left((n!)^{\alpha}n^{\alpha(\Re c-1)}\right).\label{eq:estimate_tk}
\end{equation}

Since $G_{c,\alpha}e_{n}=t_{n}e_{n}$, we get 
\[
\left\langle e_{n}^{\pm}\mid G_{c,\alpha}T'_{\pm}\psi_{k}\right\rangle =t_{n}\left\langle e_{n}^{\pm}\mid T'_{\pm}\psi_{k}\right\rangle 
\]
up to complex conjugation depending on the scalar-product convention;
in any case, after taking absolute values, 
\[
\left|\left\langle e_{n}^{\pm}\mid G_{c,\alpha}T'_{\pm}\psi_{k}\right\rangle \right|=|t_{n}|\left|\left\langle e_{n}^{\pm}\mid T'_{\pm}\psi_{k}\right\rangle \right|.
\]
Using \eqref{eq:estimate_tk} and Lemma~\ref{lem:estimate_Tprime_pm},
we obtain 
\[
\left|\left\langle e_{n}^{\pm}\mid G_{c,\alpha}T'_{\pm}\psi_{k}\right\rangle \right|=O\left(n^{\alpha(\Re c-1)}(n!)^{\alpha}n^{|k|-\frac{1}{2}}(n!)^{\pm\frac{1}{2}}\right).
\]
Now take $c=c_{\pm}=-2b_{\pm}$. Since 
\[
b_{\pm}=-\frac{1}{2}\pm i\lambda,\qquad c_{\pm}=1\mp2i\lambda,
\]
we have $\Re c_{\pm}=1$. Hence the extra polynomial factor disappears,
and 
\[
\left|\left\langle e_{n}^{\pm}\mid G_{c_{\pm},\alpha}T'_{\pm}\psi_{k}\right\rangle \right|=O\left(n^{|k|-\frac{1}{2}}(n!)^{\pm\frac{1}{2}+\alpha}\right).
\]
\end{proof}
\begin{cBoxB}{}

\begin{cor}
\label{cor:Tcal_pm} Suppose $\lambda>0$. We define the diagonal
operators $G_{\pm}$ on $c_{00}(\mathbb{N})$ by 
\begin{equation}
G_{\pm}e_{n}:=t_{n}^{\pm}e_{n},\label{eq:def_G+-}
\end{equation}
where 
\[
t_{n}^{\pm}:=\prod_{j=0}^{n-1}(-2b_{\pm}+j)^{\mp1/2}=\left(\frac{\Gamma(-2b_{\pm}+n)}{\Gamma(-2b_{\pm})}\right)^{\mp1/2}.
\]
We set 
\begin{equation}
\mathcal{T}_{\pm}:=G_{\pm}T'_{\pm}\quad:\quad\mathcal{D}_{\pm}\subset L^{2}(S_{\theta}^{1};d\theta)\longrightarrow\mathcal{D}'_{\pm},\label{eq:def_cal_T_+-}
\end{equation}
where 
\[
\mathcal{D}_{\pm}:=\mathcal{U}_{\pm}^{-1}D_{\mp},\qquad\mathcal{D}'_{\pm}:=G_{\pm}\bigl(F(D_{\pm})\bigr).
\]
The operators $\mathcal{T}_{\pm}$ are injective and densely defined,
with range $\mathcal{D}'_{\pm}$. Moreover, 
\begin{equation}
\mathcal{T}_{\pm}\tilde{X}(\mathcal{T}_{\pm})^{-1}=-A+b_{\pm},\label{eq:X_conj-1}
\end{equation}
\begin{equation}
\mathcal{T}_{\pm}\tilde{U}(\mathcal{T}_{\pm})^{-1}=\mp a^{+}(A-2b_{\pm})^{1/2},\label{eq:T3_U-1}
\end{equation}
and 
\begin{equation}
\mathcal{T}_{\pm}\tilde{S}(\mathcal{T}_{\pm})^{-1}=\pm(A-2b_{\pm})^{1/2}a^{-}.\label{eq:T3_S-1}
\end{equation}
Finally, for every $k\in\mathbb{Z}$, there exists $C_{k}>0$ such
that, for all $n\in\mathbb{N}$, 
\begin{equation}
\left|\langle e_{n}^{\pm}\mid\mathcal{T}_{\pm}\psi_{k}\rangle\right|\le C_{k}\langle n\rangle^{|k|-\frac{1}{2}},\label{eq:Fourier_decay}
\end{equation}
and 
\begin{equation}
\left|\left\langle e_{n}^{\pm}\middle|(\mathcal{T}_{\pm}^{-1})^{\dagger}\psi_{k}\right\rangle \right|\le C_{k}\langle n\rangle^{|k|-\frac{1}{2}},\label{eq:Fourier_decay-1}
\end{equation}
where $\langle n\rangle:=(1+n^{2})^{1/2}$. 
\end{cor}

\end{cBoxB}

\begin{rem}
The spaces $\mathcal{D}'_{\pm}$ should be understood as algebraic
image spaces. In general, $\mathcal{D}'_{-}$ is not contained in
the standard Hilbert space $\ell^{2}(\mathbb{N})$. For the $+$-branch,
the gauge removes the factorial growth, since 
\[
|t_{n}^{+}|=O((n!)^{-1/2}).
\]
An adequate Hilbert-space realization will be constructed later. 
\end{rem}

\begin{proof}
We have 
\[
G_{\pm}=G_{c,\alpha}\qquad\text{with}\qquad c=-2b_{\pm},\quad\alpha=\mp\frac{1}{2}.
\]
Thus 
\[
\mathcal{T}_{\pm}\eq{\ref{eq:def_cal_T_+-}}G_{\pm}T'_{\pm}=G_{c,\alpha}T'_{\pm}.
\]

Since $T'_{\pm}$ is injective and densely defined by Lemma~\ref{lem:Tprime_pm},
and since $G_{\pm}$ is diagonal and invertible on $c_{00}(\mathbb{N})$,
the operator 
\[
\mathcal{T}_{\pm}=G_{\pm}T'_{\pm}
\]
is injective and densely defined, with range 
\[
\mathcal{D}'_{\pm}=G_{\pm}(F(D_{\pm})).
\]

The conjugation formula for $\tilde{X}$ is unchanged, because $G_{\pm}$
commutes with $A$. Hence 
\[
\mathcal{T}_{\pm}\tilde{X}\mathcal{T}_{\pm}^{-1}=G_{\pm}\bigl(T'_{\pm}\tilde{X}(T'_{\pm})^{-1}\bigr)G_{\pm}^{-1}\eq{\ref{eq:X_conj}}G_{\pm}(-A+b_{\pm})G_{\pm}^{-1}=-A+b_{\pm}.
\]
This proves \eqref{eq:X_conj-1}.

For $\tilde{U}$, we use \eqref{eq:T3_U} and Lemma~\ref{lem:diagonal_algebraic_gauge}.
In the $+$-case, 
\[
\mathcal{T}_{+}\tilde{U}\mathcal{T}_{+}^{-1}=G_{+}\bigl(-a^{+}(A-2b_{+})\bigr)G_{+}^{-1}.
\]
Here $c=-2b_{+}$ and $\alpha=-1/2$, hence 
\[
G_{+}\,a^{+}(A-2b_{+})\,G_{+}^{-1}=a^{+}(A-2b_{+})^{1/2}.
\]
Therefore 
\[
\mathcal{T}_{+}\tilde{U}\mathcal{T}_{+}^{-1}=-a^{+}(A-2b_{+})^{1/2}.
\]

In the $-$-case, 
\[
\mathcal{T}_{-}\tilde{U}\mathcal{T}_{-}^{-1}=G_{-}a^{+}G_{-}^{-1}.
\]
Here $c=-2b_{-}$ and $\alpha=1/2$, so 
\[
G_{-}a^{+}G_{-}^{-1}=a^{+}(A-2b_{-})^{1/2}.
\]
This proves \eqref{eq:T3_U-1}.

For $\tilde{S}$, we use \eqref{eq:T3_S}. In the $+$-case, 
\[
\mathcal{T}_{+}\tilde{S}\mathcal{T}_{+}^{-1}=G_{+}a^{-}G_{+}^{-1}.
\]
Since $c=-2b_{+}$ and $\alpha=-1/2$, 
\[
G_{+}a^{-}G_{+}^{-1}=(A-2b_{+})^{1/2}a^{-}.
\]
Thus 
\[
\mathcal{T}_{+}\tilde{S}\mathcal{T}_{+}^{-1}=(A-2b_{+})^{1/2}a^{-}.
\]

In the $-$-case, 
\[
\mathcal{T}_{-}\tilde{S}\mathcal{T}_{-}^{-1}=-G_{-}\bigl((A-2b_{-})a^{-}\bigr)G_{-}^{-1}.
\]
Since $c=-2b_{-}$, $\alpha=1/2$, and $\beta=1$, Lemma~\ref{lem:diagonal_algebraic_gauge}
gives 
\[
G_{-}(A-2b_{-})a^{-}G_{-}^{-1}=(A-2b_{-})^{1/2}a^{-}.
\]
Therefore 
\[
\mathcal{T}_{-}\tilde{S}\mathcal{T}_{-}^{-1}=-(A-2b_{-})^{1/2}a^{-}.
\]
This proves \eqref{eq:T3_S-1}.

It remains to prove the Fourier estimates. By Lemma~\ref{lem:gauge_estimate},
with $c=-2b_{\pm}$ and $\alpha=\mp1/2$, we get, for fixed $k$,
\[
\left|\langle e_{n}^{\pm}\mid G_{\pm}T'_{\pm}\psi_{k}\rangle\right|=O\left(n^{|k|-\frac{1}{2}}(n!)^{\pm\frac{1}{2}\mp\frac{1}{2}}\right)=O\left(n^{|k|-\frac{1}{2}}\right),\qquad n\to+\infty.
\]
After increasing the constant $C_{k}$, this gives 
\[
\left|\langle e_{n}^{\pm}\mid\mathcal{T}_{\pm}\psi_{k}\rangle\right|\le C_{k}\langle n\rangle^{|k|-\frac{1}{2}}
\]
for all $n\in\mathbb{N}$, proving \eqref{eq:Fourier_decay}.

Finally, 
\[
\mathcal{T}_{\pm}^{-1}=(T'_{\pm})^{-1}G_{\pm}^{-1},
\]
hence 
\[
(\mathcal{T}_{\pm}^{-1})^{\dagger}=(G_{\pm}^{-1})^{\dagger}\bigl((T'_{\pm})^{-1}\bigr)^{\dagger}.
\]
Since $G_{\pm}^{-1}e_{n}=(t_{n}^{\pm})^{-1}e_{n}$, the $n$-th coefficient
is multiplied, in absolute value, by $|t_{n}^{\pm}|^{-1}$. By Stirling,
\[
|t_{n}^{\pm}|^{-1}=O\left((n!)^{\pm\frac{1}{2}}\right),
\]
because $\Re(-2b_{\pm})=1$. Combining this with \eqref{eq:estimate_Tinvdagger_decay},
namely 
\[
\left|\left\langle e_{n}^{\pm}\middle|((T'_{\pm})^{-1})^{\dagger}\psi_{k}\right\rangle \right|=O\left(n^{|k|-\frac{1}{2}}(n!)^{\mp\frac{1}{2}}\right),
\]
we obtain 
\[
\left|\left\langle e_{n}^{\pm}\middle|(\mathcal{T}_{\pm}^{-1})^{\dagger}\psi_{k}\right\rangle \right|=O\left(n^{|k|-\frac{1}{2}}\right).
\]
Again increasing $C_{k}$ gives \eqref{eq:Fourier_decay-1} for all
$n\in\mathbb{N}$. 
\end{proof}

\subsection{Operators $\mathcal{T}_{\pm}$ and Lagrangian states}

Using the weak expansions of $T_{\pm}$ and $T_{\pm}^{-1}$ in \eqref{eq:T_+_T_-}--\eqref{eq:T_-_T_+},
we now obtain corresponding weak expansions for the operators $\mathcal{T}_{\pm}$
defined in \eqref{eq:def_cal_T_+-}. The resulting distributions on
$S^{1}$ will be interpreted as Lagrangian states associated with
the stable and unstable manifolds of the two fixed points.
\begin{rem}
As before on $\mathbb{R}$, we use the notation 
\[
\langle\cdot\mid\cdot\rangle_{L^{2}(S^{1};d\theta)}
\]
also for the natural extension of the $L^{2}(d\theta)$-pairing between
distributions and test functions, whenever this pairing is well-defined. 
\end{rem}

\begin{cBoxB}{}

\begin{lem}
\label{lem:Tcal_Lagrangian_states} Define 
\begin{equation}
\mathcal{S}_{n,+}:=\frac{(-1)^{n}\overline{t_{n}^{+}}}{\sqrt{n!}}\,(\mathcal{U}_{+})^{\dagger}\delta^{(n)},\qquad\mathcal{U}_{n,+}:=\frac{1}{t_{n}^{+}\sqrt{n!}}\,(\mathcal{U}_{+})^{\dagger}x^{n},\label{eq:def_Sk+_Uk+}
\end{equation}
and 
\begin{equation}
\mathcal{S}_{n,-}:=\frac{\overline{t_{n}^{-}}}{\sqrt{n!}}\,(\mathcal{U}_{-})^{\dagger}x^{n},\qquad\mathcal{U}_{n,-}:=\frac{(-1)^{n}}{t_{n}^{-}\sqrt{n!}}\,(\mathcal{U}_{-})^{\dagger}\delta^{(n)}.\label{eq:def_S_k-_Uk-}
\end{equation}
These are Lagrangian distributions on $S^{1}$, associated microlocally
with the stable and unstable manifolds of the two fixed points of
the vector field $\cos\theta\,\partial_{\theta}$; see Figure~\ref{fig:S_U_distributions}.

The operators $\mathcal{T}_{\pm}$ admit the weak expansions 
\begin{equation}
\mathcal{T}_{+}=\sum_{n\in\mathbb{N}}e_{n}^{+}\langle\mathcal{S}_{n,+}\mid\,\cdot\,\rangle_{L^{2}(S^{1})},\qquad\mathcal{T}_{+}^{-1}=\sum_{n\in\mathbb{N}}\mathcal{U}_{n,+}\langle e_{n}^{+}\mid\,\cdot\,\rangle,\label{eq:def_Tcal_+}
\end{equation}
and 
\begin{equation}
\mathcal{T}_{-}=\sum_{n\in\mathbb{N}}e_{n}^{-}\langle\mathcal{S}_{n,-}\mid\,\cdot\,\rangle_{L^{2}(S^{1})},\qquad\mathcal{T}_{-}^{-1}=\sum_{n\in\mathbb{N}}\mathcal{U}_{n,-}\langle e_{n}^{-}\mid\,\cdot\,\rangle.\label{eq:def_Tcal_-}
\end{equation}
\end{lem}

\end{cBoxB}

\begin{proof}
By definition, 
\[
\mathcal{T}_{\pm}\eq{\ref{eq:def_cal_T_+-}}G_{\pm}T'_{\pm}\eq{\ref{eq:def_T3}}G_{\pm}T_{\pm}\mathcal{U}_{\pm}.
\]

We first consider the $+$-branch. From \eqref{eq:T_+_T_-}, 
\[
T_{+}=\sum_{n\in\mathbb{N}}e_{n}^{+}\left\langle \frac{(-1)^{n}}{\sqrt{n!}}\delta^{(n)}\middle|\,\cdot\,\right\rangle _{L^{2}(\mathbb{R})}.
\]
Applying $G_{+}$ on the left and using $G_{+}e_{n}^{+}=t_{n}^{+}e_{n}^{+}$,
we get 
\[
G_{+}T_{+}=\sum_{n\in\mathbb{N}}e_{n}^{+}\left\langle \frac{(-1)^{n}t_{n}^{+}}{\sqrt{n!}}\delta^{(n)}\middle|\,\cdot\,\right\rangle _{L^{2}(\mathbb{R})}.
\]
Composing on the right with $\mathcal{U}_{+}$, and using the unitary
adjoint $(\mathcal{U}_{+})^{\dagger}$, gives 
\[
\mathcal{T}_{+}=G_{+}T_{+}\mathcal{U}_{+}=\sum_{n\in\mathbb{N}}e_{n}^{+}\left\langle \frac{(-1)^{n}t_{n}^{+}}{\sqrt{n!}}(\mathcal{U}_{+})^{\dagger}\delta^{(n)}\middle|\,\cdot\,\right\rangle _{L^{2}(S^{1})}.
\]
This is the first formula in \eqref{eq:def_Tcal_+}.

Similarly, 
\[
\mathcal{T}_{+}^{-1}=\mathcal{U}_{+}^{-1}T_{+}^{-1}G_{+}^{-1}=(\mathcal{U}_{+})^{\dagger}T_{+}^{-1}G_{+}^{-1}.
\]
Using the weak expansion of $T_{+}^{-1}$, equivalently the adjoint
of \eqref{eq:T_-_T_+}, and $G_{+}^{-1}e_{n}^{+}=(t_{n}^{+})^{-1}e_{n}^{+}$,
we obtain 
\[
\mathcal{T}_{+}^{-1}=\sum_{n\in\mathbb{N}}\frac{1}{t_{n}^{+}\sqrt{n!}}(\mathcal{U}_{+})^{\dagger}x^{n}\langle e_{n}^{+}\mid\,\cdot\,\rangle,
\]
which is the second formula in \eqref{eq:def_Tcal_+}.

The $-$-branch is identical. From \eqref{eq:T_-_T_+}, 
\[
T_{-}=\sum_{n\in\mathbb{N}}e_{n}^{-}\left\langle \frac{x^{n}}{\sqrt{n!}}\middle|\,\cdot\,\right\rangle _{L^{2}(\mathbb{R})}.
\]
Thus 
\[
\mathcal{T}_{-}=G_{-}T_{-}\mathcal{U}_{-}=\sum_{n\in\mathbb{N}}e_{n}^{-}\left\langle \frac{t_{n}^{-}}{\sqrt{n!}}(\mathcal{U}_{-})^{\dagger}x^{n}\middle|\,\cdot\,\right\rangle _{L^{2}(S^{1})}.
\]
This is the first formula in \eqref{eq:def_Tcal_-}.

Finally, 
\[
\mathcal{T}_{-}^{-1}=\mathcal{U}_{-}^{-1}T_{-}^{-1}G_{-}^{-1}=(\mathcal{U}_{-})^{\dagger}T_{-}^{-1}G_{-}^{-1}.
\]
Using the weak expansion of $T_{-}^{-1}$, equivalently the adjoint
of \eqref{eq:T_+_T_-}, together with $G_{-}^{-1}e_{n}^{-}=(t_{n}^{-})^{-1}e_{n}^{-}$,
we get 
\[
\mathcal{T}_{-}^{-1}=\sum_{n\in\mathbb{N}}\frac{(-1)^{n}}{t_{n}^{-}\sqrt{n!}}(\mathcal{U}_{-})^{\dagger}\delta^{(n)}\langle e_{n}^{-}\mid\,\cdot\,\rangle.
\]
This is the second formula in \eqref{eq:def_Tcal_-}. 
\end{proof}
\begin{figure}[h]
\begin{centering}
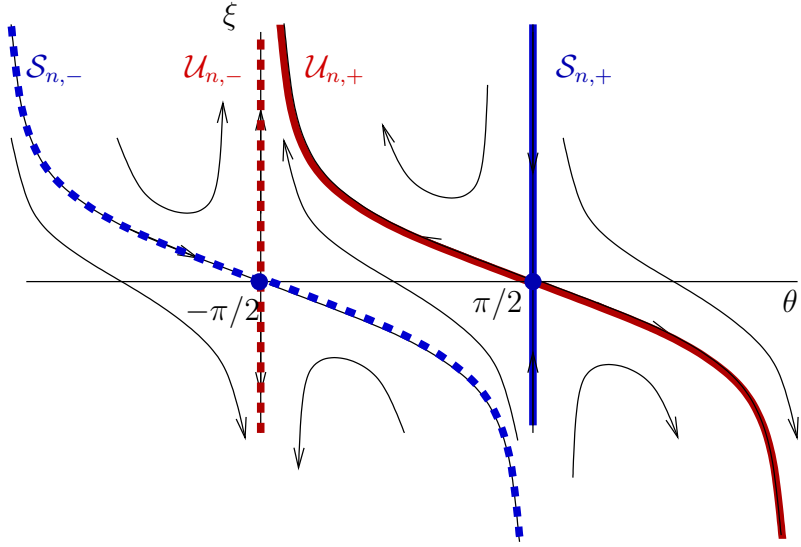
\par\end{centering}
\caption{\protect\protect\label{fig:S_U_distributions} The vector field $X$,
given in \eqref{eq:expression_X_sl2R}, has two fixed points at $\theta=\pm\frac{\pi}{2}$
and induces a Hamiltonian dynamics on $T^{*}S^{1}$, with coordinates
$(\theta,\xi)$. The figure illustrates the wave front sets of the
distributions $\mathcal{S}_{n,\pm}$ and $\mathcal{U}_{n,\pm}$ defined
in \eqref{eq:def_Sk+_Uk+} and \eqref{eq:def_S_k-_Uk-}. The distributions
obtained from $\delta^{(n)}$ are supported at the corresponding fixed
point and have conormal wave front set there. The distributions obtained
from $x^{n}$ are Lagrangian states whose oscillatory factor is $(1+h_{\pm}(\theta)^{2})^{-b}=(1+h_{\pm}(\theta)^{2})^{\frac{1}{2}-i\lambda}.$
Thus their phase is $-\lambda\log(1+h_{\pm}(\theta)^{2}),$ and the
associated Lagrangian graph is $\xi(\theta)=-\lambda\,\frac{d}{d\theta}\log(1+h_{\pm}(\theta)^{2})=-\lambda h_{\pm}(\theta).$
Compare with Figure~\ref{fig:projective}.}
\end{figure}

The next lemma computes the overlaps of these Lagrangian states.

\begin{cBoxB}{}

\begin{lem}
\label{lem:overlaps_Lagrangian_states} For all $n',n\in\mathbb{N}$,
one has 
\begin{equation}
\langle\mathcal{S}_{n',\pm}\mid\mathcal{U}_{n,\pm}\rangle_{L^{2}(S^{1})}=\delta_{n'=n}.\label{eq:S_U-1}
\end{equation}
Moreover, 
\[
\langle\mathcal{S}_{n',+}\mid\mathcal{U}_{n,-}\rangle_{L^{2}(S^{1})}=0,
\]
and, in the regularized sense above, if $\lambda>0$, then 
\begin{equation}
\langle\mathcal{S}_{n',-}\mid\mathcal{U}_{n,+}\rangle_{L^{2}(S^{1})}=0.\label{eq:SU_zero}
\end{equation}
\end{lem}

\end{cBoxB}

\begin{proof}
For the diagonal terms, we use the weak expansions \eqref{eq:def_Tcal_+}
and \eqref{eq:def_Tcal_-}. For every $m\in\mathbb{N}$, 
\[
e_{m}^{\pm}=\mathcal{T}_{\pm}\mathcal{T}_{\pm}^{-1}e_{m}^{\pm}=\sum_{n\in\mathbb{N}}e_{n}^{\pm}\langle\mathcal{S}_{n,\pm}\mid\mathcal{U}_{m,\pm}\rangle_{L^{2}(S^{1})}.
\]
Since $(e_{n}^{\pm})_{n\in\mathbb{N}}$ is a basis, this gives 
\[
\langle\mathcal{S}_{n',\pm}\mid\mathcal{U}_{n,\pm}\rangle_{L^{2}(S^{1})}=\delta_{n'=n}.
\]

Next, 
\[
\operatorname{supp}(\mathcal{S}_{n,+})=\left\{ \frac{\pi}{2}\right\} ,\qquad\operatorname{supp}(\mathcal{U}_{n,-})=\left\{ -\frac{\pi}{2}\right\} .
\]
The supports are disjoint, hence 
\[
\langle\mathcal{S}_{n',+}\mid\mathcal{U}_{n,-}\rangle_{L^{2}(S^{1})}=0.
\]

It remains to compute 
\[
\langle\mathcal{S}_{n',-}\mid\mathcal{U}_{n,+}\rangle_{L^{2}(S^{1})}.
\]
Set 
\[
\Upsilon(x):=1+x^{2}.
\]
With the normalization 
\[
\mathcal{U}_{\pm}=\sqrt{2}\,\Upsilon^{b}h_{\pm}^{-\circ},
\]
one has 
\begin{equation}
(\mathcal{U}_{+})^{\dagger}=\frac{1}{\sqrt{2}}\,h_{+}^{\circ}\,\Upsilon^{1+\overline{b}}.\label{eq:U+_dagger}
\end{equation}
Moreover, 
\[
h_{+}\circ h_{-}^{-1}(x)=-\frac{1}{x}.
\]
Denoting this map by $I(x):=-1/x$, we obtain 
\begin{align}
\mathcal{U}_{-}(\mathcal{U}_{+})^{\dagger} & =\Upsilon^{b}h_{-}^{-\circ}h_{+}^{\circ}\Upsilon^{1+\overline{b}}\nonumber \\
 & =\Upsilon^{b}(1+x^{-2})^{1+\overline{b}}I^{\circ}\nonumber \\
 & =|x|^{-2(1+\overline{b})}\Upsilon^{1+b+\overline{b}}I^{\circ}\nonumber \\
 & =|x|^{-1+2i\lambda}I^{\circ}.\label{eq:UU_dagger}
\end{align}

Now set 
\[
f:=\frac{t_{n'}^{-}}{t_{n}^{+}\sqrt{n!}\sqrt{n'!}}.
\]
Using \eqref{eq:def_S_k-_Uk-}, \eqref{eq:def_Sk+_Uk+} and \eqref{eq:UU_dagger},
we get, in the regularized sense, 
\begin{align*}
\langle\mathcal{S}_{n',-}\mid\mathcal{U}_{n,+}\rangle_{L^{2}(S^{1})} & =f\,\left\langle x^{n'}\middle|\mathcal{U}_{-}(\mathcal{U}_{+})^{\dagger}x^{n}\right\rangle _{L^{2}(\mathbb{R})}\\
 & =f\,(-1)^{n}\int_{\mathbb{R}}x^{n'-n}|x|^{-1+2i\lambda}\,dx.
\end{align*}
Splitting $\mathbb{R}=\mathbb{R}_{+}\sqcup\mathbb{R}_{-}$, this becomes
\[
\langle\mathcal{S}_{n',-}\mid\mathcal{U}_{n,+}\rangle_{L^{2}(S^{1})}=f\,(-1)^{n}\left(1+(-1)^{n'-n}\right)\int_{0}^{\infty}x^{n'-n-1+2i\lambda}\,dx.
\]

We interpret this integral in the following Mellin-symmetric regularized
sense. For $m\in\mathbb{Z}$, set 
\[
\int_{0}^{\infty}x^{m-1+2i\lambda}\,dx:=\lim_{\varepsilon\to0^{+}}\frac{2\varepsilon}{\varepsilon^{2}-(m+2i\lambda)^{2}}.
\]
Indeed, with $x=e^{y}$, this is the meromorphic continuation, from
the strip $|\Re(m+2i\lambda)|<\varepsilon$, of 
\[
\int_{\mathbb{R}}e^{y(m+2i\lambda)}e^{-\varepsilon|y|}\,dy=\frac{1}{\varepsilon-(m+2i\lambda)}+\frac{1}{\varepsilon+(m+2i\lambda)}.
\]
Thus, if $(m,\lambda)\neq(0,0)$, 
\[
\int_{0}^{\infty}x^{m-1+2i\lambda}\,dx=0
\]
in this regularized sense. Since here $\lambda>0$, the regularized
integral vanishes for every $m\in\mathbb{Z}$. Hence 
\[
\langle\mathcal{S}_{n',-}\mid\mathcal{U}_{n,+}\rangle_{L^{2}(S^{1})}=0.
\]
This proves \eqref{eq:SU_zero}. 
\end{proof}

\subsection{\protect\label{subsec:Merging-the-two}Merging the two projective
representations}

We now merge the two representations defined by the bijective maps
\[
\mathcal{T}_{\pm}:\mathcal{D}_{\pm}\subset L^{2}(S_{\theta}^{1};d\theta)\longrightarrow\mathcal{D}'_{\pm},
\]
defined in \eqref{eq:def_cal_T_+-}.

\begin{cBoxB}{}

\begin{lem}
\label{lem:TCal_inverse} Assume that $\lambda>0$. Then 
\[
\mathcal{D}_{+}\cap\mathcal{D}_{-}=\{0\}.
\]
Hence the algebraic sum 
\begin{equation}
\mathcal{D}:=\mathcal{D}_{+}+\mathcal{D}_{-}\label{eq:Gaussian_domain_D}
\end{equation}
is a direct sum: 
\[
\mathcal{D}=\mathcal{D}_{+}\oplus\mathcal{D}_{-}.
\]
Moreover, $\mathcal{D}$ is dense in $L^{2}(S_{\theta}^{1};d\theta)$.

Set 
\[
\mathcal{D}':=\mathcal{D}'_{+}\oplus\mathcal{D}'_{-}.
\]
For $u=u_{+}+u_{-}\in\mathcal{D}$, with $u_{\pm}\in\mathcal{D}_{\pm}$,
define 
\begin{equation}
\mathcal{T}u:=\begin{pmatrix}\mathcal{T}_{+}u_{+}\\
\mathcal{T}_{-}u_{-}
\end{pmatrix}\in\mathcal{D}'.\label{eq:def_cal_T}
\end{equation}
Then 
\[
\mathcal{T}:\mathcal{D}\longrightarrow\mathcal{D}'
\]
is a bijection. Its inverse is 
\begin{equation}
\mathcal{T}^{-1}=\begin{pmatrix}\mathcal{T}_{+}^{-1} & \mathcal{T}_{-}^{-1}\end{pmatrix},\qquad\mathcal{T}^{-1}\begin{pmatrix}v_{+}\\
v_{-}
\end{pmatrix}=\mathcal{T}_{+}^{-1}v_{+}+\mathcal{T}_{-}^{-1}v_{-}.\label{eq:Tau_inverse}
\end{equation}
\end{lem}

\end{cBoxB}

\begin{proof}
By Lemma~\ref{lem:overlaps_Lagrangian_states}, for $\lambda>0$
the mixed overlaps vanish in the Mellin-symmetric regularized sense
fixed there. Using the weak expansions \eqref{eq:def_Tcal_+} and
\eqref{eq:def_Tcal_-}, this gives 
\[
\mathcal{T}_{+}\mathcal{T}_{-}^{-1}=0\quad\text{on }\mathcal{D}'_{-},\qquad\mathcal{T}_{-}\mathcal{T}_{+}^{-1}=0\quad\text{on }\mathcal{D}'_{+}.
\]
Indeed, for example, if $v_{-}\in\mathcal{D}'_{-}$, then 
\[
\mathcal{T}_{+}\bigl(\mathcal{T}_{-}^{-1}v_{-}\bigr)=\sum_{n',n\in\mathbb{N}}e_{n'}^{+}\langle\mathcal{S}_{n',+}\mid\mathcal{U}_{n,-}\rangle_{L^{2}(S^{1})}\langle e_{n}^{-}\mid v_{-}\rangle=0.
\]
The other identity is identical.

Let $u\in\mathcal{D}_{+}\cap\mathcal{D}_{-}$. Since $u=\mathcal{T}_{-}^{-1}\mathcal{T}_{-}u$,
applying $\mathcal{T}_{+}$ gives 
\[
\mathcal{T}_{+}u=\mathcal{T}_{+}\mathcal{T}_{-}^{-1}\mathcal{T}_{-}u=0.
\]
As $\mathcal{T}_{+}$ is injective on $\mathcal{D}_{+}$, we get $u=0$.
Thus 
\[
\mathcal{D}_{+}\cap\mathcal{D}_{-}=\{0\},
\]
and the algebraic sum 
\[
\mathcal{D}=\mathcal{D}_{+}+\mathcal{D}_{-}
\]
is a direct sum.

Both $\mathcal{D}_{+}$ and $\mathcal{D}_{-}$ are dense in $L^{2}(S_{\theta}^{1};d\theta)$,
because they are inverse images of dense spaces under the unitary
maps $\mathcal{U}_{\pm}$. Hence $\mathcal{D}$ is dense.

Since the decomposition $u=u_{+}+u_{-}$ is unique, the map 
\[
\mathcal{T}u=\begin{pmatrix}\mathcal{T}_{+}u_{+}\\
\mathcal{T}_{-}u_{-}
\end{pmatrix}
\]
is well defined. It is injective because $\mathcal{T}_{+}$ and $\mathcal{T}_{-}$
are injective. It is surjective onto $\mathcal{D}'_{+}\oplus\mathcal{D}'_{-}$,
since for 
\[
\begin{pmatrix}v_{+}\\
v_{-}
\end{pmatrix}\in\mathcal{D}'_{+}\oplus\mathcal{D}'_{-}
\]
one has 
\[
\mathcal{T}\left(\mathcal{T}_{+}^{-1}v_{+}+\mathcal{T}_{-}^{-1}v_{-}\right)=\begin{pmatrix}v_{+}\\
v_{-}
\end{pmatrix}.
\]
Therefore $\mathcal{T}$ is bijective and its inverse is 
\[
\mathcal{T}^{-1}=\begin{pmatrix}\mathcal{T}_{+}^{-1} & \mathcal{T}_{-}^{-1}\end{pmatrix}.
\]
\end{proof}

\subsection{\protect\label{subsec:Hilbert-space-realization-from}Hilbert-space
realization from trigonometric domains}

The domains $\mathcal{D}_{\pm}$ introduced above are well adapted
to the projective charts and to the explicit formulas for $\mathcal{T}_{\pm}$.
However, in the original geometric problem, the natural algebraic
core is the space of $K$-finite vectors, which in the compact picture
is precisely the space of trigonometric polynomials 
\[
\mathrm{Pol}(S^{1}):=\mathrm{Span}\{\psi_{k}\ ;\ k\in\mathbb{Z}\}.
\]
The transforms $\mathcal{T}_{\pm}$ are not directly adapted to this
core in the form needed below. We therefore propagate it by the geodesic
generator and consider, for $t>0$, 
\[
D_{t}:=e^{t\tilde{X}}\mathrm{Pol}(S^{1}).
\]
This positive-time propagation places the natural geometric core in
a domain where the weak transforms have the required analyticity and
decay properties.

The length of this subsection comes from the fact that the natural
geometric core, namely the space of trigonometric polynomials, is
not directly adapted to the projective charts. We therefore first
construct the algebraic transforms on Gaussian-type domains, then
pass to the propagated trigonometric domains $D_{\pm t}$, and finally
complete them into Hilbert spaces. This provides a genuine Hilbert
model in which $\tilde{X}$ becomes a normal operator with discrete
spectrum.

The next two lemmas establish the injectivity of the weak transforms
$\mathcal{T}$ and $(\mathcal{T}^{-1})^{\dagger}$ on these propagated
trigonometric domains.
\begin{rem}
We keep the same notation $\mathcal{T}=(\mathcal{T}_{+},\mathcal{T}_{-})$
for the weak transform on the propagated trigonometric domains. This
should be distinguished from the algebraic map on $\mathcal{D}_{+}\oplus\mathcal{D}_{-}$,
where $\mathcal{T}$ acts after decomposing $u=u_{+}+u_{-}$. No ambiguity
will arise from the context.
\end{rem}

\begin{cBoxB}{}

\begin{lem}
\label{lem:injectivity_on_Dt} Assume that $\lambda>0$, and let $t>0$.
Then the weak transform 
\[
\mathcal{T}=\begin{pmatrix}\mathcal{T}_{+}\\
\mathcal{T}_{-}
\end{pmatrix}
\]
is injective on 
\[
D_{t}=e^{t\tilde{X}}\mathrm{Pol}(S^{1}).
\]
More precisely, $\mathcal{T}_{+}$ is already injective on $D_{t}$. 
\end{lem}

\end{cBoxB}

\begin{proof}
Let $u\in D_{t}$, and set 
\[
f_{+}:=\mathcal{U}_{+}u.
\]
By definition of $D_{t}$, there exists a trigonometric polynomial
\[
p=\sum_{|k|\le K}a_{k}\psi_{k}
\]
such that 
\[
u=e^{t\tilde{X}}p.
\]
Using 
\[
\mathcal{U}_{+}\tilde{X}\mathcal{U}_{+}^{-1}\eq{\ref{eq:def_Xtilde}}-x\frac{d}{dx}+b,
\]
we get 
\[
f_{+}=\mathcal{U}_{+}e^{t\tilde{X}}p=e^{t(-x\partial_{x}+b)}\mathcal{U}_{+}p=e^{tb}(\mathcal{U}_{+}p)(e^{-t}x).
\]
Moreover, for every $k\in\mathbb{Z}$, \eqref{eq:Upsi_pm_common}
gives 
\[
\mathcal{U}_{+}\psi_{k}=\frac{e^{ik\pi/2}}{\sqrt{\pi}}(1+ix)^{b+k}(1-ix)^{b-k}.
\]
Hence 
\[
f_{+}(x)=\sum_{|k|\le K}c_{k}(e^{t}+ix)^{b+k}(e^{t}-ix)^{b-k}
\]
for suitable constants $c_{k}$. In particular, $f_{+}$ is real-analytic
on $\mathbb{R}$ and holomorphic in a neighbourhood of $0$.

Assume now that $\mathcal{T}u=0$. Then $\mathcal{T}_{+}u=0$. Therefore,
for every $n\in\mathbb{N}$, 
\begin{align*}
0 & =\langle e_{n}^{+}\mid\mathcal{T}_{+}u\rangle\eq{\ref{eq:def_Tcal_+}}\langle\mathcal{S}_{n,+}\mid u\rangle_{L^{2}(S^{1})}\\
 & \eq{\ref{eq:def_Sk+_Uk+}}\frac{(-1)^{n}t_{n}^{+}}{\sqrt{n!}}\langle(\mathcal{U}_{+})^{\dagger}\delta^{(n)}\mid u\rangle_{L^{2}(S^{1})}\\
 & =\frac{(-1)^{n}t_{n}^{+}}{\sqrt{n!}}\langle\delta^{(n)}\mid\mathcal{U}_{+}u\rangle_{L^{2}(\mathbb{R})}\\
 & =\frac{(-1)^{n}t_{n}^{+}}{\sqrt{n!}}\langle\delta^{(n)}\mid f_{+}\rangle_{L^{2}(\mathbb{R})}=\frac{t_{n}^{+}}{\sqrt{n!}}f_{+}^{(n)}(0).
\end{align*}
Since $t_{n}^{+}\neq0$, all derivatives of $f_{+}$ at $0$ vanish:
\[
f_{+}^{(n)}(0)=0,\qquad n\in\mathbb{N}.
\]
Thus $f_{+}$ vanishes in a neighbourhood of $0$. Since $f_{+}$
is real-analytic on the connected real line, and vanishes on a neighbourhood
of $0$, the identity theorem gives 
\[
f_{+}\equiv0\qquad\text{on }\mathbb{R}.
\]
Therefore 
\[
u=\mathcal{U}_{+}^{-1}f_{+}=0.
\]
Thus $\mathcal{T}_{+}$, and hence $\mathcal{T}$, is injective on
$D_{t}$. 
\end{proof}
\begin{cBoxB}{}

\begin{lem}
\label{lem:injectivity_dual_on_Dmt} Assume that $\lambda>0$, and
let $t>0$. Then 
\[
(\mathcal{T}^{-1})^{\dagger}=\begin{pmatrix}(\mathcal{T}_{+}^{-1})^{\dagger}\\
(\mathcal{T}_{-}^{-1})^{\dagger}
\end{pmatrix}
\]
is injective on 
\[
D_{-t}:=e^{-t\tilde{X}}\mathrm{Pol}(S^{1}).
\]
More precisely, $(\mathcal{T}_{-}^{-1})^{\dagger}$ is already injective
on $D_{-t}$. 
\end{lem}

\end{cBoxB}

\begin{proof}
Let $u\in D_{-t}$, and set 
\[
f_{-}:=\mathcal{U}_{-}u.
\]
By definition of $D_{-t}$, there exists a trigonometric polynomial
\[
p=\sum_{|k|\le K}a_{k}\psi_{k}
\]
such that 
\[
u=e^{-t\tilde{X}}p.
\]
Using 
\[
\mathcal{U}_{-}\tilde{X}\mathcal{U}_{-}^{-1}\eq{\ref{eq:def_Xtilde}}x\frac{d}{dx}-b,
\]
we get 
\[
f_{-}=\mathcal{U}_{-}e^{-t\tilde{X}}p=e^{-t(x\partial_{x}-b)}\mathcal{U}_{-}p=e^{tb}(\mathcal{U}_{-}p)(e^{-t}x).
\]
Moreover, for every $k\in\mathbb{Z}$, \eqref{eq:Upsi_pm_common}
gives 
\[
\mathcal{U}_{-}\psi_{k}=\frac{e^{-ik\pi/2}}{\sqrt{\pi}}(1+ix)^{b+k}(1-ix)^{b-k}.
\]
Hence 
\[
f_{-}(x)=\sum_{|k|\le K}c_{k}(e^{t}+ix)^{b+k}(e^{t}-ix)^{b-k}
\]
for suitable constants $c_{k}$. In particular, $f_{-}$ is real-analytic
on $\mathbb{R}$ and holomorphic in a neighbourhood of $0$.

Assume now that 
\[
(\mathcal{T}^{-1})^{\dagger}u=0.
\]
Then $(\mathcal{T}_{-}^{-1})^{\dagger}u=0$. Therefore, for every
$n\in\mathbb{N}$, using \eqref{eq:def_Tcal_-} and \eqref{eq:def_S_k-_Uk-},
we obtain 
\begin{align*}
0 & =\left\langle e_{n}^{-}\middle|(\mathcal{T}_{-}^{-1})^{\dagger}u\right\rangle =\langle\mathcal{U}_{n,-}\mid u\rangle_{L^{2}(S^{1})}\\
 & =\frac{(-1)^{n}}{t_{n}^{-}\sqrt{n!}}\langle(\mathcal{U}_{-})^{\dagger}\delta^{(n)}\mid u\rangle_{L^{2}(S^{1})}\\
 & =\frac{(-1)^{n}}{t_{n}^{-}\sqrt{n!}}\langle\delta^{(n)}\mid\mathcal{U}_{-}u\rangle_{L^{2}(\mathbb{R})}\\
 & =\frac{(-1)^{n}}{t_{n}^{-}\sqrt{n!}}\langle\delta^{(n)}\mid f_{-}\rangle_{L^{2}(\mathbb{R})}=\frac{1}{t_{n}^{-}\sqrt{n!}}\,f_{-}^{(n)}(0).
\end{align*}
Since $t_{n}^{-}\neq0$, all derivatives of $f_{-}$ at $0$ vanish:
\[
f_{-}^{(n)}(0)=0,\qquad n\in\mathbb{N}.
\]
Thus $f_{-}$ vanishes in a neighbourhood of $0$. Since $f_{-}$
is real-analytic on the connected real line, the identity theorem
gives 
\[
f_{-}\equiv0\qquad\text{on }\mathbb{R}.
\]
Therefore 
\[
u=\mathcal{U}_{-}^{-1}f_{-}=0.
\]
Thus $(\mathcal{T}_{-}^{-1})^{\dagger}$, and hence $(\mathcal{T}^{-1})^{\dagger}$,
is injective on $D_{-t}$. 
\end{proof}
\begin{cBoxB}{}

\begin{prop}
\label{prop:propagated_estimates_and_X_intertwining} Assume that
$\lambda>0$, and let $t>0$. Set 
\[
N:=(-A+b_{+})\oplus(-A+b_{-}),\qquad N^{\dagger}:=(-A+\overline{b_{+}})\oplus(-A+\overline{b_{-}})
\]
on 
\[
\ell_{+}^{2}(\mathbb{N})\oplus\ell_{-}^{2}(\mathbb{N}).
\]
Then: 
\begin{enumerate}
\item For every $k\in\mathbb{Z}$, there exists $C_{k,t}>0$ such that,
for all $n\in\mathbb{N}$, 
\begin{equation}
\left|\left\langle e_{n}^{\pm}\middle|\mathcal{T}_{\pm}\bigl(e^{t\tilde{X}}\psi_{k}\bigr)\right\rangle \right|\le C_{k,t}e^{-tn}\langle n\rangle^{|k|-\frac{1}{2}},\label{eq:Fourier_decay_t}
\end{equation}
and 
\begin{equation}
\left|\left\langle e_{n}^{\pm}\middle|(\mathcal{T}_{\pm}^{-1})^{\dagger}\bigl(e^{-t\tilde{X}}\psi_{k}\bigr)\right\rangle \right|\le C_{k,t}e^{-tn}\langle n\rangle^{|k|-\frac{1}{2}},\label{eq:Fourier_decay_t_dual}
\end{equation}
where $\langle n\rangle:=(1+n^{2})^{1/2}$.
\item For every $u\in D_{t}$, one has 
\[
\mathcal{T}u\in\mathcal{D}(N),
\]
and 
\begin{equation}
N\,\mathcal{T}u=\mathcal{T}(\tilde{X}u).\label{eq:X_intertwining_on_Dt}
\end{equation}
\item For every $v\in D_{-t}$, one has 
\[
(\mathcal{T}^{-1})^{\dagger}v\in\mathcal{D}(N^{\dagger}),
\]
and 
\begin{equation}
N^{\dagger}(\mathcal{T}^{-1})^{\dagger}v=(\mathcal{T}^{-1})^{\dagger}(-\tilde{X}v).\label{eq:X_dual_intertwining_on_Dmt}
\end{equation}
\end{enumerate}
\end{prop}

\end{cBoxB}

\begin{proof}
Using \eqref{eq:X_conj-1}, we have, in the weak sense, 
\[
\mathcal{T}_{\pm}e^{s\tilde{X}}=e^{s(-A+b_{\pm})}\mathcal{T}_{\pm},\qquad s\ge0.
\]
Taking $s=t$, we get 
\[
\left\langle e_{n}^{\pm}\middle|\mathcal{T}_{\pm}\bigl(e^{t\tilde{X}}\psi_{k}\bigr)\right\rangle =e^{t(b_{\pm}-n)}\left\langle e_{n}^{\pm}\middle|\mathcal{T}_{\pm}\psi_{k}\right\rangle .
\]
Using \eqref{eq:Fourier_decay} and $\Re b_{\pm}=-\frac{1}{2}$, we
obtain, after increasing the constant if necessary, 
\[
\left|\left\langle e_{n}^{\pm}\middle|\mathcal{T}_{\pm}\bigl(e^{t\tilde{X}}\psi_{k}\bigr)\right\rangle \right|\le C_{k,t}e^{-tn}\langle n\rangle^{|k|-\frac{1}{2}}.
\]
This proves \eqref{eq:Fourier_decay_t}.

Similarly, by taking adjoints of the preceding weak intertwining relation,
\[
(\mathcal{T}_{\pm}^{-1})^{\dagger}e^{-s\tilde{X}}=e^{s(-A+\overline{b_{\pm}})}(\mathcal{T}_{\pm}^{-1})^{\dagger},\qquad s\ge0.
\]
Thus, at $s=t$, 
\[
\left\langle e_{n}^{\pm}\middle|(\mathcal{T}_{\pm}^{-1})^{\dagger}\bigl(e^{-t\tilde{X}}\psi_{k}\bigr)\right\rangle =e^{t(\overline{b_{\pm}}-n)}\left\langle e_{n}^{\pm}\middle|(\mathcal{T}_{\pm}^{-1})^{\dagger}\psi_{k}\right\rangle .
\]
Using \eqref{eq:Fourier_decay-1} and again $\Re b_{\pm}=-\frac{1}{2}$,
we get \eqref{eq:Fourier_decay_t_dual}.

Now let 
\[
u=\sum_{|k|\le K}a_{k}e^{t\tilde{X}}\psi_{k}\in D_{t}.
\]
By \eqref{eq:Fourier_decay_t}, each component $\mathcal{T}_{\pm}u$
has exponentially decaying coefficients. Hence 
\[
\mathcal{T}u\in\mathcal{D}(N).
\]
Differentiating the weak identity 
\[
\mathcal{T}_{\pm}e^{s\tilde{X}}=e^{s(-A+b_{\pm})}\mathcal{T}_{\pm}
\]
with respect to $s$ and then taking $s=t$, we obtain 
\[
\mathcal{T}_{\pm}\tilde{X}e^{t\tilde{X}}=(-A+b_{\pm})\mathcal{T}_{\pm}e^{t\tilde{X}}.
\]
By linearity, 
\[
N\,\mathcal{T}u=\mathcal{T}(\tilde{X}u),
\]
which proves \eqref{eq:X_intertwining_on_Dt}.

The proof of the dual statement is the same. If 
\[
v=\sum_{|k|\le K}a_{k}e^{-t\tilde{X}}\psi_{k}\in D_{-t},
\]
then \eqref{eq:Fourier_decay_t_dual} implies that 
\[
(\mathcal{T}^{-1})^{\dagger}v\in\mathcal{D}(N^{\dagger}).
\]
Differentiating 
\[
(\mathcal{T}_{\pm}^{-1})^{\dagger}e^{-s\tilde{X}}=e^{s(-A+\overline{b_{\pm}})}(\mathcal{T}_{\pm}^{-1})^{\dagger}
\]
with respect to $s$ and then taking $s=t$, we obtain 
\[
(\mathcal{T}_{\pm}^{-1})^{\dagger}(-\tilde{X})e^{-t\tilde{X}}=(-A+\overline{b_{\pm}})(\mathcal{T}_{\pm}^{-1})^{\dagger}e^{-t\tilde{X}}.
\]
By linearity, 
\[
N^{\dagger}(\mathcal{T}^{-1})^{\dagger}v=(\mathcal{T}^{-1})^{\dagger}(-\tilde{X}v),
\]
which proves \eqref{eq:X_dual_intertwining_on_Dmt}. 
\end{proof}
\begin{cBoxB}{}

\begin{lem}
\label{lem:each_mode_is_seen} Assume that $\lambda>0$. For every
$n\in\mathbb{N}$ and each sign $\pm$, the functions 
\[
k\longmapsto\langle e_{n}^{\pm}\mid\mathcal{T}_{\pm}\psi_{k}\rangle,\qquad k\longmapsto\langle e_{n}^{\pm}\mid(\mathcal{T}_{\pm}^{-1})^{\dagger}\psi_{k}\rangle
\]
are not identically zero on $\mathbb{Z}$. 
\end{lem}

\end{cBoxB}

\begin{proof}
By the weak expansions \eqref{eq:def_Tcal_+}--\eqref{eq:def_Tcal_-},
\[
\langle e_{n}^{\pm}\mid\mathcal{T}_{\pm}\psi_{k}\rangle=\langle\mathcal{S}_{n,\pm}\mid\psi_{k}\rangle_{L^{2}(S^{1})},
\]
and 
\[
\langle e_{n}^{\pm}\mid(\mathcal{T}_{\pm}^{-1})^{\dagger}\psi_{k}\rangle=\langle\mathcal{U}_{n,\pm}\mid\psi_{k}\rangle_{L^{2}(S^{1})}.
\]
If the first function were identically zero in $k$, then all Fourier
coefficients of the distribution $\mathcal{S}_{n,\pm}$ would vanish,
hence 
\[
\mathcal{S}_{n,\pm}=0
\]
as a distribution on $S^{1}$. This contradicts the biorthogonality
relation 
\[
\langle\mathcal{S}_{n,\pm}\mid\mathcal{U}_{n,\pm}\rangle_{L^{2}(S^{1})}=1
\]
from \eqref{eq:S_U-1}. Thus 
\[
k\longmapsto\langle e_{n}^{\pm}\mid\mathcal{T}_{\pm}\psi_{k}\rangle
\]
is not identically zero.

The second statement is identical: if all Fourier coefficients of
$\mathcal{U}_{n,\pm}$ vanished, then $\mathcal{U}_{n,\pm}=0$, again
contradicting \eqref{eq:S_U-1}. 
\end{proof}
\begin{cBoxB}{}

\begin{prop}
\label{prop:density_of_images} Assume that $\lambda>0$, and let
$t>0$. Then 
\[
\mathcal{T}(D_{t})\quad\text{is dense in}\quad\ell_{+}^{2}(\mathbb{N})\oplus\ell_{-}^{2}(\mathbb{N}).
\]
Similarly, 
\[
(\mathcal{T}^{-1})^{\dagger}(D_{-t})\quad\text{is dense in}\quad\ell_{+}^{2}(\mathbb{N})\oplus\ell_{-}^{2}(\mathbb{N}).
\]
\end{prop}

\end{cBoxB}

\begin{proof}
We prove the first statement. Set 
\[
W_{t}:=\mathcal{T}(D_{t}),\qquad\mathcal{K}_{t}:=\overline{W_{t}}^{\,\ell_{+}^{2}\oplus\ell_{-}^{2}},
\]
and 
\[
N:=(-A+b_{+})\oplus(-A+b_{-}).
\]
By Proposition~\ref{prop:propagated_estimates_and_X_intertwining},
$W_{t}$ is invariant under $N$. Moreover, by \eqref{eq:Fourier_decay_t},
vectors in $W_{t}$ have exponentially decaying coefficients. Hence
they are analytic vectors for $N$, and therefore 
\[
e^{sN}w\in\mathcal{K}_{t},\qquad s\ge0,\quad w\in W_{t}.
\]

We prove by induction on $n$ that 
\[
e_{n}^{+}\in\mathcal{K}_{t},\qquad e_{n}^{-}\in\mathcal{K}_{t}.
\]
Write 
\[
z_{n,+}:=-n-\frac{1}{2}+i\lambda,\qquad z_{n,-}:=-n-\frac{1}{2}-i\lambda.
\]
Thus 
\[
Ne_{n}^{\pm}=z_{n,\pm}e_{n}^{\pm}.
\]

Fix a sign $\sigma\in\{+,-\}$, and assume that 
\[
e_{m}^{\pm}\in\mathcal{K}_{t}\qquad\text{for all }m<n.
\]
By Lemma~\ref{lem:each_mode_is_seen}, there exists $k\in\mathbb{Z}$
such that 
\[
\langle e_{n}^{\sigma}\mid\mathcal{T}_{\sigma}\psi_{k}\rangle\neq0.
\]
Using the intertwining relation, 
\[
\mathcal{T}_{\sigma}(e^{t\tilde{X}}\psi_{k})=e^{t(-A+b_{\sigma})}\mathcal{T}_{\sigma}\psi_{k},
\]
we get 
\[
\langle e_{n}^{\sigma}\mid\mathcal{T}_{\sigma}(e^{t\tilde{X}}\psi_{k})\rangle=e^{tz_{n,\sigma}}\langle e_{n}^{\sigma}\mid\mathcal{T}_{\sigma}\psi_{k}\rangle\neq0.
\]
Hence there exists $w\in W_{t}$ whose $e_{n}^{\sigma}$-coefficient
is nonzero.

By the induction hypothesis, we may subtract from $w$ its components
of order $m<n$. We obtain a vector $v\in\mathcal{K}_{t}$ such that
all its coefficients of order $m<n$ vanish, while 
\[
\langle e_{n}^{\sigma}\mid v\rangle\neq0.
\]

For $T>0$, define the Cesàro spectral average 
\[
P_{T,n}^{\sigma}:=\frac{1}{T}\int_{0}^{T}e^{-sz_{n,\sigma}}e^{sN}\,ds.
\]
Then $P_{T,n}^{\sigma}v\in\mathcal{K}_{t}$. On the coefficient of
$e_{m}^{\tau}$, $\tau\in\{+,-\}$, this operator acts by multiplication
with 
\[
\frac{1}{T}\int_{0}^{T}e^{s(z_{m,\tau}-z_{n,\sigma})}\,ds.
\]
Since the coefficients of $v$ vanish for $m<n$, only $m\ge n$ occur.
If $m>n$, then 
\[
\Re(z_{m,\tau}-z_{n,\sigma})=-(m-n)<0,
\]
so the average tends to $0$. If $m=n$ and $\tau\neq\sigma$, the
average also tends to $0$, because it is 
\[
\frac{1}{T}\int_{0}^{T}e^{\pm2i\lambda s}\,ds\longrightarrow0\qquad(\lambda>0).
\]
For $m=n$ and $\tau=\sigma$, the average is $1$. Since, after the
subtraction, only coefficients with $m\ge n$ remain, all averaging
factors are uniformly bounded by $1$. Dominated convergence in $\ell_{+}^{2}(\mathbb{N})\oplus\ell_{-}^{2}(\mathbb{N})$
gives
\[
P_{T,n}^{\sigma}v\longrightarrow\langle e_{n}^{\sigma}\mid v\rangle e_{n}^{\sigma}.
\]
Thus $e_{n}^{\sigma}\in\mathcal{K}_{t}$.

This proves the induction for both signs. Hence every basis vector
$e_{n}^{\pm}$ belongs to $\mathcal{K}_{t}$, and therefore 
\[
\mathcal{K}_{t}=\ell_{+}^{2}(\mathbb{N})\oplus\ell_{-}^{2}(\mathbb{N}).
\]
Thus $\mathcal{T}(D_{t})$ is dense.

The proof for 
\[
(\mathcal{T}^{-1})^{\dagger}(D_{-t})
\]
is identical, using $N^{\dagger}$ instead of $N$, the estimate \eqref{eq:Fourier_decay_t_dual},
and the second part of Lemma~\ref{lem:each_mode_is_seen}. 
\end{proof}
\begin{cBoxB}{}

\begin{cor}
\label{cor:Dt_functional_properties} Assume that $\lambda>0$, and
let $t>0$. Then: 
\begin{enumerate}
\item The spaces 
\[
D_{t}=e^{t\tilde{X}}\mathrm{Pol}(S^{1}),\qquad D_{-t}=e^{-t\tilde{X}}\mathrm{Pol}(S^{1}),
\]
are dense linear subspaces of $L^{2}(S_{\theta}^{1};d\theta)$.
\item The weak transform $\mathcal{T}$ is injective on $D_{t}$, and 
\[
\mathcal{T}(D_{t})\subset\ell_{+}^{2}(\mathbb{N})\oplus\ell_{-}^{2}(\mathbb{N})
\]
is dense.
\item The weak transform $(\mathcal{T}^{-1})^{\dagger}$ is injective on
$D_{-t}$, and 
\[
(\mathcal{T}^{-1})^{\dagger}(D_{-t})\subset\ell_{+}^{2}(\mathbb{N})\oplus\ell_{-}^{2}(\mathbb{N})
\]
is dense. 
\end{enumerate}
Moreover, the propagated coefficient estimates are those of \eqref{eq:Fourier_decay_t}
and \eqref{eq:Fourier_decay_t_dual}. 
\end{cor}

\end{cBoxB}

\begin{proof}
The density of $D_{t}$ and $D_{-t}$ follows from the invertibility
of $e^{t\tilde{X}}$ on $L^{2}(S_{\theta}^{1};d\theta)$ and from
the density of $\mathrm{Pol}(S^{1})$.

The injectivity statements are Lemmas~\ref{lem:injectivity_on_Dt}
and~\ref{lem:injectivity_dual_on_Dmt}. The inclusions into $\ell_{+}^{2}(\mathbb{N})\oplus\ell_{-}^{2}(\mathbb{N})$
follow from the exponential estimates \eqref{eq:Fourier_decay_t}
and \eqref{eq:Fourier_decay_t_dual}. Finally, the density of the
two images is Proposition~\ref{prop:density_of_images}. 
\end{proof}
\begin{cBoxB}{}

\begin{cor}
\label{cor:Ht_Hminus_t} Assume that $\lambda>0$, and let $t>0$.
Define 
\begin{equation}
\mathcal{H}_{t}:=\overline{D_{t}}^{\,\|\mathcal{T}(\cdot)\|_{\ell^{2}\oplus\ell^{2}}},\label{eq:def_H}
\end{equation}
and 
\begin{equation}
\mathcal{H}_{-t}:=\overline{D_{-t}}^{\,\|(\mathcal{T}^{-1})^{\dagger}(\cdot)\|_{\ell^{2}\oplus\ell^{2}}}.\label{eq:def_H-1}
\end{equation}
Then $\mathcal{H}_{t}$ and $\mathcal{H}_{-t}$ are Hilbert spaces,
and 
\[
\mathcal{T}:\mathcal{H}_{t}\xrightarrow{\sim}\ell_{+}^{2}(\mathbb{N})\oplus\ell_{-}^{2}(\mathbb{N}),\qquad(\mathcal{T}^{-1})^{\dagger}:\mathcal{H}_{-t}\xrightarrow{\sim}\ell_{+}^{2}(\mathbb{N})\oplus\ell_{-}^{2}(\mathbb{N})
\]
extend uniquely as unitary isomorphisms.

Moreover, the $L^{2}(S_{\theta}^{1};d\theta)$-pairing on $D_{-t}\times D_{t}$
extends uniquely to a continuous sesquilinear pairing 
\begin{equation}
\langle\cdot\mid\cdot\rangle:\mathcal{H}_{-t}\times\mathcal{H}_{t}\longrightarrow\mathbb{C},\label{eq:pairing_Hmt_Ht}
\end{equation}
given by 
\begin{equation}
\langle v\mid u\rangle=\left\langle (\mathcal{T}^{-1})^{\dagger}v\middle|\mathcal{T}u\right\rangle _{\ell_{+}^{2}(\mathbb{N})\oplus\ell_{-}^{2}(\mathbb{N})}.\label{eq:pairing_transport}
\end{equation}
In particular, $\mathcal{H}_{-t}$ identifies anti-linearly with the
continuous dual $\mathcal{H}_{t}'$. 
\end{cor}

\end{cBoxB}

\begin{proof}
By Corollary~\ref{cor:Dt_functional_properties}, the maps 
\[
\mathcal{T}:D_{t}\longrightarrow\ell_{+}^{2}(\mathbb{N})\oplus\ell_{-}^{2}(\mathbb{N}),\qquad(\mathcal{T}^{-1})^{\dagger}:D_{-t}\longrightarrow\ell_{+}^{2}(\mathbb{N})\oplus\ell_{-}^{2}(\mathbb{N})
\]
are injective and have dense range. Hence the norms defining $\mathcal{H}_{t}$
and $\mathcal{H}_{-t}$ are genuine norms, and the two maps extend
uniquely by completion to unitary isomorphisms onto $\ell_{+}^{2}(\mathbb{N})\oplus\ell_{-}^{2}(\mathbb{N})$.

We define the pairing between $\mathcal{H}_{-t}$ and $\mathcal{H}_{t}$
by transport through these two unitary identifications: 
\[
\langle v\mid u\rangle:=\left\langle (\mathcal{T}^{-1})^{\dagger}v\middle|\mathcal{T}u\right\rangle _{\ell_{+}^{2}(\mathbb{N})\oplus\ell_{-}^{2}(\mathbb{N})}.
\]
It is continuous by Cauchy--Schwarz.

On the dense subspaces $D_{-t}\times D_{t}$, this transported pairing
coincides with the original $L^{2}(S^{1})$-pairing. Indeed, the weak
definition of $(\mathcal{T}^{-1})^{\dagger}$ gives, for $v\in D_{-t}$
and $u\in D_{t}$, 
\[
\left\langle (\mathcal{T}^{-1})^{\dagger}v\middle|\mathcal{T}u\right\rangle _{\ell^{2}\oplus\ell^{2}}=\langle v\mid u\rangle_{L^{2}(S^{1})}.
\]
Equivalently, this follows from the weak expansions of $\mathcal{T}$
and $\mathcal{T}^{-1}$, together with the overlap relations of Lemma~\ref{lem:overlaps_Lagrangian_states}.

Thus the transported pairing is the unique continuous extension of
the $L^{2}$-pairing from $D_{-t}\times D_{t}$ to $\mathcal{H}_{-t}\times\mathcal{H}_{t}$.
Under the two unitary identifications with $\ell_{+}^{2}(\mathbb{N})\oplus\ell_{-}^{2}(\mathbb{N})$,
this pairing becomes the standard Hilbert pairing. Therefore $\mathcal{H}_{-t}$
identifies anti-linearly with the continuous dual $\mathcal{H}_{t}'$. 
\end{proof}
\begin{cBoxB}{}

\begin{thm}
\label{thm:summary_spectral} Assume that $\lambda>0$, and let $t>0$.
Then, in the Hilbert space $\mathcal{H}_{t}$, the operator $\tilde{X}$,
initially defined on the core $D_{t}$, extends to the normal operator
\begin{equation}
\tilde{X}=\mathcal{T}^{-1}\left((-A+b_{+})\oplus(-A+b_{-})\right)\mathcal{T}.\label{eq:X_tilde}
\end{equation}
Its spectrum is discrete and given by 
\begin{equation}
z_{n,\pm}=-n-\frac{1}{2}\pm i\lambda,\qquad n\in\mathbb{N}.\label{eq:def_z_n}
\end{equation}
In the original weak realization, the associated rank-one spectral
projectors are
\begin{equation}
\Pi_{n,\pm}=\mathcal{U}_{n,\pm}\langle\mathcal{S}_{n,\pm}\mid\cdot\rangle_{L^{2}(S^{1})}\label{eq:def_Pi_k}
\end{equation}
in the weak sense, and extend continuously to $\mathcal{H}_{t}$.

Moreover, for every $\tau>0$, 
\begin{equation}
e^{\tau\tilde{X}}=\sum_{n\in\mathbb{N}}\sum_{\sigma\in\{+,-\}}e^{\tau z_{n,\sigma}}\Pi_{n,\sigma},\label{eq:spectral_semigroup}
\end{equation}
where the series converges in operator norm on $\mathcal{H}_{t}$.

Finally, on the Gaussian core 
\[
\mathcal{D}=\mathcal{D}_{+}\oplus\mathcal{D}_{-},
\]
one has the algebraic identities 
\begin{equation}
\tilde{U}=\mathcal{T}^{-1}\left(\bigl(-a^{+}(A-2b_{+})^{1/2}\bigr)\oplus\bigl(a^{+}(A-2b_{-})^{1/2}\bigr)\right)\mathcal{T},\label{eq:U_tilde}
\end{equation}
and 
\begin{equation}
\tilde{S}=\mathcal{T}^{-1}\left(\bigl((A-2b_{+})^{1/2}a^{-}\bigr)\oplus\bigl(-(A-2b_{-})^{1/2}a^{-}\bigr)\right)\mathcal{T}.\label{eq:S_tilde}
\end{equation}
\end{thm}

\end{cBoxB}

\begin{proof}
By Corollary~\ref{cor:Ht_Hminus_t}, 
\[
\mathcal{T}:\mathcal{H}_{t}\xrightarrow{\sim}\ell_{+}^{2}(\mathbb{N})\oplus\ell_{-}^{2}(\mathbb{N})
\]
is unitary. For brevity, set 
\[
N:=(-A+b_{+})\oplus(-A+b_{-}).
\]
The intertwining relation 
\[
N\mathcal{T}u=\mathcal{T}(\tilde{X}u),\qquad u\in D_{t},
\]
proved in \eqref{eq:X_intertwining_on_Dt}, gives 
\[
\tilde{X}=\mathcal{T}^{-1}N\mathcal{T}
\]
on the core $D_{t}$. Hence $\tilde{X}$ extends to the normal operator
$\mathcal{T}^{-1}N\mathcal{T}$ on $\mathcal{H}_{t}$.

Since $Ae_{n}^{\pm}=ne_{n}^{\pm}$, one has 
\[
Ne_{n}^{\pm}=(-n+b_{\pm})e_{n}^{\pm}=z_{n,\pm}e_{n}^{\pm}.
\]
Thus the spectrum is discrete, with eigenvalues \eqref{eq:def_z_n}.
The spectral projector of $N$ associated with $z_{n,\pm}$ is 
\[
e_{n}^{\pm}\langle e_{n}^{\pm}\mid\cdot\rangle.
\]
Transporting it through $\mathcal{T}$, we get 
\[
\Pi_{n,\pm}=\mathcal{T}^{-1}\bigl(e_{n}^{\pm}\langle e_{n}^{\pm}\mid\cdot\rangle\bigr)\mathcal{T}.
\]
Using the weak expansions \eqref{eq:def_Tcal_+}--\eqref{eq:def_Tcal_-}
and the formulas for $\mathcal{T}_{\pm}^{-1}$, this becomes 
\[
\Pi_{n,\pm}=\mathcal{U}_{n,\pm}\langle\mathcal{S}_{n,\pm}\mid\cdot\rangle_{L^{2}(S^{1})}.
\]

In the diagonal model, 
\[
e^{\tau N}=\sum_{n\in\mathbb{N}}\sum_{\sigma\in\{+,-\}}e^{\tau z_{n,\sigma}}e_{n}^{\sigma}\langle e_{n}^{\sigma}\mid\cdot\rangle.
\]
Since 
\[
|e^{\tau z_{n,\sigma}}|=e^{-\tau(n+\frac{1}{2})},
\]
the series converges in operator norm for every $\tau>0$. Transporting
this identity by $\mathcal{T}^{-1}$ and $\mathcal{T}$ gives \eqref{eq:spectral_semigroup}.

Finally, the formulas \eqref{eq:U_tilde} and \eqref{eq:S_tilde}
are the branchwise conjugation identities \eqref{eq:T3_U-1} and \eqref{eq:T3_S-1},
written after merging the two branches through 
\[
\mathcal{T}=\begin{pmatrix}\mathcal{T}_{+}\\
\mathcal{T}_{-}
\end{pmatrix},\qquad\mathcal{T}^{-1}=\begin{pmatrix}\mathcal{T}_{+}^{-1} & \mathcal{T}_{-}^{-1}\end{pmatrix}.
\]
They hold algebraically on the Gaussian core $\mathcal{D}=\mathcal{D}_{+}\oplus\mathcal{D}_{-}$. 
\end{proof}
\begin{cBoxB}{}

\begin{cor}
\label{cor:correlation_formula_trigo} Assume that $\lambda>0$. For
every $k,k'\in\mathbb{Z}$ and every $\tau>0$, one has the absolutely
convergent expansion 
\begin{equation}
\langle\psi_{k'}\mid e^{\tau\tilde{X}}\psi_{k}\rangle_{L^{2}(S^{1})}=\sum_{n\in\mathbb{N}}\sum_{\sigma\in\{+,-\}}e^{\tau z_{n,\sigma}}\,\langle\psi_{k'}\mid\mathcal{U}_{n,\sigma}\rangle_{L^{2}(S^{1})}\,\langle\mathcal{S}_{n,\sigma}\mid\psi_{k}\rangle_{L^{2}(S^{1})}.\label{eq:correlation_formula}
\end{equation}
\end{cor}

\end{cBoxB}

\begin{proof}
Set $s=\tau/2$. Since $\tilde{X}$ is skew-adjoint on $L^{2}(S^{1})$,
we have 
\[
\langle\psi_{k'}\mid e^{\tau\tilde{X}}\psi_{k}\rangle_{L^{2}(S^{1})}=\langle e^{-s\tilde{X}}\psi_{k'}\mid e^{s\tilde{X}}\psi_{k}\rangle_{L^{2}(S^{1})}.
\]
Now 
\[
e^{s\tilde{X}}\psi_{k}\in D_{s},\qquad e^{-s\tilde{X}}\psi_{k'}\in D_{-s}.
\]
Using the pairing between $\mathcal{H}_{-s}$ and $\mathcal{H}_{s}$
from Corollary~\ref{cor:Ht_Hminus_t}, together with the spectral
resolution \eqref{eq:spectral_semigroup}, we get 
\[
\langle\psi_{k'}\mid e^{\tau\tilde{X}}\psi_{k}\rangle_{L^{2}(S^{1})}=\sum_{n\in\mathbb{N}}\sum_{\sigma\in\{+,-\}}e^{\tau z_{n,\sigma}}\,\langle\psi_{k'}\mid\mathcal{U}_{n,\sigma}\rangle_{L^{2}(S^{1})}\,\langle\mathcal{S}_{n,\sigma}\mid\psi_{k}\rangle_{L^{2}(S^{1})}.
\]
The convergence is absolute because the two coefficient factors have
at most polynomial growth in $n$, while 
\[
|e^{\tau z_{n,\sigma}}|=e^{-\tau(n+1/2)}.
\]
\end{proof}
We have thus constructed, from the natural trigonometric core, a Hilbert
space $\mathcal{H}_{t}$ in which the hyperbolic generator $\tilde{X}$
is conjugated to the diagonal normal operator 
\[
(-A+b_{+})\oplus(-A+b_{-}).
\]
This completes the Hilbert realization of one spherical principal
series.

\subsection{\protect\label{subsec:Complementary-series}Complementary series
$0<\mu<\frac{1}{4}$}

We now indicate the modifications needed for the complementary series.
Let 
\[
0<\mu<\frac{1}{4},\qquad\nu:=\sqrt{\frac{1}{4}-\mu}\in\left(0,\frac{1}{2}\right),\qquad\lambda=i\nu.
\]

Since $0<\nu<1/2$, one has 
\[
-2b_{+}=1+2\nu>0,\qquad-2b_{-}=1-2\nu>0.
\]
Thus the diagonal gauge factors remain well defined, using the real
positive branch.

The two spectral branches are now real: 
\[
z_{n,+}:=-n-\frac{1}{2}-\nu,\qquad z_{n,-}:=-n-\frac{1}{2}+\nu,\qquad n\in\mathbb{N}.
\]
They are strictly interlaced: 
\[
z_{0,-}>z_{0,+}>z_{1,-}>z_{1,+}>\cdots.
\]

The algebraic formulas defining the projective transforms $\mathcal{T}_{\pm}$
are obtained from the principal-series case by the substitution $\lambda=i\nu$.
The intertwining relations remain valid: 
\[
\mathcal{T}_{\pm}\widetilde{X}=(-A+b_{\pm})\mathcal{T}_{\pm}.
\]
However, the corresponding projective models are no longer unitary
for the standard $L^{2}(S^{1};d\theta)$-pairing. Accordingly, in
the Hilbert space constructed below, we denote the eigenvectors by
\[
V_{n,\pm}:=\mathcal{T}^{-1}e_{n}^{\pm},
\]
rather than identifying them a priori with the distributions $\mathcal{U}_{n,\pm}$.

\begin{cBoxB}{}

\begin{thm}[Complementary series]
\label{thm:complementary_series} Assume that $0<\mu<\frac{1}{4}$,
and use the notation above. For $t>0$, set 
\[
D_{t}:=e^{t\widetilde{X}}\mathrm{Pol}(S^{1}),
\]
and define weakly on $D_{t}$ 
\begin{equation}
\mathcal{T}u:=\begin{pmatrix}\mathcal{T}_{+}u\\
\mathcal{T}_{-}u
\end{pmatrix}.\label{eq:def_cal_T-1}
\end{equation}
Then 
\[
\mathcal{T}:D_{t}\longrightarrow\ell_{+}^{2}(\mathbb{N})\oplus\ell_{-}^{2}(\mathbb{N})
\]
is well defined, injective, and has dense image. Consequently, 
\[
\mathcal{H}_{t}:=\overline{D_{t}}^{\,\|\mathcal{T}(\cdot)\|_{\ell^{2}\oplus\ell^{2}}}
\]
is a Hilbert space, and $\mathcal{T}$ extends uniquely to a unitary
isomorphism 
\[
\mathcal{T}:\mathcal{H}_{t}\xrightarrow{\sim}\ell_{+}^{2}(\mathbb{N})\oplus\ell_{-}^{2}(\mathbb{N}).
\]

In $\mathcal{H}_{t}$, the operator $\widetilde{X}$ extends to the
self-adjoint operator 
\[
\widetilde{X}=\mathcal{T}^{-1}\left((-A+b_{+})\oplus(-A+b_{-})\right)\mathcal{T}.
\]
Its spectrum is discrete and given by 
\[
\operatorname{Spec}(\widetilde{X})=\left\{ -n-\frac{1}{2}-\nu,\,-n-\frac{1}{2}+\nu\ ;\ n\in\mathbb{N}\right\} .
\]
The associated rank-one spectral projectors are 
\[
\Pi_{n,\pm}=V_{n,\pm}\left\langle \mathcal{S}_{n,\pm}\mid\cdot\right\rangle _{L^{2}(S^{1})},\qquad V_{n,\pm}:=\mathcal{T}^{-1}e_{n}^{\pm}.
\]
For every $\tau>0$, one has the norm-convergent expansion 
\[
e^{\tau\widetilde{X}}=\sum_{n\ge0}\sum_{\sigma=\pm}e^{\tau z_{n,\sigma}}\Pi_{n,\sigma}
\]
on $\mathcal{H}_{t}$. 
\end{thm}

\end{cBoxB}

\begin{proof}
All algebraic conjugation identities are obtained from the principal
series by substituting $\lambda=i\nu$. In particular, 
\[
\mathcal{T}_{\pm}\widetilde{X}=(-A+b_{\pm})\mathcal{T}_{\pm}.
\]

The coefficient estimates used in the principal-series case remain
valid after this substitution, with polynomial changes in $n$. Indeed,
the real parts of the gauge parameters are now $1\pm2\nu$, so the
factorial weights are still cancelled, up to additional polynomial
powers. Thus, for each $k\in\mathbb{Z}$, 
\[
\left|\left\langle e_{n}^{\pm}\middle|\mathcal{T}_{\pm}(e^{t\widetilde{X}}\psi_{k})\right\rangle \right|\le C_{k,t}e^{-tn}\langle n\rangle^{M_{k}},
\]
for some $M_{k}$. Hence $\mathcal{T}(D_{t})\subset\ell_{+}^{2}\oplus\ell_{-}^{2}$.

The injectivity proof is the same as in Lemma~\ref{lem:injectivity_on_Dt}.
If $\mathcal{T}u=0$, then $\mathcal{T}_{+}u=0$. After applying $\mathcal{U}_{+}$,
all derivatives at the origin of the corresponding analytic function
vanish. Since $u\in D_{t}$, this function is holomorphic near $0$
and real-analytic on $\mathbb{R}$, hence it vanishes identically.
Therefore $u=0$.

It remains to prove density. Let 
\[
W_{t}:=\mathcal{T}(D_{t}),\qquad\mathcal{K}_{t}:=\overline{W_{t}}^{\,\ell^{2}\oplus\ell^{2}},
\]
and set 
\[
N:=(-A+b_{+})\oplus(-A+b_{-}).
\]
The intertwining relation implies that $W_{t}$, hence $\mathcal{K}_{t}$,
is invariant under $e^{sN}$, $s\ge0$.

Order the basis vectors according to the strictly decreasing real
eigenvalues: 
\[
E_{0}=e_{0}^{-},\quad E_{1}=e_{0}^{+},\quad E_{2}=e_{1}^{-},\quad E_{3}=e_{1}^{+},\quad\ldots
\]
so that 
\[
NE_{j}=\zeta_{j}E_{j},\qquad\zeta_{0}>\zeta_{1}>\zeta_{2}>\cdots.
\]
As in Lemma~\ref{lem:each_mode_is_seen}, each coordinate functional
$E_{j}$ is nonzero on $W_{t}$. We prove by induction that $E_{j}\in\mathcal{K}_{t}$
for every $j$. Suppose $E_{0},\ldots,E_{j-1}\in\mathcal{K}_{t}$.
Choose $w\in W_{t}$ with nonzero $E_{j}$-coefficient, and subtract
the already constructed components. We get 
\[
v=c_{j}E_{j}+\sum_{i>j}c_{i}E_{i},\qquad c_{j}\neq0,
\]
with $v\in\mathcal{K}_{t}$. Then, as $s\to+\infty$, 
\[
e^{-s\zeta_{j}}e^{sN}v=c_{j}E_{j}+\sum_{i>j}c_{i}e^{s(\zeta_{i}-\zeta_{j})}E_{i}\longrightarrow c_{j}E_{j}
\]
in $\ell^{2}\oplus\ell^{2}$, since $\zeta_{i}-\zeta_{j}<0$ for $i>j$.
Thus $E_{j}\in\mathcal{K}_{t}$. Hence all basis vectors belong to
$\mathcal{K}_{t}$, and $\mathcal{T}(D_{t})$ is dense.

The completion of $D_{t}$ for the norm $\|\mathcal{T}(\cdot)\|$
now makes $\mathcal{T}$ a unitary map from $\mathcal{H}_{t}$ onto
$\ell_{+}^{2}\oplus\ell_{-}^{2}$. Transporting the diagonal operator
$N$ gives 
\[
\widetilde{X}=\mathcal{T}^{-1}N\mathcal{T}.
\]
Since $b_{+}$ and $b_{-}$ are real, $N$ is self-adjoint. Its spectrum
is discrete and consists of the interlaced eigenvalues $z_{n,\pm}$.

The spectral projector of $N$ onto $\mathbb{C}e_{n}^{\pm}$ is 
\[
e_{n}^{\pm}\langle e_{n}^{\pm}\mid\cdot\rangle.
\]
Transporting this projector by $\mathcal{T}$ gives 
\[
\Pi_{n,\pm}=V_{n,\pm}\left\langle \mathcal{S}_{n,\pm}\mid\cdot\right\rangle _{L^{2}(S^{1})}.
\]
Finally, the expansion of $e^{\tau\widetilde{X}}$ is the transported
spectral resolution of $e^{\tau N}$. Since $e^{\tau z_{n,\pm}}$
decays exponentially in $n$ for every $\tau>0$, the series converges
in operator norm. 
\end{proof}

\subsection{\protect\label{subsec:Threshold-case-mu=00003D1/4}Threshold case
\texorpdfstring{$\mu=\frac{1}{4}$}{mu=1/4} : coalescence and Jordan
model}

We now discuss the threshold value 
\[
\mu=\frac{1}{4}.
\]
In the principal-series notation 
\[
\mu=\lambda^{2}+\frac{1}{4},
\]
the threshold corresponds to 
\[
\lambda=0.
\]
For $\lambda\neq0$, the two resonance branches are 
\[
z_{n,+}(\lambda)=-n-\frac{1}{2}+i\lambda,\qquad z_{n,-}(\lambda)=-n-\frac{1}{2}-i\lambda.
\]
At $\lambda=0$, they coalesce: 
\[
z_{n,+}(0)=z_{n,-}(0)=z_{n},\qquad z_{n}:=-n-\frac{1}{2}.
\]

The important point is that the two branches should not be kept as
two independent eigendirections at the threshold. The representation
depends naturally on the Casimir parameter $\mu$, whereas the individual
branches depend on the choice of a square root 
\[
\lambda=\sqrt{\mu-\frac{1}{4}}.
\]
Changing $\lambda$ into $-\lambda$ exchanges the two branches. Therefore
the object which has a regular limit at $\mu=1/4$ is not each branch
separately, but the two-dimensional space spanned by the symmetric
and divided-difference combinations. This is the mechanism which produces
the Jordan block at the threshold.

Set 
\begin{equation}
L:=-A-\frac{1}{2}\label{eq:def_L_threshold}
\end{equation}
on $\ell^{2}(\mathbb{N})$. Thus 
\[
Le_{n}=z_{n}e_{n},\qquad z_{n}=-n-\frac{1}{2}.
\]

\subsubsection{Renormalized branches and coefficientwise coalescence}

We identify the two copies $\ell_{+}^{2}(\mathbb{N})$ and $\ell_{-}^{2}(\mathbb{N})$
with a single copy of $\ell^{2}(\mathbb{N})$ by 
\[
J_{+}e_{n}^{+}=e_{n},\qquad J_{-}e_{n}^{-}=e_{n}.
\]
We also introduce the parity operator 
\begin{equation}
Pe_{n}:=(-1)^{n}e_{n}.\label{eq:def_parity_threshold}
\end{equation}
Finally, set 
\begin{equation}
\rho(\lambda):=\frac{\Gamma(\frac{1}{2}-i\lambda)}{\sqrt{\pi}\,\Gamma(-i\lambda)}.\label{eq:def_rho_threshold}
\end{equation}
By the duplication formula, 
\begin{equation}
\rho(\lambda)=\frac{\Gamma(\frac{1}{2}-i\lambda)^{2}}{\pi\,2^{1+2i\lambda}\Gamma(-2i\lambda)}.\label{eq:rho_duplication_threshold}
\end{equation}
In particular, 
\[
\rho(\lambda)=-i\lambda+O(\lambda^{2}).
\]

\begin{cBoxA}{}

\begin{defn}[Renormalized threshold branches]
\label{def:renormalized_threshold_branches} For $\lambda\neq0$,
define 
\begin{equation}
\widehat{\mathcal{T}}_{+,\lambda}:=J_{+}\mathcal{T}_{+,\lambda},\qquad\widehat{\mathcal{T}}_{-,\lambda}:=P\,\rho(\lambda)\,J_{-}\mathcal{T}_{-,\lambda}.\label{eq:def_renormalized_Tpm_threshold}
\end{equation}
\end{defn}

\end{cBoxA}

Since $P$ commutes with $A$, and since $\rho(\lambda)$ is a scalar,
this renormalization does not change the diagonal action. Thus, for
$\lambda\neq0$, 
\begin{equation}
\widehat{\mathcal{T}}_{+,\lambda}\widetilde{X}_{\lambda}=(L+i\lambda)\widehat{\mathcal{T}}_{+,\lambda},\qquad\widehat{\mathcal{T}}_{-,\lambda}\widetilde{X}_{\lambda}=(L-i\lambda)\widehat{\mathcal{T}}_{-,\lambda}.\label{eq:renormalized_intertwining_threshold}
\end{equation}

For $n\in\mathbb{N}$ and $k\in\mathbb{Z}$, set 
\begin{equation}
s_{n,k}^{+}(\lambda):=\left\langle e_{n}\mid\widehat{\mathcal{T}}_{+,\lambda}\psi_{k}\right\rangle ,\qquad s_{n,k}^{-}(\lambda):=\left\langle e_{n}\mid\widehat{\mathcal{T}}_{-,\lambda}\psi_{k}\right\rangle .\label{eq:def_s_nk_pm_threshold}
\end{equation}

\begin{cBoxB}{}

\begin{lem}[Coefficientwise coalescence]
\label{lem:threshold_coefficientwise_confluence} For every $n\in\mathbb{N}$
and $k\in\mathbb{Z}$, the functions 
\[
\lambda\longmapsto s_{n,k}^{+}(\lambda),\qquad\lambda\longmapsto s_{n,k}^{-}(\lambda)
\]
extend holomorphically to a neighbourhood of $\lambda=0$, and satisfy
\begin{equation}
s_{n,k}^{+}(0)=s_{n,k}^{-}(0).\label{eq:s_plus_minus_equal_threshold}
\end{equation}
Consequently, 
\[
\frac{s_{n,k}^{+}(\lambda)-s_{n,k}^{-}(\lambda)}{2i\lambda}
\]
has a finite limit as $\lambda\to0$. 
\end{lem}

\end{cBoxB}

\begin{proof}
Write 
\[
b=-\frac{1}{2}+i\lambda.
\]
We choose the holomorphic branches of the gauge factors near $\lambda=0$
so that 
\[
t_{n}^{+}(0)=(n!)^{-1/2},\qquad t_{n}^{-}(0)=(n!)^{1/2}.
\]

For the positive branch, the weak expansion of $\mathcal{T}_{+,\lambda}$
gives 
\begin{equation}
s_{n,k}^{+}(\lambda)=t_{n}^{+}(\lambda)\sqrt{n!}\,\frac{e^{ik\pi/2}}{\sqrt{\pi}}\,[x^{n}]\left((1+ix)^{b+k}(1-ix)^{b-k}\right).\label{eq:s_plus_coeff_threshold}
\end{equation}
This is holomorphic near $\lambda=0$. In particular, 
\begin{equation}
s_{n,k}^{+}(0)=\frac{e^{ik\pi/2}}{\sqrt{\pi}}\,[x^{n}]\left((1+ix)^{k-\frac{1}{2}}(1-ix)^{-k-\frac{1}{2}}\right).\label{eq:s_plus_zero_threshold}
\end{equation}

For the negative branch, the weak expansion of $\mathcal{T}_{-,\lambda}$
gives, before the renormalization by $P\rho(\lambda)$, 
\[
\left\langle e_{n}^{-}\mid\mathcal{T}_{-,\lambda}\psi_{k}\right\rangle =\frac{t_{n}^{-}(\lambda)}{\sqrt{n!}}\,\frac{e^{-ik\pi/2}}{\sqrt{\pi}}\,I_{n,k}(\lambda),
\]
where 
\[
I_{n,k}(\lambda):=\int_{\mathbb{R}}x^{n}(1+ix)^{b+k}(1-ix)^{b-k}\,dx
\]
is understood by meromorphic continuation.

Using 
\[
x^{n}=(2i)^{-n}\bigl((1+ix)-(1-ix)\bigr)^{n}
\]
and the beta identity 
\[
\int_{\mathbb{R}}(1+ix)^{\alpha}(1-ix)^{\beta}\,dx=\pi\,2^{\alpha+\beta+2}\frac{\Gamma(-\alpha-\beta-1)}{\Gamma(-\alpha)\Gamma(-\beta)},
\]
we obtain 
\begin{align}
I_{n,k}(\lambda) & =\frac{\pi\,2^{n+1+2i\lambda}}{(2i)^{n}}\Gamma(-n-2i\lambda)\nonumber \\
 & \quad\times\sum_{j=0}^{n}(-1)^{n-j}\binom{n}{j}\frac{1}{\Gamma(\frac{1}{2}-i\lambda-k-j)\Gamma(\frac{1}{2}-i\lambda+k-n+j)}.\label{eq:I_nk_threshold}
\end{align}
Multiplying by $P\rho(\lambda)$, and using \eqref{eq:rho_duplication_threshold},
gives 
\begin{align}
s_{n,k}^{-}(\lambda) & =(-1)^{n}\frac{t_{n}^{-}(\lambda)}{\sqrt{n!}}\,\frac{e^{-ik\pi/2}}{\sqrt{\pi}}\,i^{-n}\frac{\Gamma(\frac{1}{2}-i\lambda)^{2}\Gamma(-n-2i\lambda)}{\Gamma(-2i\lambda)}\nonumber \\
 & \quad\times\sum_{j=0}^{n}(-1)^{n-j}\binom{n}{j}\frac{1}{\Gamma(\frac{1}{2}-i\lambda-k-j)\Gamma(\frac{1}{2}-i\lambda+k-n+j)}.\label{eq:s_minus_renormalized_threshold}
\end{align}
The quotient 
\[
\frac{\Gamma(-n-2i\lambda)}{\Gamma(-2i\lambda)}
\]
is holomorphic at $\lambda=0$, and 
\begin{equation}
\frac{\Gamma(-n-2i\lambda)}{\Gamma(-2i\lambda)}\longrightarrow\frac{(-1)^{n}}{n!}\qquad(\lambda\to0).\label{eq:Gamma_ratio_threshold}
\end{equation}
Thus $s_{n,k}^{-}(\lambda)$ is holomorphic near $\lambda=0$.

It remains to compare the two values at $\lambda=0$. Setting $\lambda=0$
in \eqref{eq:s_minus_renormalized_threshold}, and using \eqref{eq:Gamma_ratio_threshold},
gives 
\[
s_{n,k}^{-}(0)=\frac{e^{-ik\pi/2}}{\sqrt{\pi}}\,\frac{\pi i^{-n}}{n!}\sum_{j=0}^{n}(-1)^{n-j}\binom{n}{j}\frac{1}{\Gamma(\frac{1}{2}-k-j)\Gamma(\frac{1}{2}+k-n+j)}.
\]
On the other hand, the elementary binomial expansion of 
\[
(1+ix)^{k-\frac{1}{2}}(1-ix)^{-k-\frac{1}{2}}
\]
gives 
\[
[x^{n}]\left((1+ix)^{k-\frac{1}{2}}(1-ix)^{-k-\frac{1}{2}}\right)=\frac{\pi i^{n}}{n!}(-1)^{k+n}\sum_{j=0}^{n}(-1)^{n-j}\binom{n}{j}\frac{1}{\Gamma(\frac{1}{2}-k-j)\Gamma(\frac{1}{2}+k-n+j)}.
\]
Since 
\[
e^{-ik\pi/2}i^{-n}=e^{ik\pi/2}i^{n}(-1)^{k+n},
\]
this is exactly the equality 
\[
s_{n,k}^{-}(0)=s_{n,k}^{+}(0).
\]
This proves \eqref{eq:s_plus_minus_equal_threshold}. The existence
of the divided-difference limit follows from holomorphy. 
\end{proof}

\subsubsection{Uniform estimates for the threshold limit}

We now pass from the coefficientwise coalescence to an $\ell^{2}$-limit.
The key point is that the coefficients have at most polynomial growth
in $n$, uniformly for $\lambda$ near $0$, while the propagation
by $e^{t\widetilde{X}_{\lambda}}$ produces the exponential factor
$e^{-tn}$.

\begin{cBoxB}{}

\begin{lem}[Uniform threshold estimates]
\label{lem:uniform_threshold_estimates} For every $k\in\mathbb{Z}$,
there exist constants $C_{k}>0$, $M_{k}>0$, and $\lambda_{0}>0$
such that, for every $n\in\mathbb{N}$ and $|\lambda|<\lambda_{0}$,
\begin{equation}
|s_{n,k}^{\pm}(\lambda)|\le C_{k}\langle n\rangle^{M_{k}},\label{eq:uniform_threshold_bound}
\end{equation}
and 
\begin{equation}
|\partial_{\lambda}s_{n,k}^{\pm}(\lambda)|\le C_{k}\langle n\rangle^{M_{k}}\bigl(1+\log\langle n\rangle\bigr).\label{eq:uniform_threshold_derivative_bound}
\end{equation}
Here 
\[
\langle n\rangle:=(1+n^{2})^{1/2}.
\]
\end{lem}

\end{cBoxB}

\begin{proof}
We only recall the estimates needed later.

For the positive branch, formula \eqref{eq:s_plus_coeff_threshold}
expresses $s_{n,k}^{+}(\lambda)$ as 
\[
t_{n}^{+}(\lambda)\sqrt{n!}
\]
times the $n$-th Taylor coefficient of 
\[
(1+ix)^{k-\frac{1}{2}+i\lambda}(1-ix)^{-k-\frac{1}{2}+i\lambda}.
\]
For $\lambda$ in a small neighbourhood of $0$, the algebraic singularities
at $x=\pm i$ stay of fixed type. Standard Darboux estimates, uniformly
in $\lambda$, give polynomial bounds for these Taylor coefficients,
and differentiating with respect to $\lambda$ introduces at most
one logarithmic factor. The factor $t_{n}^{+}(\lambda)\sqrt{n!}$,
and its $\lambda$-derivative, have the same type of polynomial/logarithmic
bounds by Stirling's formula. This proves \eqref{eq:uniform_threshold_bound}
and \eqref{eq:uniform_threshold_derivative_bound} for $s_{n,k}^{+}$.

For the negative branch, after the renormalization by $\rho(\lambda)$,
the coefficients are finite sums of beta-type terms of the form 
\[
\rho(\lambda)\int_{\mathbb{R}}x^{m}(1+x^{2})^{-\frac{1}{2}+i\lambda-q}\,dx,
\]
with $q$ fixed and $m$ depending on $n$. For fixed $k$, the number
of such terms and the corresponding values of $q$ are bounded in
terms of $k$ only. These integrals are understood by meromorphic
continuation. These integrals are understood by meromorphic continuation.
The factor $\rho(\lambda)$ cancels the pole at $\lambda=0$. More
explicitly, for even $m=2r$, the beta identity gives 
\[
\rho(\lambda)\int_{\mathbb{R}}x^{2r}(1+x^{2})^{-\frac{1}{2}+i\lambda-q}\,dx=\frac{\Gamma(\frac{1}{2}-i\lambda)}{\sqrt{\pi}\Gamma(-i\lambda)}\frac{\Gamma(r+\frac{1}{2})\Gamma(q-r-i\lambda)}{\Gamma(q+\frac{1}{2}-i\lambda)}.
\]
When $r\ge q$, the quotient 
\[
\frac{\Gamma(q-r-i\lambda)}{\Gamma(-i\lambda)}
\]
is holomorphic at $\lambda=0$ and has polynomial growth in $r$;
its derivative has at most one logarithmic loss. The finitely many
cases $r<q$ are harmless. Combining these estimates with Stirling's
formula and with the polynomial bounds on $t_{n}^{-}(\lambda)/\sqrt{n!}$
gives the same estimates for $s_{n,k}^{-}$. 
\end{proof}
\begin{cBoxB}{}

\begin{prop}[Threshold $\ell^{2}$-limit]
\label{prop:threshold_l2_limit} For every $t>0$ and every $k\in\mathbb{Z}$,
the two sequences 
\[
\left(\frac{\left\langle e_{n}\mid\widehat{\mathcal{T}}_{+,\lambda}e^{t\widetilde{X}_{\lambda}}\psi_{k}\right\rangle +\left\langle e_{n}\mid\widehat{\mathcal{T}}_{-,\lambda}e^{t\widetilde{X}_{\lambda}}\psi_{k}\right\rangle }{2}\right)_{n\ge0}
\]
and 
\[
\left(\frac{\left\langle e_{n}\mid\widehat{\mathcal{T}}_{+,\lambda}e^{t\widetilde{X}_{\lambda}}\psi_{k}\right\rangle -\left\langle e_{n}\mid\widehat{\mathcal{T}}_{-,\lambda}e^{t\widetilde{X}_{\lambda}}\psi_{k}\right\rangle }{2i\lambda}\right)_{n\ge0}
\]
have limits in $\ell^{2}(\mathbb{N})$ as $\lambda\to0$. 
\end{prop}

\end{cBoxB}

\begin{proof}
By the renormalized intertwining relations \eqref{eq:renormalized_intertwining_threshold},
we have 
\begin{equation}
\left\langle e_{n}\mid\widehat{\mathcal{T}}_{+,\lambda}e^{t\widetilde{X}_{\lambda}}\psi_{k}\right\rangle =e^{-t(n+\frac{1}{2})}e^{it\lambda}s_{n,k}^{+}(\lambda),\label{eq:threshold_plus_propagated_coeff}
\end{equation}
and 
\begin{equation}
\left\langle e_{n}\mid\widehat{\mathcal{T}}_{-,\lambda}e^{t\widetilde{X}_{\lambda}}\psi_{k}\right\rangle =e^{-t(n+\frac{1}{2})}e^{-it\lambda}s_{n,k}^{-}(\lambda).\label{eq:threshold_minus_propagated_coeff}
\end{equation}

For each fixed $n$, Lemma~\ref{lem:threshold_coefficientwise_confluence}
implies 
\[
s_{n,k}^{+}(0)=s_{n,k}^{-}(0),
\]
hence both the average and the divided difference have pointwise limits
as $\lambda\to0$.

It remains to justify convergence in $\ell^{2}$. By Lemma~\ref{lem:uniform_threshold_estimates},
the average is bounded by 
\[
C_{k,t}e^{-tn}\langle n\rangle^{M_{k}},
\]
and the divided difference is bounded by 
\[
C_{k,t}e^{-tn}\langle n\rangle^{M_{k}}\bigl(1+\log\langle n\rangle\bigr).
\]
The second estimate follows from the derivative bound \eqref{eq:uniform_threshold_derivative_bound},
applied to the functions 
\[
e^{it\lambda}s_{n,k}^{+}(\lambda),\qquad e^{-it\lambda}s_{n,k}^{-}(\lambda),
\]
whose values agree at $\lambda=0$.

Both dominating sequences belong to $\ell^{2}(\mathbb{N})$, because
$t>0$. Therefore the two limits hold in $\ell^{2}(\mathbb{N})$. 
\end{proof}

\subsubsection{Algebraic Jordan normal form}

We now define the threshold transform obtained from the coalescence
of the two renormalized branches. For $\lambda\neq0$, write 
\[
f_{+,\lambda}:=\widehat{\mathcal{T}}_{+,\lambda}e^{t\widetilde{X}_{\lambda}}p,\qquad f_{-,\lambda}:=\widehat{\mathcal{T}}_{-,\lambda}e^{t\widetilde{X}_{\lambda}}p.
\]
The natural threshold coordinates are 
\[
D_{\lambda}:=\frac{f_{+,\lambda}-f_{-,\lambda}}{2i\lambda},\qquad S_{\lambda}:=\frac{f_{+,\lambda}+f_{-,\lambda}}{2}.
\]
Equivalently, 
\[
f_{+,\lambda}=S_{\lambda}+i\lambda D_{\lambda},\qquad f_{-,\lambda}=S_{\lambda}-i\lambda D_{\lambda}.
\]

\begin{cBoxA}{}

\begin{defn}[Renormalized threshold transform]
\label{def:threshold_transform} For $t>0$, set 
\[
D_{t}^{0}:=e^{t\widetilde{X}_{0}}\mathrm{Pol}(S^{1}).
\]
If 
\[
u=e^{t\widetilde{X}_{0}}p,\qquad p=\sum_{|k|\le K}a_{k}\psi_{k},
\]
define 
\begin{equation}
\mathcal{T}_{0}u:=\sum_{|k|\le K}a_{k}\lim_{\lambda\to0}\begin{pmatrix}\dfrac{\widehat{\mathcal{T}}_{+,\lambda}e^{t\widetilde{X}_{\lambda}}\psi_{k}-\widehat{\mathcal{T}}_{-,\lambda}e^{t\widetilde{X}_{\lambda}}\psi_{k}}{2i\lambda}\\[2.2ex]
\dfrac{\widehat{\mathcal{T}}_{+,\lambda}e^{t\widetilde{X}_{\lambda}}\psi_{k}+\widehat{\mathcal{T}}_{-,\lambda}e^{t\widetilde{X}_{\lambda}}\psi_{k}}{2}
\end{pmatrix}.\label{eq:def_T0_threshold}
\end{equation}
The limit is taken in 
\[
\ell^{2}(\mathbb{N})\oplus\ell^{2}(\mathbb{N}),
\]
and exists by Proposition~\ref{prop:threshold_l2_limit}. 
\end{defn}

\end{cBoxA}

\begin{rem}
The definition uses the same initial trigonometric polynomial $p$
for all $\lambda$: 
\[
u=e^{t\widetilde{X}_{0}}p,\qquad\mathcal{T}_{0}u=\lim_{\lambda\to0}\mathcal{T}_{\lambda}(e^{t\widetilde{X}_{\lambda}}p),
\]
where $\mathcal{T}_{\lambda}$ denotes the two-component transform
in the coordinates $(D_{\lambda},S_{\lambda})$. Thus $\mathcal{T}_{0}u$
is not defined as a limit with $u$ fixed independently of $\lambda$. 
\end{rem}

\begin{cBoxB}{}

\begin{lem}[Algebraic Jordan normal form]
\label{lem:threshold_algebraic_Jordan_normal_form} In the threshold
coordinates 
\[
D=\frac{f_{+}-f_{-}}{2i\lambda},\qquad S=\frac{f_{+}+f_{-}}{2},
\]
the diagonal model 
\[
(L+i\lambda)\oplus(L-i\lambda)
\]
becomes 
\begin{equation}
\mathbb{J}_{\lambda}=\begin{pmatrix}L & 1\\
-\lambda^{2} & L
\end{pmatrix}.\label{eq:J_lambda_threshold}
\end{equation}
Equivalently, since $\mu=\lambda^{2}+\frac{1}{4}$, 
\[
\mathbb{J}_{\lambda}=\begin{pmatrix}L & 1\\
\frac{1}{4}-\mu & L
\end{pmatrix}.
\]
At the threshold $\lambda=0$, this gives 
\begin{equation}
\mathbb{J}_{0}=\begin{pmatrix}L & 1\\
0 & L
\end{pmatrix}.\label{eq:J_threshold}
\end{equation}
\end{lem}

\end{cBoxB}

\begin{proof}
We have 
\[
f_{+}'=(L+i\lambda)f_{+},\qquad f_{-}'=(L-i\lambda)f_{-}.
\]
Using 
\[
f_{+}=S+i\lambda D,\qquad f_{-}=S-i\lambda D,
\]
we obtain 
\[
D'=\frac{f_{+}'-f_{-}'}{2i\lambda}=LD+S,
\]
and 
\[
S'=\frac{f_{+}'+f_{-}'}{2}=LS-\lambda^{2}D.
\]
Therefore 
\[
\begin{pmatrix}D'\\
S'
\end{pmatrix}=\begin{pmatrix}L & 1\\
-\lambda^{2} & L
\end{pmatrix}\begin{pmatrix}D\\
S
\end{pmatrix}.
\]
This proves \eqref{eq:J_lambda_threshold}, and the threshold case
is obtained by setting $\lambda=0$. 
\end{proof}
\begin{cBoxB}{}

\begin{lem}[Threshold intertwining]
\label{lem:threshold_intertwining_T0} On the core $D_{t}^{0}$,
one has 
\begin{equation}
\mathcal{T}_{0}\widetilde{X}_{0}=\begin{pmatrix}L & 1\\
0 & L
\end{pmatrix}\mathcal{T}_{0}.\label{eq:T0_intertwining_threshold}
\end{equation}
\end{lem}

\end{cBoxB}

\begin{proof}
For $\lambda\neq0$, Lemma~\ref{lem:threshold_algebraic_Jordan_normal_form}
and the renormalized intertwining relations \eqref{eq:renormalized_intertwining_threshold}
give 
\[
\mathcal{T}_{\lambda}\widetilde{X}_{\lambda}=\begin{pmatrix}L & 1\\
-\lambda^{2} & L
\end{pmatrix}\mathcal{T}_{\lambda}.
\]
Let 
\[
u=e^{t\widetilde{X}_{0}}p\in D_{t}^{0}.
\]
Applying the preceding identity to $e^{t\widetilde{X}_{\lambda}}p$,
and letting $\lambda\to0$, gives the desired identity. Indeed, 
\[
\widetilde{X}_{\lambda}p=\widetilde{X}_{0}p+O(\lambda)
\]
in the finite-dimensional space $\mathrm{Pol}(S^{1})$. The passage
to the limit is justified by the uniform estimates of Proposition~\ref{prop:threshold_l2_limit},
applied to the finitely many trigonometric polynomials occurring in
\[
\widetilde{X}_{\lambda}p=\widetilde{X}_{0}p+i\lambda\sin\theta\,p.
\]
Hence 
\[
\mathcal{T}_{0}\widetilde{X}_{0}u=\begin{pmatrix}L & 1\\
0 & L
\end{pmatrix}\mathcal{T}_{0}u.
\]
\end{proof}

\subsubsection{Hilbertian Jordan model at the threshold}

We now complete the construction at the threshold.

\begin{cBoxB}{}

\begin{thm}[Hilbertian Jordan model at the threshold]
\label{thm:threshold_hilbertian_jordan_model} Let $\mu=\frac{1}{4}$,
and let $t>0$. Define 
\[
\mathcal{H}_{t}^{\mathrm{thr}}:=\overline{D_{t}^{0}}^{\,\|\mathcal{T}_{0}(\cdot)\|_{\ell^{2}\oplus\ell^{2}}},\qquad D_{t}^{0}=e^{t\widetilde{X}_{0}}\mathrm{Pol}(S^{1}).
\]
Then $\mathcal{H}_{t}^{\mathrm{thr}}$ is a Hilbert space, and 
\[
\mathcal{T}_{0}:\mathcal{H}_{t}^{\mathrm{thr}}\xrightarrow{\sim}\ell^{2}(\mathbb{N})\oplus\ell^{2}(\mathbb{N})
\]
extends uniquely as a unitary isomorphism.

Moreover, in $\mathcal{H}_{t}^{\mathrm{thr}}$, the operator $\widetilde{X}_{0}$
is unitarily equivalent to 
\begin{equation}
\mathbb{J}_{0}=\begin{pmatrix}L & 1\\
0 & L
\end{pmatrix},\qquad L=-A-\frac{1}{2}.\label{eq:def_J0_threshold_final}
\end{equation}
Equivalently, 
\begin{equation}
\widetilde{X}_{0}=\mathcal{T}_{0}^{-1}\begin{pmatrix}L & 1\\
0 & L
\end{pmatrix}\mathcal{T}_{0}.\label{eq:X0_Jordan_model}
\end{equation}

Therefore, as an operator on $\mathcal{H}_{t}^{\mathrm{thr}}$, 
\[
\operatorname{Spec}(\widetilde{X}_{0})=\left\{ -n-\frac{1}{2}\ ;\ n\in\mathbb{N}\right\} .
\]
For each 
\[
z_{n}=-n-\frac{1}{2},
\]
the generalized spectral space is two-dimensional, and the restriction
of $\widetilde{X}_{0}$ to this generalized spectral space is conjugated
to the Jordan block 
\[
\begin{pmatrix}z_{n} & 1\\
0 & z_{n}
\end{pmatrix}.
\]
In particular, there are no Jordan blocks of size $>2$.

Finally, in the model coordinates, for every $\tau>0$, one has the
norm-convergent expansion 
\begin{equation}
e^{\tau\mathbb{J}_{0}}=\sum_{n\in\mathbb{N}}e^{\tau z_{n}}\left(\Pi_{n}+\tau\mathcal{N}_{n}\right),\label{eq:threshold_semigroup_expansion_model}
\end{equation}
where 
\[
\Pi_{n}=\begin{pmatrix}e_{n}\langle e_{n}\mid\cdot\rangle & 0\\
0 & e_{n}\langle e_{n}\mid\cdot\rangle
\end{pmatrix},\qquad\mathcal{N}_{n}=\begin{pmatrix}0 & e_{n}\langle e_{n}\mid\cdot\rangle\\
0 & 0
\end{pmatrix}.
\]
Equivalently, transporting this identity by $\mathcal{T}_{0}^{-1}$
and $\mathcal{T}_{0}$ gives the corresponding norm-convergent expansion
for $e^{\tau\widetilde{X}_{0}}$ on $\mathcal{H}_{t}^{\mathrm{thr}}$.
\end{thm}

\end{cBoxB}

\begin{proof}
We first prove that $\mathcal{T}_{0}$ is injective on $D_{t}^{0}$.
Let 
\[
u\in D_{t}^{0},\qquad u=e^{t\widetilde{X}_{0}}p,\qquad p\in\mathrm{Pol}(S^{1}),
\]
and assume that 
\[
\mathcal{T}_{0}u=0.
\]
The second component of $\mathcal{T}_{0}u$ is the common threshold
limit of the two renormalized branches. Using the positive branch,
set 
\[
f_{+}:=\mathcal{U}_{+,0}u.
\]
The vanishing of the second component gives, for every $n\in\mathbb{N}$,
\[
f_{+}^{(n)}(0)=0.
\]
Since $u\in D_{t}^{0}$, the function $f_{+}$ is holomorphic near
$0$ and real-analytic on its interval of definition. Hence $f_{+}$
vanishes identically. Therefore 
\[
u=\mathcal{U}_{+,0}^{-1}f_{+}=0.
\]
Thus $\mathcal{T}_{0}$ is injective.

We next prove that $\mathcal{T}_{0}(D_{t}^{0})$ is dense in $\ell^{2}(\mathbb{N})\oplus\ell^{2}(\mathbb{N})$.
Set 
\[
\mathcal{K}:=\overline{\mathcal{T}_{0}(D_{t}^{0})}^{\,\ell^{2}\oplus\ell^{2}}.
\]
By the threshold intertwining relation \eqref{eq:T0_intertwining_threshold},
the image $\mathcal{T}_{0}(D_{t}^{0})$ is invariant under 
\[
\mathbb{J}_{0}=\begin{pmatrix}L & 1\\
0 & L
\end{pmatrix}.
\]
Moreover, the exponential estimates obtained in Proposition~\ref{prop:threshold_l2_limit}
imply that the vectors of $\mathcal{T}_{0}(D_{t}^{0})$ are analytic
vectors for $\mathbb{J}_{0}$. Hence $\mathcal{K}$ is invariant under
\[
e^{s\mathbb{J}_{0}},\qquad s\ge0.
\]

We prove by induction on $n$ that 
\[
\binom{e_{n}}{0}\in\mathcal{K},\qquad\binom{0}{e_{n}}\in\mathcal{K}.
\]
For fixed $n$, the $e_{n}$-coordinate of the second component of
$\mathcal{T}_{0}(e^{t\widetilde{X}_{0}}\psi_{k})$ is not identically
zero as a function of $k$. Indeed, it is proportional to the $n$-th
Taylor coefficient at $0$ of 
\[
(1+x^{2})^{-1/2}\left(\frac{1+ix}{1-ix}\right)^{k},
\]
which is a polynomial in $k$ of degree $n$ with nonzero leading
coefficient.

Assume that all basis vectors of order $m<n$ already belong to $\mathcal{K}$.
Choose $u\in D_{t}^{0}$ such that 
\[
v:=\mathcal{T}_{0}u
\]
has nonzero $e_{n}$-coefficient in its second component. Subtracting
the already constructed lower-order components, we may write 
\[
v=d_{n}\binom{e_{n}}{0}+c_{n}\binom{0}{e_{n}}+\text{terms of order }m>n,\qquad c_{n}\neq0.
\]
Let 
\[
z_{n}=-n-\frac{1}{2}.
\]
Since 
\[
e^{s\mathbb{J}_{0}}=\begin{pmatrix}e^{sL} & se^{sL}\\
0 & e^{sL}
\end{pmatrix},
\]
we have, as $s\to+\infty$, 
\[
\frac{1}{s}e^{-sz_{n}}e^{s\mathbb{J}_{0}}v\longrightarrow c_{n}\binom{e_{n}}{0}.
\]
Hence 
\[
\binom{e_{n}}{0}\in\mathcal{K}.
\]
Subtracting this first component from $v$, we may suppose 
\[
v=c_{n}\binom{0}{e_{n}}+\text{terms of order }m>n.
\]
Then, as $s\to+\infty$, 
\[
e^{-sz_{n}}e^{s\mathbb{J}_{0}}v-sc_{n}\binom{e_{n}}{0}\longrightarrow c_{n}\binom{0}{e_{n}}.
\]
Thus 
\[
\binom{0}{e_{n}}\in\mathcal{K}.
\]
This completes the induction. Therefore $\mathcal{K}$ contains the
canonical orthonormal basis of 
\[
\ell^{2}(\mathbb{N})\oplus\ell^{2}(\mathbb{N}),
\]
and hence 
\[
\mathcal{K}=\ell^{2}(\mathbb{N})\oplus\ell^{2}(\mathbb{N}).
\]

The injectivity and density just proved imply that the norm 
\[
\|u\|_{\mathcal{H}_{t}^{\mathrm{thr}}}:=\|\mathcal{T}_{0}u\|_{\ell^{2}\oplus\ell^{2}}
\]
is a genuine norm on $D_{t}^{0}$, and that the completion gives a
Hilbert space for which 
\[
\mathcal{T}_{0}:\mathcal{H}_{t}^{\mathrm{thr}}\xrightarrow{\sim}\ell^{2}(\mathbb{N})\oplus\ell^{2}(\mathbb{N})
\]
is unitary.

The intertwining formula 
\[
\mathcal{T}_{0}\widetilde{X}_{0}=\mathbb{J}_{0}\mathcal{T}_{0}
\]
then gives 
\[
\widetilde{X}_{0}=\mathcal{T}_{0}^{-1}\mathbb{J}_{0}\mathcal{T}_{0}.
\]
Since 
\[
Le_{n}=\left(-n-\frac{1}{2}\right)e_{n}=z_{n}e_{n},
\]
the two-dimensional space 
\[
F_{n}:=\mathbb{C}e_{n}\oplus\mathbb{C}e_{n}
\]
is invariant under $\mathbb{J}_{0}$, and the restriction of $\mathbb{J}_{0}$
to $F_{n}$ is 
\[
\begin{pmatrix}z_{n} & 1\\
0 & z_{n}
\end{pmatrix}.
\]
The spectral and Jordan statements follow by conjugation through $\mathcal{T}_{0}$.

Finally, 
\[
e^{\tau\mathbb{J}_{0}}=\begin{pmatrix}e^{\tau L} & \tau e^{\tau L}\\
0 & e^{\tau L}
\end{pmatrix},
\]
and 
\[
e^{\tau L}=\sum_{n\in\mathbb{N}}e^{\tau z_{n}}e_{n}\langle e_{n}\mid\cdot\rangle.
\]
Thus 
\[
e^{\tau\mathbb{J}_{0}}=\sum_{n\in\mathbb{N}}e^{\tau z_{n}}\left(\Pi_{n}+\tau\mathcal{N}_{n}\right)
\]
in model coordinates. Since 
\[
e^{\tau z_{n}}=e^{-\tau(n+\frac{1}{2})},
\]
the series converges in operator norm for every $\tau>0$. Transporting
this identity by $\mathcal{T}_{0}^{-1}$ and $\mathcal{T}_{0}$ gives
\eqref{eq:threshold_semigroup_expansion_model}. 
\end{proof}
\begin{rem}
The factor 
\[
\tau e^{\tau z_{n}}
\]
in \eqref{eq:threshold_semigroup_expansion_model} is the signature
of the Jordan block. Equivalently, it is the contribution of the logarithmic
partner produced by the coalescence of the two branches 
\[
z_{n,+}(\lambda)=z_{n}+i\lambda,\qquad z_{n,-}(\lambda)=z_{n}-i\lambda.
\]
\end{rem}

\subsection{\protect\label{subsec:Trace-of--1}Trace of $e^{t\tilde{X}}$ in
the Hilbert model}

\begin{cBoxB}{}

\begin{thm}
\label{thm:flat_trace_equals_Hilbert_trace} Assume that $\lambda>0$.
For every $t>0$, the operator $e^{t\tilde{X}}$ is not trace class
on $L^{2}(S^{1})$. However, its flat trace is well defined and is
given by 
\begin{equation}
\operatorname{Tr}_{S^{1}}^{\flat}\left(e^{t\tilde{X}}\right):=\int_{S^{1}}\langle\delta_{\theta}\mid e^{t\tilde{X}}\delta_{\theta}\rangle\,d\theta=2\cos(t\lambda)\frac{e^{-t/2}}{1-e^{-t}}.\label{eq:trace_flat_spherical}
\end{equation}
On the other hand, for every $\tau>0$, the operator $e^{t\tilde{X}}$
is trace class in the Hilbert model $\mathcal{H}_{\tau}$, and 
\[
\operatorname{Tr}_{\mathcal{H}_{\tau}}\left(e^{t\tilde{X}}\right)=\operatorname{Tr}_{S^{1}}^{\flat}\left(e^{t\tilde{X}}\right).
\]
\end{thm}

\end{cBoxB}

\begin{proof}
On $L^{2}(S^{1})$, the operator $\tilde{X}$ is skew-adjoint, hence
$e^{t\tilde{X}}$ is unitary. Since $L^{2}(S^{1})$ is infinite-dimensional,
this unitary operator is not compact, and therefore not trace class.

We now compute its flat trace. Recall that 
\[
\tilde{X}\eq{\ref{eq:expression_X_sl2R}}\cos\theta\,\frac{d}{d\theta}+b\sin\theta,\qquad b\eq{\ref{eq:def_b}}-\frac{1}{2}+i\lambda.
\]
Let $\phi^{t}:S^{1}\to S^{1}$ be the flow generated by 
\[
\cos\theta\,\frac{d}{d\theta}.
\]
Its fixed points are 
\[
\theta=\frac{\pi}{2},\qquad\theta=-\frac{\pi}{2},
\]
and 
\[
d\phi_{\pi/2}^{t}=e^{-t},\qquad d\phi_{-\pi/2}^{t}=e^{t}.
\]
The propagator is the weighted transport operator 
\[
(e^{t\tilde{X}}u)(\theta)=\exp\left(\int_{0}^{t}b\sin(\phi^{r}(\theta))\,dr\right)u(\phi^{t}(\theta)).
\]
Therefore the flat trace is the sum of the two fixed-point contributions:
\[
\operatorname{Tr}_{S^{1}}^{\flat}\left(e^{t\tilde{X}}\right)=\frac{e^{tb}}{|1-e^{-t}|}+\frac{e^{-tb}}{|1-e^{t}|}.
\]
Since $t>0$, this becomes 
\[
\operatorname{Tr}_{S^{1}}^{\flat}\left(e^{t\tilde{X}}\right)=\frac{e^{tb}}{1-e^{-t}}+\frac{e^{-tb}}{e^{t}-1}.
\]
Using 
\[
\frac{1}{e^{t}-1}=\frac{e^{-t}}{1-e^{-t}},
\]
and $b=-\frac{1}{2}+i\lambda$, we get 
\[
\operatorname{Tr}_{S^{1}}^{\flat}\left(e^{t\tilde{X}}\right)=\frac{e^{-t/2+it\lambda}}{1-e^{-t}}+\frac{e^{-t/2-it\lambda}}{1-e^{-t}},
\]
hence 
\[
\operatorname{Tr}_{S^{1}}^{\flat}\left(e^{t\tilde{X}}\right)=2\cos(t\lambda)\frac{e^{-t/2}}{1-e^{-t}}.
\]

In the Hilbert model $\mathcal{H}_{\tau}$, Theorem~\ref{thm:summary_spectral}
gives the discrete spectrum 
\[
z_{n,\pm}\eq{\ref{eq:def_z_n}}-n-\frac{1}{2}\pm i\lambda,\qquad n\in\mathbb{N}.
\]
Thus $e^{t\tilde{X}}$ is trace class on $\mathcal{H}_{\tau}$, and
\[
\operatorname{Tr}_{\mathcal{H}_{\tau}}\left(e^{t\tilde{X}}\right)=\sum_{n\ge0}\left(e^{tz_{n,+}}+e^{tz_{n,-}}\right).
\]
Therefore 
\[
\operatorname{Tr}_{\mathcal{H}_{\tau}}\left(e^{t\tilde{X}}\right)=2\cos(t\lambda)\sum_{n\ge0}e^{-t(n+\frac{1}{2})}=2\cos(t\lambda)\frac{e^{-t/2}}{1-e^{-t}}.
\]
This is exactly \eqref{eq:trace_flat_spherical}. 
\end{proof}
The same identity extends to the complementary series by analytic
continuation in the spectral parameter $\lambda$, with $\lambda=i\nu$,
$0<\nu<1/2$. At the threshold $\lambda=0$, the two branches coalesce
into a Jordan block; the nilpotent part has zero trace, and the trace
is the limit of the preceding formula.

\subsection{Global expressions for spherical principal and complementary series}

We now globalize the preceding constructions over the spectral decomposition
of the Laplacian on the compact surface 
\[
\mathcal{N}=\Gamma\backslash\mathbb{H}^{2}.
\]
Since we assume $-\mathrm{Id}\in\Gamma$, only even spherical representations
occur in the present part.

Let 
\[
L_{0}^{2}(\mathcal{N}):=(\mathbb{C}\mathbf{1})^{\perp}\subset L^{2}(\mathcal{N}).
\]
The positive spectrum of $\Delta$ on $L_{0}^{2}(\mathcal{N})$ is
discrete. For $\mu>0$, we write 
\[
E_{\mu}:=\ker(\Delta-\mu).
\]
We separate the possible threshold eigenspace: 
\[
E_{1/4}:=\ker\left(\Delta-\frac{1}{4}\right),\qquad E_{\neq1/4}:=\mathbf{1}_{(0,\infty)\setminus\{1/4\}}(\Delta)L_{0}^{2}(\mathcal{N}).
\]
Thus 
\[
L_{0}^{2}(\mathcal{N})=E_{\neq1/4}\oplus E_{1/4},
\]
where $E_{1/4}$ may be zero.

On $E_{\neq1/4}$, define by spectral calculus 
\begin{equation}
\Lambda:=\left(\Delta-\frac{1}{4}\right)^{1/2},\label{eq:def_global_Lambda_spherical}
\end{equation}
with the convention 
\[
\Lambda_{\mid E_{\mu}}=\begin{cases}
\sqrt{\mu-\frac{1}{4}}\,\mathrm{Id}, & \mu>\frac{1}{4},\\[0.4em]
i\sqrt{\frac{1}{4}-\mu}\,\mathrm{Id}, & 0<\mu<\frac{1}{4}.
\end{cases}
\]
We also set 
\[
L:=-A-\frac{1}{2}
\]
on $\ell^{2}(\mathbb{N})$, and 
\[
\Sigma:=\begin{pmatrix}1 & 0\\
0 & -1
\end{pmatrix}
\]
on $\mathbb{C}^{2}$.

The model Hilbert space for the non-threshold spherical part is 
\begin{equation}
\widetilde{\mathcal{K}}_{\mathrm{sph},\neq1/4}:=\ell^{2}(\mathbb{N})\otimes E_{\neq1/4}\otimes\mathbb{C}^{2}.\label{eq:model_spherical_non_threshold}
\end{equation}
On this space, the global model operator is 
\begin{equation}
\mathbb{X}_{\mathrm{sph},\neq1/4}:=L\otimes\mathrm{Id}\otimes\mathrm{Id}+i\,\mathrm{Id}\otimes\Lambda\otimes\Sigma.\label{eq:X_model_spherical_non_threshold}
\end{equation}
Indeed, on $E_{\mu}$ its eigenvalues are 
\[
z_{n,\pm}(\mu)=-n-\frac{1}{2}\pm i\sqrt{\mu-\frac{1}{4}},
\]
where the square root is understood with the convention above. Thus
this single formula contains both the principal and complementary
series.

At the threshold $\mu=\frac{1}{4}$, the two branches coalesce and
the diagonal model must be replaced by the Jordan model 
\[
\mathbb{J}_{0}=\begin{pmatrix}L & 1\\
0 & L
\end{pmatrix}.
\]
If $E_{1/4}\neq0$, we set 
\begin{equation}
\widetilde{\mathcal{K}}_{\mathrm{sph},1/4}:=\ell^{2}(\mathbb{N})\otimes E_{1/4}\otimes\mathbb{C}^{2},\label{eq:model_spherical_threshold}
\end{equation}
with model operator 
\begin{equation}
\mathbb{X}_{\mathrm{sph},1/4}:=\mathbb{J}_{0}\otimes\mathrm{Id}_{E_{1/4}}.\label{eq:X_model_spherical_threshold}
\end{equation}
If $E_{1/4}=0$, this summand is absent.

Finally, we set 
\begin{equation}
\widetilde{\mathcal{K}}_{\mathrm{sph}}:=\widetilde{\mathcal{K}}_{\mathrm{sph},\neq1/4}\oplus\widetilde{\mathcal{K}}_{\mathrm{sph},1/4}.\label{eq:model_spherical_total}
\end{equation}

Let 
\[
\mathcal{P}_{\mathrm{sph}}:=\mathcal{P}(M)\cap\mathcal{H}^{\mathrm{sph}}
\]
be the $K$-finite and spectrally finite algebraic core of the spherical
part. For $t\in\mathbb{R}$, set 
\begin{equation}
\mathcal{D}_{\mathrm{sph},t}:=e^{tX}\mathcal{P}_{\mathrm{sph}}.\label{eq:def_D_sph_t}
\end{equation}

\begin{cBoxB}{}

\begin{thm}[Global Hilbert model for the spherical part]
\label{thm:global_spherical_model_X} Let $t>0$. With the convention
that $\mathcal{T}$ denotes, in each spectral fiber, the appropriate
transform constructed above in (\ref{eq:def_cal_T}, \ref{eq:def_cal_T-1})
and (\ref{eq:def_T0_threshold}), and $\mathcal{U}_{1}$ defined in
(\ref{eq:def_U}), we set 
\begin{equation}
\mathbb{T}_{\mathrm{sph}}:=\bigoplus_{\mu>0}\mathcal{T}\mathcal{U}_{1}\qquad:\mathcal{D}_{\mathrm{sph},t}\longrightarrow\widetilde{\mathcal{K}}_{\mathrm{sph}}.\label{eq:def_TT}
\end{equation}
This map is injective and has dense image. We define 
\[
\|u\|_{\mathcal{H}_{\mathrm{sph},t}}:=\|\mathbb{T}_{\mathrm{sph}}u\|_{\widetilde{\mathcal{K}}_{\mathrm{sph}}},\qquad u\in\mathcal{D}_{\mathrm{sph},t},
\]
and 
\[
\mathcal{H}_{\mathrm{sph},t}:=\overline{\mathcal{D}_{\mathrm{sph},t}}^{\|\cdot\|_{\mathcal{H}_{\mathrm{sph},t}}}.
\]
Then 
\[
\mathbb{T}_{\mathrm{sph}}:\mathcal{H}_{\mathrm{sph},t}\xrightarrow{\sim}\widetilde{\mathcal{K}}_{\mathrm{sph}}
\]
extends uniquely as a unitary isomorphism.

On the non-threshold spherical part, one has 
\begin{equation}
\mathbb{T}_{\mathrm{sph}}X\mathbb{T}_{\mathrm{sph}}^{-1}=L+i\Sigma\Lambda=-A-\frac{1}{2}+i\Sigma\Lambda.\label{eq:global_X_non_threshold}
\end{equation}
Moreover, on the same non-threshold part, 
\begin{equation}
\mathbb{T}_{\mathrm{sph}}U\mathbb{T}_{\mathrm{sph}}^{-1}=-\Sigma\,a^{+}\left(A+1-2i\Sigma\Lambda\right)^{1/2},\label{eq:global_U_spherical}
\end{equation}
and 
\begin{equation}
\mathbb{T}_{\mathrm{sph}}S\mathbb{T}_{\mathrm{sph}}^{-1}=\Sigma\left(A+1-2i\Sigma\Lambda\right)^{1/2}a^{-}.\label{eq:global_S_spherical}
\end{equation}
The square root in $(A+1-2i\Sigma\Lambda)^{1/2}$ is understood fiberwise,
with the branch determined by the diagonal gauge of Lemma~\ref{lem:diagonal_algebraic_gauge},
as used in Corollary~\ref{cor:Tcal_pm} and Subsection~\ref{subsec:Complementary-series}.
Here tensor products with identity operators are omitted from the
notation.

On the threshold part $E_{1/4}$, if it is nonzero, one has 
\begin{equation}
\mathbb{T}_{\mathrm{sph}}X\mathbb{T}_{\mathrm{sph}}^{-1}=\mathbb{J}_{0}=\begin{pmatrix}L & 1\\
0 & L
\end{pmatrix}.\label{eq:global_X_threshold}
\end{equation}

Finally, on the whole spherical part, 
\begin{equation}
\mathbb{T}_{\mathrm{sph}}\Omega\mathbb{T}_{\mathrm{sph}}^{-1}=\mathrm{Id}_{\ell^{2}(\mathbb{N})}\otimes\Delta\otimes\mathrm{Id}_{\mathbb{C}^{2}}.\label{eq:global_Omega_spherical}
\end{equation}

Consequently, on the non-threshold spherical part, the resonances
of $X$ are 
\begin{equation}
z_{n,\pm}(\mu)=-n-\frac{1}{2}\pm i\sqrt{\mu-\frac{1}{4}},\qquad n\in\mathbb{N},\qquad\mu\neq\frac{1}{4},\label{eq:global_spherical_resonances}
\end{equation}
where for $0<\mu<1/4$ the square root is understood through the convention
\eqref{eq:def_global_Lambda_spherical}. At the threshold, the resonance
\[
z_{n,0}=-n-\frac{1}{2}
\]
has a Jordan block of size $2$ on 
\[
e_{n}\otimes E_{1/4}\otimes\mathbb{C}^{2}.
\]
\end{thm}

\end{cBoxB}

\begin{proof}
The construction is obtained by applying the preceding fiberwise models
to each eigenspace 
\[
E_{\mu}=\ker(\Delta-\mu).
\]
For $\mu\neq1/4$, the spherical model is diagonal and is given by
\[
L+i\Sigma\Lambda.
\]
Indeed, if $\mu>1/4$, then 
\[
\Lambda_{\mid E_{\mu}}=\sqrt{\mu-\frac{1}{4}},
\]
and this gives the two principal-series branches 
\[
-n-\frac{1}{2}\pm i\sqrt{\mu-\frac{1}{4}}.
\]
If $0<\mu<1/4$, then 
\[
\Lambda_{\mid E_{\mu}}=i\sqrt{\frac{1}{4}-\mu},
\]
and the same formula gives the two complementary-series branches 
\[
-n-\frac{1}{2}\mp\sqrt{\frac{1}{4}-\mu}.
\]
Thus \eqref{eq:global_X_non_threshold} contains both cases.

The formulas for $U$ and $S$ are obtained by the same fiberwise
globalization of the local formulas. On the $+$-branch one has 
\[
U=-a^{+}(A+1-2i\Lambda)^{1/2},\qquad S=(A+1-2i\Lambda)^{1/2}a^{-},
\]
while on the $-$-branch one has 
\[
U=a^{+}(A+1+2i\Lambda)^{1/2},\qquad S=-(A+1+2i\Lambda)^{1/2}a^{-}.
\]
These two branch formulas are exactly \eqref{eq:global_U_spherical}
and \eqref{eq:global_S_spherical}. The order of the factors is inherited
from the local formulas; in particular, $a^{\pm}$ should not be commuted
with functions of $A$.

At $\mu=1/4$, the two branches coalesce, and the threshold construction
gives the Jordan model 
\[
\mathbb{J}_{0}=\begin{pmatrix}L & 1\\
0 & L
\end{pmatrix}.
\]
This proves \eqref{eq:global_X_threshold}. The corresponding threshold
formulas for $U$ and $S$ can be obtained by the same limiting procedure,
but they are not needed here.

Taking the Hilbert direct sum over the spectral decomposition of $\Delta$
gives the global transform 
\[
\mathbb{T}_{\mathrm{sph}}:\mathcal{D}_{\mathrm{sph},t}\longrightarrow\widetilde{\mathcal{K}}_{\mathrm{sph}}.
\]
The fiberwise injectivity and density imply that $\mathbb{T}_{\mathrm{sph}}$
is injective and has dense image. Hence the norm 
\[
\|u\|_{\mathcal{H}_{\mathrm{sph},t}}=\|\mathbb{T}_{\mathrm{sph}}u\|_{\widetilde{\mathcal{K}}_{\mathrm{sph}}}
\]
is a genuine norm on $\mathcal{D}_{\mathrm{sph},t}$, and the completion
gives the unitary isomorphism 
\[
\mathbb{T}_{\mathrm{sph}}:\mathcal{H}_{\mathrm{sph},t}\xrightarrow{\sim}\widetilde{\mathcal{K}}_{\mathrm{sph}}.
\]

Finally, the Casimir operator acts on the spherical summand associated
with $E_{\mu}$ by the scalar $\mu$. Therefore, globally, it is transported
to 
\[
\mathrm{Id}_{\ell^{2}(\mathbb{N})}\otimes\Delta\otimes\mathrm{Id}_{\mathbb{C}^{2}},
\]
which proves \eqref{eq:global_Omega_spherical}. The spectral statement
then follows directly from \eqref{eq:global_X_non_threshold} and
\eqref{eq:global_X_threshold}. 
\end{proof}
\begin{rem}
The formulas \eqref{eq:global_U_spherical} and \eqref{eq:global_S_spherical}
are stated only on the non-threshold spherical part, which is sufficient
for the applications below. The essential global formula is \eqref{eq:global_X_non_threshold},
together with the Jordan replacement \eqref{eq:global_X_threshold}
at $\mu=1/4$. 
\end{rem}

\section{\protect\label{sec:Discrete-series}Discrete series}

As before, we assume that $-\mathrm{Id}\in\Gamma$, so that only even
$K$-types occur. We keep the notation of Section~\ref{subsec:Decomposition-of-}:
the genuine $K$-type is denoted by 
\[
l\in2\mathbb{Z}.
\]
Thus the holomorphic discrete series have lowest $K$-type 
\[
l\in2\mathbb{N}^{*},
\]
as in \eqref{eq:def_ds_plus}, while the anti-holomorphic discrete
series have highest $K$-type $-l$, as in \eqref{eq:def_ds_minus}.

The purpose of this section is to construct Hilbert spaces adapted
to the generator $X$ on the discrete-series part. As in the spherical
case, the $K$-finite basis is not an eigenbasis for $X$. The eigenvectors
of $X$ will appear only after a non-unitary transform, and they should
be understood as generalized resonant states. We therefore again use
propagated domains.

We first treat one irreducible holomorphic discrete series 
\[
\mathcal{H}_{l}^{\mathrm{d.s.},+}(v),
\]
defined in \eqref{eq:def_ds_plus}. The anti-holomorphic case will
be obtained by complex conjugation.

The construction proceeds in three steps. We first identify $\mathcal{H}_{l}^{\mathrm{d.s.},+}(v)$
with a weighted Bergman space $H_{l}(\mathbb{D})$. We then apply
a non-unitary Cayley transform $T_{\mathcal{C}^{-1}}$, and finally
pass to the canonical basis of $\ell^{2}(\mathbb{N})$. Schematically,
\[
\xymatrix{\mathcal{H}_{l}^{\mathrm{d.s.},+}(v)\ar[r]_{\eqref{eq:def_U1}}^{\mathcal{U}_{1}} & H_{l}(\mathbb{D})\ar[r]_{\eqref{eq:def_Cayley_operator_ds}}^{T_{\mathcal{C}^{-1}}} & H_{l}(\mathbb{D})\ar[r]_{\eqref{eq:def_F-1}}^{F} & \ell^{2}(\mathbb{N}).}
\]
We write 
\begin{equation}
\mathbb{T}_{\mathrm{d.s.}}\eq{\eqref{eq:def_TT_ds}}F\,T_{\mathcal{C}^{-1}}\mathcal{U}_{1}.\label{eq:def_T_l_alg_ds_intro}
\end{equation}
Later, after summing over all discrete-series components, $\mathbb{T}_{\mathrm{d.s.}}$
will form the discrete-series part of the global transform $\mathbb{T}$
appearing in the main theorem, in (\ref{eq:TT}).

\subsection{Holomorphic realization on the disk}

We recall standard results, see for instance \cite[p.~71]{perelomov1}.
Let 
\[
\mathbb{D}:=\{w\in\mathbb{C}\ ;\ |w|<1\}.
\]
For $l\ge2$, define the weighted Bergman space $H_{l}(\mathbb{D})$
by the scalar product 
\begin{equation}
\langle f\mid g\rangle_{H_{l}}:=\frac{l-1}{\pi}\int_{\mathbb{D}}\overline{f(w)}g(w)(1-|w|^{2})^{l-2}\,dA(w),\label{eq:def_Hl_scalar_product_ds}
\end{equation}
where 
\[
dA(w)=d(\operatorname{Re}w)\,d(\operatorname{Im}w)
\]
denotes the Euclidean area measure. The space $H_{l}(\mathbb{D})$
admits the orthonormal basis 
\begin{equation}
\psi_{k}(w):=\binom{l+k-1}{k}^{1/2}w^{k},\qquad k\in\mathbb{N}.\label{eq:def_ek_disk_ds}
\end{equation}

\begin{cBoxB}{}

\begin{lem}[Holomorphic model of one discrete series]
\label{lem:holomorphic_model_ds} \cite[p.~71]{perelomov1} Let $l\in2\mathbb{N}^{*}$,
and let $(\varphi_{k})_{k\in\mathbb{N}}$ be the orthonormal basis
of $\mathcal{H}_{l}^{\mathrm{d.s.},+}(v)$ constructed in Subsection~\ref{subsec:Orthonormal-basis-for}.
The map 
\begin{equation}
\mathcal{U}_{1}:\mathcal{H}_{l}^{\mathrm{d.s.},+}(v)\longrightarrow H_{l}(\mathbb{D}),\qquad\mathcal{U}_{1}\varphi_{k}=\psi_{k},\label{eq:def_U1}
\end{equation}
is unitary and gives, on the holomorphic polynomial core, 
\begin{equation}
\mathcal{U}_{1}\Theta(\mathcal{U}_{1})^{-1}=2w\partial_{w}+l,\label{eq:Theta_disk_ds}
\end{equation}
\begin{equation}
\mathcal{U}_{1}N_{+}(\mathcal{U}_{1})^{-1}=w^{2}\partial_{w}+lw,\label{eq:Nplus_disk_ds}
\end{equation}
and 
\begin{equation}
\mathcal{U}_{1}N_{-}(\mathcal{U}_{1})^{-1}=-\partial_{w}.\label{eq:Nminus_disk_ds}
\end{equation}
Consequently, 
\begin{equation}
\mathcal{U}_{1}\Omega(\mathcal{U}_{1})^{-1}=-\frac{1}{4}l(l-2)\,\mathrm{Id}.\label{eq:Omega_disk_ds}
\end{equation}
\end{lem}

\end{cBoxB}

\begin{proof}
Orthogonality follows from the angular integration in polar coordinates.
Moreover, 
\[
\frac{l-1}{\pi}\int_{\mathbb{D}}|w|^{2k}(1-|w|^{2})^{l-2}\,dA(w)=\frac{k!(l-1)!}{(l+k-1)!}.
\]
Hence the functions $\psi_{k}$ form an orthonormal family. Since
holomorphic polynomials are dense in $H_{l}(\mathbb{D})$, this family
is an orthonormal basis. By construction, $\mathcal{U}_{1}$ maps
an orthonormal basis onto an orthonormal basis, hence is unitary.

We compute 
\[
(2w\partial_{w}+l)\psi_{k}=(2k+l)\psi_{k}.
\]
Since 
\[
\Theta\varphi_{k}=(l+2k)\varphi_{k},
\]
this proves \eqref{eq:Theta_disk_ds}.

Next, 
\[
(w^{2}\partial_{w}+lw)\psi_{k}=(k+l)\sqrt{\binom{l+k-1}{k}}\,w^{k+1}.
\]
Using 
\[
\binom{l+k}{k+1}=\frac{l+k}{k+1}\binom{l+k-1}{k},
\]
we get 
\[
(w^{2}\partial_{w}+lw)\psi_{k}=\sqrt{(k+1)(l+k)}\,\psi_{k+1}.
\]
This agrees with the action of $N_{+}$ on the basis $(\varphi_{k})$,
and proves \eqref{eq:Nplus_disk_ds}.

Similarly, 
\[
-\partial_{w}\psi_{k}=-\sqrt{k(l+k-1)}\,\psi_{k-1},
\]
which agrees with the action of $N_{-}$. This proves \eqref{eq:Nminus_disk_ds}.

Finally, using 
\[
\Omega=-\left(\frac{1}{2}\Theta\right)^{2}-\frac{1}{2}N_{+}N_{-}-\frac{1}{2}N_{-}N_{+},
\]
it is enough to compute on the lowest weight vector $\psi_{0}$. Since
\[
N_{-}\psi_{0}=0,\qquad\Theta\psi_{0}=l\psi_{0},\qquad N_{-}N_{+}\psi_{0}=-l\psi_{0},
\]
we get 
\[
\Omega\psi_{0}=-\frac{l^{2}}{4}\psi_{0}+\frac{l}{2}\psi_{0}=-\frac{1}{4}l(l-2)\psi_{0}.
\]
Since $\Omega$ is central, it acts by this scalar on the whole irreducible
representation. 
\end{proof}
\begin{cBoxB}{}

\begin{lem}[Right action in the disk model]
\label{lem:right_action_disk_ds} Let $l\in2\mathbb{N}^{*}$, and
let $g\in SL_{2}(\mathbb{R})$. Recall that we use the right action
\[
(R_{g}^{\circ}v)(g')=v(g'g).
\]
Let 
\begin{equation}
\mathcal{C}:=\left(\frac{1+i}{2}\right)\begin{pmatrix}1 & i\\
1 & -i
\end{pmatrix}\in SL(2,\mathbb{C})\label{eq:def_Cayley}
\end{equation}
be the Cayley transform, and write 
\[
\mathcal{C}g^{-1}\mathcal{C}^{-1}=\begin{pmatrix}\alpha_{g} & \beta_{g}\\
\overline{\beta_{g}} & \overline{\alpha_{g}}
\end{pmatrix}\in SU(1,1).
\]
Then 
\[
\mathcal{U}_{1}R_{g}^{\circ}\mathcal{U}_{1}^{-1}=T_{g},
\]
where $T_{g}:H_{l}(\mathbb{D})\to H_{l}(\mathbb{D})$ is the unitary
operator 
\begin{equation}
(T_{g}u)(w)=(\overline{\beta_{g}}w+\overline{\alpha_{g}})^{-l}u\left(\frac{\alpha_{g}w+\beta_{g}}{\overline{\beta_{g}}w+\overline{\alpha_{g}}}\right).\label{eq:def_Tg_disk_ds}
\end{equation}
\end{lem}

\end{cBoxB}

\begin{proof}
Let 
\[
h=\begin{pmatrix}\alpha & \beta\\
\overline{\beta} & \overline{\alpha}
\end{pmatrix}\in SU(1,1).
\]
The standard holomorphic action of $SU(1,1)$ on $H_{l}(\mathbb{D})$
is 
\[
(S_{h}u)(w)=(\overline{\beta}w+\overline{\alpha})^{-l}u\left(\frac{\alpha w+\beta}{\overline{\beta}w+\overline{\alpha}}\right).
\]
It satisfies 
\[
S_{h}S_{h'}=S_{h'h},
\]
because the automorphy factor satisfies the usual cocycle identity.

Since the representation on $L^{2}(M)$ is the right regular action,
we set 
\[
T_{g}:=S_{\mathcal{C}g^{-1}\mathcal{C}^{-1}}.
\]
Then, for $g_{1},g_{2}\in SL_{2}(\mathbb{R})$, 
\[
T_{g_{1}}T_{g_{2}}=S_{\mathcal{C}g_{1}^{-1}\mathcal{C}^{-1}}S_{\mathcal{C}g_{2}^{-1}\mathcal{C}^{-1}}=S_{\mathcal{C}g_{2}^{-1}g_{1}^{-1}\mathcal{C}^{-1}}=S_{\mathcal{C}(g_{1}g_{2})^{-1}\mathcal{C}^{-1}}=T_{g_{1}g_{2}}.
\]
This agrees with 
\[
R_{g_{1}}^{\circ}R_{g_{2}}^{\circ}=R_{g_{1}g_{2}}^{\circ}.
\]

Differentiating this action at the identity gives 
\[
\Theta\mapsto2w\partial_{w}+l,\qquad N_{+}\mapsto w^{2}\partial_{w}+lw,\qquad N_{-}\mapsto-\partial_{w},
\]
which agrees with Lemma~\ref{lem:holomorphic_model_ds}. Therefore
the two strongly continuous representations 
\[
g\longmapsto\mathcal{U}_{1}R_{g}^{\circ}\mathcal{U}_{1}^{-1},\qquad g\longmapsto T_{g}
\]
have the same infinitesimal action. Since $SL_{2}(\mathbb{R})$ is
connected, they coincide. Hence 
\[
\mathcal{U}_{1}R_{g}^{\circ}\mathcal{U}_{1}^{-1}=T_{g}.
\]
\end{proof}

\subsection{Algebraic Cayley transform and diagonalization of $X$}

In the disk model of Subsection~\ref{lem:holomorphic_model_ds},
the operator 
\[
\mathcal{U}_{1}\Theta\mathcal{U}_{1}^{-1}=2w\partial_{w}+l
\]
is diagonal in the orthonormal basis $(\psi_{k})_{k\in\mathbb{N}}$.
Our aim is to construct a non-unitary model in which the hyperbolic
generator $X$, rather than the elliptic generator $\Theta$, becomes
diagonal.

The key algebraic fact is that the Cayley transform exchanges the
elliptic and hyperbolic directions. We use the matrix $\mathcal{C}$
defined in \eqref{eq:def_Cayley}.

\begin{cBoxB}{}

\begin{lem}[Cayley conjugation]
\label{lem:With-the-Cayley} In the complexified Lie algebra $\mathfrak{sl}_{2}(\mathbb{C})$,
one has 
\begin{equation}
\mathcal{C}\Theta\mathcal{C}^{-1}=-2X,\qquad\mathcal{C}N_{+}\mathcal{C}^{-1}=U,\qquad\mathcal{C}N_{-}\mathcal{C}^{-1}=S.\label{eq:conjugason_C}
\end{equation}
\end{lem}

\end{cBoxB}

\begin{proof}
Recall that 
\[
\Theta\eq{\ref{eq:def_J}}\begin{pmatrix}0 & -i\\
i & 0
\end{pmatrix},
\]
and 
\[
N_{+}\eq{\ref{eq:def_N+-},\ref{eq:X,U,S}}\frac{1}{2}\begin{pmatrix}1 & i\\
i & -1
\end{pmatrix},\qquad N_{-}=\frac{1}{2}\begin{pmatrix}1 & -i\\
-i & -1
\end{pmatrix}.
\]
Together with the definitions of $X,U,S$ in \eqref{eq:X,U,S}, the
identities \eqref{eq:conjugason_C} follow by direct matrix multiplication. 
\end{proof}
We now use the Cayley transform to diagonalize the generator $X$
in one holomorphic discrete series.

Let $l\in2\mathbb{N}^{*}$ be fixed. For simplicity, set 
\[
\tilde{X}:=\mathcal{U}_{1}X\mathcal{U}_{1}^{-1},\qquad\tilde{U}:=\mathcal{U}_{1}U\mathcal{U}_{1}^{-1},\qquad\tilde{S}:=\mathcal{U}_{1}S\mathcal{U}_{1}^{-1}.
\]
As in \eqref{eq:def_F}, we introduce the coefficient map 
\begin{equation}
F:=\sum_{n\in\mathbb{N}}e_{n}\langle\psi_{n}\mid\,\cdot\,\rangle_{H_{l}(\mathbb{D})}\quad:\quad H_{l}(\mathbb{D})\longrightarrow\ell^{2}(\mathbb{N}).\label{eq:def_F-1}
\end{equation}
We shall also use the algebraic space of $K$-finite vectors 
\[
\mathcal{P}_{l}:=\operatorname{Span}\{\psi_{k}\ ;\ k\in\mathbb{N}\}\subset H_{l}(\mathbb{D}).
\]

We algebraically extend the formula \eqref{eq:def_Tg_disk_ds} for
$T_{g}$ to the complex matrix 
\[
g=\mathcal{C}^{-1}\in SL_{2}(\mathbb{C}).
\]
With the right-action convention used above, this gives, for $u\in\mathcal{P}_{l}$,
\begin{equation}
\left(T_{\mathcal{C}^{-1}}u\right)(w)=\left(\frac{1-i}{w-i}\right)^{l}u\left(\frac{w+i}{w-i}\right).\label{eq:def_Cayley_operator_ds}
\end{equation}
The operator $T_{\mathcal{C}^{-1}}$ is not unitary. It maps $\mathcal{P}_{l}$
to holomorphic functions on $\mathbb{D}$ which extend meromorphically
to a neighbourhood of $\overline{\mathbb{D}}$, with a possible pole
at the boundary point $w=i$. More explicitly, 
\begin{equation}
T_{\mathcal{C}^{-1}}\psi_{k}(w)=\binom{l+k-1}{k}^{1/2}(1-i)^{l}\frac{(w+i)^{k}}{(w-i)^{l+k}}.\label{eq:T_C_inv_psi_k_ds}
\end{equation}
Thus $T_{\mathcal{C}^{-1}}\psi_{k}$ has at most a pole of order $l+k$
at $w=i$, and is not in $H_{l}(\mathbb{D})$ in general.

We introduce the space of polynomially growing sequences 
\[
\mathscr{S}'_{\mathrm{pol}}(\mathbb{N}):=\left\{ (a_{n})_{n\ge0}\ ;\ \exists C,N>0,\ |a_{n}|\le C(1+n)^{N}\right\} .
\]
Since $T_{\mathcal{C}^{-1}}u$ is holomorphic in $\mathbb{D}$, it
has a Taylor expansion at $0$. We define 
\begin{equation}
\mathcal{T}:=FT_{\mathcal{C}^{-1}}\qquad:\mathcal{P}_{l}\longrightarrow\mathscr{S}'_{\mathrm{pol}}(\mathbb{N})\label{eq:def_T_l_alg_ds-1}
\end{equation}
by taking the coefficients of this Taylor expansion in the normalized
basis $(\psi_{n})_{n\in\mathbb{N}}$. Equivalently, if 
\[
T_{\mathcal{C}^{-1}}u(w)=\sum_{n\ge0}b_{n}\psi_{n}(w),
\]
then 
\[
\mathcal{T}u=(b_{n})_{n\ge0}.
\]

\begin{cBoxB}{}

\begin{thm}[Algebraic Cayley model for one discrete series]
\label{thm:algebraic_Cayley_model_ds} Let $l\in2\mathbb{N}^{*}$.
The operator 
\begin{equation}
\mathcal{T}=FT_{\mathcal{C}^{-1}}\qquad:\mathcal{P}_{l}\longrightarrow\mathscr{S}'_{\mathrm{pol}}(\mathbb{N})\label{eq:def_T_l_alg_ds}
\end{equation}
is injective. On this algebraic domain, one has 
\begin{equation}
\mathcal{T}\tilde{X}\mathcal{T}^{-1}=-A-\frac{l}{2},\label{eq:T_l_alg_X_ds}
\end{equation}
\begin{equation}
\mathcal{T}\tilde{U}\mathcal{T}^{-1}=a^{+}(A+l)^{1/2},\label{eq:T_l_alg_U_ds}
\end{equation}
and 
\begin{equation}
\mathcal{T}\tilde{S}\mathcal{T}^{-1}=-(A+l)^{1/2}a^{-}.\label{eq:T_l_alg_S_ds}
\end{equation}
Moreover, for every $k\in\mathbb{N}$, there exist constants $C_{l,k},C'_{l,k}>0$
such that, for every $n\in\mathbb{N}$, 
\begin{equation}
\left|\left\langle e_{n}\mid\mathcal{T}\psi_{k}\right\rangle \right|\le C_{l,k}(1+n)^{k+\frac{l-1}{2}},\label{eq:estimate_T_ds}
\end{equation}
and 
\begin{equation}
\left|\left\langle e_{n}\mid(\mathcal{T}^{-1})^{\dagger}\psi_{k}\right\rangle \right|\le C'_{l,k}(1+n)^{k+\frac{l-1}{2}}.\label{eq:estimate_T_dual_ds}
\end{equation}
Here $(\mathcal{T}^{-1})^{\dagger}$ is understood algebraically,
through the coefficients of $T_{\mathcal{C}}\psi_{n}$ in the basis
$(\psi_{k})$. 
\end{thm}

\end{cBoxB}

\begin{proof}
We first prove injectivity. If $\mathcal{T}u=0$, then all Taylor
coefficients of $T_{\mathcal{C}^{-1}}u$ vanish. Hence 
\[
T_{\mathcal{C}^{-1}}u=0
\]
as a holomorphic function on $\mathbb{D}$. Since the prefactor 
\[
\left(\frac{1-i}{w-i}\right)^{l}
\]
does not vanish in $\mathbb{D}$, and since the map 
\[
w\longmapsto\frac{w+i}{w-i}
\]
has open image, the polynomial $u$ vanishes on a nonempty open set.
Thus $u=0$.

We now compute the conjugated infinitesimal action. For the right
action, if $A\in\mathfrak{sl}_{2}(\mathbb{R})$, then 
\[
(Au)(g')=\left.\frac{d}{ds}\right|_{s=0}u(g'\exp(sA)).
\]
Thus, formally, for $h$ in the group, or in the complexified group,
\[
R_{h}^{\circ}A(R_{h}^{\circ})^{-1}=\operatorname{Ad}_{h}(A).
\]
Indeed, 
\[
\begin{aligned}\left(R_{h}^{\circ}A(R_{h}^{\circ})^{-1}u\right)(g') & =\left.\frac{d}{ds}\right|_{s=0}u(g'h\exp(sA)h^{-1})\\
 & =\left.\frac{d}{ds}\right|_{s=0}u(g'\exp(s\,\operatorname{Ad}_{h}A)).
\end{aligned}
\]
Applying this with $h=\mathcal{C}^{-1}$, and using \eqref{eq:conjugason_C},
gives algebraically 
\[
T_{\mathcal{C}^{-1}}\tilde{X}T_{\mathcal{C}}=-\frac{1}{2}\tilde{\Theta},\qquad T_{\mathcal{C}^{-1}}\tilde{U}T_{\mathcal{C}}=\tilde{N}_{+},\qquad T_{\mathcal{C}^{-1}}\tilde{S}T_{\mathcal{C}}=\tilde{N}_{-}.
\]

By Lemma~\ref{lem:holomorphic_model_ds}, 
\[
\tilde{\Theta}=2w\partial_{w}+l,\qquad\tilde{N}_{+}=w^{2}\partial_{w}+lw,\qquad\tilde{N}_{-}=-\partial_{w}.
\]
Passing to the normalized basis $(\psi_{k})_{k\in\mathbb{N}}$, we
have 
\[
F\left(-\frac{1}{2}\tilde{\Theta}\right)F^{-1}=-A-\frac{l}{2},
\]
because 
\[
-\frac{1}{2}(2w\partial_{w}+l)\psi_{k}=-\left(k+\frac{l}{2}\right)\psi_{k}.
\]
Similarly, 
\[
(w^{2}\partial_{w}+lw)\psi_{k}=\sqrt{(k+1)(l+k)}\,\psi_{k+1},
\]
hence 
\[
F\tilde{N}_{+}F^{-1}=a^{+}(A+l)^{1/2}.
\]
Finally, 
\[
-\partial_{w}\psi_{k}=-\sqrt{k(l+k-1)}\,\psi_{k-1},
\]
hence 
\[
F\tilde{N}_{-}F^{-1}=-(A+l)^{1/2}a^{-}.
\]
Combining these identities with 
\[
\mathcal{T}=FT_{\mathcal{C}^{-1}}
\]
gives \eqref{eq:T_l_alg_X_ds}, \eqref{eq:T_l_alg_U_ds}, and \eqref{eq:T_l_alg_S_ds}.

It remains to prove the coefficient estimates. Up to a constant depending
only on $l$ and $k$, one has 
\[
T_{\mathcal{C}^{-1}}\psi_{k}(w)=(1-iw)^{k}(1+iw)^{-l-k}.
\]
The coefficient of $w^{n}$ in this function is a finite linear combination
of coefficients of $(1+iw)^{-l-k}$, hence is bounded by 
\[
C_{l,k}(1+n)^{l+k-1}.
\]
Since 
\[
\psi_{n}(w)=\sqrt{\binom{l+n-1}{n}}\,w^{n},\qquad\binom{l+n-1}{n}\asymp(1+n)^{l-1},
\]
the coefficient in the normalized basis $(\psi_{n})$ satisfies 
\[
\left|\left\langle e_{n}\mid\mathcal{T}\psi_{k}\right\rangle \right|\le C_{l,k}(1+n)^{l+k-1-\frac{l-1}{2}}=C_{l,k}(1+n)^{k+\frac{l-1}{2}}.
\]
This proves \eqref{eq:estimate_T_ds}.

For the dual estimate, we use the algebraic inverse 
\[
\mathcal{T}^{-1}=T_{\mathcal{C}}F^{-1}.
\]
Thus the coefficient 
\[
\left\langle e_{n}\mid(\mathcal{T}^{-1})^{\dagger}\psi_{k}\right\rangle 
\]
is, by definition, the coefficient of $\psi_{k}$ in $T_{\mathcal{C}}\psi_{n}$.
Now 
\[
(T_{\mathcal{C}}u)(w)=\left(\frac{1-i}{1-w}\right)^{l}u\left(-i\,\frac{1+w}{1-w}\right).
\]
Hence, up to a constant depending only on $l$ and $n$ in modulus,
\[
T_{\mathcal{C}}\psi_{n}(w)=\sqrt{\binom{l+n-1}{n}}\,(1+w)^{n}(1-w)^{-l-n}.
\]
For fixed $k$, the coefficient of $w^{k}$ in 
\[
(1+w)^{n}(1-w)^{-l-n}
\]
is bounded by 
\[
C_{l,k}(1+n)^{k}.
\]
Since the normalization factor of $\psi_{k}$ is fixed when $k$ is
fixed, we obtain 
\[
\left|\left\langle e_{n}\mid(\mathcal{T}^{-1})^{\dagger}\psi_{k}\right\rangle \right|\le C'_{l,k}\sqrt{\binom{l+n-1}{n}}\,(1+n)^{k}.
\]
Using again 
\[
\binom{l+n-1}{n}\asymp(1+n)^{l-1},
\]
we get 
\[
\left|\left\langle e_{n}\mid(\mathcal{T}^{-1})^{\dagger}\psi_{k}\right\rangle \right|\le C'_{l,k}(1+n)^{k+\frac{l-1}{2}}.
\]
This proves \eqref{eq:estimate_T_dual_ds}. 
\end{proof}

\subsection{Propagated domains and Hilbert realization}

For $t\in\mathbb{R}$, we define 
\[
\mathcal{D}_{t}:=e^{t\tilde{X}}\mathcal{P}_{l}.
\]
Thus $\mathcal{D}_{t}$ is the propagated $K$-finite domain associated
with the fixed holomorphic discrete series 
\[
\mathcal{H}_{l}^{\mathrm{d.s.},+}(v).
\]

For $t>0$, we extend the algebraic transform $\mathcal{T}$ from
$\mathcal{P}_{l}$ to $\mathcal{D}_{t}$ by 
\[
\mathcal{T}(e^{t\tilde{X}}p):=e^{-t(A+l/2)}\mathcal{T}p,\qquad p\in\mathcal{P}_{l}.
\]
Similarly, we extend the algebraic dual transform $(\mathcal{T}^{-1})^{\dagger}$
to $\mathcal{D}_{-t}$ by 
\[
(\mathcal{T}^{-1})^{\dagger}(e^{-t\tilde{X}}p):=e^{-t(A+l/2)}(\mathcal{T}^{-1})^{\dagger}p,\qquad p\in\mathcal{P}_{l}.
\]

\begin{cBoxB}{}

\begin{prop}[Functional properties of the propagated domains]
\label{prop:ds_functional_properties_short} Let $l\in2\mathbb{N}^{*}$,
and let $t>0$. Then: 
\begin{enumerate}
\item The spaces $\mathcal{D}_{t}$ and $\mathcal{D}_{-t}$ are dense in
$H_{l}(\mathbb{D})$.
\item The propagated transform $\mathcal{T}$ is injective on $\mathcal{D}_{t}$,
and 
\[
\mathcal{T}(\mathcal{D}_{t})\subset\ell^{2}(\mathbb{N})
\]
is dense.
\item The propagated dual transform $(\mathcal{T}^{-1})^{\dagger}$ is injective
on $\mathcal{D}_{-t}$, and 
\[
(\mathcal{T}^{-1})^{\dagger}(\mathcal{D}_{-t})\subset\ell^{2}(\mathbb{N})
\]
is dense.
\item On these domains, one has 
\begin{equation}
\left(A+\frac{l}{2}\right)\mathcal{T}u=-\mathcal{T}(\tilde{X}u),\qquad u\in\mathcal{D}_{t},\label{eq:ds_intertwining_short}
\end{equation}
and 
\begin{equation}
\left(A+\frac{l}{2}\right)(\mathcal{T}^{-1})^{\dagger}v=(\mathcal{T}^{-1})^{\dagger}(\tilde{X}v),\qquad v\in\mathcal{D}_{-t}.\label{eq:ds_dual_intertwining_short}
\end{equation}
\item For every $k\in\mathbb{N}$, there exists $C_{l,k,t}>0$ such that,
for every $n\in\mathbb{N}$, 
\begin{equation}
\left|\left\langle e_{n}\mid\mathcal{T}(e^{t\tilde{X}}\psi_{k})\right\rangle \right|\le C_{l,k,t}e^{-tn}(1+n)^{k+\frac{l-1}{2}},\label{eq:ds_decay_T_short}
\end{equation}
and 
\begin{equation}
\left|\left\langle e_{n}\mid(\mathcal{T}^{-1})^{\dagger}(e^{-t\tilde{X}}\psi_{k})\right\rangle \right|\le C_{l,k,t}e^{-tn}(1+n)^{k+\frac{l-1}{2}}.\label{eq:ds_decay_Tdual_short}
\end{equation}
\end{enumerate}
\end{prop}

\end{cBoxB}

\begin{proof}
The density of $\mathcal{D}_{\pm t}$ follows from the density of
$\mathcal{P}_{l}$ in $H_{l}(\mathbb{D})$ and from the unitarity
of $e^{t\tilde{X}}$ in the original Hilbert space.

From Theorem~\ref{thm:algebraic_Cayley_model_ds}, one has on $\mathcal{P}_{l}$
\[
\mathcal{T}\tilde{X}=-\left(A+\frac{l}{2}\right)\mathcal{T}.
\]
Therefore 
\[
\mathcal{T}e^{t\tilde{X}}=e^{-t(A+l/2)}\mathcal{T}.
\]
Since 
\[
e^{-t(A+l/2)}e_{n}=e^{-t(n+l/2)}e_{n},
\]
the polynomial estimate \eqref{eq:estimate_T_ds} gives 
\[
\left|\left\langle e_{n}\mid\mathcal{T}(e^{t\tilde{X}}\psi_{k})\right\rangle \right|\le C_{l,k}e^{-t(n+l/2)}(1+n)^{k+\frac{l-1}{2}}.
\]
After absorbing $e^{-tl/2}$ into the constant, this proves \eqref{eq:ds_decay_T_short}.
In particular, 
\[
\mathcal{T}(\mathcal{D}_{t})\subset\ell^{2}(\mathbb{N}).
\]

For the dual transform, the algebraic relation 
\[
\tilde{X}\mathcal{T}^{-1}=-\mathcal{T}^{-1}\left(A+\frac{l}{2}\right)
\]
gives, after taking adjoints and using $\tilde{X}^{\dagger}=-\tilde{X}$,
\[
(\mathcal{T}^{-1})^{\dagger}\tilde{X}=\left(A+\frac{l}{2}\right)(\mathcal{T}^{-1})^{\dagger}.
\]
Hence 
\[
(\mathcal{T}^{-1})^{\dagger}e^{-t\tilde{X}}=e^{-t(A+l/2)}(\mathcal{T}^{-1})^{\dagger}.
\]
Together with \eqref{eq:estimate_T_dual_ds}, this proves \eqref{eq:ds_decay_Tdual_short},
and therefore 
\[
(\mathcal{T}^{-1})^{\dagger}(\mathcal{D}_{-t})\subset\ell^{2}(\mathbb{N}).
\]

The intertwining identities \eqref{eq:ds_intertwining_short} and
\eqref{eq:ds_dual_intertwining_short} follow by differentiating the
two propagated identities above.

It remains to prove injectivity and density. The injectivity of $\mathcal{T}$
on $\mathcal{D}_{t}$ follows from the injectivity of $\mathcal{T}$
on $\mathcal{P}_{l}$, since 
\[
\mathcal{T}(e^{t\tilde{X}}p)=e^{-t(A+l/2)}\mathcal{T}p
\]
and $e^{-t(A+l/2)}$ is injective. The injectivity of $(\mathcal{T}^{-1})^{\dagger}$
on $\mathcal{D}_{-t}$ is analogous.

Finally, we prove density of $\mathcal{T}(\mathcal{D}_{t})$. Let
\[
W_{t}:=\mathcal{T}(\mathcal{D}_{t}),\qquad\mathcal{K}_{t}:=\overline{W_{t}}^{\,\ell^{2}}.
\]
The intertwining relation shows that $\mathcal{K}_{t}$ is invariant
under $e^{-s(A+l/2)}$, $s\ge0$. Moreover, 
\[
\mathcal{T}(e^{t\tilde{X}}\psi_{0})=e^{-t(A+l/2)}\mathcal{T}\psi_{0}.
\]
By \eqref{eq:T_C_inv_psi_k_ds} with $k=0$, all Taylor coefficients
of $T_{\mathcal{C}^{-1}}\psi_{0}$ are nonzero. Hence $\mathcal{T}(e^{t\tilde{X}}\psi_{0})$
has nonzero $e_{n}$-coefficient for every $n$.

Using the diagonal semigroup $e^{-s(A+l/2)}$, we isolate the coefficients
successively. More precisely, if the basis vectors $e_{0},\ldots,e_{n-1}$
are already in $\mathcal{K}_{t}$, subtract their components from
$\mathcal{T}(e^{t\tilde{X}}\psi_{0})$. Applying 
\[
e^{s(n+l/2)}e^{-s(A+l/2)}
\]
and letting $s\to+\infty$, we obtain a nonzero multiple of $e_{n}$.
Thus $e_{n}\in\mathcal{K}_{t}$ for every $n$, and so 
\[
\mathcal{K}_{t}=\ell^{2}(\mathbb{N}).
\]

The proof of the density of 
\[
(\mathcal{T}^{-1})^{\dagger}(\mathcal{D}_{-t})
\]
is the same, using $T_{\mathcal{C}}\psi_{0}$ and the estimate \eqref{eq:ds_decay_Tdual_short}. 
\end{proof}
\begin{cBoxB}{}

\begin{thm}[Hilbert realization of one holomorphic discrete series]
\label{thm:ds_Hilbert_realization_short} Let $l\in2\mathbb{N}^{*}$
and $t>0$. Define 
\[
\|u\|_{\mathcal{H}_{t}}:=\|\mathcal{T}u\|_{\ell^{2}(\mathbb{N})},\qquad u\in\mathcal{D}_{t},
\]
and 
\[
\mathcal{H}_{t}:=\overline{\mathcal{D}_{t}}^{\|\cdot\|_{\mathcal{H}_{t}}}.
\]
Then $\mathcal{T}$ extends uniquely to a unitary isomorphism 
\[
\mathcal{T}:\mathcal{H}_{t}\xrightarrow{\sim}\ell^{2}(\mathbb{N}).
\]
In this Hilbert space, 
\begin{equation}
\tilde{X}=-\mathcal{T}^{-1}\left(A+\frac{l}{2}\right)\mathcal{T}.\label{eq:ds_X_model_one_irrep}
\end{equation}
Thus $\tilde{X}$ is self-adjoint in $\mathcal{H}_{t}$, with discrete
spectrum 
\begin{equation}
z_{n,l}=-n-\frac{l}{2},\qquad n\in\mathbb{N}.\label{eq:ds_resonances_one_irrep}
\end{equation}
The associated rank-one spectral projectors are 
\[
\Pi_{n,l}=\mathcal{T}^{-1}\left(e_{n}\langle e_{n}\mid\cdot\rangle\right)\mathcal{T}.
\]
Moreover, for every $\tau>0$, 
\begin{equation}
e^{\tau\tilde{X}}=\mathcal{T}^{-1}e^{-\tau(A+l/2)}\mathcal{T}=\sum_{n\in\mathbb{N}}e^{-\tau(n+l/2)}\Pi_{n,l},\label{eq:ds_semigroup_one_irrep}
\end{equation}
where the series converges in operator norm. 
\end{thm}

\end{cBoxB}

\begin{proof}
By definition of the norm, $\mathcal{T}$ is an isometry from $\mathcal{D}_{t}$
into $\ell^{2}(\mathbb{N})$. By Proposition~\ref{prop:ds_functional_properties_short},
its image is dense. Hence $\mathcal{T}$ extends uniquely by completion
to a unitary isomorphism 
\[
\mathcal{T}:\mathcal{H}_{t}\xrightarrow{\sim}\ell^{2}(\mathbb{N}).
\]

The intertwining identity \eqref{eq:ds_intertwining_short} gives
\[
\mathcal{T}\tilde{X}=-\left(A+\frac{l}{2}\right)\mathcal{T},
\]
hence 
\[
\tilde{X}=-\mathcal{T}^{-1}\left(A+\frac{l}{2}\right)\mathcal{T}.
\]
Since $A+\frac{l}{2}$ is self-adjoint and diagonal on $\ell^{2}(\mathbb{N})$,
the operator $\tilde{X}$ is self-adjoint in $\mathcal{H}_{t}$. Its
eigenvalues are 
\[
z_{n,l}=-n-\frac{l}{2},\qquad n\in\mathbb{N},
\]
with rank-one spectral projectors 
\[
\Pi_{n,l}=\mathcal{T}^{-1}\left(e_{n}\langle e_{n}\mid\cdot\rangle\right)\mathcal{T}.
\]
The formula for the semigroup follows by functional calculus. Since
\[
e^{-\tau(n+l/2)}
\]
decays exponentially in $n$ for every $\tau>0$, the series converges
in operator norm. 
\end{proof}
\begin{cBoxB}{}

\begin{cor}[Correlation formula]
\label{cor:ds_correlation_short} For every $k,k'\in\mathbb{N}$
and every $\tau>0$, one has the absolutely convergent expansion 
\begin{equation}
\langle\psi_{k'}\mid e^{\tau\tilde{X}}\psi_{k}\rangle_{H_{l}}=\sum_{n\ge0}e^{-\tau(n+l/2)}\overline{\left\langle e_{n}\mid(\mathcal{T}^{-1})^{\dagger}\psi_{k'}\right\rangle }\left\langle e_{n}\mid\mathcal{T}\psi_{k}\right\rangle .\label{eq:ds_correlation_short}
\end{equation}
\end{cor}

\end{cBoxB}

\begin{proof}
For $u$ in the propagated domain, the algebraic identities give 
\[
\langle\psi_{k'}\mid u\rangle_{H_{l}}=\left\langle (\mathcal{T}^{-1})^{\dagger}\psi_{k'}\middle|\mathcal{T}u\right\rangle _{\ell^{2}}.
\]
Applying this to $u=e^{\tau\tilde{X}}\psi_{k}$, and using 
\[
\mathcal{T}e^{\tau\tilde{X}}=e^{-\tau(A+l/2)}\mathcal{T},
\]
we get 
\[
\langle\psi_{k'}\mid e^{\tau\tilde{X}}\psi_{k}\rangle_{H_{l}}=\left\langle (\mathcal{T}^{-1})^{\dagger}\psi_{k'}\middle|e^{-\tau(A+l/2)}\mathcal{T}\psi_{k}\right\rangle _{\ell^{2}}.
\]
Expanding in the canonical basis $(e_{n})_{n\in\mathbb{N}}$, this
gives 
\[
\langle\psi_{k'}\mid e^{\tau\tilde{X}}\psi_{k}\rangle_{H_{l}}=\sum_{n\ge0}e^{-\tau(n+l/2)}\overline{\left\langle e_{n}\mid(\mathcal{T}^{-1})^{\dagger}\psi_{k'}\right\rangle }\left\langle e_{n}\mid\mathcal{T}\psi_{k}\right\rangle .
\]

Finally, by Theorem~\ref{thm:algebraic_Cayley_model_ds}, 
\[
\left|\left\langle e_{n}\mid\mathcal{T}\psi_{k}\right\rangle \right|\le C_{l,k}(1+n)^{k+\frac{l-1}{2}},
\]
and 
\[
\left|\left\langle e_{n}\mid(\mathcal{T}^{-1})^{\dagger}\psi_{k'}\right\rangle \right|\le C'_{l,k'}(1+n)^{k'+\frac{l-1}{2}}.
\]
Thus the $n$-th term is bounded by 
\[
Ce^{-\tau n}(1+n)^{k+k'+l-1},
\]
which is summable for every $\tau>0$. Hence the expansion is absolutely
convergent. 
\end{proof}

\subsection{\protect\label{subsec:Trace-of-}Trace of $e^{t\tilde{X}}$ in one
discrete series}

We fix $l\in2\mathbb{N}^{*}$ and one holomorphic discrete series
\[
\mathcal{H}_{l}^{\mathrm{d.s.},+}(v).
\]
In the holomorphic disk model, recall that 
\[
\tilde{X}\eq{\ref{eq:Nplus_disk_ds},\ref{eq:X_from_R}}\frac{1}{2}\left((w^{2}-1)\partial_{w}+lw\right).
\]

Let $\phi^{t}$ be the holomorphic flow generated by 
\[
\frac{1}{2}(w^{2}-1)\partial_{w}.
\]
We write 
\[
\rho_{t}(w):=\exp\left(\int_{0}^{t}\frac{l}{2}\,\phi^{s}(w)\,ds\right).
\]
Then, formally, 
\[
(e^{t\tilde{X}}u)(w)=\rho_{t}(w)\,u(\phi^{t}(w)).
\]

For a holomorphic function $u$, Cauchy's formula gives 
\[
u(\phi^{t}(w))=\frac{1}{2\pi i}\int_{\partial\mathbb{D}}\frac{u(\zeta)}{\zeta-\phi^{t}(w)}\,d\zeta.
\]
Thus the holomorphic kernel of $e^{t\tilde{X}}$, with respect to
the Cauchy boundary measure $d\zeta/(2\pi i)$, is 
\begin{equation}
K_{t}^{+}(w,\zeta):=\frac{\rho_{t}(w)}{\zeta-\phi^{t}(w)}.\label{eq:holomorphic_kernel_discrete}
\end{equation}

\begin{cBoxB}{}

\begin{thm}
\label{thm:trace_one_discrete_series} Let $l\in2\mathbb{N}^{*}$.
For $t>0$, the operator $e^{t\tilde{X}}$ is not trace class in the
original Hilbert space $H_{l}(\mathbb{D})$. However its holomorphic
flat trace is well defined by 
\begin{equation}
\operatorname{Tr}_{\partial\mathbb{D}}^{\flat,+}\left(e^{t\tilde{X}}\right):=\frac{1}{2\pi i}\int_{\partial\mathbb{D}}^{\flat,+}K_{t}^{+}(w,w)\,dw.\label{eq:def_holomorphic_flat_trace_kernel}
\end{equation}
Equivalently, 
\begin{equation}
\operatorname{Tr}_{\partial\mathbb{D}}^{\flat,+}\left(e^{t\tilde{X}}\right)=\sum_{\substack{\phi^{t}(w)=w\\
|(\phi^{t})'(w)|<1
}
}\frac{\rho_{t}(w)}{1-(\phi^{t})'(w)}.\label{eq:def_holomorphic_flat_trace_fixed_points}
\end{equation}
For $t>0$, this gives 
\begin{equation}
\operatorname{Tr}_{\partial\mathbb{D}}^{\flat,+}\left(e^{t\tilde{X}}\right)=\frac{e^{-tl/2}}{1-e^{-t}}.\label{eq:trace_flat_discrete_one}
\end{equation}
On the other hand, for every $\tau>0$, the same operator is trace
class in the Hilbert model $\mathcal{H}_{\tau}$, and 
\[
\operatorname{Tr}_{\mathcal{H}_{\tau}}\left(e^{t\tilde{X}}\right)=\operatorname{Tr}_{\partial\mathbb{D}}^{\flat,+}\left(e^{t\tilde{X}}\right).
\]
\end{thm}

\end{cBoxB}

\begin{proof}
On the original Hilbert space $H_{l}(\mathbb{D})$, the representation
is unitary. Hence $e^{t\tilde{X}}$ is unitary. Since $H_{l}(\mathbb{D})$
is infinite-dimensional, this operator is not compact, and therefore
not trace class.

The fixed points of 
\[
\frac{1}{2}(w^{2}-1)\partial_{w}
\]
on $\partial\mathbb{D}$ are 
\[
w=-1,\qquad w=1.
\]
For $t>0$, $w=-1$ is attracting and 
\[
(\phi^{t})'(-1)=e^{-t},
\]
whereas $w=1$ is repelling and 
\[
(\phi^{t})'(1)=e^{t}.
\]

The diagonal of the holomorphic kernel is 
\[
K_{t}^{+}(w,w)=\frac{\rho_{t}(w)}{w-\phi^{t}(w)}.
\]
Thus the diagonal singularities occur exactly at the fixed points
of $\phi^{t}$. The holomorphic, or positive-frequency, prescription
keeps the attracting boundary contribution. Hence 
\[
\operatorname{Tr}_{\partial\mathbb{D}}^{\flat,+}\left(e^{t\tilde{X}}\right)=\frac{\rho_{t}(-1)}{1-(\phi^{t})'(-1)}.
\]
Since 
\[
\rho_{t}(-1)=\exp\left(\int_{0}^{t}\frac{l}{2}(-1)\,ds\right)=e^{-tl/2},
\]
we obtain 
\[
\operatorname{Tr}_{\partial\mathbb{D}}^{\flat,+}\left(e^{t\tilde{X}}\right)=\frac{e^{-tl/2}}{1-e^{-t}}.
\]

In the Hilbert model $\mathcal{H}_{\tau}$, Theorem~\ref{thm:ds_Hilbert_realization_short}
gives 
\[
\mathcal{T}\tilde{X}\mathcal{T}^{-1}=-A-\frac{l}{2}.
\]
Thus 
\[
\operatorname{Tr}_{\mathcal{H}_{\tau}}\left(e^{t\tilde{X}}\right)=\sum_{n\ge0}e^{-t(n+l/2)}=\frac{e^{-tl/2}}{1-e^{-t}}.
\]
This proves the equality between the Hilbert trace and the holomorphic
flat trace. 
\end{proof}
\begin{rem}
The notation $\operatorname{Tr}_{\partial\mathbb{D}}^{\flat,+}$ emphasizes
that we use the holomorphic Cauchy kernel and keep the positive-frequency,
attracting boundary contribution. This is analogous to the fixed-point
contribution in holomorphic Lefschetz formulas, but here the fixed
point lies on the boundary and the operator is not trace class on
the original Bergman space. 
\end{rem}

\begin{cBoxB}{}

\begin{thm}[Global Hilbert model for the discrete series]
\label{thm:global_discrete_model_X} Let $t>0$. The fiberwise transforms
constructed above define an injective map 
\begin{equation}
\mathbb{T}_{\mathrm{d.s.}}:\mathcal{D}_{\mathrm{d.s.},t}\longrightarrow\widetilde{\mathcal{K}}_{\mathrm{d.s.}},\label{eq:def_TT_ds}
\end{equation}
with dense image. More explicitly, $\mathbb{T}_{\mathrm{d.s.}}$ is
the direct sum, over all holomorphic and anti-holomorphic discrete
series components, of the corresponding maps 
\[
\mathcal{T}\mathcal{U}_{1}.
\]
We define 
\[
\|u\|_{\mathcal{H}_{\mathrm{d.s.},t}}:=\|\mathbb{T}_{\mathrm{d.s.}}u\|_{\widetilde{\mathcal{K}}_{\mathrm{d.s.}}},\qquad u\in\mathcal{D}_{\mathrm{d.s.},t},
\]
and 
\[
\mathcal{H}_{\mathrm{d.s.},t}:=\overline{\mathcal{D}_{\mathrm{d.s.},t}}^{\|\cdot\|_{\mathcal{H}_{\mathrm{d.s.},t}}}.
\]
Then 
\[
\mathbb{T}_{\mathrm{d.s.}}:\mathcal{H}_{\mathrm{d.s.},t}\xrightarrow{\sim}\widetilde{\mathcal{K}}_{\mathrm{d.s.}}
\]
extends uniquely as a unitary isomorphism.

On the algebraic core, one has 
\begin{equation}
\mathbb{T}_{\mathrm{d.s.}}X\mathbb{T}_{\mathrm{d.s.}}^{-1}=-A-\frac{1}{2}-\Lambda_{\mathrm{d.s.}},\label{eq:global_X_discrete}
\end{equation}
\begin{equation}
\mathbb{T}_{\mathrm{d.s.}}U\mathbb{T}_{\mathrm{d.s.}}^{-1}=a^{+}\left(A+2\Lambda_{\mathrm{d.s.}}+1\right)^{1/2},\label{eq:global_U_discrete}
\end{equation}
and 
\begin{equation}
\mathbb{T}_{\mathrm{d.s.}}S\mathbb{T}_{\mathrm{d.s.}}^{-1}=-\left(A+2\Lambda_{\mathrm{d.s.}}+1\right)^{1/2}a^{-}.\label{eq:global_S_discrete}
\end{equation}

Moreover, 
\begin{equation}
\mathbb{T}_{\mathrm{d.s.}}\Omega\mathbb{T}_{\mathrm{d.s.}}^{-1}=\frac{1}{4}-\Lambda_{\mathrm{d.s.}}^{2}.\label{eq:global_Omega_discrete}
\end{equation}
Consequently, the discrete-series resonances of $X$ are 
\begin{equation}
z_{n,l}=-n-\frac{l}{2},\qquad n\in\mathbb{N},\quad l\in2\mathbb{N}^{*},\label{eq:global_discrete_resonances}
\end{equation}
with multiplicity $2m_{l}$. 
\end{thm}

\end{cBoxB}

\begin{proof}
Fix $l\in2\mathbb{N}^{*}$. On each holomorphic discrete-series copy
generated by 
\[
0\neq v\in\mathcal{H}_{\mathrm{hol},+}(l),
\]
Theorem~\ref{thm:ds_Hilbert_realization_short} gives the local model
\[
\mathcal{T}X\mathcal{T}^{-1}=-A-\frac{l}{2}.
\]
Moreover, on the algebraic core, Theorem~\ref{thm:algebraic_Cayley_model_ds}
gives 
\[
\mathcal{T}U\mathcal{T}^{-1}=a^{+}(A+l)^{1/2},\qquad\mathcal{T}S\mathcal{T}^{-1}=-(A+l)^{1/2}a^{-}.
\]

By definition of $\Lambda_{\mathrm{d.s.}}$, 
\[
\Lambda_{\mathrm{d.s.}}v=\frac{l-1}{2}v.
\]
Thus, on the summand $\ell^{2}(\mathbb{N})\otimes\mathbb{C}v$, 
\[
-A-\frac{1}{2}-\Lambda_{\mathrm{d.s.}}=-A-\frac{1}{2}-\frac{l-1}{2}=-A-\frac{l}{2},
\]
and 
\[
A+2\Lambda_{\mathrm{d.s.}}+1=A+l.
\]
This gives the formulas \eqref{eq:global_X_discrete}, \eqref{eq:global_U_discrete},
and \eqref{eq:global_S_discrete} on the holomorphic part. The order
of the factors in \eqref{eq:global_U_discrete} and \eqref{eq:global_S_discrete}
is inherited from the local formulas; in particular, one should not
commute $a^{\pm}$ with functions of $A$.

The anti-holomorphic part is obtained by complex conjugation. Since
$X,U,S$ and $\Omega$ are real operators, the same formulas hold
on the anti-holomorphic copies. We keep the direct sum 
\[
\mathcal{H}_{\mathrm{hol},+}\oplus\mathcal{H}_{\mathrm{hol},-}
\]
rather than identifying the two spaces linearly, because complex conjugation
is anti-linear.

Taking the Hilbert direct sum over all $l\in2\mathbb{N}^{*}$ and
over orthonormal bases of the finite-dimensional lowest/highest-weight
spaces gives the global transform 
\[
\mathbb{T}_{\mathrm{d.s.}}:\mathcal{D}_{\mathrm{d.s.},t}\longrightarrow\widetilde{\mathcal{K}}_{\mathrm{d.s.}}.
\]
The local injectivity and density statements imply that $\mathbb{T}_{\mathrm{d.s.}}$
is injective and has dense image. Hence the norm 
\[
\|u\|_{\mathcal{H}_{\mathrm{d.s.},t}}=\|\mathbb{T}_{\mathrm{d.s.}}u\|_{\widetilde{\mathcal{K}}_{\mathrm{d.s.}}}
\]
is a genuine norm, and completion gives the unitary isomorphism 
\[
\mathbb{T}_{\mathrm{d.s.}}:\mathcal{H}_{\mathrm{d.s.},t}\xrightarrow{\sim}\widetilde{\mathcal{K}}_{\mathrm{d.s.}}.
\]

Finally, on the summand of lowest weight $l$, 
\[
\Omega=-\frac{1}{4}l(l-2),\qquad\Lambda_{\mathrm{d.s.}}=\frac{l-1}{2}.
\]
Therefore 
\[
\frac{1}{4}-\Lambda_{\mathrm{d.s.}}^{2}=\frac{1}{4}-\frac{(l-1)^{2}}{4}=-\frac{1}{4}l(l-2),
\]
which proves \eqref{eq:global_Omega_discrete}.

Since $Ae_{n}=ne_{n}$, the eigenvalues of 
\[
-A-\frac{1}{2}-\Lambda_{\mathrm{d.s.}}
\]
on the summand of weight $l$ are 
\[
-n-\frac{1}{2}-\frac{l-1}{2}=-n-\frac{l}{2}.
\]
This gives \eqref{eq:global_discrete_resonances}. The multiplicity
is $2m_{l}$, because for each $l$ there are $m_{l}$ holomorphic
and $m_{l}$ anti-holomorphic copies. 
\end{proof}

\section{\protect\label{sec:Global-Hilbert-model}Global Hilbert model}

We now collect the spherical, discrete and trivial components into
a single Hilbert model.

Recall the decomposition 
\[
L^{2}(M)=\mathcal{H}^{\mathrm{triv}}\widehat{\oplus}\mathcal{H}^{\mathrm{sph}}\widehat{\oplus}\mathcal{H}^{\mathrm{d.s.}},
\]
where 
\[
\mathcal{H}^{\mathrm{triv}}=\mathbb{C}\boldsymbol{1}.
\]
Let 
\begin{equation}
\mathcal{P}(M):=\bigoplus_{l\in2\mathbb{Z}}^{\mathrm{alg}}\bigoplus_{\mu\in\mathrm{Spec}(\Omega_{l})}^{\mathrm{alg}}\mathcal{H}_{\mu,l}\label{eq:def_calP_M}
\end{equation}
be the algebraic core of $K$-finite and spectrally finite vectors.
For $r\in\mathbb{R}$, set 
\begin{equation}
\mathcal{D}_{r}:=e^{rX}\mathcal{P}(M).\label{eq:def_global_D_t}
\end{equation}

We set 
\[
N_{2}:=\begin{pmatrix}0 & 1\\
0 & 0
\end{pmatrix}
\]
and 
\begin{equation}
\mathcal{H}_{\mathrm{hol}}:=\mathcal{H}_{\mathrm{hol},+}\oplus\mathcal{H}_{\mathrm{hol},-}.\label{eq:def_Hol}
\end{equation}

Using the decomposition 
\[
L_{0}^{2}(\mathcal{N})=E_{\neq1/4}\oplus E_{1/4},
\]
we introduce the parameter space 
\begin{equation}
\mathcal{B}:=\left(E_{\neq1/4}\otimes\mathbb{C}^{2}\right)\oplus\left(E_{1/4}\otimes\mathbb{C}^{2}\right)\oplus\mathcal{H}_{\mathrm{hol}}.\label{eq:def_global_B}
\end{equation}
The global model space is 
\begin{equation}
\mathcal{K}:=\left(\ell^{2}(\mathbb{N})\otimes\mathcal{B}\right)\oplus\mathbb{C}\boldsymbol{1}.\label{eq:def_global_K}
\end{equation}

\begin{cBoxB}{}

\begin{thm}[Global Hilbert model]
\label{thm:global_Hilbert_model_XUS} Let $r>0$. There exists an
injective map 
\[
\mathbb{T}:\mathcal{D}_{r}\longrightarrow\mathcal{K}
\]
with dense image. We define 
\[
\|u\|_{\mathcal{H}_{r}}:=\|\mathbb{T}u\|_{\mathcal{K}},\qquad u\in\mathcal{D}_{r},
\]
and 
\begin{equation}
\mathcal{H}_{r}:=\overline{\mathcal{D}_{r}}^{\|\cdot\|_{\mathcal{H}_{r}}}.\label{eq:def_global_H_t}
\end{equation}
Then 
\[
\mathbb{T}:\mathcal{H}_{r}\xrightarrow{\sim}\mathcal{K}
\]
extends uniquely as a unitary isomorphism.

In the model space, the generator $X$ is represented by 
\begin{equation}
\mathbb{T}X\mathbb{T}^{-1}=\left(-\left(A+\frac{1}{2}\right)\otimes\mathrm{Id}_{\mathcal{B}}+\mathrm{Id}_{\ell^{2}(\mathbb{N})}\otimes X_{\mathcal{B}}\right)\oplus0_{\mathbb{C}\boldsymbol{1}},\label{eq:global_X_all_blocks}
\end{equation}
where 
\begin{equation}
X_{\mathcal{B}}=i\Sigma\Lambda\oplus N_{2}\oplus\left(-\Lambda_{\mathrm{d.s.}}\right)\label{eq:def_X_B}
\end{equation}
on 
\[
\mathcal{B}=\left(E_{\neq1/4}\otimes\mathbb{C}^{2}\right)\oplus\left(E_{1/4}\otimes\mathbb{C}^{2}\right)\oplus\mathcal{H}_{\mathrm{hol}}.
\]
Here $\Sigma$ acts on the factor $\mathbb{C}^{2}$, and $\Lambda$
acts on $E_{\neq1/4}$.

Consequently, for every $\tau\ge0$, 
\begin{equation}
\mathbb{T}e^{\tau X}\mathbb{T}^{-1}=\left(e^{-\tau(A+\frac{1}{2})}\otimes e^{\tau X_{\mathcal{B}}}\right)\oplus\mathrm{Id}_{\mathbb{C}\boldsymbol{1}},\label{eq:global_exp_X}
\end{equation}
with 
\begin{equation}
e^{\tau X_{\mathcal{B}}}=e^{i\tau\Sigma\Lambda}\oplus\left(\mathrm{Id}_{2}+\tau N_{2}\right)\oplus e^{-\tau\Lambda_{\mathrm{d.s.}}}.\label{eq:global_exp_X_B}
\end{equation}

The operators $U$ and $S$ are represented blockwise as follows.
On the non-threshold spherical part, 
\begin{equation}
U_{\mathrm{sph},\neq1/4}=-\Sigma\,a^{+}\left(A+1-2i\Sigma\Lambda\right)^{1/2},\label{eq:U_sph_global_again}
\end{equation}
\begin{equation}
S_{\mathrm{sph},\neq1/4}=\Sigma\left(A+1-2i\Sigma\Lambda\right)^{1/2}a^{-}.\label{eq:S_sph_global_again}
\end{equation}
On the discrete-series part, 
\begin{equation}
U_{\mathrm{d.s.}}=a^{+}\left(A+2\Lambda_{\mathrm{d.s.}}+1\right)^{1/2},\label{eq:U_ds_global_again}
\end{equation}
\begin{equation}
S_{\mathrm{d.s.}}=-\left(A+2\Lambda_{\mathrm{d.s.}}+1\right)^{1/2}a^{-}.\label{eq:S_ds_global_again}
\end{equation}
On the trivial representation, 
\[
X\boldsymbol{1}=U\boldsymbol{1}=S\boldsymbol{1}=0.
\]

Moreover, 
\begin{equation}
\mathbb{T}\Omega\mathbb{T}^{-1}=\left(\mathrm{Id}_{\ell^{2}(\mathbb{N})}\otimes\Omega_{\mathcal{B}}\right)\oplus0_{\mathbb{C}\boldsymbol{1}},\label{eq:global_Omega_all_blocks}
\end{equation}
where 
\begin{equation}
\Omega_{\mathcal{B}}=\Delta_{|E_{\neq1/4}}\otimes\mathrm{Id}_{\mathbb{C}^{2}}\oplus\frac{1}{4}\,\mathrm{Id}_{E_{1/4}\otimes\mathbb{C}^{2}}\oplus\left(\frac{1}{4}-\Lambda_{\mathrm{d.s.}}^{2}\right).\label{eq:def_Omega_B}
\end{equation}
\end{thm}

\end{cBoxB}

\begin{proof}
The decomposition 
\[
L^{2}(M)=\mathcal{H}^{\mathrm{triv}}\widehat{\oplus}\mathcal{H}^{\mathrm{sph}}\widehat{\oplus}\mathcal{H}^{\mathrm{d.s.}}
\]
is orthogonal and invariant under the right action of $\mathrm{SL}_{2}(\mathbb{R})$.
The transforms $\mathbb{T}_{\mathrm{sph}}$ in \eqref{eq:def_TT}
and $\mathbb{T}_{\mathrm{d.s.}}$ in \eqref{eq:def_TT_ds} are defined
fiberwise on the spherical and discrete parts, and have dense image
in their corresponding model spaces. On the trivial summand we take
the identity map 
\[
\mathbb{C}\boldsymbol{1}\longrightarrow\mathbb{C}\boldsymbol{1}.
\]
Taking their Hilbert direct sum gives 
\[
\mathbb{T}:\mathcal{D}_{r}\longrightarrow\mathcal{K}.
\]
The injectivity and density of the image follow from the corresponding
properties on each block.

By definition of the norm 
\[
\|u\|_{\mathcal{H}_{r}}=\|\mathbb{T}u\|_{\mathcal{K}},
\]
the map $\mathbb{T}$ is an isometry on $\mathcal{D}_{r}$. Hence
it extends by completion to a unitary isomorphism 
\[
\mathbb{T}:\mathcal{H}_{r}\xrightarrow{\sim}\mathcal{K}.
\]

The formula for $X$ is the direct sum of the spherical formula 
\[
L+i\Sigma\Lambda
\]
on the non-threshold spherical part, the Jordan threshold formula
\[
L+N_{2},\qquad L=-A-\frac{1}{2},
\]
on $E_{1/4}$, and the discrete-series formula 
\[
-A-\frac{1}{2}-\Lambda_{\mathrm{d.s.}}.
\]
This is exactly \eqref{eq:global_X_all_blocks}.

The formulas for $U$, $S$, and $\Omega$ are the direct sums of
the formulas proved in the spherical model, Theorem~\ref{thm:global_spherical_model_X},
and in the discrete model, Theorem~\ref{thm:global_discrete_model_X}.
On the trivial representation all vector fields act by zero, and the
Casimir also acts by zero. This proves the theorem. 
\end{proof}
\begin{rem}
The threshold part $E_{1/4}$, if non-trivial, is represented by the
Jordan block 
\[
L+N_{2}=\begin{pmatrix}L & 1\\
0 & L
\end{pmatrix}.
\]
The formulas \eqref{eq:U_sph_global_again}--\eqref{eq:S_sph_global_again}
are stated only on the non-threshold spherical part. The threshold
formulas for $U$ and $S$ can be obtained from the same limiting
construction as the Jordan model for $X$, but they are not needed
below. 
\end{rem}

\subsection{Spectral trace of $e^{\tau X}$ in the Hilbert model}

Let $r>0$ be fixed. By Theorem~\ref{thm:global_Hilbert_model_XUS},
we have constructed a Hilbert space $\mathcal{H}_{r}$ and a unitary
transform 
\[
\mathbb{T}:\mathcal{H}_{r}\xrightarrow{\sim}\mathcal{K}
\]
such that $X$ is represented blockwise on $\mathcal{K}$.

For each irreducible summand, $e^{\tau X}$, $\tau>0$, is trace class
in the corresponding Hilbert model. Globally, however, the sum over
the Laplace spectrum is not an ordinary trace-class trace. We therefore
define the global spectral trace as the sum of the fiberwise traces,
understood as a distribution in the time variable $\tau$.

In this subsection, we extend the operator 
\[
\Lambda=\left(\Delta-\frac{1}{4}\right)^{1/2}
\]
to all of $L_{0}^{2}(\mathcal{N})$ by setting 
\[
\Lambda_{\mid E_{1/4}}=0,
\]
with the same convention as in \eqref{eq:def_global_Lambda_spherical}
on $E_{\neq1/4}$.

\begin{cBoxB}{}

\begin{prop}[Spectral trace as a sum over irreducible representations]
\label{prop:spectral_trace_sum_irreps} As a distribution in $\tau>0$,
the spectral trace of $e^{\tau X}$ in the global Hilbert model is
\begin{align}
\operatorname{Tr}_{\mathcal{H}_{r}}^{\mathrm{spec}}\left(e^{\tau X}\right) & =\operatorname{Tr}_{\mathcal{H}_{\mathrm{triv}}}\left(e^{\tau X}\right)\nonumber \\
 & \quad+\sum_{\mu\in\operatorname{Spec}(\Delta|_{L_{0}^{2}(\mathcal{N})})}\operatorname{Tr}_{\mathcal{H}_{\mu}^{\mathrm{sph}}}\left(e^{\tau X}\right)\nonumber \\
 & \quad+\sum_{l\in2\mathbb{N}^{*}}\left[\operatorname{Tr}_{\mathcal{H}_{l}^{\mathrm{d.s.},+}}\left(e^{\tau X}\right)+\operatorname{Tr}_{\mathcal{H}_{-l}^{\mathrm{d.s.},-}}\left(e^{\tau X}\right)\right].\label{eq:spectral_trace_sum_irreps}
\end{align}
Equivalently, 
\begin{align}
\operatorname{Tr}_{\mathcal{H}_{r}}^{\mathrm{spec}}\left(e^{\tau X}\right) & =1+\frac{2e^{-\tau/2}}{1-e^{-\tau}}\,\operatorname{Tr}_{L_{0}^{2}(\mathcal{N})}\left(\cos(\tau\Lambda)\right)\nonumber \\
 & \quad+\frac{2}{1-e^{-\tau}}\sum_{l\in2\mathbb{N}^{*}}m_{l}e^{-\tau l/2}.\label{eq:spectral_trace_model_before_RR}
\end{align}
\end{prop}

\end{cBoxB}

\begin{proof}
The trivial representation contributes 
\[
\operatorname{Tr}_{\mathcal{H}_{\mathrm{triv}}}(e^{\tau X})=1.
\]

For the spherical part, the non-threshold model gives the two branches
\[
z_{n,\pm}(\mu)=-n-\frac{1}{2}\pm i\sqrt{\mu-\frac{1}{4}},\qquad n\in\mathbb{N}.
\]
Thus the contribution of an eigenspace $E_{\mu}$, $\mu\neq1/4$,
is 
\[
\sum_{n\ge0}\left(e^{\tau z_{n,+}(\mu)}+e^{\tau z_{n,-}(\mu)}\right)=\frac{2e^{-\tau/2}}{1-e^{-\tau}}\,\cos(\tau\Lambda_{\mid E_{\mu}}).
\]
At the threshold $\mu=1/4$, the two branches coalesce into a Jordan
block. The nilpotent part has trace zero, and the contribution is
the same formula with 
\[
\Lambda_{\mid E_{1/4}}=0.
\]
Summing over the spectral decomposition of $\Delta$ on $L_{0}^{2}(\mathcal{N})$
gives the spherical contribution 
\[
\frac{2e^{-\tau/2}}{1-e^{-\tau}}\,\operatorname{Tr}_{L_{0}^{2}(\mathcal{N})}\left(\cos(\tau\Lambda)\right),
\]
understood as the usual wave-trace distribution.

For one holomorphic discrete series of lowest $K$-type $l$, the
model gives 
\[
z_{n,l}=-n-\frac{l}{2},\qquad n\in\mathbb{N}.
\]
Hence 
\[
\operatorname{Tr}(e^{\tau X})=\sum_{n\ge0}e^{-\tau(n+l/2)}=\frac{e^{-\tau l/2}}{1-e^{-\tau}}.
\]
The anti-holomorphic copy gives the same contribution. Since the multiplicity
of the holomorphic discrete series of lowest $K$-type $l$ is $m_{l}$,
the full discrete contribution is 
\[
\frac{2}{1-e^{-\tau}}\sum_{l\in2\mathbb{N}^{*}}m_{l}e^{-\tau l/2}.
\]

Combining the trivial, spherical and discrete contributions gives
\eqref{eq:spectral_trace_model_before_RR}. 
\end{proof}
\begin{cBoxB}{}

\begin{cor}[Spectral trace after Riemann--Roch]
\label{cor:spectral_trace_after_RR} Using Proposition~\ref{prop:RR_discrete_multiplicities},
the spectral trace becomes 
\begin{align}
\operatorname{Tr}_{\mathcal{H}_{r}}^{\mathrm{spec}}\left(e^{\tau X}\right) & =1+\frac{2e^{-\tau/2}}{1-e^{-\tau}}\,\operatorname{Tr}_{L_{0}^{2}(\mathcal{N})}\left(\cos(\tau\Lambda)\right)\nonumber \\
 & \quad+\frac{2e^{-\tau}}{1-e^{-\tau}}+|\chi(\mathcal{N})|\frac{e^{-\tau}(1+e^{-\tau})}{(1-e^{-\tau})^{3}}.\label{eq:spectral_trace_model_after_RR}
\end{align}
\end{cor}

\end{cBoxB}

\begin{proof}
Write 
\[
l=2q,\qquad x=e^{-\tau}.
\]
By Proposition~\ref{prop:RR_discrete_multiplicities}, 
\[
m_{2}=g,\qquad m_{2q}=(2q-1)(g-1),\qquad q\ge2.
\]
Equivalently, 
\[
m_{2q}=(2q-1)(g-1)+\delta_{q,1},\qquad q\ge1.
\]
Therefore 
\[
\sum_{q\ge1}m_{2q}x^{q}=x+(g-1)\sum_{q\ge1}(2q-1)x^{q}.
\]
Since 
\[
\sum_{q\ge1}(2q-1)x^{q}=\frac{x(1+x)}{(1-x)^{2}},
\]
we obtain 
\[
\frac{2}{1-x}\sum_{q\ge1}m_{2q}x^{q}=\frac{2x}{1-x}+2(g-1)\frac{x(1+x)}{(1-x)^{3}}.
\]
Finally, 
\[
|\chi(\mathcal{N})|=2g-2=2(g-1).
\]
Substituting $x=e^{-\tau}$ into \eqref{eq:spectral_trace_model_before_RR}
gives \eqref{eq:spectral_trace_model_after_RR}. 
\end{proof}

\section{\protect\label{sec:Atiyah=002013Bott-and-Selberg}Atiyah--Bott and
Selberg trace formula}

In the previous sections, we constructed Hilbert models in which the
geodesic flow $e^{tX}$ has a discrete spectral representation and
becomes trace class. We now compare the resulting spectral trace with
the dynamical flat trace, and recover the Selberg trace formula.

\subsection{Atiyah--Bott--Guillemin flat trace formula}

Let 
\[
\varphi^{t}:M\to M
\]
be the flow generated by $X$. We denote by $\mathcal{P}_{X}$ the
set of primitive closed orbits of $\varphi^{t}$. If $\gamma\in\mathcal{P}_{X}$,
we write $\ell_{\gamma}>0$ for its primitive period. The $m$-th
iterate $\gamma^{m}$ has period $m\ell_{\gamma}$.

\begin{cBoxB}{}

\begin{thm}[Atiyah--Bott--Guillemin flat trace formula]
\label{thm:ABG_flat_trace_geodesic_flow} For $t>0$, the flat trace
of the propagator $e^{tX}$ is the distribution 
\begin{equation}
\operatorname{Tr}_{L^{2}(M)}^{\flat}\left(e^{tX}\right)=\sum_{\gamma\in\mathcal{P}_{X}}\sum_{m\ge1}\frac{\ell_{\gamma}\,\delta(t-m\ell_{\gamma})}{\left|\det\left(\mathrm{Id}-d\varphi_{|E_{s}\oplus E_{u}}^{m\ell_{\gamma}}\right)\right|}.\label{eq:ABG_flat_trace_general}
\end{equation}
Since the curvature of $\mathcal{N}$ is constant equal to $-1$,
one has, for $T>0$, 
\[
d\varphi_{|E_{s}\oplus E_{u}}^{T}\simeq\begin{pmatrix}e^{-T} & 0\\
0 & e^{T}
\end{pmatrix}.
\]
Hence 
\[
\left|\det\left(\mathrm{Id}-d\varphi_{|E_{s}\oplus E_{u}}^{T}\right)\right|=(1-e^{-T})(e^{T}-1)=4\sinh^{2}\left(\frac{T}{2}\right).
\]
Therefore 
\begin{equation}
\operatorname{Tr}_{L^{2}(M)}^{\flat}\left(e^{tX}\right)=\sum_{\gamma\in\mathcal{P}_{X}}\sum_{m\ge1}\frac{\ell_{\gamma}}{4\sinh^{2}\left(\frac{m\ell_{\gamma}}{2}\right)}\,\delta(t-m\ell_{\gamma}).\label{eq:ABG_flat_trace_geodesic}
\end{equation}
\end{thm}

\end{cBoxB}

\begin{proof}
The Schwartz kernel of $e^{tX}$ is supported on the graph of the
flow. Formally, 
\[
K_{t}(x,y)=\delta\bigl(y-\varphi^{t}(x)\bigr).
\]
Thus the flat trace is the diagonal flat integral 
\[
\operatorname{Tr}^{\flat}(e^{tX})=\int_{M}^{\flat}K_{t}(x,x)\,dx.
\]
Its singularities occur precisely when 
\[
\varphi^{t}(x)=x,
\]
that is, along points belonging to closed orbits whose period divides
$t$.

The Atiyah--Bott--Guillemin fixed point formula for flows gives
the contribution of the $m$-th iterate of a primitive closed orbit
$\gamma$: 
\[
\frac{\ell_{\gamma}\,\delta(t-m\ell_{\gamma})}{\left|\det\left(\mathrm{Id}-d\varphi_{|E_{s}\oplus E_{u}}^{m\ell_{\gamma}}\right)\right|}.
\]
Summing over primitive closed orbits and their iterates gives \eqref{eq:ABG_flat_trace_general}.

In constant curvature $-1$, the stable and unstable multipliers over
a closed orbit of period $T$ are respectively 
\[
e^{-T}\qquad\text{and}\qquad e^{T}.
\]
Therefore 
\[
\left|\det\left(\mathrm{Id}-d\varphi_{|E_{s}\oplus E_{u}}^{T}\right)\right|=(1-e^{-T})(e^{T}-1)=4\sinh^{2}\left(\frac{T}{2}\right).
\]
This gives \eqref{eq:ABG_flat_trace_geodesic}. 
\end{proof}
\begin{rem}
Here $\mathcal{P}_{X}$ denotes primitive closed orbits of the oriented
geodesic flow on $M=\Gamma\backslash SL_{2}(\mathbb{R})$. If one
sums instead over primitive unoriented closed geodesics on $\mathcal{N}$,
one must account for the two possible orientations. 
\end{rem}

\subsection{Decomposition into irreducible components}

The representation-theoretic decomposition gives 
\[
L^{2}(M)=\mathcal{H}_{\mathrm{triv}}\widehat{\oplus}\mathcal{H}^{\mathrm{d.s.}}\widehat{\oplus}\mathcal{H}^{\mathrm{sph}},
\]
with 
\[
\mathcal{H}_{\mathrm{triv}}=\mathbb{C}\boldsymbol{1},
\]
\[
\mathcal{H}^{\mathrm{d.s.}}\eq{\ref{eq:def_H_ds}}\widehat{\bigoplus}_{l\in2\mathbb{N}^{*}}\left(\mathcal{H}_{l}^{\mathrm{d.s.},+}\oplus\mathcal{H}_{-l}^{\mathrm{d.s.},-}\right),
\]
and 
\[
\mathcal{H}^{\mathrm{sph}}\eq{\ref{eq:H_sph}}\widehat{\bigoplus}_{\mu\in\operatorname{Spec}(\Delta_{|L_{0}^{2}(\mathcal{N})})}\mathcal{H}_{\mu}^{\mathrm{sph}}.
\]
Here $\mathcal{H}_{l}^{\mathrm{d.s.},+}$ denotes the full multiplicity
space generated by $\mathcal{H}_{\mathrm{hol},+}(l)$, and similarly
for $\mathcal{H}_{-l}^{\mathrm{d.s.},-}$.

\begin{cBoxB}{}

\begin{lem}[Flat trace decomposition]
\label{lem:flat_trace_decomposition} For $t>0$, in the sense of
flat traces, one has 
\begin{align}
\operatorname{Tr}_{L^{2}(M)}^{\flat}\left(e^{tX}\right) & =\operatorname{Tr}_{\mathcal{H}_{\mathrm{triv}}}^{\flat}\left(e^{tX}\right)\nonumber \\
 & \quad+\sum_{\mu\in\operatorname{Spec}(\Delta_{|L_{0}^{2}(\mathcal{N})})}\operatorname{Tr}_{\mathcal{H}_{\mu}^{\mathrm{sph}}}^{\flat}\left(e^{tX}\right)\nonumber \\
 & \quad+\sum_{l\in2\mathbb{N}^{*}}\left[\operatorname{Tr}_{\mathcal{H}_{l}^{\mathrm{d.s.},+}}^{\flat}\left(e^{tX}\right)+\operatorname{Tr}_{\mathcal{H}_{-l}^{\mathrm{d.s.},-}}^{\flat}\left(e^{tX}\right)\right].\label{eq:flat_trace_decomposition}
\end{align}
The possible threshold contribution $\mu=1/4$ is included in the
spherical sum. 
\end{lem}

\end{cBoxB}

\begin{proof}
We define the decomposition of the flat trace by smoothing regularization.
Let 
\[
\chi_{\varepsilon}:L^{2}(M)\to C^{\infty}(M)
\]
be a family of smoothing operators such that $\chi_{\varepsilon}\to\mathrm{Id}$
strongly on $L^{2}(M)$, and such that $\chi_{\varepsilon}e^{tX}$
is trace class for every $\varepsilon>0$. The ordinary trace of this
regularized operator will be used only as a device to define the corresponding
flat trace in the limit $\varepsilon\to0$.

Let 
\[
\Pi_{\mathrm{triv}},\qquad\Pi_{\mu}^{\mathrm{sph}},\qquad\Pi_{l}^{+},\qquad\Pi_{l}^{-}
\]
be the orthogonal projectors onto 
\[
\mathcal{H}_{\mathrm{triv}},\qquad\mathcal{H}_{\mu}^{\mathrm{sph}},\qquad\mathcal{H}_{l}^{\mathrm{d.s.},+},\qquad\mathcal{H}_{-l}^{\mathrm{d.s.},-}.
\]
They give the orthogonal decomposition 
\[
\mathrm{Id}_{L^{2}(M)}=\Pi_{\mathrm{triv}}+\sum_{\mu\in\operatorname{Spec}(\Delta_{|L_{0}^{2}(\mathcal{N})})}\Pi_{\mu}^{\mathrm{sph}}+\sum_{l\in2\mathbb{N}^{*}}\left(\Pi_{l}^{+}+\Pi_{l}^{-}\right),
\]
with convergence in the strong operator topology.

For fixed $\varepsilon>0$, set 
\[
T_{\varepsilon}(t):=\chi_{\varepsilon}e^{tX}.
\]
Since $T_{\varepsilon}(t)$ is trace class, its trace is the sum of
the traces of its diagonal blocks with respect to the above orthogonal
decomposition. Indeed, the off-diagonal blocks have zero trace: if
$\Pi_{\alpha}\Pi_{\beta}=0$, then by cyclicity, 
\[
\operatorname{Tr}(\Pi_{\alpha}T_{\varepsilon}(t)\Pi_{\beta})=\operatorname{Tr}(\Pi_{\beta}\Pi_{\alpha}T_{\varepsilon}(t))=0.
\]
Therefore 
\begin{align*}
\operatorname{Tr}_{L^{2}(M)}(T_{\varepsilon}(t)) & =\operatorname{Tr}_{\mathcal{H}_{\mathrm{triv}}}\left(\Pi_{\mathrm{triv}}T_{\varepsilon}(t)\Pi_{\mathrm{triv}}\right)\\
 & \quad+\sum_{\mu}\operatorname{Tr}_{\mathcal{H}_{\mu}^{\mathrm{sph}}}\left(\Pi_{\mu}^{\mathrm{sph}}T_{\varepsilon}(t)\Pi_{\mu}^{\mathrm{sph}}\right)\\
 & \quad+\sum_{l\in2\mathbb{N}^{*}}\operatorname{Tr}_{\mathcal{H}_{l}^{\mathrm{d.s.},+}}\left(\Pi_{l}^{+}T_{\varepsilon}(t)\Pi_{l}^{+}\right)\\
 & \quad+\sum_{l\in2\mathbb{N}^{*}}\operatorname{Tr}_{\mathcal{H}_{-l}^{\mathrm{d.s.},-}}\left(\Pi_{l}^{-}T_{\varepsilon}(t)\Pi_{l}^{-}\right),
\end{align*}
where the sums are absolutely convergent for fixed $\varepsilon>0$.
For fixed $\varepsilon>0$, this is an identity of ordinary traces.
Passing to the limit $\varepsilon\to0$ in the distributional sense
in the time variable defines the corresponding block flat traces and
gives \ref{eq:flat_trace_decomposition}. Namely, for $\rho\in C_{c}^{\infty}((0,\infty))$,
\[
\left\langle \operatorname{Tr}_{L^{2}(M)}^{\flat}(e^{tX}),\rho(t)\right\rangle :=\lim_{\varepsilon\to0}\int_{0}^{\infty}\rho(t)\,\operatorname{Tr}_{L^{2}(M)}\left(\chi_{\varepsilon}e^{tX}\right)\,dt.
\]
Passing to the limit in the preceding block decomposition defines
the corresponding block flat traces and gives \eqref{eq:flat_trace_decomposition}.
The possible threshold contribution $\mu=1/4$ is included in the
spherical sum. 
\end{proof}

\subsection{Comparison with the spectral trace}

We now compare the geometric flat trace \eqref{eq:ABG_flat_trace_geodesic}
with the representation-theoretic decomposition \eqref{eq:flat_trace_decomposition}.
The traces of the irreducible components were computed in the Hilbert
model and summarized in Corollary~\ref{cor:spectral_trace_after_RR}.
Hence, in the sense of distributions in the time variable, 
\[
\operatorname{Tr}_{L^{2}(M)}^{\flat}\left(e^{tX}\right)=\operatorname{Tr}_{\mathcal{H}_{r}}^{\mathrm{spec}}\left(e^{tX}\right).
\]
Equating this spectral expression with the Atiyah--Bott--Guillemin
flat trace gives the Selberg trace formula in flow form.

\begin{cBoxB}{}

\begin{cor}[Selberg trace formula in flow form]
\label{cor:Selberg_trace_flow_form} As distributions on $(0,\infty)$,
one has 
\begin{align}
\sum_{\gamma\in\mathcal{P}_{X}}\sum_{m\ge1}\frac{\ell_{\gamma}}{4\sinh^{2}(m\ell_{\gamma}/2)}\,\delta(t-m\ell_{\gamma}) & =1+\frac{2e^{-t/2}}{1-e^{-t}}\,\operatorname{Tr}_{L_{0}^{2}(\mathcal{N})}\left(\cos(t\Lambda)\right)\nonumber \\
 & \quad+\frac{2e^{-t}}{1-e^{-t}}+|\chi(\mathcal{N})|\frac{e^{-t}(1+e^{-t})}{(1-e^{-t})^{3}}.\label{eq:Selberg_trace_flow_form}
\end{align}
Here $\Lambda=(\Delta-\frac{1}{4})^{1/2}$, with the convention used
in \eqref{eq:def_global_Lambda_spherical}, extended by $0$ on the
possible eigenspace $E_{1/4}$. 
\end{cor}

\end{cBoxB}

\begin{proof}
The left-hand side is the Atiyah--Bott--Guillemin flat trace of
$e^{tX}$, given by \eqref{eq:ABG_flat_trace_geodesic}. By Lemma~\ref{lem:flat_trace_decomposition},
this flat trace decomposes as the sum of the flat traces on the irreducible
components. By the componentwise trace computations, this equals the
spectral trace in the Hilbert model. Corollary~\ref{cor:spectral_trace_after_RR}
gives the explicit spectral expression. Comparing the two sides gives
\eqref{eq:Selberg_trace_flow_form}. 
\end{proof}

\subsection{From the flow trace to the Selberg wave trace}

We now rewrite the flow trace formula in the usual wave-trace form
of Selberg's trace formula. Set 
\[
J(t):=4\sinh^{2}\left(\frac{t}{2}\right),\qquad J(t)^{1/2}=2\sinh\left(\frac{t}{2}\right)=e^{t/2}(1-e^{-t}).
\]
Thus 
\[
J(t)^{-1/2}=\frac{1}{2\sinh(t/2)}=\frac{e^{-t/2}}{1-e^{-t}}.
\]

\begin{cBoxB}{}

\begin{lem}
\label{lem:tanh_Fourier_identity} For $t>0$, in the distributional
sense, 
\begin{equation}
\int_{\mathbb{R}}re^{itr}\tanh(\pi r)\,dr=-\frac{e^{-t/2}(1+e^{-t})}{(1-e^{-t})^{2}}=-\frac{1}{2}\frac{\cosh(t/2)}{\sinh^{2}(t/2)}.\label{eq:tanh_Fourier_identity}
\end{equation}
\end{lem}

\end{cBoxB}

\begin{proof}
The poles of $\tanh(\pi r)$ in the upper half-plane are 
\[
r=i\left(k+\frac{1}{2}\right),\qquad k\in\mathbb{N},
\]
and each has residue $1/\pi$. For $t>0$, closing the contour in
the upper half-plane gives, with the usual distributional regularization,
\[
\int_{\mathbb{R}}re^{itr}\tanh(\pi r)\,dr=2\pi i\sum_{k\ge0}\frac{1}{\pi}i\left(k+\frac{1}{2}\right)e^{-t(k+1/2)}.
\]
Hence 
\[
\int_{\mathbb{R}}re^{itr}\tanh(\pi r)\,dr=-2\sum_{k\ge0}\left(k+\frac{1}{2}\right)e^{-t(k+1/2)}.
\]
Since 
\[
\sum_{k\ge0}e^{-t(k+1/2)}=\frac{e^{-t/2}}{1-e^{-t}},
\]
differentiating with respect to $t$ gives 
\[
\sum_{k\ge0}\left(k+\frac{1}{2}\right)e^{-t(k+1/2)}=\frac{e^{-t/2}(1+e^{-t})}{2(1-e^{-t})^{2}}.
\]
This proves \eqref{eq:tanh_Fourier_identity}. 
\end{proof}
We start from the flow form obtained above: 
\begin{align}
\operatorname{Tr}_{L^{2}(M)}^{\flat}\left(e^{tX}\right) & =1+\frac{2e^{-t/2}}{1-e^{-t}}\,\operatorname{Tr}_{L_{0}^{2}(\mathcal{N})}\left(\cos(t\Lambda)\right)\nonumber \\
 & \quad+\frac{2e^{-t}}{1-e^{-t}}+|\chi(\mathcal{N})|\frac{e^{-t}(1+e^{-t})}{(1-e^{-t})^{3}}.\label{eq:flow_trace_spectral_before_wave}
\end{align}
Multiplying by $J(t)^{1/2}$, and using 
\[
J(t)^{1/2}\frac{2e^{-t/2}}{1-e^{-t}}=2,
\]
we obtain 
\begin{align}
J(t)^{1/2}\operatorname{Tr}_{L^{2}(M)}^{\flat}\left(e^{tX}\right) & =J(t)^{1/2}+2\,\operatorname{Tr}_{L_{0}^{2}(\mathcal{N})}\left(\cos(t\Lambda)\right)\nonumber \\
 & \quad+2e^{-t/2}+|\chi(\mathcal{N})|\frac{e^{-t/2}(1+e^{-t})}{(1-e^{-t})^{2}}.\label{eq:Dhalf_spectral_trace}
\end{align}
By \eqref{eq:tanh_Fourier_identity}, 
\[
|\chi(\mathcal{N})|\frac{e^{-t/2}(1+e^{-t})}{(1-e^{-t})^{2}}=-|\chi(\mathcal{N})|\int_{\mathbb{R}}re^{itr}\tanh(\pi r)\,dr.
\]
Therefore 
\begin{align}
2\,\operatorname{Tr}_{L_{0}^{2}(\mathcal{N})}\left(\cos(t\Lambda)\right) & =J(t)^{1/2}\operatorname{Tr}_{L^{2}(M)}^{\flat}\left(e^{tX}\right)\nonumber \\
 & \quad-J(t)^{1/2}-2e^{-t/2}+|\chi(\mathcal{N})|\int_{\mathbb{R}}re^{itr}\tanh(\pi r)\,dr.\label{eq:L0_wave_from_flow}
\end{align}
Since 
\[
J(t)^{1/2}=e^{t/2}-e^{-t/2},
\]
we have 
\[
J(t)^{1/2}+2e^{-t/2}=e^{t/2}+e^{-t/2}.
\]
Thus 
\begin{align}
2\,\operatorname{Tr}_{L_{0}^{2}(\mathcal{N})}\left(\cos(t\Lambda)\right) & =J(t)^{1/2}\operatorname{Tr}_{L^{2}(M)}^{\flat}\left(e^{tX}\right)\nonumber \\
 & \quad-e^{t/2}-e^{-t/2}+|\chi(\mathcal{N})|\int_{\mathbb{R}}re^{itr}\tanh(\pi r)\,dr.\label{eq:L0_wave_from_flow_simplified}
\end{align}

The constant eigenfunction of $\Delta$ has eigenvalue $0$, and therefore
contributes 
\[
e^{it\sqrt{-1/4}}+e^{-it\sqrt{-1/4}}=e^{-t/2}+e^{t/2}.
\]
Adding this constant contribution to \eqref{eq:L0_wave_from_flow_simplified},
the elementary exponential terms cancel. Hence 
\begin{equation}
\operatorname{Tr}_{L^{2}(\mathcal{N})}^{\mathrm{dist}}\left(e^{it\sqrt{\Delta-\frac{1}{4}}}+e^{-it\sqrt{\Delta-\frac{1}{4}}}\right)=J(t)^{1/2}\operatorname{Tr}_{L^{2}(M)}^{\flat}\left(e^{tX}\right)+|\chi(\mathcal{N})|\int_{\mathbb{R}}re^{itr}\tanh(\pi r)\,dr.\label{eq:wave_trace_from_flow_trace}
\end{equation}

Finally, by the Atiyah--Bott--Guillemin formula \eqref{eq:ABG_flat_trace_geodesic},
\[
\operatorname{Tr}_{L^{2}(M)}^{\flat}\left(e^{tX}\right)=\sum_{\gamma\in\mathcal{P}_{X}}\sum_{m\ge1}\frac{\ell_{\gamma}}{4\sinh^{2}(m\ell_{\gamma}/2)}\,\delta(t-m\ell_{\gamma}).
\]
Multiplying by $J(t)^{1/2}$, and using $t=m\ell_{\gamma}$ under
the delta distribution, gives 
\[
J(t)^{1/2}\operatorname{Tr}_{L^{2}(M)}^{\flat}\left(e^{tX}\right)=\sum_{\gamma\in\mathcal{P}_{X}}\sum_{m\ge1}\frac{\ell_{\gamma}}{2\sinh(m\ell_{\gamma}/2)}\,\delta(t-m\ell_{\gamma}).
\]
Substituting into \eqref{eq:wave_trace_from_flow_trace}, we obtain,
as distributions on $(0,\infty)$, 
\begin{equation}
\boxed{\begin{aligned} & \operatorname{Tr}_{L^{2}(\mathcal{N})}^{\mathrm{dist}}\left(e^{it\sqrt{\Delta-\frac{1}{4}}}+e^{-it\sqrt{\Delta-\frac{1}{4}}}\right)\\
 & \qquad=|\chi(\mathcal{N})|\int_{\mathbb{R}}re^{itr}\tanh(\pi r)\,dr+\sum_{\gamma\in\mathcal{P}_{X}}\sum_{m\ge1}\frac{\ell_{\gamma}}{2\sinh(m\ell_{\gamma}/2)}\,\delta(t-m\ell_{\gamma}).
\end{aligned}
}\label{eq:Selberg_wave_trace_final}
\end{equation}
Here $\mathcal{P}_{X}$ denotes primitive oriented closed geodesics,
as in the flow trace formula above. This is the Selberg wave trace
formula in the usual form; see also \cite{duistermaat_guillemin_1975,mckean1972selberg,lax_phillips_1976,borthwick2016spectral}.

\section{\protect\label{sec:Spherical-averages}Spherical averages and emergence
of the wave propagator }

In this section we discuss another manifestation of the main relation
between the geodesic flow operator $e^{tX}$ and the wave dynamics
in the Hilbert model constructed above. Averaging over the circle
fibers of the bundle map 
\[
S^{*}\mathcal{N}\longrightarrow\mathcal{N}
\]
gives the spherical mean operator. We deduce an asymptotic expansion
for this operator, with control in operator norm on $L^{2}(\mathcal{N})$.

\subsection{The scalar spherical function}

We first recall the spherical matrix coefficient in one spherical
principal-series representation. Let $\lambda>0$, and set 
\[
b=-\frac{1}{2}+i\lambda.
\]
We keep the normalization used above, 
\[
\psi_{0}(\theta)=\frac{1}{\sqrt{2\pi}},\qquad\langle f\mid g\rangle_{L^{2}(S^{1})}=\int_{0}^{2\pi}\overline{f(\theta)}g(\theta)\,d\theta.
\]
With this convention, the scalar spherical coefficient is normalized
by the factor $1/(2\pi)$.

\begin{cBoxB}{}

\begin{prop}[Spherical function]
\label{prop:spherical_average} For every $t>0$, 
\begin{equation}
\left\langle \psi_{0}\mid e^{t\widetilde{X}}\psi_{0}\right\rangle _{L^{2}(S^{1})}=\frac{1}{2\pi}\int_{0}^{2\pi}\left(\cosh t+\sinh t\cos\theta\right)^{-\frac{1}{2}+i\lambda}\,d\theta.\label{eq:spherical_average_integral}
\end{equation}
Equivalently, 
\begin{equation}
\left\langle \psi_{0}\mid e^{t\widetilde{X}}\psi_{0}\right\rangle _{L^{2}(S^{1})}=P_{-\frac{1}{2}+i\lambda}(\cosh t),\label{eq:spherical_average_legendre}
\end{equation}
where $P_{\nu}$ is the Legendre function of the first kind. 
\end{prop}

\end{cBoxB}

\begin{proof}
In the compact picture used above, one has 
\[
\widetilde{X}=\cos\theta\,\partial_{\theta}+b\sin\theta,\qquad b=-\frac{1}{2}+i\lambda.
\]
Let 
\[
F(t,\theta):=\frac{1}{\sqrt{2\pi}}(\cosh t+\sinh t\sin\theta)^{b}.
\]
Then 
\[
F(0,\theta)=\psi_{0}(\theta).
\]
Set 
\[
D(t,\theta):=\cosh t+\sinh t\sin\theta.
\]
We have 
\[
\partial_{t}F=bD^{b-1}\left(\sinh t+\cosh t\sin\theta\right),
\]
and 
\[
\partial_{\theta}F=bD^{b-1}\sinh t\cos\theta.
\]
Therefore 
\begin{align*}
\widetilde{X}F & =\cos\theta\,\partial_{\theta}F+b\sin\theta\,F\\
 & =bD^{b-1}\left(\sinh t\cos^{2}\theta+\sin\theta(\cosh t+\sinh t\sin\theta)\right)\\
 & =bD^{b-1}\left(\sinh t+\cosh t\sin\theta\right)\\
 & =\partial_{t}F.
\end{align*}
Thus 
\[
\partial_{t}F=\widetilde{X}F,\qquad F(0,\theta)=\psi_{0}(\theta),
\]
and hence 
\[
(e^{t\widetilde{X}}\psi_{0})(\theta)=\frac{1}{\sqrt{2\pi}}(\cosh t+\sinh t\sin\theta)^{b}.
\]
Taking the scalar product with $\psi_{0}(\theta)=1/\sqrt{2\pi}$,
we get
\[
\left\langle \psi_{0}\mid e^{t\widetilde{X}}\psi_{0}\right\rangle _{L^{2}(S^{1})}=\frac{1}{2\pi}\int_{0}^{2\pi}\left(\cosh t+\sinh t\sin\theta\right)^{b}\,d\theta.
\]
By the change of variable $\theta\mapsto\theta-\pi/2$, this is equivalent
to 
\[
\frac{1}{2\pi}\int_{0}^{2\pi}\left(\cosh t+\sinh t\cos\theta\right)^{b}\,d\theta,
\]
which proves \eqref{eq:spherical_average_integral}. The identification
with 
\[
P_{-\frac{1}{2}+i\lambda}(\cosh t)
\]
is the standard integral representation of the spherical function
on the hyperbolic plane. 
\end{proof}

\subsection{Harish--Chandra expansion}

The same matrix coefficient admits the following large-time expansion.
In the resonant model, it is obtained by summing the resonant states;
equivalently, it is the classical Harish--Chandra expansion of the
spherical function. We shall use this expansion together with its
uniform estimates in the spectral parameter. For the general theory
of spherical functions, see Helgason~\cite{Helgason1984}; for uniform
estimates of the Harish--Chandra series, see Gangolli--Varadarajan~\cite[Chapter~IV, Sections~4.4--4.6]{GangolliVaradarajan1988}.

\begin{cBoxB}{}

\begin{prop}[Resonance expansion of the spherical average]
\label{prop:resonance_expansion_spherical_average} For every $\tau>0$,
\begin{align}
\left\langle \psi_{0}\mid e^{\tau\widetilde{X}}\psi_{0}\right\rangle _{L^{2}(S^{1})} & \eq{\eqref{eq:spherical_average_legendre}}P_{-\frac{1}{2}+i\lambda}(\cosh\tau)\nonumber \\
 & =e^{-\tau(\frac{1}{2}-i\lambda)}\sum_{m\ge0}e^{-2m\tau}W_{2m,\lambda}+e^{-\tau(\frac{1}{2}+i\lambda)}\sum_{m\ge0}e^{-2m\tau}W_{2m,-\lambda},\label{eq:resonance_expansion_spherical_average}
\end{align}
where 
\begin{equation}
W_{2m,\lambda}:=\frac{1}{\pi}\frac{\Gamma\left(m+\frac{1}{2}\right)\Gamma(i\lambda-m)}{m!\,\Gamma\left(\frac{1}{2}+i\lambda-m\right)},\qquad W_{2m+1,\lambda}:=0.\label{eq:def_W_2m_lambda}
\end{equation}
For real $\lambda$, 
\[
W_{2m,-\lambda}=\overline{W_{2m,\lambda}}.
\]
The leading coefficient is 
\begin{equation}
W_{0,\lambda}=\frac{1}{\sqrt{\pi}}\frac{\Gamma(i\lambda)}{\Gamma\left(\frac{1}{2}+i\lambda\right)}.\label{eq:leading_HC_coefficient}
\end{equation}
Moreover, for fixed $m\in\mathbb{N}$, as $\lambda\to+\infty$, 
\begin{equation}
W_{2m,\lambda}=\frac{1}{\sqrt{\pi}}\frac{\binom{2m}{m}}{4^{m}}e^{-i\pi/4}\lambda^{-1/2}\left(1-\frac{i(4m+1)}{8\lambda}+\mathcal{O}_{m}(\lambda^{-2})\right).\label{eq:W_2m_lambda_asymptotic}
\end{equation}
\end{prop}

\end{cBoxB}

\begin{proof}
We apply the correlation formula in the resonant model to $\psi_{0}$.
The positive branch gives 
\[
e^{-\tau(\frac{1}{2}-i\lambda)}\sum_{n\ge0}e^{-n\tau}\left\langle (\mathcal{T}_{+}^{-1})^{\dagger}\psi_{0}\mid e_{n}^{+}\right\rangle \left\langle e_{n}^{+}\mid\mathcal{T}_{+}\psi_{0}\right\rangle .
\]
By parity, only even indices contribute. A direct computation of the
two coefficients gives 
\[
\left\langle (\mathcal{T}_{+}^{-1})^{\dagger}\psi_{0}\mid e_{2m}^{+}\right\rangle \left\langle e_{2m}^{+}\mid\mathcal{T}_{+}\psi_{0}\right\rangle =W_{2m,\lambda},
\]
with $W_{2m,\lambda}$ given by \eqref{eq:def_W_2m_lambda}. The negative
branch gives the same expression with $\lambda$ replaced by $-\lambda$.
This proves \eqref{eq:resonance_expansion_spherical_average}.

The formula \eqref{eq:leading_HC_coefficient} is the case $m=0$.

It remains to record the large-$\lambda$ behavior. For fixed $m$,
write 
\[
z=i\lambda-m.
\]
Using 
\[
\frac{\Gamma(z)}{\Gamma(z+\frac{1}{2})}=z^{-1/2}\left(1+\frac{1}{8z}+\mathcal{O}(z^{-2})\right),\qquad|z|\to+\infty,
\]
with the principal branch, we get 
\[
\frac{\Gamma(i\lambda-m)}{\Gamma(\frac{1}{2}+i\lambda-m)}=(i\lambda)^{-1/2}\left(1-\frac{i(4m+1)}{8\lambda}+\mathcal{O}_{m}(\lambda^{-2})\right).
\]
Since 
\[
(i\lambda)^{-1/2}=e^{-i\pi/4}\lambda^{-1/2},
\]
we obtain 
\[
W_{2m,\lambda}=\frac{1}{\pi}\frac{\Gamma(m+\frac{1}{2})}{m!}e^{-i\pi/4}\lambda^{-1/2}\left(1-\frac{i(4m+1)}{8\lambda}+\mathcal{O}_{m}(\lambda^{-2})\right).
\]
Finally, 
\[
\frac{1}{\pi}\frac{\Gamma(m+\frac{1}{2})}{m!}=\frac{1}{\sqrt{\pi}}\frac{\binom{2m}{m}}{4^{m}},
\]
which gives \eqref{eq:W_2m_lambda_asymptotic}. 
\end{proof}

\subsection{Global operator formula up to finite rank}

Let 
\[
\iota_{0}:L^{2}(\mathcal{N})\longrightarrow L^{2}(M)
\]
be the natural embedding as $K$-invariant functions, and let 
\[
\pi_{0}:L^{2}(M)\longrightarrow L^{2}(\mathcal{N})
\]
be the corresponding orthogonal projection. We define the spherical
mean operator 
\begin{equation}
\mathcal{L}_{t}:=\pi_{0}e^{tX}\iota_{0}.\label{eq:def_L_t_spherical_mean}
\end{equation}

We isolate the finite-dimensional exceptional part. Fix $\eta>0$,
and let 
\[
\Pi_{\mathrm{exc}}
\]
be the spectral projector of $\Delta$ onto the finite-dimensional
subspace generated by the constant functions, the complementary spectrum,
the possible threshold space $E_{1/4}$, and the principal eigenvalues
satisfying 
\[
0<\sqrt{\mu-\frac{1}{4}}<\eta.
\]
Equivalently, $\Pi_{\mathrm{exc}}$ contains the spectrum with 
\[
\mu\le\frac{1}{4},
\]
together with a fixed finite window just above the threshold.

We set 
\[
\mathcal{H}_{\mathrm{reg}}:=(1-\Pi_{\mathrm{exc}})L_{0}^{2}(\mathcal{N}).
\]
On $\mathcal{H}_{\mathrm{reg}}$, the operator 
\[
\Lambda:=\left(\Delta-\frac{1}{4}\right)^{1/2}
\]
is self-adjoint and satisfies 
\[
\Lambda\ge\eta.
\]

Define 
\[
w_{+}(\lambda):=\frac{1}{\sqrt{\pi}}\frac{\Gamma(i\lambda)}{\Gamma\left(\frac{1}{2}+i\lambda\right)},\qquad w_{-}(\lambda):=w_{+}(-\lambda),
\]
and the corresponding spectral multipliers on $\mathcal{H}_{\mathrm{reg}}$:
\[
W_{+}:=w_{+}(\Lambda),\qquad W_{-}:=w_{-}(\Lambda).
\]
Since $\Lambda$ is self-adjoint on $\mathcal{H}_{\mathrm{reg}}$,
one has 
\[
W_{-}=W_{+}^{\dagger}.
\]

Finally, define the finite-rank contribution 
\begin{equation}
F_{t}:=\mathcal{L}_{t}\Pi_{\mathrm{exc}}.\label{eq:def_finite_rank_exceptional_part}
\end{equation}
Its rank is uniformly bounded: 
\[
\operatorname{rank}F_{t}\le\operatorname{rank}\Pi_{\mathrm{exc}}.
\]

\begin{cBoxB}{}

\begin{prop}[Large-time spherical average up to finite rank]
\label{prop:global_spherical_average_asymptotic} For $t>0$, one
has 
\begin{equation}
\mathcal{L}_{t}=F_{t}+e^{-t/2}\left(W_{+}e^{it\Lambda}+W_{-}e^{-it\Lambda}+e^{-2t}\mathcal{R}_{t}\right)(1-\Pi_{\mathrm{exc}}),\label{eq:global_spherical_average_asymptotic}
\end{equation}
where $\mathcal{R}_{t}$ is uniformly bounded for $t\ge t_{0}>0$.
More precisely, for every $s\in\mathbb{R}$, 
\[
\sup_{t\ge t_{0}}\|\mathcal{R}_{t}\|_{H^{s}(\mathcal{N})\to H^{s+1/2}(\mathcal{N})}<\infty.
\]

On the regular principal spectrum, the leading part is therefore 
\begin{equation}
e^{-t/2}\left(W_{+}e^{it\sqrt{\Delta-\frac{1}{4}}}+W_{+}^{\dagger}e^{-it\sqrt{\Delta-\frac{1}{4}}}\right).\label{eq:principal_leading_spherical_average}
\end{equation}

Moreover, $W_{+}$ is a pseudodifferential operator of order $-1/2$,
with principal term 
\begin{equation}
W_{+}=\frac{1}{\sqrt{\pi}}e^{-i\pi/4}\left(\Delta-\frac{1}{4}\right)^{-1/4}+R_{+},\label{eq:W_plus_principal}
\end{equation}
where, for every $s\in\mathbb{R}$, 
\[
R_{+}:H^{s}(\mathcal{N})\longrightarrow H^{s+3/2}(\mathcal{N}).
\]
Equivalently, in a slightly less sharp but simpler form, 
\begin{equation}
W_{+}=\frac{1}{\sqrt{\pi}}e^{-i\pi/4}\Delta^{-1/4}+O\left(H^{s}(\mathcal{N})\to H^{s+1}(\mathcal{N})\right).\label{eq:W_plus_Delta_asymptotic}
\end{equation}
\end{prop}

\end{cBoxB}

\begin{proof}
The operator $\mathcal{L}_{t}$ is the spherical mean operator. On
a spherical irreducible component with spectral parameter $\lambda$,
it acts by multiplication by the spherical function 
\[
\varphi_{\lambda}(t)=P_{-\frac{1}{2}+i\lambda}(\cosh t).
\]
On the regular spectral subspace, $|\lambda|\ge\eta$. We use the
standard uniform form of the Harish--Chandra expansion, namely 
\[
\varphi_{\lambda}(t)=e^{-t/2}\left(w_{+}(\lambda)e^{it\lambda}+w_{-}(\lambda)e^{-it\lambda}+e^{-2t}r_{t}(\lambda)\right),
\]
where $r_{t}(\lambda)$ is uniformly bounded for $t\ge t_{0}>0$ and
$|\lambda|\ge\eta$, together with the corresponding symbol estimates
in $\lambda$.

Applying this scalar identity by spectral calculus with 
\[
\lambda=\Lambda
\]
on $\mathcal{H}_{\mathrm{reg}}$, we obtain 
\[
\mathcal{L}_{t}(1-\Pi_{\mathrm{exc}})=e^{-t/2}\left(W_{+}e^{it\Lambda}+W_{-}e^{-it\Lambda}+e^{-2t}\mathcal{R}_{t}\right)(1-\Pi_{\mathrm{exc}}),
\]
where 
\[
\mathcal{R}_{t}:=r_{t}(\Lambda).
\]
The spectral theorem and the uniform symbol estimates imply 
\[
\sup_{t\ge t_{0}}\|\mathcal{R}_{t}\|_{H^{s}\to H^{s+1/2}}<\infty.
\]
Adding the finite-rank part 
\[
F_{t}=\mathcal{L}_{t}\Pi_{\mathrm{exc}}
\]
gives \eqref{eq:global_spherical_average_asymptotic}.

It remains to identify the leading pseudodifferential behavior of
$W_{+}$. From the quotient asymptotic of the Gamma function, 
\[
\frac{\Gamma(i\lambda)}{\Gamma\left(\frac{1}{2}+i\lambda\right)}=e^{-i\pi/4}\lambda^{-1/2}\left(1-\frac{i}{8\lambda}+O(\lambda^{-2})\right),\qquad\lambda\to+\infty.
\]
Therefore 
\[
w_{+}(\lambda)=\frac{1}{\sqrt{\pi}}e^{-i\pi/4}\lambda^{-1/2}+O(\lambda^{-3/2}).
\]
Since 
\[
\Lambda=\left(\Delta-\frac{1}{4}\right)^{1/2},
\]
the multiplier $\lambda^{-1/2}$ corresponds to 
\[
\left(\Delta-\frac{1}{4}\right)^{-1/4}.
\]
Thus 
\[
W_{+}=\frac{1}{\sqrt{\pi}}e^{-i\pi/4}\left(\Delta-\frac{1}{4}\right)^{-1/4}+R_{+},
\]
where $R_{+}$ is of order $-3/2$. Hence 
\[
R_{+}:H^{s}(\mathcal{N})\to H^{s+3/2}(\mathcal{N}).
\]
Replacing 
\[
\left(\Delta-\frac{1}{4}\right)^{-1/4}
\]
by 
\[
\Delta^{-1/4}
\]
changes only lower-order terms on the regular subspace, which gives
\eqref{eq:W_plus_Delta_asymptotic}. 
\end{proof}

\subsection{Emergence of the wave equation}

The previous formula shows that, after removing a uniformly finite-rank
contribution, the renormalized spherical mean is governed at large
time by the shifted wave propagator.

Define 
\begin{equation}
E_{t}:=e^{t/2}\left(\mathcal{L}_{t}-F_{t}\right)\eq{\ref{eq:def_finite_rank_exceptional_part}}e^{t/2}\mathcal{L}_{t}(1-\Pi_{\mathrm{exc}}).\label{eq:def_E_t_renormalized_spherical_mean}
\end{equation}

\begin{cBoxB}{}

\begin{prop}[Wave asymptotics of spherical means]
\label{prop:wave_asymptotics_spherical_means} For $t>0$, on the
regular spectral subspace, one has 
\begin{equation}
E_{t}=W_{+}e^{it\Lambda}+W_{+}^{\dagger}e^{-it\Lambda}+e^{-2t}Q_{t},\label{eq:E_t_wave_asymptotic}
\end{equation}
where, for every $s\in\mathbb{R}$, 
\[
\sup_{t\ge t_{0}}\|Q_{t}\|_{H^{s}(\mathcal{N})\to H^{s+1/2}(\mathcal{N})}<\infty.
\]
Moreover, 
\begin{equation}
W_{+}=\frac{1}{\sqrt{\pi}}e^{-i\pi/4}\left(\Delta-\frac{1}{4}\right)^{-1/4}+\Psi^{-3/2}.\label{eq:W_plus_principal_symbol}
\end{equation}
Consequently, $E_{t}$ satisfies the shifted wave equation modulo
an exponentially small remainder: 
\begin{equation}
\left(\frac{\partial^{2}}{\partial t^{2}}+\Delta-\frac{1}{4}\right)E_{t}=e^{-2t}P_{t},\label{eq:E_t_wave_equation_remainder}
\end{equation}
where 
\[
\sup_{t\ge t_{0}}\|P_{t}\|_{H^{s}(\mathcal{N})\to H^{s-1/2}(\mathcal{N})}<\infty.
\]
Equivalently, 
\begin{equation}
(1+\Delta)^{-1/4}\left(\frac{\partial^{2}}{\partial t^{2}}+\Delta-\frac{1}{4}\right)E_{t}=O_{L^{2}\to L^{2}}(e^{-2t}).\label{eq:E_t_wave_equation_L2_remainder}
\end{equation}
\end{prop}

\end{cBoxB}

\begin{proof}
By Proposition~\ref{prop:global_spherical_average_asymptotic}, 
\[
E_{t}=W_{+}e^{it\Lambda}+W_{-}e^{-it\Lambda}+e^{-2t}\mathcal{R}_{t}.
\]
On the regular principal spectrum, 
\[
W_{-}=W_{+}^{\dagger}.
\]
This gives \eqref{eq:E_t_wave_asymptotic}, with 
\[
Q_{t}=\mathcal{R}_{t}.
\]

The pseudodifferential expansion \eqref{eq:W_plus_principal_symbol}
is exactly the leading term identified in \eqref{eq:W_plus_principal}.

The two leading terms solve the shifted wave equation exactly. Indeed,
on the regular spectral subspace, 
\[
\Lambda^{2}=\Delta-\frac{1}{4},
\]
and therefore 
\[
\left(\partial_{t}^{2}+\Lambda^{2}\right)W_{+}e^{it\Lambda}=0,\qquad\left(\partial_{t}^{2}+\Lambda^{2}\right)W_{+}^{\dagger}e^{-it\Lambda}=0.
\]
Thus only the Harish--Chandra remainder contributes.

Using the full Harish--Chandra expansion of the remainder, one may
write 
\[
e^{-2t}Q_{t}=\sum_{m\ge1}e^{-2mt}\left(W_{2m,\Lambda}e^{it\Lambda}+W_{2m,-\Lambda}e^{-it\Lambda}\right).
\]
Applying $\partial_{t}^{2}+\Lambda^{2}$ term by term gives, for the
positive branch, 
\[
\left(\partial_{t}^{2}+\Lambda^{2}\right)\left(e^{-2mt}W_{2m,\Lambda}e^{it\Lambda}\right)=\left(4m^{2}-4im\Lambda\right)e^{-2mt}W_{2m,\Lambda}e^{it\Lambda},
\]
and similarly for the negative branch. Since $W_{2m,\Lambda}$ is
a spectral multiplier of order $-1/2$, the factor $\Lambda$ gives
an operator of order $1/2$. More precisely, by the explicit formula
\eqref{eq:def_W_2m_lambda} and Stirling estimates, the symbol seminorms
of 
\[
(4m^{2}\mp4im\Lambda)W_{2m,\pm\Lambda}
\]
grow at most polynomially in $m$. Therefore, for $t\ge t_{0}>0$,
the series 
\[
\sum_{m\ge1}e^{-2(m-1)t}(4m^{2}\mp4im\Lambda)W_{2m,\pm\Lambda}e^{\pm it\Lambda}
\]
is uniformly bounded as an operator 
\[
H^{s}(\mathcal{N})\longrightarrow H^{s-\frac{1}{2}}(\mathcal{N}).
\]
Factoring out the first exponential $e^{-2t}$, we obtain 
\[
\left(\partial_{t}^{2}+\Delta-\frac{1}{4}\right)E_{t}=e^{-2t}P_{t},
\]
with 
\[
\sup_{t\ge t_{0}}\|P_{t}\|_{H^{s}(\mathcal{N})\to H^{s-\frac{1}{2}}(\mathcal{N})}<\infty.
\]
 This proves \eqref{eq:E_t_wave_equation_remainder}. Composing with
$(1+\Delta)^{-1/4}$ gives \eqref{eq:E_t_wave_equation_L2_remainder}. 
\end{proof}
\begin{rem}
The wave equation obtained here is a large-time statement. The finite-rank
contribution $F_{t}\eq{\ref{eq:def_finite_rank_exceptional_part}}\mathcal{L}_{t}\Pi_{\mathrm{exc}}$,
coming from the trivial representation, the complementary series,
the possible threshold eigenspace, and a fixed small window near the
threshold, has been removed. On the remaining principal spectrum,
the renormalized spherical mean 
\[
E_{t}=e^{t/2}\left(\mathcal{L}_{t}-F_{t}\right)
\]
is governed, up to exponentially small corrections, by the shifted
wave operator 
\[
\partial_{t}^{2}+\Delta-\frac{1}{4}.
\]
\end{rem}

\bibliographystyle{plain}
\bibliography{/home/faure/articles/articles}

\end{document}